%% file: TorusAlg.tex
\documentclass[11pt]{amsart}
\usepackage{epsfig}
\usepackage{amscd}
\usepackage{etex} 
\usepackage{amssymb}
\usepackage[noadjust]{cite}
\usepackage{booktabs}
\usepackage{xypic}
\usepackage{url}
\usepackage{hyphenat}
\usepackage{mathtools}
\usepackage{mathrsfs}
\usepackage{cancel} 
\usepackage[T1]{fontenc}
\usepackage[utf8]{inputenc}
\usepackage{mdwtab}
\usepackage{enumitem}
\usepackage{bbding}
\usepackage{dsfont}
\usepackage[percent]{overpic}
\usepackage{xr-hyper}
\usepackage{soul} 
\usepackage{MnSymbol} 
\DeclareMathSymbol{\pitchfork}    {\mathrel}{AMSa}{"74} 
\usepackage{color}
\definecolor{darkblue}{rgb}{0,0,0.4}
  \usepackage[pdftex,lmargin=1in, rmargin=1in, tmargin=1in, bmargin=1in]{geometry}
  \usepackage[bookmarks=true, bookmarksopen=true,%
    bookmarksdepth=3,bookmarksopenlevel=2,%
    colorlinks=true,%
    ]{hyperref}
  \hypersetup{pdftitle={A bordered HF- algebra for the torus}}
  \hypersetup{pdfauthor={Robert Lipshitz, Peter S. Ozsváth and Dylan
      P. Thurston}}
\usepackage{color}
\usepackage{tikz}
\usetikzlibrary{matrix,arrows,decorations.markings,calc,intersections}
\tikzset{cdlabel/.style={above,sloped,
   execute at begin node=$\scriptstyle,execute at end node=$}}
\tikzset{alga/.style={->, thick}}
\tikzset{taa/.style={->, double}}

\input{defs}

\input{macros}
\begin{document}
\title{A bordered {$\HFm$} algebra for the torus}

\author[Lipshitz]{Robert Lipshitz}
\thanks{\texttt{RL\ \ was supported by NSF grant DMS-1810893.}}
\address{Department of Mathematics, University of Oregon\\
  Eugene, OR 97403}
\email{lipshitz@uoregon.edu}

\author[Ozsv\'ath]{Peter Ozsv\'ath}
\thanks{\texttt{PSO was supported by NSF grants DMS-1708284  and DMS-2104536.}}
\address {Department of Mathematics, Princeton University\\ New
  Jersey, 08544}
\email {petero@math.princeton.edu}

\author[Thurston]{Dylan~P.~Thurston}
\thanks{\texttt{DPT was supported by NSF grant DMS-2110143.}}
\address{Department of Mathematics\\
         Indiana University,
         Bloomington, Indiana 47405\\
         USA}
\email{dpthurst@indiana.edu}

\begin{abstract}
  We describe a weighted $\Ainf$-algebra associated to the torus.  We
  give a combinatorial construction of this algebra, and an abstract
  characterization. The abstract characterization also gives a
  relationship between our algebra and the wrapped Fukaya category of
  the torus.  These algebras underpin the (unspecialized) bordered
  Heegaard Floer homology for three-manifolds with torus boundary,
  which will be constructed in forthcoming work.
\end{abstract} 


\maketitle 

\tableofcontents


\input{intro}
\input{weightedAlgebras}
\input{graphs}
\input{gradings}
\input{unique}
\input{Fukaya}
\input{signs}

\bibliographystyle{hamsalpha}\bibliography{heegaardfloer}
\end{document}

%% file: defs.tex
\newcommand{\RR}{\mathbb R}

\newcommand{\ZZ}{\mathbb Z}

\newcommand{\FF}{\mathbb F}

\newcommand{\Image}{\mathrm{Im}}
\newcommand{\bD}{\mathbb{D}}

\newcommand{\connectsum}{\mathbin{\#}}

\newcommand{\co}{\nobreak\mskip2mu\mathpunct{}\nonscript
  \mkern-\thinmuskip{:}\penalty300\mskip6muplus1mu\relax}

\newcommand{\abs}[1]{{\lvert #1 \rvert}}
\newcommand{\OneHalf}{{\textstyle\frac{1}{2}}}

\newcommand{\onto}{\twoheadrightarrow}

\newcommand{\bdy}{\partial}

\newcommand{\lbracket}{[}
\newcommand{\rbracket}{]}

\newcommand{\hotimes}{\mathbin{\widehat{\otimes}}}

\DeclareMathOperator{\End}{End}

\DeclareMathOperator{\Hom}{Hom}

\DeclareMathOperator{\spin}{spin}
\newcommand{\SpinC}{\spin^c}
\DeclareMathOperator{\Pin}{pin}

\DeclareMathOperator{\gr}{gr}
\newcommand{\grt}{\widetilde{\gr}}

\DeclareMathOperator{\wingr}{wn}

\newcommand{\Barop}{{\mathrm{Bar}}}
\DeclareMathOperator{\Cobarop}{Cob}

\theoremstyle{plain}

\numberwithin{equation}{section}
\newtheorem{theorem}[equation]{Theorem}

\newtheorem{proposition}[equation]{Proposition}
\newtheorem{lemma}[equation]{Lemma}
\newtheorem{corollary}[equation]{Corollary}

\newtheorem{convention}[equation]{Convention}
\newtheorem{definition}[equation]{Definition}

\theoremstyle{definition}

\theoremstyle{remark}
\newtheorem{example}[equation]{Example}
\newtheorem{remark}[equation]{Remark}

\newcommand{\HF}{\mathit{HF}}
\newcommand{\HFa}{\widehat {\HF}}

\newcommand{\HFm}{{\HF}^-}

\newcommand{\CF}{{\mathit{CF}}}

\newcommand{\x}{\mathbf x}
\newcommand{\y}{\mathbf y}

\newcommand{\w}{\mathbf w}

\newcommand\HH{\mathit{HH}}
\newcommand\HC{\mathit{HC}}
\newcommand\Hochschild\HH

\newcommand{\Ainf}{A_\infty}
\newcommand{\Alg}{\mathcal{A}}

\newcommand\Blg{\mathcal{B}}

\newcommand{\DA}{\textit{DA}}

\newcommand{\cZ}{\mathcal{Z}}
\newcommand{\PtdMatchCirc}{\cZ}
\newcommand{\PMC}{\PtdMatchCirc}

\newcommand{\CircPts}{{\mathbf{a}}}

\newcommand\Id{\mathbb{I}}
\newcommand\Ground{\mathds{k}}

\newcommand\Tensor{\mathcal T}
\newcommand\Zmod[1]{\mathbb{Z}/{#1}\mathbb{Z}}
\newcommand{\Field}{{\FF_2}}
\newcommand{\FieldSup}[1]{(\FF_2)^{#1}}

\renewcommand{\th}{^\text{th}}

\newcommand{\bigGroup}{G'}
\newcommand{\smallGroup}{G}

\newcommand{\grb}{\gr'}

\newcommand{\Torus}{\mathbb{T}}

\newcommand{\Cat}{\mathscr{C}}

\DeclareMathOperator{\Mor}{Mor}

\DeclareMathOperator{\ob}{Ob}

\newcommand{\op}{\mathrm{op}}

\makeatletter
\newcommand\honestalg[3]{\bigl\lbracket
\begin{smallmatrix} #1\@ifempty{#3}{}{&#3} \\ #2 \end{smallmatrix}
\bigr\rbracket}

\makeatother

\newcommand{\lsup}[2]{{}^{#1}\mskip-.6\thinmuskip#2}
\newcommand{\lsupv}[2]{{}^{#1}#2}

\newcommand{\SmallCobar}{\Alg'}
\newcommand{\SmallCobarZ}{\SmallCobar_{\ZZ}}
 
\newcommand{\AsUnDefAlg}{\Alg_-^{0,\mathrm{as}}} 
\newcommand{\UnDefAlg}{\Alg_-^{0}}

\newcommand{\MAlg}{\Alg_-} 
\newcommand{\MAlgA}{\Alg^A_-}

\newcommand{\wild}{\mathalpha{*}}

\newcommand{\corolla}[1]{\Psi_{#1}}
\newcommand{\wcorolla}[2]{\corolla{#1}^{#2}}

\newcommand{\kotimes}[1]{\otimes}

\newcommand{\aAlg}{\widehat{\Alg}}

\newcommand{\obstr}{\mathfrak{O}}

\newcommand{\fobstr}{\mathfrak{F}}
\newcommand{\Filt}{\mathcal{F}}

\newcommand{\WFuk}{\mathcal{W}}

\newcommand{\LagGr}[2][{}]{\mathcal{L}\!\mathit{ag}_{#1}(#2)}
\newcommand{\LagGras}{\mathcal{L}\!ag}
\newcommand{\GrSet}{S}
\newcommand{\GrGroupoid}{\mathcal{G}}

\newcommand\Graphs{\mathfrak S}
\newcommand\Region{\mathcal D}

\newcommand{\CDisk}{\mathbb{D}}

\newcommand{\gpath}[3]{\gamma_{#1}^{#2\to#3}}%
\newcommand{\hpath}[4]{\eta_{#1}^{\phi^{#3}(#2)\to\phi^{#4}(#2)}}%
\newcommand{\epath}[3]{\epsilon_{#1\to#2}^{#3}}

\newcommand\AsUnDefAlgZ{\Alg_{-;\ZZ}^{0,\mathrm{as}}} 
\newcommand\UnDefAlgZ{\Alg_{-;\ZZ}^{0}} 
\newcommand\MAlgZ{\Alg_{-;\ZZ}}
\newcommand\Obj{\mathrm{Obj}}



%% file: macros.tex
\newread\testin

\graphicspath{{draws/}{mpdraws/}{}}
\makeatletter
\def\input@path{{}{draws/}}
\makeatother

\def\mathcenter#1{%
  \vcenter{\hbox{$#1$}}%
}

\makeatletter
\newcommand\mi@kern[1]{%
  \settowidth\@tempdima{$\mi@obj^{#1}$}
  \kern-\@tempdima
  #1
  \settowidth\@tempdima{$\mi@obj$}
  \kern\@tempdima
}

\newtoks\mi@toksp
\newtoks\mi@toksb
\DeclareRobustCommand{\manyindices}[5]{
  \def\mi@obj{#5}
  \mi@toksp\expandafter{\mi@kern{#2}}
  \mi@toksb\expandafter{\mi@kern{#1}}
  \@mathmeasure4\textstyle{#5_{#1}^{#2}}
  \@mathmeasure6\textstyle{#5_{#3}^{#4}}
  \dimen0-\wd6 \advance\dimen0\wd4
  \@mathmeasure8\textstyle{\hphantom{{}_{#1}^{#2}}#5^{\the\mi@toksp#4}_{\the\mi@toksb#3}}
  \hbox to \dimen0{}{\kern-\dimen0\box8}
}
\makeatother 

\usepackage{ifpdf}
\ifpdf
  
\else
  
\fi


%% file: intro.tex
\section{Introduction}

This is a tale of an algebra. In fact, it is a tale of several algebras:
\begin{itemize}
\item $\aAlg$, the algebra introduced in~\cite{LOT1} governing $\HFa$
  for 3-manifolds with torus boundary. (In~\cite{LOT1}, $\aAlg$ is
  denoted $\Alg(T^2)$.) The algebra $\aAlg$ is a finite-dimensional,
  associative algebra over~$\Field$ with a grading by a particular group
  $G$ with a distinguished central element $\lambda$.
\item $\AsUnDefAlg$, an associative algebra over $\Field$ containing
  $\aAlg$ as a subalgebra. The algebra $\AsUnDefAlg$ is also graded by
  $G$. The algebra $\AsUnDefAlg$ is
  infinite-dimensional, but is finitely generated and is
  finite-dimensional in each grading. The definition of $\AsUnDefAlg$
  is an unsurprising extension of the definition of $\aAlg$.
\item $\UnDefAlg$, an $\Ainf$-algebra over $\Field[U]$, with trivial
  differential ($m_1=0$). The algebra $\UnDefAlg$ is free over
  $\Field[U]$, and the $\Ainf$-operations are $U$-equivariant. Since
  $m_1=0$, $\UnDefAlg$ has an underlying associative algebra; this
  algebra is $\AsUnDefAlg[U]\coloneqq \AsUnDefAlg\otimes\Field[U]$.
\item $\MAlg=\MAlgA$, the main algebra of interest for studying $\HF^-$
  for $3$-manifolds with torus boundary. This is what we call a
  \emph{weighted $\Ainf$-algebra} (see
  Section~\ref{sec:weighted-alg} and \cite{LOT:abstract}). The weight-zero part of $\MAlgA$ is
  $\UnDefAlg$.
\end{itemize}

Although weighted $\Ainf$-algebras are discussed in detail in
Section~\ref{sec:weighted-alg}, perhaps a few words are in order
now. In brief, a weighted $\Ainf$-algebra is a curved $\Ainf$-algebra
$A[[t]]=A\otimes_\Field\Field[[t]]$ over $\Field[[t]]$, such that
 the curvature lies in the ideal $tA[[t]]$.
The operations $\mu\co
T^*A[[t]]\to A[[t]]$ are determined by the maps
$\mu_{n}^k\co A^{\otimes n}\to
t^kA$, which we call the \emph{weight $k$ part} of $\mu_n$. 
Weighted
$\Ainf$-algebras are also called one-parameter deformations of
$\Ainf$-algebras (see, e.g.,~\cite[Section 3b]{Seidel15:quartic}).

While this paper is self contained, we explain briefly how the algebra fits into the broader
bordered context from~\cite{LOT1,LOT:torus-mod}.  The algebra $\aAlg$ is an algebra of Reeb
chords that are not permitted to cross a certain basepoint adjacent to
the boundary; multiplication is encoded by collisions of curves at
``east infinity''. The algebra $\AsUnDefAlg$ is a larger algebra,
consisting of all Reeb chords. The $\Ainf$ deformation $\UnDefAlg$
is needed to account for boundary deformations, i.e. holomorphic
curves whose boundary lie entirely in the $\alpha$-tori. Indeed, the
deformation $\UnDefAlg$ is the one which counts the disk that covers
the torus once.
(See~\cite[Figure~\ref{TM:fig:simple-bdy-degen-eg}]{LOT:torus-mod}.)  The
weighted deformation $\MAlg$ is the algebraic object which also encodes the
Reeb orbits.
(See~\cite[Figure~\ref{TM:fig:orbit-curve-eg}]{LOT:torus-mod}.)

The main goal of this paper is to define $\MAlg$. In fact, we give two
paths to defining $\MAlg$. The first construction is based on the
combinatorics of certain kinds of planar graphs. We give this construction 
in Section~\ref{sec:def}.
Interpreting the planar graphs as coverings of the torus, one can think of this definition as an
almost trivial case of the theory of pseduoholomorphic curves; see Section~\ref{subsec:Immersions}.

The second path to defining $\MAlg$
is more indirect. As mentioned above, it is easy to
define the algebra $\AsUnDefAlg$. The algebra $\AsUnDefAlg$ turns out
to have a unique nontrivial $\Ainf$-deformation over $\Field[U]$
respecting the gradings, up to $\Ainf$-isomorphism; see Theorem~\ref{thm:UnDefAlg-unique}. This deformation
is $\UnDefAlg$. (Theorem~\ref{thm:UnDefAlg-unique} follows from a computation of
the Hochschild cohomology of $\AsUnDefAlg$.) Similarly, $\UnDefAlg$
turns out to have a unique extension to a weighted $\Ainf$-algebra
subject to the conditions that the curvature is a particular element
(the sum of the ``length 4 chords'') and respecting particular
gradings; see Theorem~\ref{thm:MAlg-unique}. (Again, Theorem~\ref{thm:MAlg-unique}
follows from computing the Hochschild cohomology of $\UnDefAlg$.) This
extension is $\MAlg$. In particular, the uniqueness theorem implies
that the two definitions of $\MAlg$ agree.

Auroux showed that $\aAlg$
is derived equivalent to a certain partially-wrapped version of the
Fukaya category of the torus~\cite{Auroux10:ICM,Auroux10:Bordered}. We show that $\UnDefAlg$ is equivalent
to the fully wrapped Fukaya category of the torus.

We note that various versions of the Fukaya categories of Riemann surfaces have
been studied extensively. For instance, Abouzaid explicitly computed $K_0$ of (a
particular variant of the) Fukaya category of a
surface~\cite{Abouzaid08:FukSurf} (see also~\cite{AB22:Fukaya}). In a previous
paper~\cite{LOT:faith}, we gave an explicit description of the mapping class
group action on a partially-wrapped Fukaya category of a surface, and showed
this action was faithful and captures the dilatation, a point rediscovered by
Dimitrov-Haiden-Katzarkov-Kontsevich~\cite{DHKK14:dynamic-cat}. Lekili-Perutz
studied a different variant of the Fukaya category of the torus, giving an
explicit description of it and showing it is not
formal~\cite{LekiliPerutz11:torus}. A reformulation of bordered $\HFa$ with
torus boundary in terms of (a version of) the Fukaya category of the torus was
given by Hanselman-Rasmussen-Watson~\cite{HRW}.  There have been many papers
about mirror symmetry for Riemann surfaces
(e.g.,~\cite{Seidel11:genus2,Efimov12:higher-g,LekPol17:arithmetic,PS19:Fuk-surf-mirror,AS:Fuk-surf},
and many others). There has also been substantial interest in the Fukaya
category from the representation theory community
(e.g.,~\cite{KhS02:BraidGpAction,CS20:Hall}).

This paper is organized as follows. In Section~\ref{sec:weighted-alg}
we introduce weighted $\Ainf$-algebras, the algebraic structure
underlying $\MAlg$. In Section~\ref{sec:def} we define $\AsUnDefAlg$,
$\UnDefAlg$ and $\MAlg$; the gradings on these algebras are deferred
to Section~\ref{sec:grading}. Section~\ref{sec:unique} uses
computations of Hochschild homology groups to prove uniqueness
theorems for $\UnDefAlg$ and $\MAlg$.  Section~\ref{sec:Fukaya}
gives a Fukaya-categorical interpretation of $\UnDefAlg$.  In
Section~\ref{sec:Signs}, we explain how to lift the constructions from
this paper to arbitrary characteristic.

\subsection*{Acknowledgements}
We thank Rumen Zarev for helpful conversations. We also thank
Mohammed Abouzaid, Tim Perutz, and Nick Sheridan, for helpful
conversations, and Tom Hockenhull for corrections to an earlier
version. We also thank the referee for many helpful comments and
corrections.


%% file: weightedAlgebras.tex
\section{Weighted \texorpdfstring{$\Ainf$}{A-infinity} algebras and modules}
\label{sec:weighted-alg}
\subsection{Definition and first properties}\label{sec:walg-first-props}
The $\HF^-$ extension of bordered Floer homology associates to the torus $T^2$
an object $\MAlg=\MAlg(T^2)$, having a particular algebraic structure,
which we will call a \emph{weighted $\Ainf$-algebra}. (Such objects
have appeared in the literature already in several forms; see
Remarks~\ref{remark:weighted-is-deformation}, 
\ref{remark:gap-in-the-literature}
and~\ref{remark:left-leaning}.) A more detailed discussion can be
found in our previous paper~\cite[Section 4]{LOT:abstract}.

For now, we will discuss ungraded weighted $\Ainf$-algebras, and defer
the discussion of gradings to Section~\ref{sec:W-gradings}.

Fix a unital commutative ring $R$ of characteristic $2$.

\begin{definition}
  A curved \emph{$\Ainf$-algebra over $R$} consists of $R$-bimodule maps
  \[
    \mu_n\co A^{\otimes_R n}\to A
  \]
  for $n\geq 0$ satisfying the structure equation
  \[
    \sum_{k=-1}^{n-1}\sum_{i=1}^{n-k}\mu_{n-k}(a_1\otimes\dots\otimes a_{i-1}\otimes\mu_{k+1}(a_i\otimes\dots\otimes a_{i+k})\otimes a_{i+k+1}\otimes\dots\otimes a_n)=0.
  \]
  for all $n$ and all $a_1,\dots,a_n\in A$. (In the graded setting,
  the map $\mu_n$ decreases the homological grading by $2-n$.)

  In the special case
  $n=0$, $A^{\otimes_R 0}=R$, so $\mu_0\co R\to A$, so $\mu_0$ is
  determined by $\mu_0(1)\eqqcolon\mu_0$, which is the
  \emph{curvature} of the curved $\Ainf$-algebra.
  
  A curved $\Ainf$-algebra $\Alg=(A,\{\mu_n\})$ over $R$ is
  \emph{strictly unital} if there is an element $1\in A$ so that for
  all $a\in A$, $\mu_2(a,1)=\mu_2(1,a)=a$ and for all $n\neq 2$, and
  $a_1,\dots,a_n\in A$, if some $a_i=1$ then
  $\mu_n(a_1,\dots,a_n)=0$.
\end{definition}

\begin{definition}\label{def:weighted-alg}
  A \emph{weighted $\Ainf$-algebra} over $R$ is a curved $\Ainf$-algebra
  $(A[[t]],\{\mu_n\co A[[t]]^{\otimes n}\to A[[t]]\}_{n=0}^\infty)$ over
  $R[[t]]$ (where $A[[t]]=A\otimes_RR[[t]]$ for some $R$-module $A$)
  such that the curvature $\mu_0$ lies in $tA[[t]]\subset A[[t]]$.
\end{definition}

\begin{convention}
  Henceforth, all curved or weighted $\Ainf$-algebras will be assumed
  strictly unital. We will identify elements $r\in R$ with their images $r\cdot 1\in A$.
\end{convention}

\begin{lemma}
  A weighted $\Ainf$-algebra over $R$ is specified by an $R$-bimodule
  $A$ and maps $\mu_n^w\co A^{\otimes_R n}\to A$ for $n,w\in\ZZ_{\geq
    0}$ such that:
  \begin{enumerate}
  \item $\mu_0^0=0$.
  \item For any $r\in R$, $\mu_2^0(a,r)=ar=ra=\mu_2^0(r,a)$.
  \item For any $(n,w)\neq (2,0)$, $\mu_n^w(a_1,\dots,a_n)=0$ if some $a_i\in R$.
  \item For each $n,w\in\ZZ_{\geq 0}$ and $a_1,\dots,a_n\in A$,
    \begin{equation}
      \label{eq:weighted-Ainf-relation}
      \sum_{\substack{p+q=n+1\\ u+v=w}}\sum_{i=0}^{n-q}
      \mu_p^u(a_1\otimes\cdots\otimes a_i\otimes
      \mu_q^v(a_{i+1}\otimes\cdots\otimes a_{i+q})\otimes
      a_{i+q+1}\otimes\cdots\otimes a_n)=0.
    \end{equation}
  \end{enumerate}
\end{lemma}
\begin{proof}
  Immediate from the definitions.
\end{proof}

\begin{definition}
  Given a weighted $\Ainf$-algebra $\Alg=(A,\{\mu_n^k\})$, the
  operations $\{\mu_n^0\}$ make $A$ into an ordinary
  $\Ainf$-algebra. We call $(A,\{\mu_n^0\})$ the \emph{undeformed
    $\Ainf$-algebra} of $\Alg$.
\end{definition}

\begin{definition} 
  \label{def:HomomorphismWeightedAinf}
  Given weighted $\Ainf$-algebras $\Alg$ and $\Blg$, a
  \emph{homomorphism} $f\co \Alg\to\Blg$ is a homomorphism of curved
  $\Ainf$-algebras (i.e., a sequence of maps $f_n\co A[[t]]^{\otimes n}\to
  B[[t]]$, $n=0,\dots,\infty$, over $R[[t]]$, satisfying the $\Ainf$-algebra homomorphism
  relations) so that $f_0\in tB[[t]]\subset B[[t]]$. A homomorphism is called
  \emph{uncurved} if $f_0=0$ (which implies that
  $f_1(\mu_0^\Alg)=\mu_0^\Blg$).

  The identity homomorphism of $\Alg$ is given by $f_1=\Id$ and
  $f_n=0$ for $n\neq 1$. The composition of homomorphisms of weighted
  $\Ainf$-algebras is induced by the usual composition of homomorphisms of
  curved $\Ainf$-algebras, i.e.,
  \[
    (g\circ f)_n=\sum_{\substack{k_1+\cdots+k_m=n\\k_i\geq 0}}
    g_m\circ(f_{k_1}\otimes\cdots\otimes f_{k_m}).
  \]
  Convergence of this sum follows from the fact that $f_0\in tB[[t]]$.
  Also, since
  \[
    (g\circ f)_0 = g_0 + g_1(f_0)+g_2(f_0,f_0)+\cdots,
  \]
  a composition of uncurved homomorphisms is uncurved.

  An \emph{isomorphism} of weighted $\Ainf$-algebras is an invertible homomorphism.
\end{definition}

Given a weighted $\Ainf$-homomorphism $f\co\Alg\to\Blg$ we let
$f_n^k\co A^{\otimes n}\to B$ be the coefficient of $t^k$ in
$f_n$ (restricted to $A^{\otimes n}\subset A[[t]]^{\otimes n}$). So,
for example, the identity homomorphism is given by
$\Id_1^0(a)=a$ and $\Id_{n}^k=0$ for $(n,k)\neq(1,0)$. 

\begin{lemma}\label{lem:bij-hom-is-iso}\cite[Lemma 4.19]{LOT:abstract}
  If $f\co\Alg\to\Blg$ is a homomorphism of weighted $\Ainf$-algebras
  so that $f_1^0$ is an isomorphism of $R$-modules then $f$ is an
  isomorphism.
\end{lemma}

\begin{definition}
  An \emph{augmentation} of a weighted $\Ainf$-algebra $\Alg$ over $R$ is
  a $\Ainf$-homomorphism $f\co A[[t]]\to R[[t]]$, where $R[[t]]$ has $\mu_n=0$ for
  $n\neq 2$ and $\mu_2$ the usual multiplication.
\end{definition}
Notice, in particular, that any augmentation $f\co A[[t]]\to R[[t]]$ sends the curvature
$\mu_0$ to $0$, i.e., $f_1(\mu_0)=0$.

With these properties established, we turn to some alternate
formulations of weighted $\Ainf$-algebras.

\begin{remark}\label{remark:weighted-is-deformation}
  A weighted $\Ainf$-algebra $\Alg=(A,\{\mu_n^k\})$ is the same as a
  one-parameter deformation of $(A,\{\mu_n^0\})$, in the sense of
  (for instance)~\cite[Section 3b]{Seidel15:quartic}, except that we
  allow curved deformations. See also Section~\ref{sec:weighted-deformation}.
\end{remark}

\begin{remark}\label{remark:gap-in-the-literature}
  A weighted $\Ainf$-algebra is a special case of a gapped, filtered $\Ainf$-algebra~\cite{FOOO1}.
\end{remark}

\begin{remark}\label{remark:left-leaning}
  Let $\Alg=(A,\{\mu_n\})$ be an uncurved $\Ainf$-algebra so that
  $(A,\mu_1,\mu_2)$ is an associative differential algebra, and let
  $\omega\in\Alg$ be an element such that:
  \begin{itemize}
  \item $\omega$ is central with respect to $\mu_2$.
  \item $A=B\oplus \omega B$ is a free $\Field[\omega]/(\omega^2)$-module.
  \item If $a_1,\dots,a_n\in B\cup\{\omega\}$ and $n\geq 2$ then $\mu_n(a_1,\dots,a_n)\in B$.
  \end{itemize}
  Then there is an associated weighted $\Ainf$-algebra
  $(B,\{\mu_n^k\})$ where
  \[
  \mu_n^k(a_1,\dots,a_n)=\mu_{n+k}(\overbrace{\omega,\dots,\omega}^k,a_1,\dots,a_n)+\mu_{n+k}(\overbrace{\omega,\dots,\omega}^{k-1},a_1,\omega,a_2,\dots,a_n)+\cdots
  \]
  is the sum of all ways of inserting $k$ copies of $\omega$ into the
  sequence $(a_1,\dots,a_n)$ and then applying the operation
  $\mu_{n+k}$. In particular, $\mu_0^k=\mu_k(\omega,\dots,\omega)$.
\end{remark}

Let $T$ be a rooted, planar tree with $n$ distinguished \emph{input}
leaves, and one distinguished output, the root.
We call the vertices of $T$ other than
the inputs and the output \emph{internal}. (Note that this definition allows for some leaves
that are neither inputs nor the output.)
A \emph{weight function} on $T$ is
a function $w$ from the internal vertices of $T$ to the non-negative
integers, with the property that any internal vertex $v$ with valence $1$
or $2$ has weight $w(v)\geq 1$. A \emph{weighted tree} is a rooted, planar
tree $T$, a choice of input leaves for $T$, and a weight function on
the internal vertices of $T$. The \emph{dimension} of a weighted tree $T$ with $n$ inputs,
$i$ internal vertices, and total weight $w$ is
\begin{equation}\label{eq:tree-dim}
  \dim(T)=n+2w-i-1.
\end{equation}
  
A weighted $\Ainf$ algebra structure on $A$ associates a map
\[
  \mu(T)\co A^{\otimes n}\to A
\]
to a weighted tree $T$ with $n$ inputs.  Specifically, the operation
$\mu^w_n$ specifies the action of the $n$ input corolla (planar graph
with one internal vertex) with weight $w$, $\wcorolla{w}{n}$; actions
of more complicated trees are obtained from these actions by suitable
compositions. In particular, the valence-1 internal vertices
correspond to operations of the form $\mu_0^w$.

If $S$ and $T$ are two rooted, planar, weighted trees, and $i$ is some
input to $S$, let $S\circ_i T$ denote the rooted, planar, weighted tree
obtained by joining the output to $T$ to the $i^{\th}$ input to $S$.
The weighted $\Ainf$ relation can now be phrased as the identities
(indexed by pairs of non-negative integers $w$ and $n$)
\[
  \sum_{\substack{a+b=n+1\\u+v=w}}\sum_{i=1}^a\mu({\wcorolla{a}{u}\circ_i\wcorolla{b}{v}})=0.
\]

\subsection{Gradings}\label{sec:W-gradings}
Fix a set $S$. By an \emph{$S$-graded $R$-module} we mean an
$R$-module $M$ together with a decomposition $M=\bigoplus_{s\in
  S}M_s$. Elements of $M_s$ are called \emph{homogeneous of grading
  $s$.} (The element $0\in M$ is viewed as being homogeneous of all
gradings.)

Now, fix a group $G$ and central elements $\lambda_d$ and $\lambda_w$
in $G$. By a \emph{$(G,\lambda_d,\lambda_w)$-graded weighted
  $\Ainf$-algebra} we mean a weighted $\Ainf$-algebra
$\Alg=(A,\{\mu_m^w\})$ and a $G$-grading of $A$ so that $\mu_m^w$ is
graded of degree $\lambda_d^{m-2}\lambda_w^w$. More explicitly, if
$a_1,\dots,a_m$ are homogeneous of gradings $g_1,\dots,g_m$ then
$\mu_m^w(a_1,\dots,a_m)$ is homogeneous of grading
$\lambda_d^{m-2}\lambda_w^wg_1\cdots g_m$. Equivalently, viewing
$\Alg$ as a curved $\Ainf$-algebra over $\Field[[t]]$, $t$ has grading
$\lambda_w^{-1}$ and $\mu_m$ has degree $\lambda_d^{m-2}$.

\begin{remark}
  In the notation of our previous paper~\cite{LOT:abstract}, where we considered only
  integer-valued gradings, $\lambda_d$ was $-1$ and $\lambda_w$ was
  denoted $\kappa$.
\end{remark}


%% file: graphs.tex
\section{Definitions of the algebras via planar graphs}\label{sec:def}%
\subsection{The \texorpdfstring{associative algebra $\AsUnDefAlg$}{underlying associative algebra}}\label{sec:AsUnDefAlg}%
The associative algebra $\AsUnDefAlg$ is the path algebra with
relations
\begin{equation}\label{eq:path-alg} \mathcenter{ \begin{tikzpicture}
  \node at (0,0) (iota0) {$\iota_0$};
  \node at (4,0) (iota1) {$\iota_1$};
  \draw[->, bend right=15] (iota0) to node[above]{$\rho_3$} (iota1);
  \draw[->, bend left=45] (iota0) to node[above]{$\rho_1$} (iota1);
  \draw[->, bend right=15] (iota1) to node[above]{$\rho_2$} (iota0);
  \draw[->, bend left=45] (iota1) to node[above]{$\rho_4$} (iota0);
\end{tikzpicture}}\big/\bigl(\rho_2\rho_1=\rho_3\rho_2=\rho_4\rho_3=\rho_1\rho_4=0\bigr)
\end{equation}
over the ground ring $\Field$.

The idempotents
$\iota_0$, $\iota_1$ and elements
\[
  \rho_{i,i+1,\dots,i+n}=\rho_i\rho_{i+1}\cdots\rho_{i+n},
\]
$i\in\ZZ/4\ZZ$, $n\in\ZZ_{>0}$ form a $\Field$-basis for $\AsUnDefAlg$. We call these
elements \emph{basic elements} of the algebra; we call the elements of
the form $\rho=\rho_i\cdots\rho_{i+n}$ \emph{Reeb elements}.
to distinguish them from 
$\iota_0$ and $\iota_1$. The \emph{length} of
$\rho_{i,\dots,i+n}$ is $n+1$ and is denoted $|\rho_{i,\dots,i+n}|$;
the idempotents are defined to have length $0$:
\[
  |\iota_0|=|\iota_1|=0.
\]

The idempotents $\iota_0$ and $\iota_1$ generate a $2$-dimensional
subalgebra $\Ground\cong \Field\oplus\Field\subset \AsUnDefAlg$.

\begin{figure}
  \centering
  \includegraphics{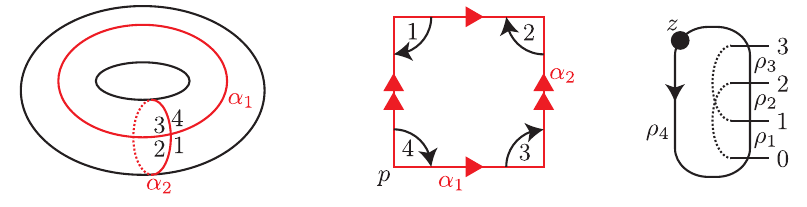}
  \caption{\textbf{The torus.} Left: the torus $T^2$ with $\alpha_1$
    and $\alpha_2$ drawn on it and the corners $1,\dots,4$
    labeled. Center: the same torus, as an identification space of the
    rectangle $S$. Right: the pointed matched circle for the torus and
    labeling of the points, matching, and indecomposable chords.}
  \label{fig:torus}
\end{figure}

To put this algebra into a wider context, recall from~\cite{LOT1} the
\emph{pointed matched circle} $\PMC=(Z,\CircPts,M,z)$ for $T^2$;
see Figure~\ref{fig:torus}. Then each $\rho_i$ is a chord in
$(Z,\CircPts)$ (as indicated in the figure), and the
$\rho_{i,\dots,i+n}$ are chords of length $n+1$. In particular, each
chord $\rho_{i,\dots,i+n}$ has a \emph{support} in
$H_1(Z,\CircPts)$. Explicitly, identify $Z=[0,4]/(0\sim4)$ and
$\CircPts=\{0,1,2,3\}$, so that $\rho_1$ corresponds to the interval
$[1,2]$, $\rho_2$ to the interval $[2,3]$, and so on. Declare that the
support of $\rho_i$ is $[i,i+1]$ and that the support of a product
$\rho\rho'$ is the sum of the support of $\rho$ and the support of
$\rho'$.
We will write
the support of $\rho$ as
$[\rho]=(n_1(\rho),n_2(\rho),n_3(\rho),n_4(\rho))$. So, for
instance, $[\rho_1]=(1,0,0,0)$, $[\rho_2]=(0,1,0,0)$, and so on. We
have
\[
  |\rho|=n_1(\rho)+n_2(\rho)+n_3(\rho)+n_4(\rho)
\]

\begin{remark}
  Let $I$ be the 2-sided ideal of $\AsUnDefAlg$ generated by $\rho_4$.
  Then $\AsUnDefAlg/I$ is the algebra $\Alg(T^2)$ associated to the
  torus in~\cite{LOT1}.
\end{remark}

We will extend the length grading on  $\Field[U]\otimes\AsUnDefAlg$,
by declaring $|U|=4$.

\subsection{The definition of \texorpdfstring{$\MAlg$}{weighted A-infinity algebra}}
\label{sec:ConstructMAlg}%

We will define a weighted $\Ainf$-algebra structure $\MAlg$ on
$\AsUnDefAlg[U]\coloneqq \Field[U]\otimes\AsUnDefAlg$.  The higher operations will be defined in
terms of combinatorial objects associated to certain planar graphs, as
follows.  These are interpreted as certain maps of the disk into the torus in
Section~\ref{subsec:Immersions}. (Note that these graphs play a
different role from the rooted, planar trees in
Section~\ref{sec:walg-first-props}.)

\begin{definition}
  \label{def:TilingPattern}
  By a {\em rooted, planar graph}, we mean a graph $\Gamma$,
  together with an embedding of $\Gamma$ into the disk $\CDisk$, so that
  $\Gamma$ meets $\partial\CDisk$ in its leaves, together with a
  choice of a distinguished
  leaf of $\Gamma$. This distinguished leaf is called
  the {\em root} (and we will no longer refer to it as a ``leaf'').
  Let $\Gamma$ be a rooted, planar, graph 
  with the
  following properties:
  \begin{itemize}
    \item $\Gamma$ is connected.
    \item $\Gamma$ has at least one internal vertex.
    \item Each internal vertex in $\Gamma$ has valence $4$.
    \item Each component of $\CDisk\setminus \Gamma$ either meets
      $\partial \CDisk$, or it meets exactly $4$ edges and vertices of
      $\Gamma$. In the second case, these embedded, $4$-edged cycles in $\Gamma$ 
      are called {\em short cycles}.
  \end{itemize}
  Let $Q$ denote the set of sectors (quadrants) around all of the internal
  vertices. A {\em valid labelling} on $\Gamma$ is a function
  $\Lambda\co 
  Q \to \{1,\dots,4\}$ with the following properties:
  \begin{itemize}
  \item If $q_1,q_2,q_3,q_4$ are the four quadrants around some 
    vertex, labelled in clockwise order (around that vertex),
    then, up to cyclic reordering, 
    \[ (\Lambda(q_1),\dots,\Lambda(q_4))=(1,2,3,4).\]
  \item Let $v_1$ and $v_2$ be two vertices that are connected by an
    edge $e$.  Orient $e$, and let $q_1$ and $q_2$ be the two sectors
    around
    $v_1$ and $v_2$ to the right of $e$, labelled in the order given
    by the orientation of $e$. Then, 
    $\Lambda(q_1)+1\equiv \Lambda(q_2)\pmod{4}$. (This is required to hold
  for both orientations of $e$.)
  \end{itemize}
  A {\em centered tiling pattern} is a rooted planar graph with the above properties,
  equipped with a valid labelling $\Lambda$.
\end{definition}

Observe that a valid labelling is uniquely determined by its value on
any single sector.
The conventions of a valid labelling are summarized in
Figure~\ref{fig:ValidLabels}.

\begin{figure}
\centering
\input{ValidLabels}
\caption{{\bf Conventions on a valid labeling.} The labels are in $\{1,\dots,4\}$ (mod $4$).}
\label{fig:ValidLabels} 
\end{figure}
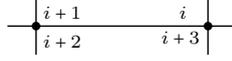

We will extend slightly our graphs:

\begin{definition}
  \label{def:ExtendedTilingPattern}
  Let $\Gamma$ be a
  centered tiling pattern as in
  Definition~\ref{def:TilingPattern}. Let $e$ be the edge of $\Gamma$
  adjacent to the root; orient $e$ so it points towards the root.  Enlarge the graph to obtain a new graph $\Gamma'$ by inserting
  a sequence of $2$-valent vertices along $e$.
  A \emph{valid labeling} on $\Gamma'$ is now a labelling on certain distinguished
  sectors, as follows. Near each $2$ valent
  vertex, there are two sectors.  For a {\em left-extended} (respectively
  {\em right-extended}) tiling the sectors to the left (respectively right) of $e$ are 
  distinguished;  in both cases, we
  will also think of all the quadrants as being distinguished sectors.
  A \emph{valid labelling} of $\Gamma'$ is a
  function $\Lambda$ from the distinguished sectors to $\{1,\dots,4\}$
  satisfying the conditions of Definition~\ref{def:TilingPattern}.  An {\em
   extended tiling pattern} is a planar, rooted graph $\Gamma'$
  as above, together with a valid labelling.
  A {\em tiling pattern} is either a centered tiling pattern or an extended tiling pattern.
  For a tiling pattern, the underlying graph has first homology $H_1(\Gamma)$ isomorphic to $\ZZ^w$ for
  some $w\geq 0$. This integer $w$ is called the {\em weight} of the tiling pattern.
\end{definition}

\begin{definition}
  \label{def:ChordSequence}
  Let $\Gamma$ be a tiling pattern. The graph $\Gamma$ divides
  $\partial\CDisk$ into $k$ intervals $I_1,\dots,I_k$, 
  starting and ending at the root, labelled in the order they appear
  with respect to the
  boundary orientation. 
  Let $(q_1,\dots,q_m)$ be those distinguished sectors
  which are visible from the boundary,
  ordered as one  traverses the boundary with its
  orientation, starting at the root vertex.
  There is a sequence $1=n_1<n_2<\dots<n_k<m$, with the property that
  $q_{n_i},\dots,q_{n_{i+1}-1}$ are the sectors visible along~$I_i$.
  The {\em chord sequence} of a labelled tiling pattern
  is defined to be the
  element of $\left(\AsUnDefAlg\right)^{\otimes_{\Ground} {m}}$ given by
    \[ \left(\prod_{i=1}^{n_2-1} \rho_{\Lambda(q_i)}\right)\otimes
    \dots\otimes \left(\prod_{i=n_\ell}^{n_{\ell+1}-1}
    \rho_{\Lambda(q_i)}\right) \otimes\dots \otimes
    \left(\prod_{i=n_k}^{m} \rho_{\Lambda(q_i)}\right).\]
\end{definition}

\begin{lemma}
  \label{lem:alg-property-nonmult}
  If $\Gamma$ is a tiling pattern, its chord sequence
  $\rho^1\otimes\dots\otimes\rho^n$ is non-zero; and moreover for each
  $i=1,\dots,n-1$, we have that $\rho^i\cdot \rho^{i+1}=0$.
\end{lemma}
\begin{proof}
  This is immediate from the conventions on a valid labelling,
  and the definition of the chord sequence.
\end{proof}

\begin{lemma}
  \label{lem:alg-even-operations}
  If $\Gamma$ is a tiling pattern, then 
  length $k$ of its chord sequence is even.
\end{lemma}

\begin{proof}
  The boundary of $\Gamma$ cuts $\partial\CDisk$ into arcs. The number
  of arcs is the length of the chord sequence. Since the graph
  underlying $\Gamma$ is $4$-valent, we can think of this graph as a
  union of $\ell$ arcs embedded in $\CDisk$, which intersect transversely.
  The boundary of each arc meets $\partial\CDisk$ in a pair of points;
  so $\Gamma$ meets $\partial\CDisk$ in $2\ell$ points. Clearly, $k=2\ell$.
\end{proof}

\begin{definition}
  \label{def:OutputElement}
  Given an extended tiling pattern, let $d$ denote the number of 
  valence $4$ vertices. The {\em output element} in $\AsUnDefAlg[U]$ of the tiling pattern is $U^d$ times the following
  element of $\AsUnDefAlg$:
  \begin{itemize}
  \item If the tiling pattern is centered, the element is the left
    idempotent of $\rho_{\Lambda(q_1)}$, which coincides with the
    right idempotent of $\rho_{\Lambda(q_m)}$.
  \item If the tiling pattern is extended, the element is the product of $\rho_{\Lambda(q)}$, taken over all the distinguished
    valence~$2$ sectors (taken in their natural order).
  \end{itemize}
\end{definition}

\begin{example}
  Consider Figure~\ref{fig:TilingPatterns}. We have illustrated five
  tiling patterns: (a), (d), and (e) are centered; (b) is right
  extended; (c) is left-extended. The weights of (a)--(d) are~$0$; (e)
  has weight~$1$.  We have chord sequences:
  (a) $\rho_4\otimes
  \rho_3\otimes\rho_2\otimes\rho_1$;
  (b) $\rho_{34}\otimes
  \rho_3\otimes\rho_2\otimes\rho_1$;
  (c) $\rho_4\otimes\rho_3\otimes\rho_2\otimes \rho_{123}$;
  (d) $\rho_4\otimes\rho_{341}\otimes\rho_4\otimes\rho_3\otimes\rho_{23}\otimes\rho_2\otimes\rho_{12}\otimes\rho_1$; and
  (e) $\rho_{41}\otimes\rho_4\otimes\rho_{34}\otimes\rho_3\otimes\rho_{23}\otimes\rho_2\otimes\rho_{12}\otimes\rho_1$.
  The outputs are (a) $U\iota_0$; (b) $U\rho_3$; (c) $U \rho_{23}$; (d)
  $U^3\iota_0$; and (e) $U^4 \iota_0$.
\end{example}

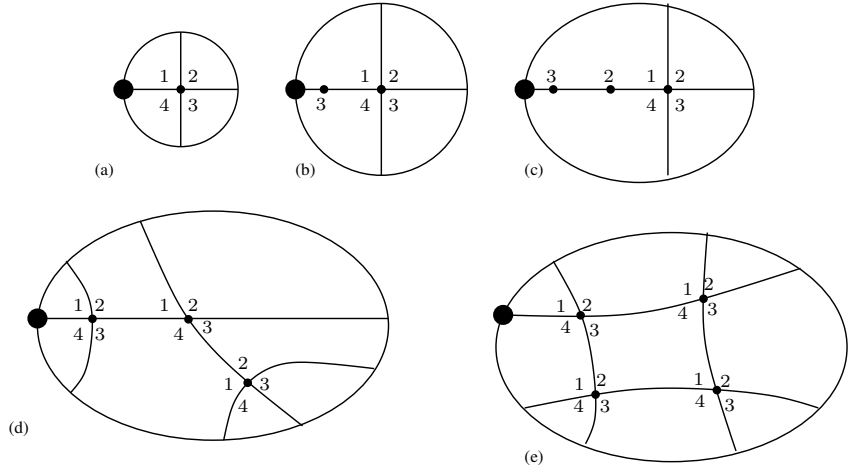
\begin{figure}
\centering
\input{TilingPatterns}
\caption{{\bf Tiling patterns.} When these are extended (as in (b) and (c) above), we indicate the distinguished sectors by writing the labels in those sectors. The root is indicated by the large dot on the boundary.}
\label{fig:TilingPatterns} 
\end{figure}

\begin{definition}
  \label{def:MAlg}
  As an $\Field[U]$-module, the weighted $\Ainf$-algebra $\MAlg$ is
  $\AsUnDefAlg[U]$. The operations $\mu_n^w\co
  (\MAlg)^{\otimes n}\to\MAlg$ are defined as follows. First, the
  operations are $U$-multilinear, so it suffices to define them
  on basic elements of $\AsUnDefAlg$. The operation $\mu_2^0$ is given
  by the multiplication on $\AsUnDefAlg$, and
  $\mu_n^w(a_1,\dots,a_n)=0$ if $(n,w)\neq(2,0)$ and some
  $a_i\in\{\iota_0,\iota_1\}$.
  So, it remains to define
  $\mu_n^w(\rho^1,\dots,\rho^n)$, where the $\rho^i$ are Reeb
  elements.

  The element $\mu_0^1$ is
  \begin{equation}\label{eq:mu-0-1}
    \mu_0^1=\rho_{1234}+\rho_{2341}+\rho_{3412}+\rho_{4123},
  \end{equation}
  the sum of the length-$4$ chords. For all $w\neq 1$, we have $\mu_0^w=0$.

  For the remaining pairs $(n,w)$, the function $\mu_n^w(\rho^1,\dots,\rho^n)$ is the sum of all the output elements
  (in the sense of Definition~\ref{def:OutputElement})
  of labelled tiling patterns which have  weight $w$ and chord sequence $\rho^1\otimes\dots\otimes \rho^n$.
\end{definition}

Interpreted in terms of the algebra, Lemma~\ref{lem:alg-even-operations} 
states that $\mu^w_n=0$ if $n$ is odd.

In Sections~\ref{sec:unique} and~\ref{sec:Fukaya} we will also be interested in the ordinary $\Ainf$-algebra underlying $\MAlg$:
\begin{definition}
  Let $\UnDefAlg$ denote the undeformed $\Ainf$-algebra underlying
  $\MAlg$, i.e., $\UnDefAlg=(\AsUnDefAlg[U],\allowbreak \{\mu_n^0\})$,
  where the operations $\mu_n^0$ are as in Definition~\ref{def:MAlg}.
\end{definition}
In other words, the operations on $\UnDefAlg$ count tile patterns whose
underlying graph $\Gamma$ is a tree.

\subsection{Verifying the \texorpdfstring{$\Ainf$}{A-infinity} relations}

The aim of this section is to prove that the operations from
Definition~\ref{def:MAlg} make $\MAlg$ into a weighted
$\Ainf$-algebra. To establish this property, it is helpful to have a
graphical representation of the composition of tiling patterns.

\begin{definition}
  \label{def:CompositePattern}
  A {\em centered composite pattern}
  consists of the following data:
  \begin{itemize}
    \item a rooted, planar graph $\Gamma$ satisfying the conditions of
      Definition~\ref{def:ExtendedTilingPattern}, except that now
      rather than being connected, the underlying graph is required to
      have exactly two components, labelled $\Gamma_1\coprod \Gamma_2$;
      and each component is required to have at least one
      $4$-valent vertex.
\item An arc $S$ on the boundary of $\CDisk$, with the following property. If we cut
      $\CDisk$ along $\Gamma_1\cup\Gamma_2$, there is a
      distinguished region $\Region$ that meets both $\Gamma_1$ and
      $\Gamma_2$. This region meets the boundary in two arcs.
      The arc $S$ is required to be one of those two arcs.
    \item A valid labelling $\Lambda$ as in Definition~\ref{def:TilingPattern}, which is also required
      to satisfy a compatibility condition at the
      distinguished edge $e$, as shown in
      Figure~\ref{fig:CompositePattern}.
  \end{itemize}
  A composite pattern is called {\em
    extremal} if the arc $S$ meets the root vertex; otherwise it
  is called {\em generic}. 
  An {\em extended composite pattern} is defined similarly, but has a sequence of valence two vertices adjacent to the root vertex,
  with distinguished sectors, all of which lie on the same side of the arc from the root to the first $4$-valent vertex.
\end{definition}

  \begin{figure}
    \centering
    \input{CompositePattern}
    \caption{{\bf Composite labelling convention.}
      We require that $\Lambda$ changes as shown, as we cross $S$
      between the two components $\Gamma_1$ and $\Gamma_2$.}
    \label{fig:CompositePattern} 
\end{figure}
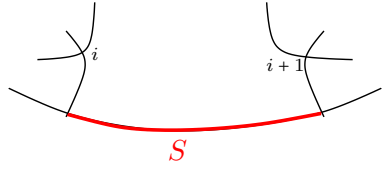

See Figure~\ref{fig:Composition} for an example.

Let $\Gamma_1\subset \CDisk_1$ be a tiling pattern with chord sequence 
$\rho^1\otimes \dots\otimes\rho^n$, fix some $i\in \{1,\dots,n\}$, and
let $\Gamma_2\subset \CDisk_2$ be an extended tiling pattern with
chord sequence $\sigma^1\otimes\dots \otimes\sigma^m$ whose output
element is $\rho^i$.
We will compose these two tiling patterns, to obtain a
composite in the sense of Definition~\ref{def:CompositePattern}, as
follows.
Let $J\subset \partial\CDisk_2$ be the boundary segment
that meets the distinguished $2$-valent sectors: i.e. this is
immediately after the root if $\Gamma_2$ is left-extended, and it is
immediately before the root if $\Gamma_2$ is right-extended.  Let
$\Gamma_2'$ be obtained from $\Gamma_2$ by removing all the $2$-valent
vertices.  The composition
$\Gamma_1\connectsum_i\Gamma_2$ is the
boundary connected sum of $\CDisk_1$ with $\CDisk_2$, gluing $I_i$ to
$J$, which we think of as a disk $\CDisk$, equipped with the graph
$\Gamma=\Gamma_1\coprod \Gamma_2'$.  Let $S\subset\partial\CDisk$ 
be the interval that connects the root vertex of $\Gamma_2'$ with
some vertex in $\Gamma_1$.  See Figure~\ref{fig:Composition} for an
illustration.

\begin{figure}
\centering
\input{Composition}
\caption{{\bf Example of a composite pattern.}  Here, $\Gamma_1$ and
  $\Gamma_2$ are two tiling patterns; we can form their composition
  $\Gamma_1 \connectsum_2 \Gamma_2$ to obtain the composite pattern~$\Gamma$ on
  the right.  The red dot is the root vertex of $\Gamma_2$, which is
  to be joined to the $2^{\textrm{nd}}$ boundary arc of $\Gamma_1$.}
\label{fig:Composition} 
\end{figure}
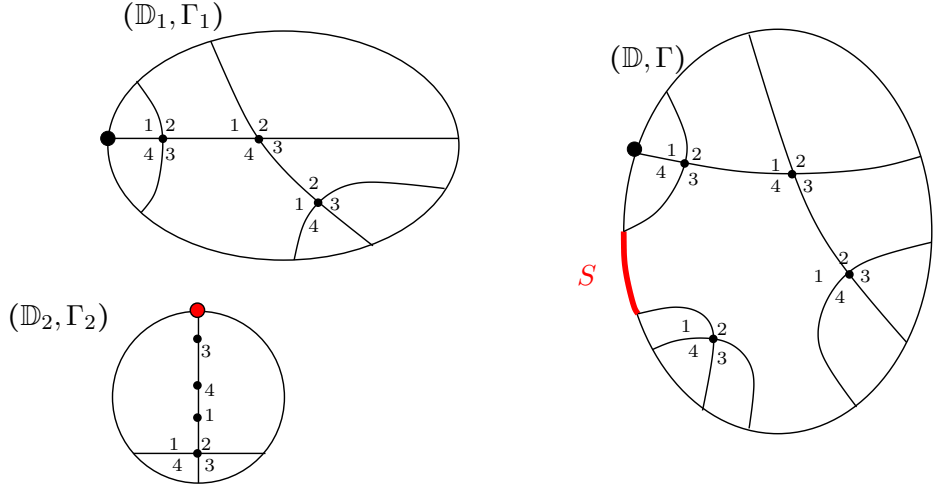

Each composite pattern~$\Gamma$ can be uniquely decomposed as
$\Gamma_1 \connectsum_i \Gamma_2$ for suitable choices of $\Gamma_1$
and~$\Gamma_2$. Explicitly, cut $\CDisk$ in two along an arc $A$ that
intersects $S$ exactly once and is disjoint from $\Gamma$. Let
$(\CDisk_1,\Gamma_1)$ be the component that contains the root of
$\Gamma$; the other component is $(\CDisk_2,\Gamma_2')$, equipped with
the root vertex induced from $S$. The sequence of $2$-valent vertices
needed to reconstruct $\Gamma_2$ from~$\Gamma_2'$ can be read off the $\Gamma_1$
side.

\begin{definition}
The \emph{chord sequence} of a composite pattern~$\Gamma$ is defined as in
Definition~\ref{def:ChordSequence}, with the understanding that the
factor in the tensor product of chords coming from the sectors adjacent to the interval $S$ is
dropped.
Let $d$ denote the number of valence $4$ vertices in $\Gamma$.
The \emph{output element} of the composite pattern~$\Gamma$ is 
$U^{d}$ times the product of the chords
associated to the distinguished valence~$2$ sectors. Equivalently, if
$\Gamma=\Gamma_1\connectsum_i\Gamma_2$, the output element of $\Gamma$ is the
output element of $\Gamma_1$ times $U^{d_2}$, where $d_2$ denotes the
number of valence $4$ vertices in $\Gamma_2$.
\end{definition}

For example, consider the picture on the right of Figure~\ref{fig:Composition}.
Reading the chords in order as they are seen from the boundary (starting at the root) gives the sequence
\[ \rho_4\otimes\rho_{3412}\otimes \rho_{1}\otimes \rho_{4}\otimes \rho_3 \otimes \rho_{2341} \otimes\rho_4\otimes\rho_3\otimes\rho_{23}\otimes\rho_2\otimes\rho_{12}\otimes\rho_1.\]
For the chord sequence, though, we drop the second tensor factor
$\rho_{3412}$, since that is the term visible from $S$; thus,
thus the chord sequence associated to the composite is:
\[ \rho_4\otimes\rho_{1}\otimes \rho_{4}\otimes \rho_3 \otimes\rho_{2341} \otimes\rho_4\otimes\rho_3\otimes\rho_{23}\otimes\rho_2\otimes\rho_{12}\otimes\rho_1.\]

By construction, the chord sequence of $\Gamma_1\connectsum_i\Gamma_2$ is obtained from the chord sequence for $\Gamma_1$ by replacing
the $i^{th}$ tensor factor with the chord sequence for $\Gamma_2$.

\begin{lemma}
  \label{lem:ProductCancellations}
  Fix a sequence of Reeb elements $\rho^1,\dots,\rho^n$ of
  $\AsUnDefAlg$, with the property that
  $\rho^1\otimes\dots\otimes\rho^n\in (\AsUnDefAlg)^{\otimes_{\Ground}
    n}$ is non-zero and $\rho^k\rho^{k+1}=0$ for all $k$
  (cf.\ Lemma~\ref{lem:alg-property-nonmult}). Fix some $1\le i\le n$ and
  fix a factorization
  $\rho^i=\sigma_1\cdot \sigma_2$ into two Reeb elements.  Consider the following sets.
  \begin{itemize}
    \item  The set $\Graphs_1$ of tiling patterns with weight $w$ and chord sequence
  $\rho^1\otimes\dots\otimes\rho^n$.
    \item The set ${\mathfrak P}$ of
      composite patterns with weight $w$ and chord sequence
      \[\rho^1\otimes\dots\otimes \rho^{i-1}\otimes
        \sigma_1\otimes\sigma_2\otimes\rho^{i+1}\otimes\dots\otimes\rho^n.\]
      These composite patterns are always generic in the sense of
      Definition~\ref{def:CompositePattern}.
    \item The set  ${\mathfrak T}$ of tiling patterns with weight
      $w-1$ and chord sequence
      \[
        \rho^1\otimes\dots\otimes \rho^{i-1}\otimes
        \sigma_1\otimes\tau\otimes\sigma_2\otimes\rho^{i+1}\otimes\dots\otimes\rho^n,
      \]
      where $\tau$ has length $4$.      
    \item (When $i=1$) the set ${\mathfrak L}$ of centered or
      left-extended tiling patterns $\Gamma$ with weight $w$ and chord sequence
      \[
        \sigma_2\otimes\rho^{2}\otimes\dots\otimes\rho^n.
      \] 
      If $\Gamma$ has output element $U^d \rho$, then when viewed as an
      element of~${\mathfrak L}$ we take its associated output element to
      be $U^d \sigma_1\cdot \rho$.
    \item (When $i=n$) the set ${\mathfrak R}$ of centered or
      right-extended tiling patterns $\Gamma$ with weight $w$ and
      chord sequence
      \[
        \rho^1\otimes\rho^2\otimes\dots\otimes \rho^{n-1}\otimes\sigma_1.
      \]
      If $\Gamma$ has output element $U^d \rho$, then when viewed as an
      element of $\mathfrak{R}$ we take its associated output element
      to be $U^d \rho\cdot \sigma_2$.
  \end{itemize}
  Then there is a one-to-one correspondence between $\Graphs_1$
  and $\mathfrak{P}\cup\mathfrak{T}\cup\mathfrak{L}\cup\mathfrak{R}$ 
  which preserves the associated output elements
\end{lemma}

\begin{proof}
  Fix an element of $\Graphs_1$.
  The factorization 
  $\rho^i=\sigma_1\cdot \sigma_2$ corresponds to an arc $A$ whose interior
  is disjoint from $\Gamma$, with one endpoint on some edge
  $e$ in $\Gamma$ and the other endpoint on $\partial \CDisk$.
  Assume first that $i\not\in\{1,n\}$.
  The one-to-one correspondence from $\Graphs_1$ to $\mathfrak{P}\cup \mathfrak{T}$ is
  obtained by pushing $e$ out to the boundary along $A$, as shown in
  Figure~\ref{fig:PushingOut}. 

  If pushing out $e$ disconnects $\Gamma$, we label the newly introduced
  arc on $\partial \CDisk$ by $S$, and the result is a
  composite pattern, which is in $\mathfrak{P}$.  If pushing out $e$ does not
  disconnect~$\Gamma$ (i.e., if $e$ is part of a cycle in~$\Gamma$), the
  result is another tiling pattern, which is in ${\mathfrak T}$.  Examples are
  illustrated in Figure~\ref{fig:PushOutExamples}.

  In the case $i=1$, pushing out $e$ from $\Gamma$ may result in a disconnected
  graph, one of whose components has no valence $4$ vertices. In this
  case, the operation of pushing $e$ out to the boundary along $A$
  results in a configuration which is not a composite pattern in the
  sense of Definition~\ref{def:CompositePattern}, as one of the two
  components $\Gamma_1$ and $\Gamma_2$ consists of a chain of
  $2$-valent vertices.  Deleting that component, and placing the new
  root as for composite patterns, gives a tiling pattern
  which is either left-extended or centered, with chord sequence
  $\sigma_2\otimes\rho^2\otimes\dots\rho^n$. View the result as an
  element of~${\mathfrak L}$. With this addition, we obtain the
  desired one-to-one correspondence
  between $\Graphs_1$ and $\mathfrak{L}\cup\mathfrak{P}\cup\mathfrak{T}$.
  
  The case where $i=n$ is analogous to the case when $i=1$, except in that case
  the tiling pattern is either centered or right-extended, and we view it now as an element 
  of $\mathfrak{R}$.
  This gives the desired one-to-one correspondence between 
  $\Graphs_1$ and ${\mathfrak R}\cup{\mathfrak P}\cup{\mathfrak T}$.
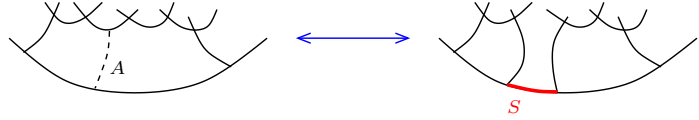
\begin{figure}
\centering
\input{PushingOut}
\caption{{\bf Pushing out the edge.}
Illustration of the $1:1$ correspondence from Lemma~\ref{lem:ProductCancellations}}
\label{fig:PushingOut} 
\end{figure}
\end{proof}

\begin{remark}
  If ${\mathfrak T}$ is non-empty, the newly-introduced length $4$ chord $\rho$
  appearing in its chord sequence is determined by $\sigma_1$: it is
  the length four chord for which $\sigma_1\otimes\rho\neq 0$
  (i.e., with matching idempotents), and for which
  $\sigma_1\cdot\rho=\rho\cdot \sigma_2 =0$.
\end{remark}

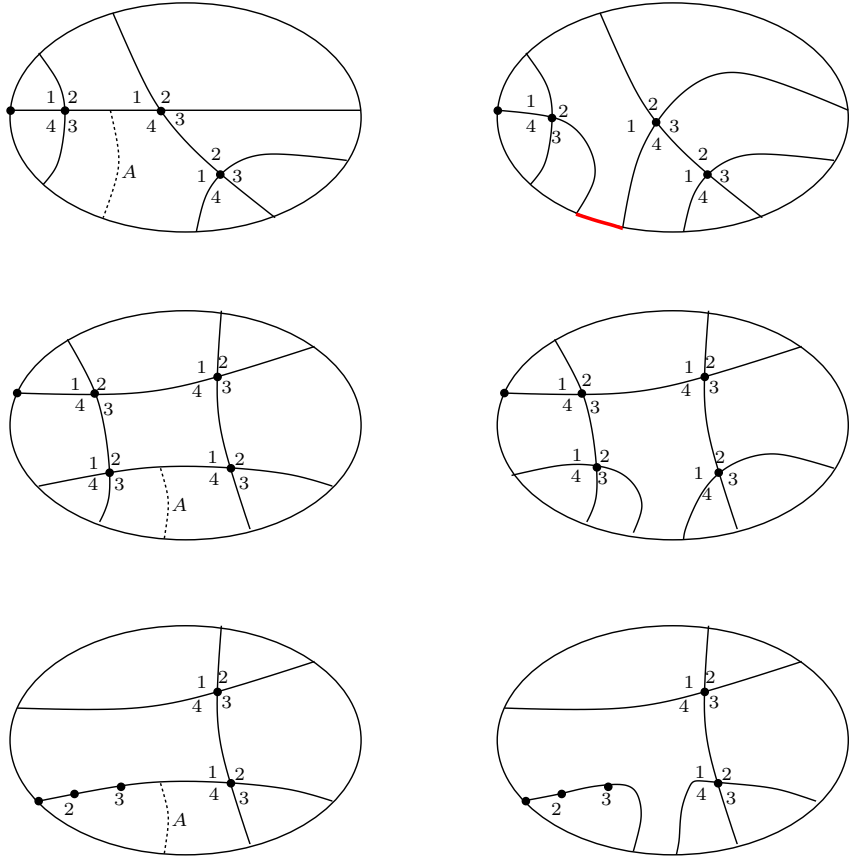
\begin{figure}
\centering
\input{PushOutExamples}
\caption{{\bf Pushing out the edge: examples.}  On the upper left is
  an arc $A$ connecting an edge to the boundary, so that if we push
  the edge out along $A$, we obtain a composite pattern, as
  illustrated on the upper right.  On the middle left is an arc
  connected to a non-disconnecting edge; pushing out along that arc
  gives another tiling pattern (with one smaller weight), as
  illustrated on the middle right.  On the bottom left is a
  left-extended tiling pattern and an arc $A$ which disconnects the
  graph; but pushing out the edge, as illustrated on the lower right,
  does not give a composite pattern. Rather, if we delete the component with no $4$-valent vertices, we obtain a (centered) tiling pattern of type ${\mathfrak L}$.}
\label{fig:PushOutExamples} 
\end{figure}

\begin{lemma}
  \label{lem:WhatIsGeneric}
  Let $\Gamma$ be a composite pattern, and write
  $\Gamma=\Gamma_1\connectsum_i\Gamma_2$. Let $n$ be the length of the chord sequence for 
  $\Gamma_1$. Then $\Gamma$ is generic (in the sense of Definition~\ref{def:CompositePattern})
  if and only if one of the following conditions holds:
  \begin{itemize}
    \item $1<i<n$;
    \item $i=1$ and  $\Gamma_2$ is  left-extended; or
    \item $i=n$ and $\Gamma_2$ is right-extended.
  \end{itemize}
  The composite pattern~$\Gamma$ is extremal, and its
  distinguished $2$-valent sectors, if any, lie in the region~$\Region$ 
  (see Definition~\ref{def:CompositePattern})
  if and only if one of the following conditions holds:
  \begin{itemize}
  \item $i=1$,  $\Gamma_1$ is centered or left-extended, and
    $\Gamma_2$ is right-extended; or
    \item $i=n$,  $\Gamma_1$ is centered or right-extended, and $\Gamma_2$ is left-extended.
  \end{itemize}
  The composite pattern~$\Gamma$ is extremal, extended, and its
  distinguished $2$-valent sectors lie in a region other than $\Region$
  if and only if one of the following conditions holds:
  \begin{itemize}
    \item $i=1$,  $\Gamma_1$ is right-extended, and $\Gamma_2$ is
      right-extended; or
    \item $i=n$,  $\Gamma_1$ is left-extended, and $\Gamma_2$ is left-extended.
  \end{itemize}
\end{lemma}

\begin{proof}
  It is clear from the definitions that $\Gamma_1\connectsum_i\Gamma_2$ is extremal
  precisely when $i=1$ and $\Gamma_2$ is right-extended or
  $i=n$ and $\Gamma_2$ is left-extended. See
  Figure~\ref{fig:WhatIsGeneric}. The case of an extremal composite pattern with no distinguished $2$-valent sectors
  is illustrated in Figure~\ref{fig:DegExtremal}.
\begin{figure}
\centering
\input{WhatIsGeneric}
\caption{{\bf Generic and extremal compositions.}
  We have drawn here compositions $\Gamma_1\connectsum_1\Gamma_2$ and $\Gamma_1\connectsum_1\Gamma_2'$. In the top line, we use a
  left-extended  $\Gamma_2$ and the result is generic; 
  while in the second line, we use  a right-extended  $\Gamma_2'$, and the result is extremal.}
\label{fig:DegExtremal} 
\end{figure}
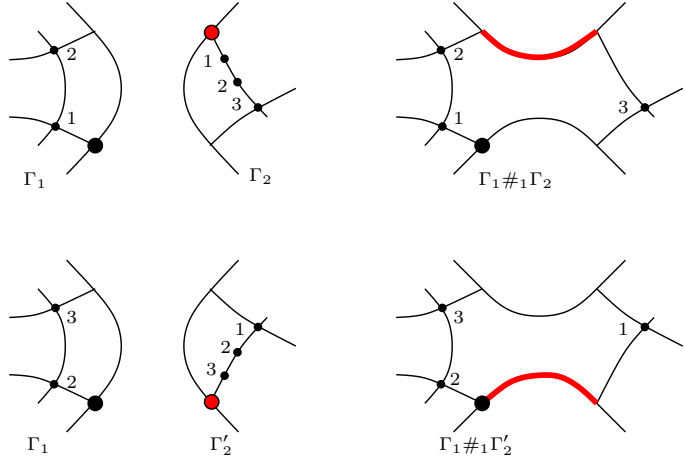
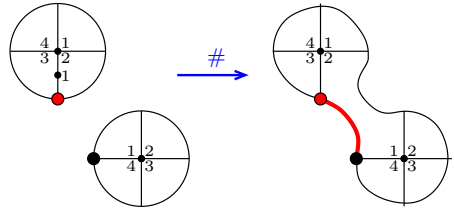
\begin{figure}
\centering
\input{DegExtremal}
\caption{{\bf An extremal composite pattern.}
  In this case, $\Gamma_1$ is centered and $\Gamma_2$ is left-extended.}
\label{fig:WhatIsGeneric} 
\end{figure}
\end{proof}

\begin{lemma}
  \label{lem:FirstInvolution}
  Let $\Graphs_2$ denote the set of extremal composite patterns whose
  distinguished $2$-valent sectors are in the region $\Region$. This
  set $\Graphs_2$
  admits a fixed-point-free involution which preserves the
  chord sequence, output element, and total weight.
\end{lemma}
\begin{proof}
  The involution is obtained by placing the root vertex on the other
  endpoint of $S$ and then moving the $2$-valent sectors (if there are
  any) so that they remain adjacent to the root vertex.  See
  Figure~\ref{fig:FirstInvolution}. (In the bottom row, we have an
  example where $2$-valent vectors need to be moved).
\end{proof}

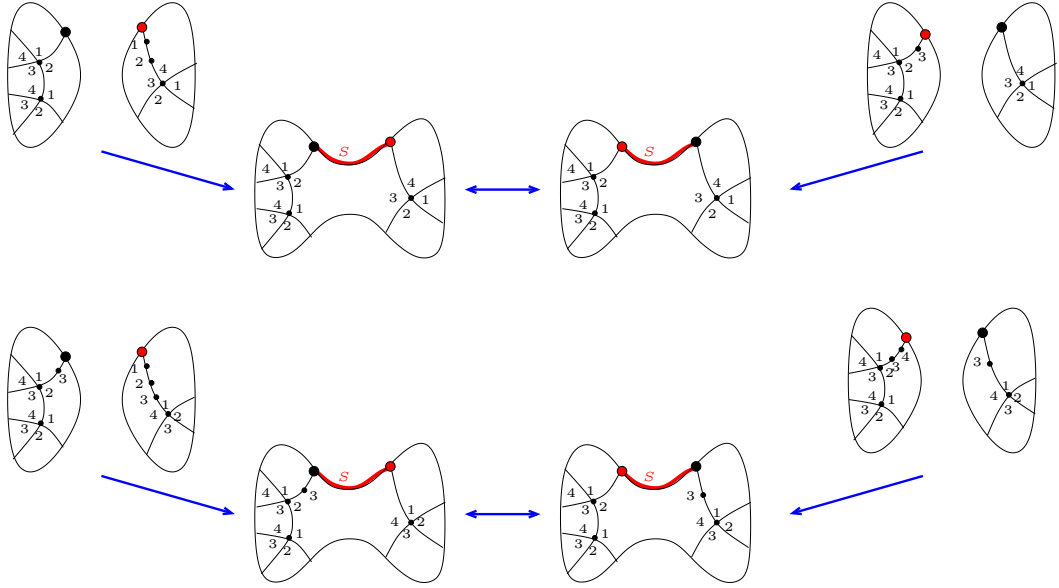
\begin{figure}
\centering
\input{FirstInvolution}
\caption{{\bf Involution from Lemma~\ref{lem:FirstInvolution}.}  The
  involution is given by moving the root vertex from one endpoint of
  $S$ to the other (as indicated by the double arrows).  In the second
  row, we also had to move a $2$-valent vertex during this operation.
  We have also included compositions that give rise to the composites.
  As before, we use the conventions on compositions: the red dot lies on the
  tree $\Gamma_2$, which is joined to some input of $\Gamma_1$.
}
\label{fig:FirstInvolution} 
\end{figure}

\begin{theorem}
  \label{thm:AinftyAlgebra}
  The operations from Definition~\ref{def:MAlg} give $\MAlg$ the
  structure of a weighted $\Ainf$-algebra.
\end{theorem}

\begin{proof}
  We must verify the $\Ainf$ relation for each fixed input
  sequence of algebra elements and fixed weight. This relation is a sum of contributions
  of weighted, rooted trees with two internal vertices. (Note that we
  consider a leaf with positive weight internal.)

  The $\Ainf$ relations with $\leq 2$ inputs or with $3$ inputs and
  weight $0$ are easy consequences of the following facts:
  \begin{itemize}
    \item $\mu^w_0=0$ except when $w=1$.
    \item $\mu^w_1=0$ for all $w\geq 0$.
    \item $\mu^w_2=0$ for all $w>0$.
    \item $\AsUnDefAlg$ is an associative algebra.
    \item $\mu^1_0$ is a central element of $\AsUnDefAlg$.
  \end{itemize}

  Consider an $\Ainf$ relation, then, with at least $4$ inputs or with
  $3$ inputs and weight $\geq 1$. 
  We describe these terms according to
  the type of the two $\Ainf$ operations involved, with the outer
  $\mu^w_n$ written first, using a symbol
  \begin{itemize}
  \item $0$ for the curvature $\mu^1_0$,
  \item $2$ for a multiplication $\mu^0_2$,
  \item $L$ for a left-extended $\mu^w_n$ with $n \ge 4$,
  \item $C$ for a centered $\mu^w_n$ with $n \ge 4$, and
  \item $R$ for a right-extended $\mu^w_n$ with $n \ge 4$.
  \end{itemize}
  In addition, we add a final symbol $+$, $-$, or~$g$ describing
  whether the inner $\mu^w_n$ is fed into the leftmost input, rightmost input,
  or any other input of the outer $\mu^w_n$. We will also use a $\wild$
  for a wild-card any of these symbols.

  We thus have terms of the following types, classified according to
  the number of inputs of the smaller~$\mu^w_n$.
  \begin{itemize}
  \item \textbf{Terms of type~$(0)$} involve $\mu^1_0$, which is
    necessarily the inner $\Ainf$ operation.
    The other operation is
    associated to a tiling pattern~$\Gamma$. The overall term is
    written, for example, $(C0+)$ for a term where $\Gamma$ is
    centered and the curvature is fed into the first input of the
    corresponding $\mu^w_n$.
  \item \textbf{Terms of type~$(2)$} involve $\mu^0_2$, which can be
    the inner or outer operation, giving terms like $(2L-)$, from a
    left-extended tiling pattern~$\Gamma$ fed into the second input of
    the $\mu^0_2$, or $(R2g)$, from a $\mu^0_2$ fed into a generic
    input of a right-extended tiling pattern. Note that if the $2$
    comes first, then the last symbol cannot be a $g$, since $\mu^0_2$
    has no generic inputs.
  \item \textbf{Terms of type~$[]$} involve two higher multiplications
    $\mu^w_n$ with $n > 2$. These correspond to composite patterns in
    the sense of Definition~\ref{def:CompositePattern}, and are written
    with square brackets to distinguish them from types~$(0)$
    and~$(2)$.
    So, for example, $[CR+]$
    means a right-extended operation feeding into the first input of a
    centered operation.
    The second letter cannot be~$C$: the output of a 
    centered operation is a power of $U$ times an idempotent, and if such
    an output is channelled into another $\mu^w_n$ with
    $(w,n)\neq (0,2)$, the result vanishes.
  \end{itemize}

  Finally, we make one further distinction: we divide terms of type
  $(2C-)$ into two types:
  \begin{itemize}
  \item $(2C-)^0$, where $\rho^1 \cdot \rho^2 = 0$, and
  \item $(2C-)^\times$, where $\rho^1 \cdot \rho^2 \ne 0$.
  \end{itemize}
  We similarly distinguish $(2C+)$ into $(2C+)^0$, where $\rho^{n-1}
  \cdot \rho^n = 0$, and $(2C+)^\times$, where $\rho^{n-1} \cdot
  \rho^n \ne 0$.

  Our goal is to explain how these terms
  cancel in the $\Ainf$ relations. Before proceeding to the main
  verification, we make some remarks about the sequence of incoming
  algebra elements.  By linearity, it suffices to verify the $\Ainf$
  relation when the sequence of incoming algebra elements consists of
  basic algebra elements. Moreover, the case where at least one of
  those elements is an idempotent can be handled easily. Since the
  $\mu_m^w(a_1,\dots,a_m)=0$ if some $a_i$ is an idempotent and
  $(n,w)\neq (2,0)$, if $(a_1,\dots,a_n)$ is a sequence of inputs to a
  non-zero term in the $\Ainf$ relation and $a_i$ is an idempotent,
  then either $i=1$, in which case we have two cancelling terms in the
  $\Ainf$ relation of types $(2\wild-)$ and a $(\wild2+)$; or $i=n$, in which
  case we have two cancelling terms of types $(2\wild+)$ and $(\wild2-)$.  For
  example, in the weight $0$ $\Ainf$ relation with inputs
  $\iota_1\otimes\rho_4\otimes\rho_3\otimes\rho_2\otimes \rho_1$,
  there are two non-zero terms of type $(2C+)$ and $(C2-)$
  respectively:
  \[ \mu^2_0(\iota_1,\mu^4_0(\rho_4,\rho_3,\rho_2,\rho_1))= 
  \mu^4_0(\mu^2_0(\iota_1,\rho_4),\rho_3,\rho_2,\rho_1)), \]

  We now proceed to the verification of the $\Ainf$ relation, assuming
  that the input sequence consists of Reeb elements
  $\rho^1\otimes \dots\otimes \rho^n$, with $n\geq 4$ or
  $n=3$ and $w\geq 1$.
  See Table~\ref{tab:Ainf-terms} for a listing of the term types, and how
they cancel in the following proof.

\begin{table}
  \centering
  \begin{tabular}{cl@{\qquad}cl@{\qquad}cl}
    \toprule
    Term&Cancellation&Term&Cancellation&Term&Cancellation \\ \midrule
    $(L0+)$ & Cancels $(R0-)$ &
      $(L2+)$ & $\Graphs_1$ in Lem.~\ref{lem:ProductCancellations}&
        $(2L+)$ & See Table~\ref{tab:remaining} \\
    $(L0g)$ & $\mathfrak{T}$ in Lem.~\ref{lem:ProductCancellations}&
      $(L2g)$ & $\Graphs_1$ in Lem.~\ref{lem:ProductCancellations}&
        $(2L-)$ & ${\mathfrak L}$  in  Lem.~\ref{lem:ProductCancellations}\\
    $(L0-)$ & See Table~\ref{tab:remaining} &
      $(L2-)$ & $\Graphs_1$ in Lem.~\ref{lem:ProductCancellations}&
        $(2C-)^\times$ & ${\mathfrak L}$ in Lem.~\ref{lem:ProductCancellations} \\
    $(C0+)$ & Cancels $(C0-)$ &
      $(C2+)$ & $\Graphs_1$ in Lem.~\ref{lem:ProductCancellations}&
        $(2C-)^0$ & See Table~\ref{tab:remaining} \\
    $(C0g)$ & $\mathfrak{T}$ in Lem.~\ref{lem:ProductCancellations}&
      $(C2g)$ & $\Graphs_1$ in Lem.~\ref{lem:ProductCancellations}&
        $(2C+)^\times$ & ${\mathfrak R}$ in Lem.~\ref{lem:ProductCancellations} \\
    $(C0-)$ & Cancels $(C0+)$ &
      $(C2-)$ & $\Graphs_1$ in Lem.~\ref{lem:ProductCancellations}&
        $(2C+)^0$ & See Table~\ref{tab:remaining} \\
    $(R0+)$ & See Table~\ref{tab:remaining} &
      $(R2+)$ & $\Graphs_1$ in Lem.~\ref{lem:ProductCancellations}&
        $(2R+)$ & ${\mathfrak R}$ in Lem.~\ref{lem:ProductCancellations} \\
    $(R0g)$ & $\mathfrak{T}$ in Lem.~\ref{lem:ProductCancellations}&
      $(R2g)$ & $\Graphs_1$ in Lem.~\ref{lem:ProductCancellations}&
        $(2R-)$ & See Table~\ref{tab:remaining} \\
    $(R0-)$ & Cancels $(L0+)$ &
      $(R2-)$ & $\Graphs_1$ in Lem.~\ref{lem:ProductCancellations} \\ \addlinespace
    $[LL+]$ & $\mathfrak{P}$ in Lem.~\ref{lem:ProductCancellations} &
      $[CL+]$ & $\mathfrak{P}$ in Lem.~\ref{lem:ProductCancellations}&
        $[RL+]$ & $\mathfrak{P}$ in Lem.~\ref{lem:ProductCancellations}\\
    $[LLg]$ & $\mathfrak{P}$ in Lem.~\ref{lem:ProductCancellations} &
      $[CLg]$ & $\mathfrak{P}$ in Lem.~\ref{lem:ProductCancellations} &
        $[RLg]$ & $\mathfrak{P}$ in Lem.~\ref{lem:ProductCancellations} \\
    $[LL-]$ & See Table~\ref{tab:remaining} &
      $[CL-]$ & $\Graphs_2$ in Lem.~\ref{lem:FirstInvolution} &
        $[RL-]$ & $\Graphs_2$ in Lem.~\ref{lem:FirstInvolution} \\
    $[LR+]$ & $\Graphs_2$ in Lem.~\ref{lem:FirstInvolution} &
      $[CR+]$ & $\Graphs_2$ in Lem.~\ref{lem:FirstInvolution} &
        $[RR+]$ & See Table~\ref{tab:remaining} \\
    $[LRg]$ & $\mathfrak{P}$ in Lem.~\ref{lem:ProductCancellations} &
      $[CRg]$ & $\mathfrak{P}$ in Lem.~\ref{lem:ProductCancellations} &
        $[RRg]$ & $\mathfrak{P}$ in Lem.~\ref{lem:ProductCancellations}\\
    $[LR-]$ & $\mathfrak{P}$ in Lem.~\ref{lem:ProductCancellations} &
      $[CR-]$ & $\mathfrak{P}$ in Lem.~\ref{lem:ProductCancellations} &
        $[RR-]$ & $\mathfrak{P}$ in Lem.~\ref{lem:ProductCancellations} \\
    \bottomrule
  \end{tabular}
  \caption{Types of terms in $\Ainf$ relations not involving idempotents, and how they cancel.}
  \label{tab:Ainf-terms}
\end{table}

  Terms of types $(L0+)$ and $(R0-)$ cancel, as follows.
A term of type $(L0+)$ is determined by a left-extended tiling pattern $\Gamma$,
which has a
string of at least $1$ but at most $3$ distinguished 2-valent sectors after
the root vertex.
The corresponding term of type $(R0-)$  is obtained from
$\Gamma$ by moving the root vertex of $\Gamma$ to the next position (with respect to the boundary orientation) of the intersection of $\Gamma$ with $\partial \CDisk$,  and moving around the valence $2$ vertices
as needed, as illustrated in Figure~\ref{fig:LeftRightC}.
For example, in the weight $1$ $\Ainf$ relation with inputs
$\rho_3\otimes\rho_2\otimes\rho_1$, there are two non-zero terms, of
type $(L0+)$ and $(R0-)$, respectively:
  \begin{align*}
    \mu^0_4(\mu_0^1,\rho_3,\rho_2,\rho_1)&=\mu^0_4(\rho_{1234},\rho_3,\rho_2,\rho_1)
    = U \rho_{123}\\
    \mu^0_4(\rho_3,\rho_2,\rho_1,\mu_0^1) &=
       \mu^0_4(\rho_3,\rho_2,\rho_1,\rho_{4123}) = U \rho_{123}.
  \end{align*}
Another pair of cancelling terms of type $(L0+)$ and $(R0-)$ is given by
\[ \mu^0_6(\rho_{3412},\rho_1,\rho_4,\rho_{34},\rho_3,\rho_2)=U^2 \rho_{34}=\mu^0_6(\rho_1, 
\rho_4, \rho_{34},\rho_3,\rho_2,\rho_{1234}).\]

There is a similar cancellation of terms of types~(C0+) and (C0-), as shown in the last picture of Figure~\ref{fig:LeftRightC}.
An example is provided by the cancellation of 
\[ \mu^0_{10}(\rho_{1234},\rho_3,\rho_2,\rho_{12},\rho_{1},\rho_{41},\rho_4,\rho_{34},\rho_3,\rho_2)=U^4\cdot \iota_0
=\mu^0_{10}(\rho_3,\rho_2,\rho_{12},\rho_{1},\rho_{41},\rho_4,\rho_{34},\rho_3,\rho_2,\rho_{1234}). \]

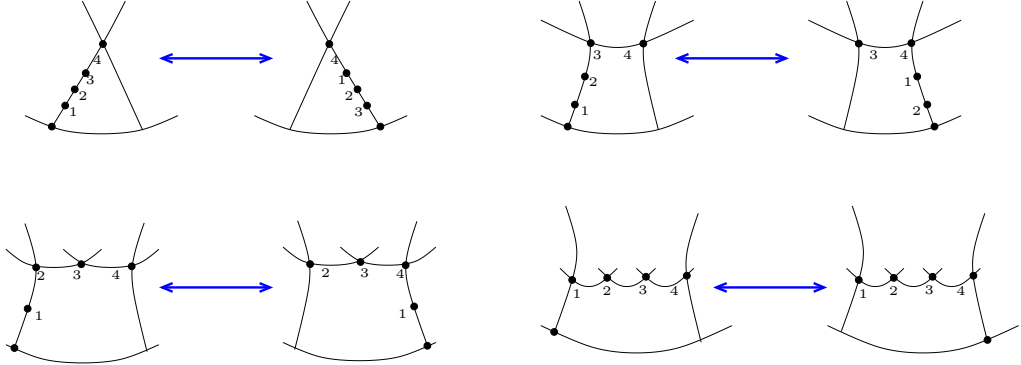
\begin{figure}
\centering
\input{LeftRightC}
\caption{{\bf Cancellation of type~$(L0+)$ and~$(R0-)$; and $(C0+)$ and $(C0-)$}.
    The first three pictures illustrate cancellations of $(L0+)$ and $(R0-)$, while the last
    one is a cancellation of $(C0+)$ with $(C0-)$. For each picture
    there are three other cases obtained by cyclically permuting the subscripts.
}
\label{fig:LeftRightC} 
\end{figure}

The
patterns of type $\Graphs_1$ from Lemma~\ref{lem:ProductCancellations}
contribute terms of types
$(\wild2\wild)$. According to that lemma, these cancel against the following types of terms:
\begin{itemize}
  \item 
    generic composition patterns (i.e., the terms in ${\mathfrak P}$) which, 
    by Lemma~\ref{lem:WhatIsGeneric}, are the terms of type 
    $[\wild\wild g]$, $[\wild L+]$,
    and $[\wild R-]$;
  \item 
    terms corresponding to configurations in $\mathfrak{T}$, which are
    of type $(\wild0g)$;
  \item 
    configurations in ${\mathfrak L}$ from the notation of
    Lemma~\ref{lem:ProductCancellations}, which in turn correspond to all terms
    of type $(2L-)$ and $(2C-)^\times$; and
  \item configurations in ${\mathfrak R}$, which in turn correspond to all terms of type 
    of type $(2R+)$ and $(2C+)^\times$.
\end{itemize}

For example, we have the following cancellations from 
Lemma~\ref{lem:ProductCancellations}.
The terms
\[
  \mu^0_8(\rho_4,\mu^0_2(\rho_3,\rho_{41}),\rho_4,\rho_3,\rho_{23},\rho_2,\rho_{12},\rho_1)
\qquad{\text{and}}\qquad
 \mu^0_4(\rho_4,\rho_3,\mu^0_6(\rho_{41},\rho_4,\rho_3,\rho_{23},\rho_2,\rho_{12}),\rho_1),
\]
are of types $(C2g)$ and $[CRg]$ respectively; they correspond to the diagrams
on the top row of Figure~\ref{fig:PushOutExamples}. Likewise,
\[
  \mu^1_8(\rho_{41},\rho_4,\mu^0_2(\rho_3,\rho_4),\rho_3,\rho_{23},\rho_2,\rho_{12},\rho_1)
  \qquad\text{and}\qquad
  \mu^0_{10}(\rho_{41},\rho_4,\rho_3,\mu^1_0,\rho_4,\rho_3,\rho_{23},\rho_2,\rho_{12},\rho_1)
\]
are of types $(C2g)$ and $(C0g)$ respectively; they correspond to the diagrams
on the middle row of Figure~\ref{fig:PushOutExamples}.
Finally, 
\[ \mu^1_6(\mu^0_2(\rho_{23},\rho_{4}),\rho_3,\rho_{23},\rho_2,\rho_1,\rho_{41})
\qquad\text{and}\qquad
\mu^0_2(\rho_{23},\mu^0_6(\rho_4,\rho_3,\rho_{23},\rho_2,\rho_4,\rho_{41}))
\]
are of types $(L2+)$ and $(2C-)^\times$ respectively; they correspond to the diagrams
from the bottom row of Figure~\ref{fig:PushOutExamples}.

By Lemma~\ref{lem:WhatIsGeneric}, the following terms correspond to
extremal composite patterns, all of whose distinguished $2$-valent
sectors lie in the region~$\Region$ from
Definition~\ref{def:CompositePattern}:
$[LR+]$, $[CR+]$, $[CL-]$, and $[RL-]$.
By Lemma~\ref{lem:FirstInvolution}, these terms cancel each other.
For example, the terms
\[
  \mu^0_6(\rho_1,\rho_4,\rho_{34},\rho_3,\rho_2,\mu^0_4(\rho_{123},\rho_2,\rho_1,\rho_4))
  \qquad{\text{and}}\qquad
  \mu^0_4(\mu^0_6(\rho_1,\rho_4,\rho_{34},\rho_3,\rho_2,\rho_{123}),\rho_2,\rho_1,\rho_4)
\]
are of types $[CL-]$ and $[CR+]$, and cancel by
Lemma~\ref{lem:FirstInvolution}; they correspond to the diagrams on
the top of Figure~\ref{fig:FirstInvolution}. Likewise,
\[
  \mu^0_6(\rho_1,\rho_4,\rho_{34},\rho_3,\rho_2,\mu^0_4(\rho_{1234},\rho_3,\rho_2,\rho_1))
  \qquad{\text{and}}\qquad
  \mu^0_4(\mu^0_6(\rho_1,\rho_4,\rho_{34},\rho_3,\rho_2,\rho_{1234}),\rho_3,\rho_2,\rho_1)
\]
are of types $[RL-]$ and $[LR+]$, and correspond to the diagrams on
the bottom of Figure~\ref{fig:FirstInvolution}.

The remaining possible terms are of types
$(2L+)$, $(2R-)$, $(2C+)^0$, $(2C-)^0$, $[LL-]$, and $[RR+]$.
For cancellations among these remaining terms, we take a closer look
at the sequence of algebra elements
$(\rho^1\otimes\dots\otimes \rho^n)$ entering the $\Ainf$ relation.
For all of these terms, we have that $\rho^i\cdot\rho^{i+1}=0$ for all
$i=1,\dots,n-1$; this follows from
Lemma~\ref{lem:alg-property-nonmult}.

We will now construct a bijection $M$ between terms of types~$[LL-]$,
$(L0-)$, $(2C-)^0$, and $(2R-)$ with terms of types~$[RR+]$, $(R0+)$, $(2C+)^0$, and
$(2L+)$, preserving the input sequence and output
element.  Indeed, for each given chord sequence
$(\rho^1\otimes\dots\otimes \rho^n)$ and output element
$U^n\otimes \rho$, exactly one of $[LL-]\cup(L0-)$, $(2C-)^0$, and
$(2R-)$ can be non-empty; similarly, exactly one of $[RR+]\cup (R0+)$,
$(2C+)^0$, and $(2L+)$ is non-empty. These possibilities are subdivided
according to these according to the relative lengths of $\rho$,
$\rho^1$, and $\rho^n$, as follows.

For a left-extended operation $[L\wild\wild]$ or $(L\wild\wild)$, the
output chord is shorter than the
first input; whereas for a $\mu^0_2$ operation $(2\wild-)$,
the output chord is
at least as long as the first input. Thus, 
if $|\rho|<|\rho^1|$, there can be no terms of type $(2\wild -)$; if
$|\rho|\geq |\rho^1|$, there can be no terms of types~$[LL-]$ or $(L0-)$.
Symmetrically, $|\rho|<|\rho^n|$ excludes terms of
type~$(2\wild +)$, and $|\rho|\geq |\rho^n|$ excludes terms of types~$[RR+]$ or $(R0+)$. Thus, the relative sizes of $(|\rho|,|\rho^1|)$ and
$(|\rho|,|\rho^n|)$ exclude all but the types of remaining terms as shown in Table~\ref{tab:remaining}.
\begin{table}
  \centering
  \begin{tabular}{r|ccc}
  \multicolumn{1}{r}{} & $|\rho|<|\rho^1|$ & $|\rho|=|\rho^1|$ & $|\rho|>|\rho^1|$ \\ \cmidrule{2-4}
  $|\rho|<|\rho^n|$ & $[LL-]$, $(L0-)$; $[RR+]$, $(R0+)$
            & $(2C-)^0$; $[RR+]$, $(R0+)$
              & $(2R-)$; $[RR+]$ $(R0+)$ \\ [2pt]
  $|\rho|=|\rho^n|$ & $[LL-]$, $(L0-)$; $(2C+)^0$
            & $(2C-)^0$; $(2C+)^0$
              & $(2R-)$; $(2C+)^0$\\[2pt]
  $|\rho|>|\rho^n|$ & $[LL-]$, $(L0-)$; $(2L+)$
            & $(2C-)^0$; $(2L+)$
                                                               & $(2R-)$; $(2L+)$\\[1em]
\end{tabular}
  \caption{Remaining cancellation of terms}
  \label{tab:remaining}
\end{table}
\begin{figure}
\centering
\input{2Rm}
\caption{{\bf Defining $M(2R-)$.}
  At the left are elements of $(2R-)$; the image of the term under $M$
  is indicated on the right. In the first line, $\rho=\rho_{1234}$; in
  the second, $\rho=\rho_{234}$; and in the third, $\rho=\rho_{34}$. We have also listed the types of the result.
  In all cases, $\rho^n=\rho_{234}$.
\label{fig:2Rm}}
\end{figure}
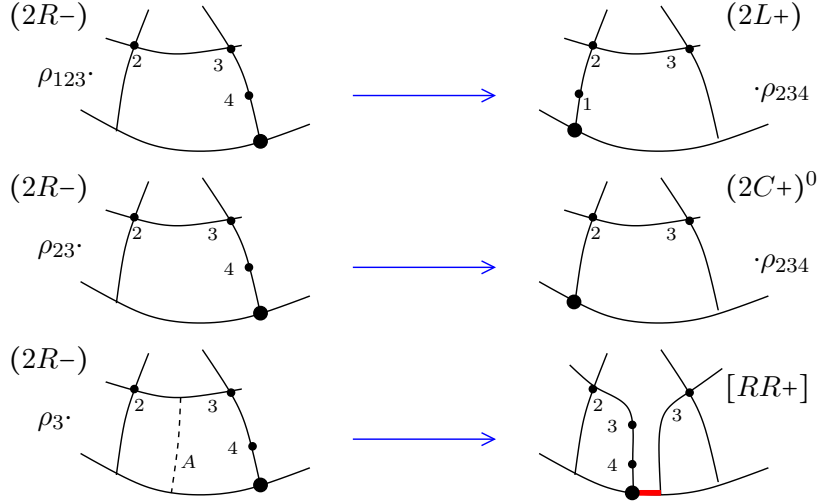

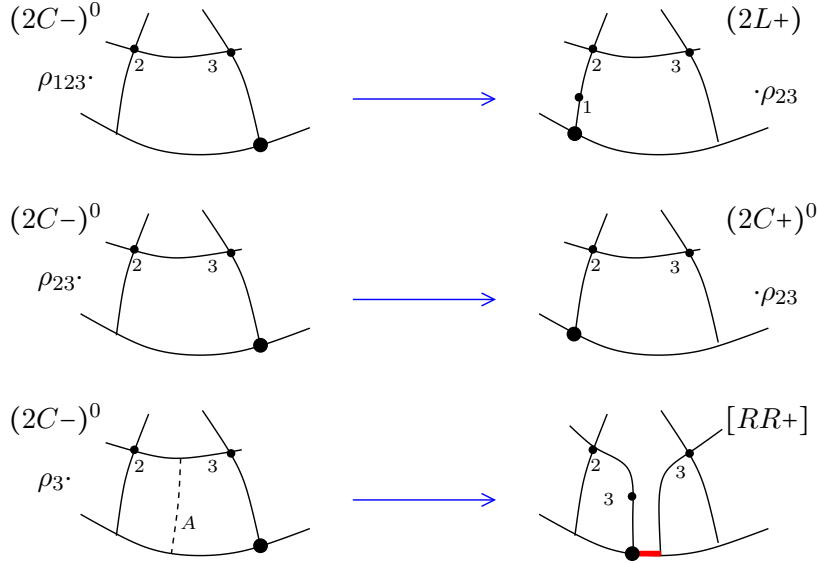
\begin{figure}
\centering
\input{2Cm}
\caption{{\bf Defining $M(2C-)^0$.}
  \label{fig:2Cm}
The element of type $(2C-)^0$ is on the left column. 
In the first line, $\rho=\rho_{123}$;
in the second, $\rho_{23}$; and in the third, $\rho=\rho_3$.
In all cases,
$\rho^n = \rho_{23}$.}
\end{figure}

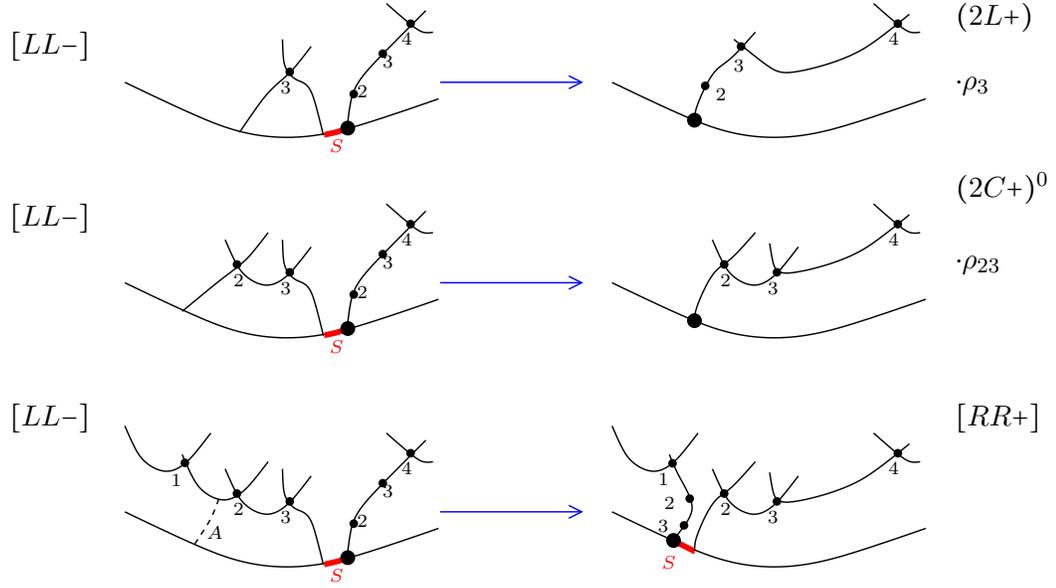
\begin{figure}
\centering
\input{LLm}
\caption{{\bf Defining $M[LL-]$.}
\label{fig:LLm} In these pictures, $\rho=\rho_{23}$.}
\end{figure}

The map $M$ is defined as follows. Suppose the input to $M$ is an
element of type $(2C-)^0$ or $(2R-)$ (so that $|\rho|\geq
|\rho^1|$). Let $\Gamma$ be the tiling pattern appearing in the
configuration. When $|\rho|\geq |\rho^n|$, the map $M$ gives a result
of type $(2C+)^0$ or $(2L+)$, whose tiling pattern is obtained from
$\Gamma$ by moving the root vertex one spot back (with respect to the
boundary orientation) and reshuffling the $2$-valent vertices, as
needed.
When $|\rho|<|\rho^n|$, there is an arc $A$ that connects some edge in
$\Gamma$ to the a point on $\partial\CDisk$, corresponding to the
factorization $\rho^n=\tau\cdot\rho$ for some suitable choice of
$\tau$ (in the sense that one endpoint of $A$ is on the edge
corresponding to the factorization). In that case, $M$ is defined by
pushing $A$ to the boundary, moving the root vertex back, and
reshuffling the $2$-valent vertices. In that case, the result
of $M$ is a term of type $[RR+]$ or $(R0+)$ (depending on whether or
not the edge pushed out disconnects).
In all these cases, the ability to do the necessary reshuffling
depends on the fact that $\rho^1 \cdot \rho^2 = 0$.
See Figures~\ref{fig:2Rm} and
\ref{fig:2Cm} for illustrations.

Starting from a composite pattern~$\Gamma$ of type $[LL-]$, we proceed
similarly. Provided that $|\rho|\geq |\rho^n|$,  we pull in the
arc~$S$
and then move the root vertex one spot backwards,
to find cancelling terms of type $(2C+)^0$ or $(2R+)$.
When $|\rho|<|\rho^n|$, we find an arc $A$ to push out to the boundary (again, as before). 
Now, when we
pull in $S$ and push out along $A$, we obtain the cancelling term of
type $[RR+]$ or $(R0+)$.
When the input is a configuration of $(L0-)$, we perform the same
operations as in the case of $[LL-]$, thinking of $S$ as the last interval before the root vertex.
See Figure~\ref{fig:LLm}.
Examples are provided by the cancelling terms
\begin{equation}
  \label{eq:LLm2Cp}
  \mu^0_4(\rho_{34},\rho_3,\rho_2,\mu^0_4(\rho_{12},\rho_{1},\rho_{4},\rho_{3})) 
\qquad{\text{and}}\qquad
\mu^0_2(\mu^0_6(\rho_{34},\rho_{3},\rho_{2},\rho_{12},\rho_{1},\rho_{4}),\rho_{3}) 
\end{equation}
which are of type
$[LL-]$ and $(2C+)^0$ respectively. Cancelling terms
\begin{equation}
  \label{eq:LLm2Lp}
  \mu^0_4(\rho_{234},\rho_{3},\rho_{2},\mu^0_4(\rho_{12},\rho_{1},\rho_{4},\rho_{3})) 
\qquad{\text{and}}\qquad
\mu^0_2(\mu^0_6(\rho_{234},\rho_{3},\rho_{2},\rho_{12},\rho_{1},\rho_{4}),\rho_{3}) 
\end{equation}
are of types $[LL-]$ and $(2L+)$ respectively.
Finally, terms
\begin{equation}
  \label{eq:LLRRex}
  \mu^0_4(\rho_{34},\rho_3,\rho_2,\mu^0_6(\rho_{12},\rho_1,\rho_{41},\rho_4,\rho_3,\rho_{23})) 
\qquad{\text{and}}\qquad
 \mu^0_4(\mu^0_6(\rho_{34},\rho_{3},\rho_{2},\rho_{12},\rho_{1},\rho_{41}),\rho_{4},\rho_{3},\rho_{23}) 
\end{equation}
are of types  $[LL-]$ and $[RR+]$ respectively. 
The map $M$ in these three cases is illustrated in
Figure~\ref{fig:LLRRex}.
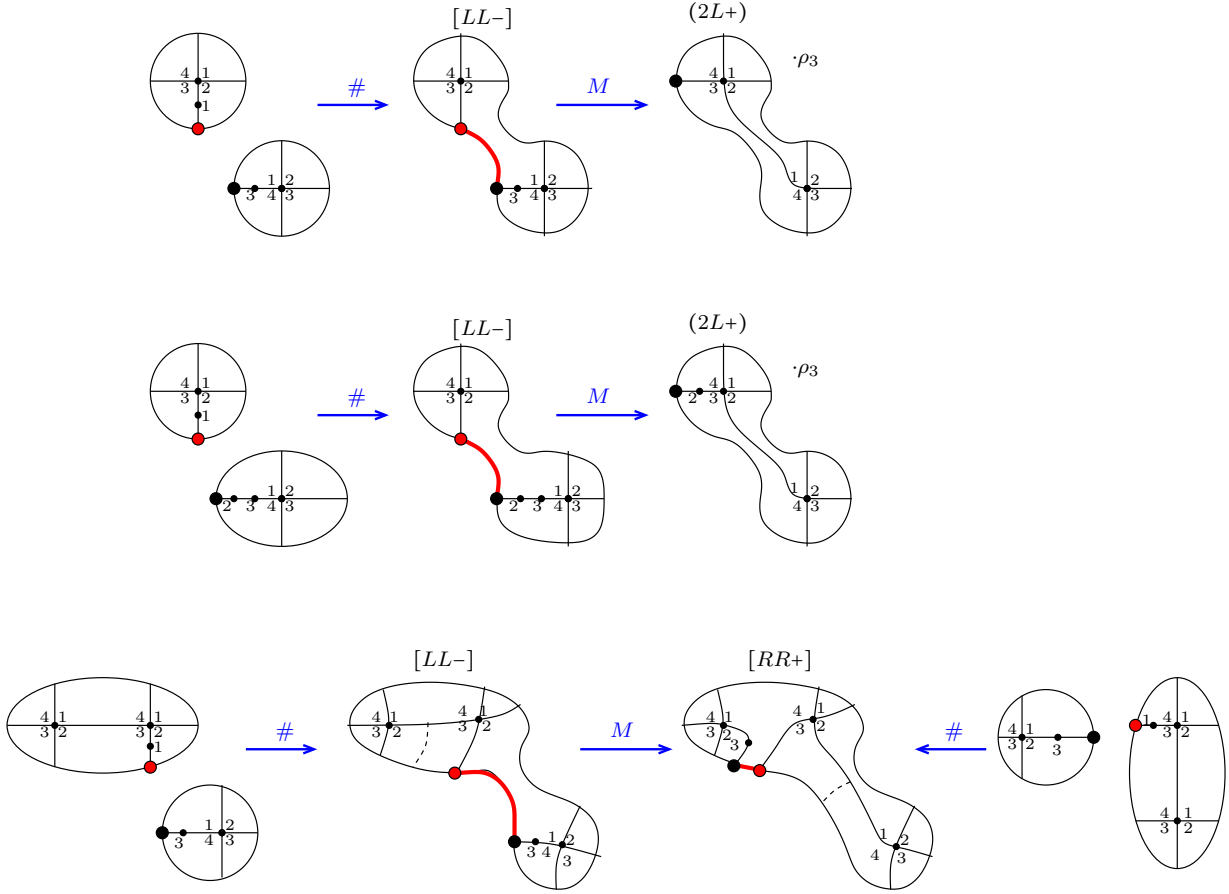
\begin{figure}
  \centering
  \input{LLRRex}
  \caption{\textbf{Cancellation of terms of type $[LL-]$.}
    The three lines represent the cancellations from Equation~\eqref{eq:LLm2Cp},~\eqref{eq:LLm2Lp},
    and~\eqref{eq:LLRRex} respectively.}
  \label{fig:LLRRex}
\end{figure}

An inverse to $M$ is constructed by pushing out along an arc $A$ and
then moving the vertex one spot forwards. In cases where
$|\rho|<|\rho^n|$, we must first pull an interval on the boundary in
before pushing out along $A$: when the configuration is of type
$[RR+]$, the interval to be pulled in is the interval $S$; when the
configuration is of type $(R0+)$, the interval to be pulled in is the first interval
after the root vertex. 
\end{proof}

\subsection{First properties of \texorpdfstring{$\MAlg$}{the weighted A-infinity
    algebra of the torus}}

\begin{lemma}
  \label{lem:WeightGrading}
  If $a^1,\dots,a^k$ are basic algebra elements, then for every
  non-zero operation,
  \[ |\mu^w_k(a^1\otimes\dots\otimes a^k)|=4w+\sum_{i=1}^k|a^i|.\]  
\end{lemma}
\begin{proof}
  The cases $\mu_2^0$ and $\mu_0^1$ are immediate. Otherwise,
  let $\Gamma$ be a graph that contributes to the operation.
  From Definition~\ref{def:MAlg}, we see that $|\mu^w_k(a^1\otimes\dots\otimes a^k)|$
  is the total number of distinguished sectors in $\Gamma$.
  On the other hand, the number of distinguished sectors in $\Gamma$ is given by
  $4w + \sum_{i=1}^k|a^i|$, since
  the sectors visible from the boundary are the ones in $\sum_{i=1}^k |a^i|$, and 
  the ones that are not occur $4$ times in each short cycle, and there are
  $w$ short cycles.
\end{proof}

The following properties of $\MAlg$ will be useful when studying the
gradings:

\begin{lemma}\label{lem:alg-property-factor} 
  For any $n+2w>4$, if $\rho^1,\dots,\rho^n$ are Reeb elements with
  $\mu_n^w(\rho^1,\dots,\rho^n)\neq 0$ then there exists an~$i$ so that
  $\rho^i$ factors nontrivially (i.e., $\rho^i=\rho\rho'$ for some
  Reeb elements $\rho$, $\rho'$). In fact, there are at least two
  such integers $i$.
\end{lemma}
\begin{proof}
  If $n+2w>4$, the tile pattern contains at least $2$ internal
  vertices.
  There is some internal edge connecting $2$ distinct vertices that is visible from the boundary.
  That edge corresponds to a factorization of
  $\rho^i=\rho\rho'$. Further, there is either more than one such
  internal edge or there is one internal edge visible from the
  boundary in two different places; so, there are at least two
  integers $i$ so that $\rho^i$ factors.
\end{proof}

\begin{lemma}\label{lem:alg-diet}
  Suppose that $\rho^1,\dots,\rho^n$ are Reeb elements so that
  $\mu_n^w(\rho^1,\dots,\rho^n)=b\neq 0$ for some $w>0$. Then there is
  an $i$ and a factorization $\rho^i=\rho\rho'$ so that
  $\mu_{n+2}^{w-1}(\rho^1,\dots,\rho^{i-1},\rho,\mu_0^1,\rho',\rho^{i+1},\dots,\rho^n)$
  has $b$ as a term.
\end{lemma}
\begin{proof}
  Some edge on a short cycle
  is visible from the boundary.
  That edge corresponds to the desired factorization.
\end{proof}

Finally, we note a boundedness property of the algebra, which is
useful for working with twisted complexes (type $D$ structures) over
it, and for defining certain kinds of bimodules over it.

\begin{definition}
  Recall that a weighted tree $T$ with $n$ inputs,
  $i$ internal vertices, and total weight $w$ has a dimension
  $\dim(T)=n-1+2w-i$ (Equation~\eqref{eq:tree-dim}).
  A weighted algebra $\Alg$ is called {\em bonsai} if 
  there is an integer $N$ with the property that 
  for all weighted trees $T$ with $\dim(T)>N$, $\mu(T)=0$.
  A weighted algebra $\Alg$ over $\Field[U]$ is called {\em
    pre-filtered bonsai} if, for every~$m$, the quotient
  $\Alg/U^m\Alg$ is bonsai.
\end{definition}

If $\Alg$ is  pre-filtered bonsai, then its completion with respect to the sequence of 
(weighted $\Ainf$) ideals 
\[ U \Alg\supset U^2\Alg \supset U^3 \Alg \supset \dots \]
is  filtered bonsai, in the sense of~\cite[Definition~\ref*{AbstDiag:def:filtered-bonsai}]{LOT:abstract}.

The algebra $\MAlg$ constructed here satisfies these hypotheses, according to the following:

\begin{proposition}
\label{lem:Bonsai}
If $T$ is a weight $w$ operation tree with $n$ inputs then $\mu(T)$
maps $A^{\otimes n}$ into $U^{\dim(T)/2} A$.
In particular, the weighted algebra $\MAlg$ is pre-filtered bonsai.
\end{proposition}
\begin{proof}
  We start by verifying that $\mu^w_n$ maps $A^{\otimes n}$ to $U^{w+n/2-1} A$; 
  this verifies the lemma when 
  $T=\wcorolla{n}{w}$.  This is vacuous for $\mu^0_2$ and $\mu^1_0$.
  Let $\Gamma$ be a tiling pattern that contributes to~$\mu^w_n$. 
  $e$ the number of edges of $\Gamma$, $v$ the number
  of vertices, $x$ the number of valence $2$ internal
  vertices,
  and $d$ the number of valence $4$ internal vertices.
  Then $v=n+x+d$, $2e=2x+4d+n$,
  and the Euler characteristic of $\Gamma$ is $1-w$; so
  $1-w=(n+x+d)-(x+2d+n/2)=n/2-d$; and the operation contributes $U^d =
  U^{w+n/2-1}$, as desired.

  Suppose now that $T$ is a weighted operation tree with $n$~inputs, $e$~edges,
  $i$~internal vertices,
  and weight~$w$, where an internal vertex~$v$ has $n_v$ inputs,
  weight~$w_v$, and contribution $u_v$ to the $U$ power.
  It is elementary to see that
  \[
    n-1 = \sum_{\text{internal vertices $v$}} (n_v-1).
  \]
  Since we have already verified that $u_v = w_v + (n_v-1)/2 - 1/2$,
  it follows that the total $U$ contribution is
  \[
    \sum u_v = w + \frac{n-1}{2} - \frac{i}{2} = \frac{\dim(T)}{2}. \qedhere
  \]
\end{proof}

\subsection{From tiling patterns to immersions}
\label{subsec:Immersions}

The weighted algebra operations have an interpretation in terms of
immersions of the disk into the torus, as follows.

Mark the torus, drawing a pair of curves $\alpha_1$, $\alpha_2$ on
the torus so that $\alpha_1\pitchfork \alpha_2$ in a single point $p$.
Label the four corners near $p$ by $1,2,3,4$ in \emph{clockwise} order
around $p$. (The reason for this ordering is that we will sometimes
think of $p$ as a puncture; then 
this is the orientation induced on the circle, thought of as 
the circle at infinity on  $T^2\setminus\{p\}$.)  Cutting along
$\alpha_1\cup\alpha_2$, we obtain a square with opposite sides
identified, as shown in Figure~\ref{fig:torus}, whose four
corners are labelled $1,\dots,4$.

Given a centered tiling pattern $(\Gamma,\Lambda)$, there is an
immersion of the disk to the torus (with possible branching at the corners),
defined as follows.  To each vertex $v$ in $\Gamma$, we associate a
copy of the standard square $S(v)$ with labelled corners (as in
Figure~\ref{fig:torus}). If $\Gamma$ has an edge from $v_1$ to $v_2$,
we perform an identification of $S(v_1)$ with $S(v_2)$ along a shared
edge. The result of these identifications is a topological disk
$\Delta(\Gamma)$, equipped with a tiling by squares. The map which
projects each tile to $T^2$ induces a map from $\Delta(\Gamma)$ to
$T^2$, which is an immersion away from the corners of
$\Delta(\Gamma)$.  

An inverse operation is given as follows. Given a square-tiled disk
with an immersion as above, let $\alpha_1^\vee$
(respectively $\alpha_2^\vee$) be a disjoint, isotopic translate of
$\alpha_2$ (respectively $\alpha_1$) so that
$\alpha_i\pitchfork \alpha_i^\vee$ and
$\alpha_1^\vee\pitchfork \alpha_2^\vee$ in a single point each. Then
$\Gamma$ is the preimage of $\alpha_1^\vee\cup\alpha_2^\vee$ under the
map $u\co \CDisk\to T^2$.  

If the tiling pattern has $d$ internal vertices, then the degree of
the corresponding immersion is $d/4$.  The weight of the operation is
the number of preimages of $p$ in the interior of $\Delta(\Gamma)$.

This compelling geometric interpretation of the $\Ainf$ operations,
as counts of immersions, is closely connected to pseudo-holomorphic
curve theory; but we will not make further use of it in this paper.


%% file: draws/ValidLabels.tex
\begin{picture}(0,0)%
\includegraphics{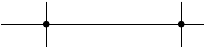}%
\end{picture}%
\setlength{\unitlength}{1184sp}%
\begin{picture}(5466,1266)(-32,-6994)
\put(1351,-6811){\makebox(0,0)[lb]{\smash{\fontsize{7}{8.4}\usefont{T1}{ptm}{m}{n}{\color[rgb]{0,0,0}$i+2$}%
}}}
\put(1051,-6811){\makebox(0,0)[rb]{\smash{\fontsize{7}{8.4}\usefont{T1}{ptm}{m}{n}{\color[rgb]{0,0,0}$i+3$}%
}}}
\put(1351,-6136){\makebox(0,0)[lb]{\smash{\fontsize{7}{8.4}\usefont{T1}{ptm}{m}{n}{\color[rgb]{0,0,0}$i+1$}%
}}}
\put(1051,-6136){\makebox(0,0)[rb]{\smash{\fontsize{7}{8.4}\usefont{T1}{ptm}{m}{n}{\color[rgb]{0,0,0}$i$}%
}}}
\put(4651,-6136){\makebox(0,0)[rb]{\smash{\fontsize{7}{8.4}\usefont{T1}{ptm}{m}{n}{\color[rgb]{0,0,0}$i$}%
}}}
\put(4651,-6811){\makebox(0,0)[rb]{\smash{\fontsize{7}{8.4}\usefont{T1}{ptm}{m}{n}{\color[rgb]{0,0,0}$i+3$}%
}}}
\end{picture}%

%% file: draws/TilingPatterns.tex
\begin{picture}(0,0)%
\includegraphics{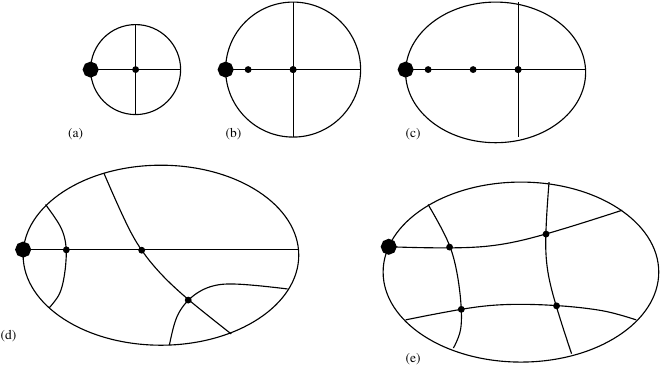}%
\end{picture}%
\setlength{\unitlength}{1184sp}%
\begingroup\makeatletter\ifx\SetFigFont\undefined%
\gdef\SetFigFont#1#2#3#4#5{%
  \reset@font\fontsize{#1}{#2pt}%
  \fontfamily{#3}\fontseries{#4}\fontshape{#5}%
  \selectfont}%
\fi\endgroup%
\begin{picture}(17588,9756)(-614,-10084)
\put(11851,-2011){\makebox(0,0)[lb]{\smash{{\SetFigFont{7}{8.4}{\rmdefault}{\mddefault}{\updefault}{\color[rgb]{0,0,0}$2$}%
}}}}
\put(10651,-2011){\makebox(0,0)[lb]{\smash{{\SetFigFont{7}{8.4}{\rmdefault}{\mddefault}{\updefault}{\color[rgb]{0,0,0}$3$}%
}}}}
\put(6751,-2611){\makebox(0,0)[lb]{\smash{{\SetFigFont{7}{8.4}{\rmdefault}{\mddefault}{\updefault}{\color[rgb]{0,0,0}$4$}%
}}}}
\put(6751,-2011){\makebox(0,0)[lb]{\smash{{\SetFigFont{7}{8.4}{\rmdefault}{\mddefault}{\updefault}{\color[rgb]{0,0,0}$1$}%
}}}}
\put(7351,-2011){\makebox(0,0)[lb]{\smash{{\SetFigFont{7}{8.4}{\rmdefault}{\mddefault}{\updefault}{\color[rgb]{0,0,0}$2$}%
}}}}
\put(7351,-2611){\makebox(0,0)[lb]{\smash{{\SetFigFont{7}{8.4}{\rmdefault}{\mddefault}{\updefault}{\color[rgb]{0,0,0}$3$}%
}}}}
\put(13351,-2011){\makebox(0,0)[lb]{\smash{{\SetFigFont{7}{8.4}{\rmdefault}{\mddefault}{\updefault}{\color[rgb]{0,0,0}$2$}%
}}}}
\put(13351,-2611){\makebox(0,0)[lb]{\smash{{\SetFigFont{7}{8.4}{\rmdefault}{\mddefault}{\updefault}{\color[rgb]{0,0,0}$3$}%
}}}}
\put(12751,-2011){\makebox(0,0)[lb]{\smash{{\SetFigFont{7}{8.4}{\rmdefault}{\mddefault}{\updefault}{\color[rgb]{0,0,0}$1$}%
}}}}
\put(12751,-2611){\makebox(0,0)[lb]{\smash{{\SetFigFont{7}{8.4}{\rmdefault}{\mddefault}{\updefault}{\color[rgb]{0,0,0}$4$}%
}}}}
\put(2551,-2011){\makebox(0,0)[lb]{\smash{{\SetFigFont{7}{8.4}{\rmdefault}{\mddefault}{\updefault}{\color[rgb]{0,0,0}$1$}%
}}}}
\put(2551,-2611){\makebox(0,0)[lb]{\smash{{\SetFigFont{7}{8.4}{\rmdefault}{\mddefault}{\updefault}{\color[rgb]{0,0,0}$4$}%
}}}}
\put(3151,-2011){\makebox(0,0)[lb]{\smash{{\SetFigFont{7}{8.4}{\rmdefault}{\mddefault}{\updefault}{\color[rgb]{0,0,0}$2$}%
}}}}
\put(3151,-2611){\makebox(0,0)[lb]{\smash{{\SetFigFont{7}{8.4}{\rmdefault}{\mddefault}{\updefault}{\color[rgb]{0,0,0}$3$}%
}}}}
\put(11401,-6811){\makebox(0,0)[lb]{\smash{{\SetFigFont{7}{8.4}{\rmdefault}{\mddefault}{\updefault}{\color[rgb]{0,0,0}$2$}%
}}}}
\put(11551,-7336){\makebox(0,0)[lb]{\smash{{\SetFigFont{7}{8.4}{\rmdefault}{\mddefault}{\updefault}{\color[rgb]{0,0,0}$3$}%
}}}}
\put(13951,-6361){\makebox(0,0)[lb]{\smash{{\SetFigFont{7}{8.4}{\rmdefault}{\mddefault}{\updefault}{\color[rgb]{0,0,0}$2$}%
}}}}
\put(14026,-6886){\makebox(0,0)[lb]{\smash{{\SetFigFont{7}{8.4}{\rmdefault}{\mddefault}{\updefault}{\color[rgb]{0,0,0}$3$}%
}}}}
\put(11701,-8386){\makebox(0,0)[lb]{\smash{{\SetFigFont{7}{8.4}{\rmdefault}{\mddefault}{\updefault}{\color[rgb]{0,0,0}$2$}%
}}}}
\put(11776,-8911){\makebox(0,0)[lb]{\smash{{\SetFigFont{7}{8.4}{\rmdefault}{\mddefault}{\updefault}{\color[rgb]{0,0,0}$3$}%
}}}}
\put(10956,-7268){\makebox(0,0)[lb]{\smash{{\SetFigFont{7}{8.4}{\rmdefault}{\mddefault}{\updefault}{\color[rgb]{0,0,0}$4$}%
}}}}
\put(10854,-6818){\makebox(0,0)[lb]{\smash{{\SetFigFont{7}{8.4}{\rmdefault}{\mddefault}{\updefault}{\color[rgb]{0,0,0}$1$}%
}}}}
\put(11243,-8467){\makebox(0,0)[lb]{\smash{{\SetFigFont{7}{8.4}{\rmdefault}{\mddefault}{\updefault}{\color[rgb]{0,0,0}$1$}%
}}}}
\put(11228,-8918){\makebox(0,0)[lb]{\smash{{\SetFigFont{7}{8.4}{\rmdefault}{\mddefault}{\updefault}{\color[rgb]{0,0,0}$4$}%
}}}}
\put(13736,-8358){\makebox(0,0)[lb]{\smash{{\SetFigFont{7}{8.4}{\rmdefault}{\mddefault}{\updefault}{\color[rgb]{0,0,0}$1$}%
}}}}
\put(13764,-8816){\makebox(0,0)[lb]{\smash{{\SetFigFont{7}{8.4}{\rmdefault}{\mddefault}{\updefault}{\color[rgb]{0,0,0}$4$}%
}}}}
\put(14297,-8393){\makebox(0,0)[lb]{\smash{{\SetFigFont{7}{8.4}{\rmdefault}{\mddefault}{\updefault}{\color[rgb]{0,0,0}$2$}%
}}}}
\put(14386,-8905){\makebox(0,0)[lb]{\smash{{\SetFigFont{7}{8.4}{\rmdefault}{\mddefault}{\updefault}{\color[rgb]{0,0,0}$3$}%
}}}}
\put(13404,-6968){\makebox(0,0)[lb]{\smash{{\SetFigFont{7}{8.4}{\rmdefault}{\mddefault}{\updefault}{\color[rgb]{0,0,0}$4$}%
}}}}
\put(13506,-6471){\makebox(0,0)[lb]{\smash{{\SetFigFont{7}{8.4}{\rmdefault}{\mddefault}{\updefault}{\color[rgb]{0,0,0}$1$}%
}}}}
\put(5831,-2601){\makebox(0,0)[lb]{\smash{{\SetFigFont{7}{8.4}{\rmdefault}{\mddefault}{\updefault}{\color[rgb]{0,0,0}$3$}%
}}}}
\put(4201,-8911){\makebox(0,0)[lb]{\smash{{\SetFigFont{7}{8.4}{\rmdefault}{\mddefault}{\updefault}{\color[rgb]{0,0,0}$4$}%
}}}}
\put(4201,-8011){\makebox(0,0)[lb]{\smash{{\SetFigFont{7}{8.4}{\rmdefault}{\mddefault}{\updefault}{\color[rgb]{0,0,0}$2$}%
}}}}
\put(3151,-6811){\makebox(0,0)[lb]{\smash{{\SetFigFont{7}{8.4}{\rmdefault}{\mddefault}{\updefault}{\color[rgb]{0,0,0}$2$}%
}}}}
\put(2551,-6811){\makebox(0,0)[lb]{\smash{{\SetFigFont{7}{8.4}{\rmdefault}{\mddefault}{\updefault}{\color[rgb]{0,0,0}$1$}%
}}}}
\put(2851,-7411){\makebox(0,0)[lb]{\smash{{\SetFigFont{7}{8.4}{\rmdefault}{\mddefault}{\updefault}{\color[rgb]{0,0,0}$4$}%
}}}}
\put(3451,-7261){\makebox(0,0)[lb]{\smash{{\SetFigFont{7}{8.4}{\rmdefault}{\mddefault}{\updefault}{\color[rgb]{0,0,0}$3$}%
}}}}
\put(3901,-8461){\makebox(0,0)[lb]{\smash{{\SetFigFont{7}{8.4}{\rmdefault}{\mddefault}{\updefault}{\color[rgb]{0,0,0}$1$}%
}}}}
\put(4651,-8461){\makebox(0,0)[lb]{\smash{{\SetFigFont{7}{8.4}{\rmdefault}{\mddefault}{\updefault}{\color[rgb]{0,0,0}$3$}%
}}}}
\put(751,-7411){\makebox(0,0)[lb]{\smash{{\SetFigFont{7}{8.4}{\rmdefault}{\mddefault}{\updefault}{\color[rgb]{0,0,0}$4$}%
}}}}
\put(751,-6811){\makebox(0,0)[lb]{\smash{{\SetFigFont{7}{8.4}{\rmdefault}{\mddefault}{\updefault}{\color[rgb]{0,0,0}$1$}%
}}}}
\put(1201,-6811){\makebox(0,0)[lb]{\smash{{\SetFigFont{7}{8.4}{\rmdefault}{\mddefault}{\updefault}{\color[rgb]{0,0,0}$2$}%
}}}}
\put(1201,-7411){\makebox(0,0)[lb]{\smash{{\SetFigFont{7}{8.4}{\rmdefault}{\mddefault}{\updefault}{\color[rgb]{0,0,0}$3$}%
}}}}
\end{picture}%

%% file: draws/CompositePattern.tex
\begin{picture}(0,0)%
\includegraphics{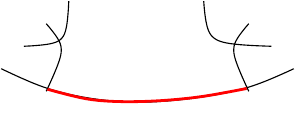}%
\end{picture}%
\setlength{\unitlength}{1184sp}%
\begin{picture}(7866,3499)(2968,-9227)
\put(6334,-9068){\makebox(0,0)[lb]{\smash{\fontsize{11}{13.2}\usefont{T1}{ptm}{m}{n}{\color[rgb]{1,0,0}$S$}%
}}}
\put(4724,-6968){\makebox(0,0)[lb]{\smash{\fontsize{7}{8.4}\usefont{T1}{ptm}{m}{n}{\color[rgb]{0,0,0}$i$}%
}}}
\put(8394,-7188){\makebox(0,0)[lb]{\smash{\fontsize{7}{8.4}\usefont{T1}{ptm}{m}{n}{\color[rgb]{0,0,0}$i-1$}%
}}}
\end{picture}%

%% file: draws/Composition.tex
\begin{picture}(0,0)%
\includegraphics{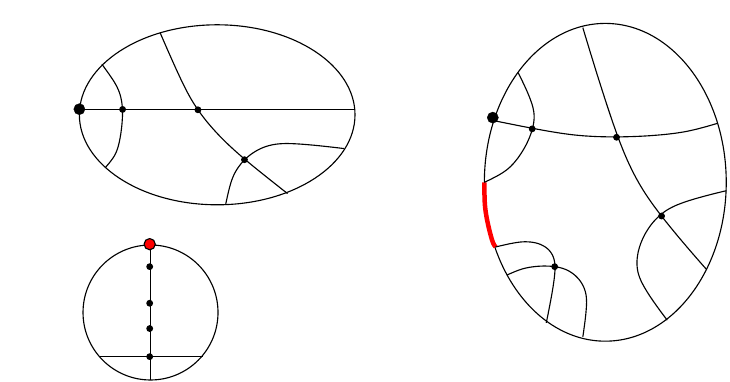}%
\end{picture}%
\setlength{\unitlength}{1184sp}%
\begingroup\makeatletter\ifx\SetFigFont\undefined%
\gdef\SetFigFont#1#2#3#4#5{%
  \reset@font\fontsize{#1}{#2pt}%
  \fontfamily{#3}\fontseries{#4}\fontshape{#5}%
  \selectfont}%
\fi\endgroup%
\begin{picture}(19398,10114)(-2114,-14216)
\put(4201,-8911){\makebox(0,0)[lb]{\smash{{\SetFigFont{7}{8.4}{\rmdefault}{\mddefault}{\updefault}{\color[rgb]{0,0,0}$4$}%
}}}}
\put(4201,-8011){\makebox(0,0)[lb]{\smash{{\SetFigFont{7}{8.4}{\rmdefault}{\mddefault}{\updefault}{\color[rgb]{0,0,0}$2$}%
}}}}
\put(3151,-6811){\makebox(0,0)[lb]{\smash{{\SetFigFont{7}{8.4}{\rmdefault}{\mddefault}{\updefault}{\color[rgb]{0,0,0}$2$}%
}}}}
\put(2551,-6811){\makebox(0,0)[lb]{\smash{{\SetFigFont{7}{8.4}{\rmdefault}{\mddefault}{\updefault}{\color[rgb]{0,0,0}$1$}%
}}}}
\put(2851,-7411){\makebox(0,0)[lb]{\smash{{\SetFigFont{7}{8.4}{\rmdefault}{\mddefault}{\updefault}{\color[rgb]{0,0,0}$4$}%
}}}}
\put(3451,-7261){\makebox(0,0)[lb]{\smash{{\SetFigFont{7}{8.4}{\rmdefault}{\mddefault}{\updefault}{\color[rgb]{0,0,0}$3$}%
}}}}
\put(3901,-8461){\makebox(0,0)[lb]{\smash{{\SetFigFont{7}{8.4}{\rmdefault}{\mddefault}{\updefault}{\color[rgb]{0,0,0}$1$}%
}}}}
\put(4651,-8461){\makebox(0,0)[lb]{\smash{{\SetFigFont{7}{8.4}{\rmdefault}{\mddefault}{\updefault}{\color[rgb]{0,0,0}$3$}%
}}}}
\put(751,-7411){\makebox(0,0)[lb]{\smash{{\SetFigFont{7}{8.4}{\rmdefault}{\mddefault}{\updefault}{\color[rgb]{0,0,0}$4$}%
}}}}
\put(751,-6811){\makebox(0,0)[lb]{\smash{{\SetFigFont{7}{8.4}{\rmdefault}{\mddefault}{\updefault}{\color[rgb]{0,0,0}$1$}%
}}}}
\put(1201,-6811){\makebox(0,0)[lb]{\smash{{\SetFigFont{7}{8.4}{\rmdefault}{\mddefault}{\updefault}{\color[rgb]{0,0,0}$2$}%
}}}}
\put(1201,-7411){\makebox(0,0)[lb]{\smash{{\SetFigFont{7}{8.4}{\rmdefault}{\mddefault}{\updefault}{\color[rgb]{0,0,0}$3$}%
}}}}
\put(11701,-7336){\makebox(0,0)[lb]{\smash{{\SetFigFont{7}{8.4}{\rmdefault}{\mddefault}{\updefault}{\color[rgb]{0,0,0}$1$}%
}}}}
\put(12226,-7411){\makebox(0,0)[lb]{\smash{{\SetFigFont{7}{8.4}{\rmdefault}{\mddefault}{\updefault}{\color[rgb]{0,0,0}$2$}%
}}}}
\put(11476,-7786){\makebox(0,0)[lb]{\smash{{\SetFigFont{7}{8.4}{\rmdefault}{\mddefault}{\updefault}{\color[rgb]{0,0,0}$4$}%
}}}}
\put(12151,-7936){\makebox(0,0)[lb]{\smash{{\SetFigFont{7}{8.4}{\rmdefault}{\mddefault}{\updefault}{\color[rgb]{0,0,0}$3$}%
}}}}
\put(13876,-7636){\makebox(0,0)[lb]{\smash{{\SetFigFont{7}{8.4}{\rmdefault}{\mddefault}{\updefault}{\color[rgb]{0,0,0}$1$}%
}}}}
\put(14401,-7561){\makebox(0,0)[lb]{\smash{{\SetFigFont{7}{8.4}{\rmdefault}{\mddefault}{\updefault}{\color[rgb]{0,0,0}$2$}%
}}}}
\put(13876,-8086){\makebox(0,0)[lb]{\smash{{\SetFigFont{7}{8.4}{\rmdefault}{\mddefault}{\updefault}{\color[rgb]{0,0,0}$4$}%
}}}}
\put(14551,-8086){\makebox(0,0)[lb]{\smash{{\SetFigFont{7}{8.4}{\rmdefault}{\mddefault}{\updefault}{\color[rgb]{0,0,0}$3$}%
}}}}
\put(15301,-9586){\makebox(0,0)[lb]{\smash{{\SetFigFont{7}{8.4}{\rmdefault}{\mddefault}{\updefault}{\color[rgb]{0,0,0}$2$}%
}}}}
\put(14776,-9961){\makebox(0,0)[lb]{\smash{{\SetFigFont{7}{8.4}{\rmdefault}{\mddefault}{\updefault}{\color[rgb]{0,0,0}$1$}%
}}}}
\put(15226,-10411){\makebox(0,0)[lb]{\smash{{\SetFigFont{7}{8.4}{\rmdefault}{\mddefault}{\updefault}{\color[rgb]{0,0,0}$4$}%
}}}}
\put(15751,-9961){\makebox(0,0)[lb]{\smash{{\SetFigFont{7}{8.4}{\rmdefault}{\mddefault}{\updefault}{\color[rgb]{0,0,0}$3$}%
}}}}
\put(12751,-11086){\makebox(0,0)[lb]{\smash{{\SetFigFont{7}{8.4}{\rmdefault}{\mddefault}{\updefault}{\color[rgb]{0,0,0}$2$}%
}}}}
\put(12001,-11011){\makebox(0,0)[lb]{\smash{{\SetFigFont{7}{8.4}{\rmdefault}{\mddefault}{\updefault}{\color[rgb]{0,0,0}$1$}%
}}}}
\put(12151,-11536){\makebox(0,0)[lb]{\smash{{\SetFigFont{7}{8.4}{\rmdefault}{\mddefault}{\updefault}{\color[rgb]{0,0,0}$4$}%
}}}}
\put(12751,-11686){\makebox(0,0)[lb]{\smash{{\SetFigFont{7}{8.4}{\rmdefault}{\mddefault}{\updefault}{\color[rgb]{0,0,0}$3$}%
}}}}
\put(1951,-11536){\makebox(0,0)[lb]{\smash{{\SetFigFont{7}{8.4}{\rmdefault}{\mddefault}{\updefault}{\color[rgb]{0,0,0}$3$}%
}}}}
\put(2026,-12361){\makebox(0,0)[lb]{\smash{{\SetFigFont{7}{8.4}{\rmdefault}{\mddefault}{\updefault}{\color[rgb]{0,0,0}$4$}%
}}}}
\put(2026,-12886){\makebox(0,0)[lb]{\smash{{\SetFigFont{7}{8.4}{\rmdefault}{\mddefault}{\updefault}{\color[rgb]{0,0,0}$1$}%
}}}}
\put(1951,-13486){\makebox(0,0)[lb]{\smash{{\SetFigFont{7}{8.4}{\rmdefault}{\mddefault}{\updefault}{\color[rgb]{0,0,0}$2$}%
}}}}
\put(2026,-13936){\makebox(0,0)[lb]{\smash{{\SetFigFont{7}{8.4}{\rmdefault}{\mddefault}{\updefault}{\color[rgb]{0,0,0}$3$}%
}}}}
\put(1351,-13936){\makebox(0,0)[lb]{\smash{{\SetFigFont{7}{8.4}{\rmdefault}{\mddefault}{\updefault}{\color[rgb]{0,0,0}$4$}%
}}}}
\put(1276,-13486){\makebox(0,0)[lb]{\smash{{\SetFigFont{7}{8.4}{\rmdefault}{\mddefault}{\updefault}{\color[rgb]{0,0,0}$1$}%
}}}}
\put(10501,-5461){\makebox(0,0)[lb]{\smash{{\SetFigFont{11}{13.2}{\rmdefault}{\mddefault}{\updefault}{\color[rgb]{0,0,0}$(\CDisk,\Gamma)$}%
}}}}
\put(-2099,-10861){\makebox(0,0)[lb]{\smash{{\SetFigFont{11}{13.2}{\rmdefault}{\mddefault}{\updefault}{\color[rgb]{0,0,0}$(\CDisk_2,\Gamma_2)$}%
}}}}
\put(301,-4561){\makebox(0,0)[lb]{\smash{{\SetFigFont{11}{13.2}{\rmdefault}{\mddefault}{\updefault}{\color[rgb]{0,0,0}$(\CDisk_1,\Gamma_1)$}%
}}}}
\put(9826,-10036){\makebox(0,0)[lb]{\smash{{\SetFigFont{11}{13.2}{\rmdefault}{\mddefault}{\updefault}{\color[rgb]{1,0,0}$S$}%
}}}}
\end{picture}%

%% file: draws/PushingOut.tex
\begin{picture}(0,0)%
\includegraphics{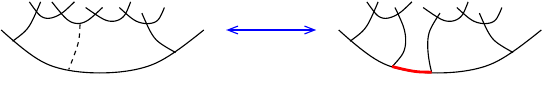}%
\end{picture}%
\setlength{\unitlength}{1184sp}%
\begingroup\makeatletter\ifx\SetFigFont\undefined%
\gdef\SetFigFont#1#2#3#4#5{%
  \reset@font\fontsize{#1}{#2pt}%
  \fontfamily{#3}\fontseries{#4}\fontshape{#5}%
  \selectfont}%
\fi\endgroup%
\begin{picture}(14466,2481)(3568,-8059)
\put(5701,-7111){\makebox(0,0)[lb]{\smash{{\SetFigFont{7}{8.4}{\rmdefault}{\mddefault}{\updefault}{\color[rgb]{0,0,0}$A$}%
}}}}
\put(14026,-7936){\makebox(0,0)[lb]{\smash{{\SetFigFont{7}{8.4}{\rmdefault}{\mddefault}{\updefault}{\color[rgb]{1,0,0}$S$}%
}}}}
\end{picture}%

%% file: draws/PushOutExamples.tex
\begin{picture}(0,0)%
\includegraphics{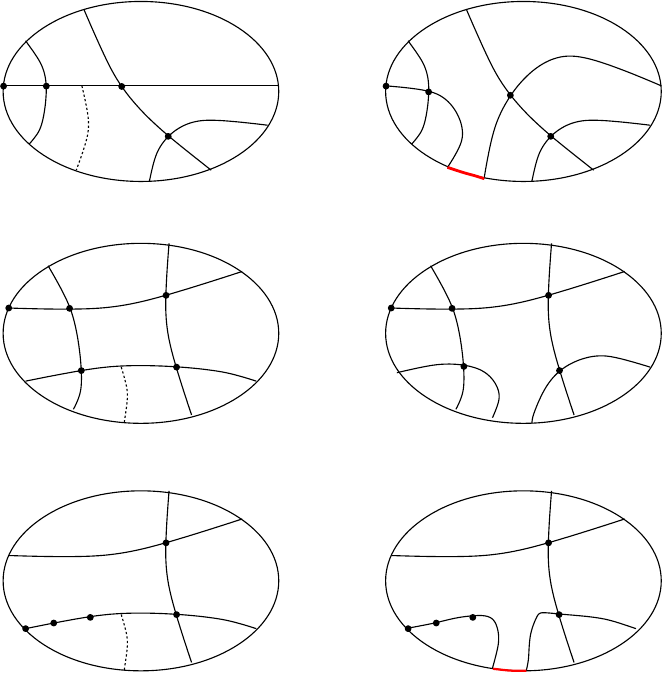}%
\end{picture}%
\setlength{\unitlength}{1184sp}%
\begin{picture}(17669,17926)(-85,-19615)
\put(1801,-9811){\makebox(0,0)[lb]{\smash{\fontsize{7}{8.4}\usefont{T1}{ptm}{m}{n}{\color[rgb]{0,0,0}$2$}%
}}}
\put(1951,-10336){\makebox(0,0)[lb]{\smash{\fontsize{7}{8.4}\usefont{T1}{ptm}{m}{n}{\color[rgb]{0,0,0}$3$}%
}}}
\put(4351,-9361){\makebox(0,0)[lb]{\smash{\fontsize{7}{8.4}\usefont{T1}{ptm}{m}{n}{\color[rgb]{0,0,0}$2$}%
}}}
\put(4426,-9886){\makebox(0,0)[lb]{\smash{\fontsize{7}{8.4}\usefont{T1}{ptm}{m}{n}{\color[rgb]{0,0,0}$3$}%
}}}
\put(2101,-11386){\makebox(0,0)[lb]{\smash{\fontsize{7}{8.4}\usefont{T1}{ptm}{m}{n}{\color[rgb]{0,0,0}$2$}%
}}}
\put(2176,-11911){\makebox(0,0)[lb]{\smash{\fontsize{7}{8.4}\usefont{T1}{ptm}{m}{n}{\color[rgb]{0,0,0}$3$}%
}}}
\put(1356,-10268){\makebox(0,0)[lb]{\smash{\fontsize{7}{8.4}\usefont{T1}{ptm}{m}{n}{\color[rgb]{0,0,0}$4$}%
}}}
\put(1254,-9818){\makebox(0,0)[lb]{\smash{\fontsize{7}{8.4}\usefont{T1}{ptm}{m}{n}{\color[rgb]{0,0,0}$1$}%
}}}
\put(1643,-11467){\makebox(0,0)[lb]{\smash{\fontsize{7}{8.4}\usefont{T1}{ptm}{m}{n}{\color[rgb]{0,0,0}$1$}%
}}}
\put(1628,-11918){\makebox(0,0)[lb]{\smash{\fontsize{7}{8.4}\usefont{T1}{ptm}{m}{n}{\color[rgb]{0,0,0}$4$}%
}}}
\put(4136,-11358){\makebox(0,0)[lb]{\smash{\fontsize{7}{8.4}\usefont{T1}{ptm}{m}{n}{\color[rgb]{0,0,0}$1$}%
}}}
\put(4164,-11816){\makebox(0,0)[lb]{\smash{\fontsize{7}{8.4}\usefont{T1}{ptm}{m}{n}{\color[rgb]{0,0,0}$4$}%
}}}
\put(4697,-11393){\makebox(0,0)[lb]{\smash{\fontsize{7}{8.4}\usefont{T1}{ptm}{m}{n}{\color[rgb]{0,0,0}$2$}%
}}}
\put(4786,-11905){\makebox(0,0)[lb]{\smash{\fontsize{7}{8.4}\usefont{T1}{ptm}{m}{n}{\color[rgb]{0,0,0}$3$}%
}}}
\put(3804,-9968){\makebox(0,0)[lb]{\smash{\fontsize{7}{8.4}\usefont{T1}{ptm}{m}{n}{\color[rgb]{0,0,0}$4$}%
}}}
\put(3906,-9471){\makebox(0,0)[lb]{\smash{\fontsize{7}{8.4}\usefont{T1}{ptm}{m}{n}{\color[rgb]{0,0,0}$1$}%
}}}
\put(3376,-12361){\makebox(0,0)[lb]{\smash{\fontsize{7}{8.4}\usefont{T1}{ptm}{m}{n}{\color[rgb]{0,0,0}$A$}%
}}}
\put(14551,-9361){\makebox(0,0)[lb]{\smash{\fontsize{7}{8.4}\usefont{T1}{ptm}{m}{n}{\color[rgb]{0,0,0}$2$}%
}}}
\put(14626,-9886){\makebox(0,0)[lb]{\smash{\fontsize{7}{8.4}\usefont{T1}{ptm}{m}{n}{\color[rgb]{0,0,0}$3$}%
}}}
\put(11556,-10268){\makebox(0,0)[lb]{\smash{\fontsize{7}{8.4}\usefont{T1}{ptm}{m}{n}{\color[rgb]{0,0,0}$4$}%
}}}
\put(11454,-9818){\makebox(0,0)[lb]{\smash{\fontsize{7}{8.4}\usefont{T1}{ptm}{m}{n}{\color[rgb]{0,0,0}$1$}%
}}}
\put(14106,-9471){\makebox(0,0)[lb]{\smash{\fontsize{7}{8.4}\usefont{T1}{ptm}{m}{n}{\color[rgb]{0,0,0}$1$}%
}}}
\put(14757,-11353){\makebox(0,0)[lb]{\smash{\fontsize{7}{8.4}\usefont{T1}{ptm}{m}{n}{\color[rgb]{0,0,0}$2$}%
}}}
\put(11761,-11311){\makebox(0,0)[lb]{\smash{\fontsize{7}{8.4}\usefont{T1}{ptm}{m}{n}{\color[rgb]{0,0,0}$1$}%
}}}
\put(11741,-11751){\makebox(0,0)[lb]{\smash{\fontsize{7}{8.4}\usefont{T1}{ptm}{m}{n}{\color[rgb]{0,0,0}$4$}%
}}}
\put(12341,-11331){\makebox(0,0)[lb]{\smash{\fontsize{7}{8.4}\usefont{T1}{ptm}{m}{n}{\color[rgb]{0,0,0}$2$}%
}}}
\put(12301,-11801){\makebox(0,0)[lb]{\smash{\fontsize{7}{8.4}\usefont{T1}{ptm}{m}{n}{\color[rgb]{0,0,0}$3$}%
}}}
\put(15026,-11805){\makebox(0,0)[lb]{\smash{\fontsize{7}{8.4}\usefont{T1}{ptm}{m}{n}{\color[rgb]{0,0,0}$3$}%
}}}
\put(14044,-9938){\makebox(0,0)[lb]{\smash{\fontsize{7}{8.4}\usefont{T1}{ptm}{m}{n}{\color[rgb]{0,0,0}$4$}%
}}}
\put(11961,-9741){\makebox(0,0)[lb]{\smash{\fontsize{7}{8.4}\usefont{T1}{ptm}{m}{n}{\color[rgb]{0,0,0}$2$}%
}}}
\put(12141,-10296){\makebox(0,0)[lb]{\smash{\fontsize{7}{8.4}\usefont{T1}{ptm}{m}{n}{\color[rgb]{0,0,0}$3$}%
}}}
\put(14306,-11678){\makebox(0,0)[lb]{\smash{\fontsize{7}{8.4}\usefont{T1}{ptm}{m}{n}{\color[rgb]{0,0,0}$1$}%
}}}
\put(14484,-12136){\makebox(0,0)[lb]{\smash{\fontsize{7}{8.4}\usefont{T1}{ptm}{m}{n}{\color[rgb]{0,0,0}$4$}%
}}}
\put(4351,-15961){\makebox(0,0)[lb]{\smash{\fontsize{7}{8.4}\usefont{T1}{ptm}{m}{n}{\color[rgb]{0,0,0}$2$}%
}}}
\put(4426,-16486){\makebox(0,0)[lb]{\smash{\fontsize{7}{8.4}\usefont{T1}{ptm}{m}{n}{\color[rgb]{0,0,0}$3$}%
}}}
\put(4136,-17958){\makebox(0,0)[lb]{\smash{\fontsize{7}{8.4}\usefont{T1}{ptm}{m}{n}{\color[rgb]{0,0,0}$1$}%
}}}
\put(4164,-18416){\makebox(0,0)[lb]{\smash{\fontsize{7}{8.4}\usefont{T1}{ptm}{m}{n}{\color[rgb]{0,0,0}$4$}%
}}}
\put(4697,-17993){\makebox(0,0)[lb]{\smash{\fontsize{7}{8.4}\usefont{T1}{ptm}{m}{n}{\color[rgb]{0,0,0}$2$}%
}}}
\put(4786,-18505){\makebox(0,0)[lb]{\smash{\fontsize{7}{8.4}\usefont{T1}{ptm}{m}{n}{\color[rgb]{0,0,0}$3$}%
}}}
\put(3804,-16568){\makebox(0,0)[lb]{\smash{\fontsize{7}{8.4}\usefont{T1}{ptm}{m}{n}{\color[rgb]{0,0,0}$4$}%
}}}
\put(3906,-16071){\makebox(0,0)[lb]{\smash{\fontsize{7}{8.4}\usefont{T1}{ptm}{m}{n}{\color[rgb]{0,0,0}$1$}%
}}}
\put(3376,-18961){\makebox(0,0)[lb]{\smash{\fontsize{7}{8.4}\usefont{T1}{ptm}{m}{n}{\color[rgb]{0,0,0}$A$}%
}}}
\put(1126,-18736){\makebox(0,0)[lb]{\smash{\fontsize{7}{8.4}\usefont{T1}{ptm}{m}{n}{\color[rgb]{0,0,0}$2$}%
}}}
\put(2176,-18511){\makebox(0,0)[lb]{\smash{\fontsize{7}{8.4}\usefont{T1}{ptm}{m}{n}{\color[rgb]{0,0,0}$3$}%
}}}
\put(14551,-15961){\makebox(0,0)[lb]{\smash{\fontsize{7}{8.4}\usefont{T1}{ptm}{m}{n}{\color[rgb]{0,0,0}$2$}%
}}}
\put(14626,-16486){\makebox(0,0)[lb]{\smash{\fontsize{7}{8.4}\usefont{T1}{ptm}{m}{n}{\color[rgb]{0,0,0}$3$}%
}}}
\put(14336,-17958){\makebox(0,0)[lb]{\smash{\fontsize{7}{8.4}\usefont{T1}{ptm}{m}{n}{\color[rgb]{0,0,0}$1$}%
}}}
\put(14364,-18416){\makebox(0,0)[lb]{\smash{\fontsize{7}{8.4}\usefont{T1}{ptm}{m}{n}{\color[rgb]{0,0,0}$4$}%
}}}
\put(14897,-17993){\makebox(0,0)[lb]{\smash{\fontsize{7}{8.4}\usefont{T1}{ptm}{m}{n}{\color[rgb]{0,0,0}$2$}%
}}}
\put(14986,-18505){\makebox(0,0)[lb]{\smash{\fontsize{7}{8.4}\usefont{T1}{ptm}{m}{n}{\color[rgb]{0,0,0}$3$}%
}}}
\put(14004,-16568){\makebox(0,0)[lb]{\smash{\fontsize{7}{8.4}\usefont{T1}{ptm}{m}{n}{\color[rgb]{0,0,0}$4$}%
}}}
\put(14106,-16071){\makebox(0,0)[lb]{\smash{\fontsize{7}{8.4}\usefont{T1}{ptm}{m}{n}{\color[rgb]{0,0,0}$1$}%
}}}
\put(11326,-18736){\makebox(0,0)[lb]{\smash{\fontsize{7}{8.4}\usefont{T1}{ptm}{m}{n}{\color[rgb]{0,0,0}$2$}%
}}}
\put(12376,-18511){\makebox(0,0)[lb]{\smash{\fontsize{7}{8.4}\usefont{T1}{ptm}{m}{n}{\color[rgb]{0,0,0}$3$}%
}}}
\put(4201,-5911){\makebox(0,0)[lb]{\smash{\fontsize{7}{8.4}\usefont{T1}{ptm}{m}{n}{\color[rgb]{0,0,0}$4$}%
}}}
\put(4201,-5011){\makebox(0,0)[lb]{\smash{\fontsize{7}{8.4}\usefont{T1}{ptm}{m}{n}{\color[rgb]{0,0,0}$2$}%
}}}
\put(3151,-3811){\makebox(0,0)[lb]{\smash{\fontsize{7}{8.4}\usefont{T1}{ptm}{m}{n}{\color[rgb]{0,0,0}$2$}%
}}}
\put(2551,-3811){\makebox(0,0)[lb]{\smash{\fontsize{7}{8.4}\usefont{T1}{ptm}{m}{n}{\color[rgb]{0,0,0}$1$}%
}}}
\put(2851,-4411){\makebox(0,0)[lb]{\smash{\fontsize{7}{8.4}\usefont{T1}{ptm}{m}{n}{\color[rgb]{0,0,0}$4$}%
}}}
\put(3451,-4261){\makebox(0,0)[lb]{\smash{\fontsize{7}{8.4}\usefont{T1}{ptm}{m}{n}{\color[rgb]{0,0,0}$3$}%
}}}
\put(3901,-5461){\makebox(0,0)[lb]{\smash{\fontsize{7}{8.4}\usefont{T1}{ptm}{m}{n}{\color[rgb]{0,0,0}$1$}%
}}}
\put(4651,-5461){\makebox(0,0)[lb]{\smash{\fontsize{7}{8.4}\usefont{T1}{ptm}{m}{n}{\color[rgb]{0,0,0}$3$}%
}}}
\put(751,-4411){\makebox(0,0)[lb]{\smash{\fontsize{7}{8.4}\usefont{T1}{ptm}{m}{n}{\color[rgb]{0,0,0}$4$}%
}}}
\put(751,-3811){\makebox(0,0)[lb]{\smash{\fontsize{7}{8.4}\usefont{T1}{ptm}{m}{n}{\color[rgb]{0,0,0}$1$}%
}}}
\put(1201,-3811){\makebox(0,0)[lb]{\smash{\fontsize{7}{8.4}\usefont{T1}{ptm}{m}{n}{\color[rgb]{0,0,0}$2$}%
}}}
\put(1201,-4411){\makebox(0,0)[lb]{\smash{\fontsize{7}{8.4}\usefont{T1}{ptm}{m}{n}{\color[rgb]{0,0,0}$3$}%
}}}
\put(14401,-5911){\makebox(0,0)[lb]{\smash{\fontsize{7}{8.4}\usefont{T1}{ptm}{m}{n}{\color[rgb]{0,0,0}$4$}%
}}}
\put(14401,-5011){\makebox(0,0)[lb]{\smash{\fontsize{7}{8.4}\usefont{T1}{ptm}{m}{n}{\color[rgb]{0,0,0}$2$}%
}}}
\put(14101,-5461){\makebox(0,0)[lb]{\smash{\fontsize{7}{8.4}\usefont{T1}{ptm}{m}{n}{\color[rgb]{0,0,0}$1$}%
}}}
\put(14851,-5461){\makebox(0,0)[lb]{\smash{\fontsize{7}{8.4}\usefont{T1}{ptm}{m}{n}{\color[rgb]{0,0,0}$3$}%
}}}
\put(2326,-5386){\makebox(0,0)[lb]{\smash{\fontsize{7}{8.4}\usefont{T1}{ptm}{m}{n}{\color[rgb]{0,0,0}$A$}%
}}}
\put(10801,-4411){\makebox(0,0)[lb]{\smash{\fontsize{7}{8.4}\usefont{T1}{ptm}{m}{n}{\color[rgb]{0,0,0}$4$}%
}}}
\put(11326,-4636){\makebox(0,0)[lb]{\smash{\fontsize{7}{8.4}\usefont{T1}{ptm}{m}{n}{\color[rgb]{0,0,0}$3$}%
}}}
\put(11476,-4111){\makebox(0,0)[lb]{\smash{\fontsize{7}{8.4}\usefont{T1}{ptm}{m}{n}{\color[rgb]{0,0,0}$2$}%
}}}
\put(10801,-3886){\makebox(0,0)[lb]{\smash{\fontsize{7}{8.4}\usefont{T1}{ptm}{m}{n}{\color[rgb]{0,0,0}$1$}%
}}}
\put(13801,-4411){\makebox(0,0)[lb]{\smash{\fontsize{7}{8.4}\usefont{T1}{ptm}{m}{n}{\color[rgb]{0,0,0}$3$}%
}}}
\put(12901,-4411){\makebox(0,0)[lb]{\smash{\fontsize{7}{8.4}\usefont{T1}{ptm}{m}{n}{\color[rgb]{0,0,0}$1$}%
}}}
\put(13351,-3961){\makebox(0,0)[lb]{\smash{\fontsize{7}{8.4}\usefont{T1}{ptm}{m}{n}{\color[rgb]{0,0,0}$2$}%
}}}
\put(13426,-4786){\makebox(0,0)[lb]{\smash{\fontsize{7}{8.4}\usefont{T1}{ptm}{m}{n}{\color[rgb]{0,0,0}$4$}%
}}}
\end{picture}%

%% file: draws/WhatIsGeneric.tex
\begin{picture}(0,0)%
\includegraphics{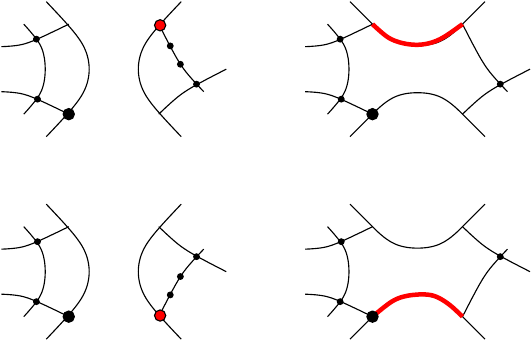}%
\end{picture}%
\setlength{\unitlength}{1184sp}%
\begingroup\makeatletter\ifx\SetFigFont\undefined%
\gdef\SetFigFont#1#2#3#4#5{%
  \reset@font\fontsize{#1}{#2pt}%
  \fontfamily{#3}\fontseries{#4}\fontshape{#5}%
  \selectfont}%
\fi\endgroup%
\begin{picture}(14166,9544)(568,-11672)
\put(9901,-8881){\makebox(0,0)[lb]{\smash{{\SetFigFont{7}{8.4}{\rmdefault}{\mddefault}{\updefault}{\color[rgb]{0,0,0}$3$}%
}}}}
\put(9901,-4711){\makebox(0,0)[lb]{\smash{{\SetFigFont{7}{8.4}{\rmdefault}{\mddefault}{\updefault}{\color[rgb]{0,0,0}$1$}%
}}}}
\put(9901,-3361){\makebox(0,0)[lb]{\smash{{\SetFigFont{7}{8.4}{\rmdefault}{\mddefault}{\updefault}{\color[rgb]{0,0,0}$2$}%
}}}}
\put(901,-5986){\makebox(0,0)[lb]{\smash{{\SetFigFont{7}{8.4}{\rmdefault}{\mddefault}{\updefault}{\color[rgb]{0,0,0}$\Gamma_1$}%
}}}}
\put(5626,-5911){\makebox(0,0)[lb]{\smash{{\SetFigFont{7}{8.4}{\rmdefault}{\mddefault}{\updefault}{\color[rgb]{0,0,0}$\Gamma_2$}%
}}}}
\put(10426,-5986){\makebox(0,0)[lb]{\smash{{\SetFigFont{7}{8.4}{\rmdefault}{\mddefault}{\updefault}{\color[rgb]{0,0,0}$\Gamma_1\#_1\Gamma_2$}%
}}}}
\put(9601,-11536){\makebox(0,0)[lb]{\smash{{\SetFigFont{7}{8.4}{\rmdefault}{\mddefault}{\updefault}{\color[rgb]{0,0,0}$\Gamma_1\#_1\Gamma_2'$}%
}}}}
\put(4801,-11536){\makebox(0,0)[lb]{\smash{{\SetFigFont{7}{8.4}{\rmdefault}{\mddefault}{\updefault}{\color[rgb]{0,0,0}$\Gamma_2'$}%
}}}}
\put(976,-11536){\makebox(0,0)[lb]{\smash{{\SetFigFont{7}{8.4}{\rmdefault}{\mddefault}{\updefault}{\color[rgb]{0,0,0}$\Gamma_1$}%
}}}}
\put(1801,-4711){\makebox(0,0)[lb]{\smash{{\SetFigFont{7}{8.4}{\rmdefault}{\mddefault}{\updefault}{\color[rgb]{0,0,0}$1$}%
}}}}
\put(1801,-3361){\makebox(0,0)[lb]{\smash{{\SetFigFont{7}{8.4}{\rmdefault}{\mddefault}{\updefault}{\color[rgb]{0,0,0}$2$}%
}}}}
\put(4951,-4036){\makebox(0,0)[lb]{\smash{{\SetFigFont{7}{8.4}{\rmdefault}{\mddefault}{\updefault}{\color[rgb]{0,0,0}$2$}%
}}}}
\put(4651,-3511){\makebox(0,0)[lb]{\smash{{\SetFigFont{7}{8.4}{\rmdefault}{\mddefault}{\updefault}{\color[rgb]{0,0,0}$1$}%
}}}}
\put(1801,-8881){\makebox(0,0)[lb]{\smash{{\SetFigFont{7}{8.4}{\rmdefault}{\mddefault}{\updefault}{\color[rgb]{0,0,0}$3$}%
}}}}
\put(1801,-10231){\makebox(0,0)[lb]{\smash{{\SetFigFont{7}{8.4}{\rmdefault}{\mddefault}{\updefault}{\color[rgb]{0,0,0}$2$}%
}}}}
\put(9826,-10186){\makebox(0,0)[lb]{\smash{{\SetFigFont{7}{8.4}{\rmdefault}{\mddefault}{\updefault}{\color[rgb]{0,0,0}$2$}%
}}}}
\put(13351,-9136){\makebox(0,0)[lb]{\smash{{\SetFigFont{7}{8.4}{\rmdefault}{\mddefault}{\updefault}{\color[rgb]{0,0,0}$1$}%
}}}}
\put(5326,-9136){\makebox(0,0)[lb]{\smash{{\SetFigFont{7}{8.4}{\rmdefault}{\mddefault}{\updefault}{\color[rgb]{0,0,0}$1$}%
}}}}
\put(13351,-4486){\makebox(0,0)[lb]{\smash{{\SetFigFont{7}{8.4}{\rmdefault}{\mddefault}{\updefault}{\color[rgb]{0,0,0}$3$}%
}}}}
\put(5251,-4411){\makebox(0,0)[lb]{\smash{{\SetFigFont{7}{8.4}{\rmdefault}{\mddefault}{\updefault}{\color[rgb]{0,0,0}$3$}%
}}}}
\put(5026,-9511){\makebox(0,0)[lb]{\smash{{\SetFigFont{7}{8.4}{\rmdefault}{\mddefault}{\updefault}{\color[rgb]{0,0,0}$2$}%
}}}}
\put(4726,-9961){\makebox(0,0)[lb]{\smash{{\SetFigFont{7}{8.4}{\rmdefault}{\mddefault}{\updefault}{\color[rgb]{0,0,0}$3$}%
}}}}
\end{picture}%

%% file: draws/DegExtremal.tex
\begin{picture}(0,0)%
\includegraphics{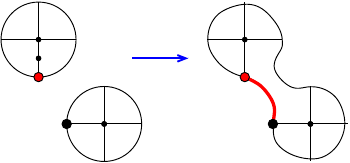}%
\end{picture}%
\setlength{\unitlength}{987sp}%
\begingroup\makeatletter\ifx\SetFigFont\undefined%
\gdef\SetFigFont#1#2#3#4#5{%
  \reset@font\fontsize{#1}{#2pt}%
  \fontfamily{#3}\fontseries{#4}\fontshape{#5}%
  \selectfont}%
\fi\endgroup%
\begin{picture}(11166,5166)(4768,10106)
\put(9751,13739){\makebox(0,0)[lb]{\smash{{\SetFigFont{8}{9.6}{\rmdefault}{\mddefault}{\updefault}{\color[rgb]{0,0,1}$\connectsum$}%
}}}}
\put(6076,13739){\makebox(0,0)[lb]{\smash{{\SetFigFont{6}{7.2}{\rmdefault}{\mddefault}{\updefault}{\color[rgb]{0,0,0}$2$}%
}}}}
\put(6076,14114){\makebox(0,0)[lb]{\smash{{\SetFigFont{6}{7.2}{\rmdefault}{\mddefault}{\updefault}{\color[rgb]{0,0,0}$1$}%
}}}}
\put(12676,13739){\makebox(0,0)[lb]{\smash{{\SetFigFont{6}{7.2}{\rmdefault}{\mddefault}{\updefault}{\color[rgb]{0,0,0}$2$}%
}}}}
\put(12676,14114){\makebox(0,0)[lb]{\smash{{\SetFigFont{6}{7.2}{\rmdefault}{\mddefault}{\updefault}{\color[rgb]{0,0,0}$1$}%
}}}}
\put(6076,13289){\makebox(0,0)[lb]{\smash{{\SetFigFont{6}{7.2}{\rmdefault}{\mddefault}{\updefault}{\color[rgb]{0,0,0}$1$}%
}}}}
\put(12226,14114){\makebox(0,0)[lb]{\smash{{\SetFigFont{6}{7.2}{\rmdefault}{\mddefault}{\updefault}{\color[rgb]{0,0,0}$4$}%
}}}}
\put(12226,13739){\makebox(0,0)[lb]{\smash{{\SetFigFont{6}{7.2}{\rmdefault}{\mddefault}{\updefault}{\color[rgb]{0,0,0}$3$}%
}}}}
\put(5551,14114){\makebox(0,0)[lb]{\smash{{\SetFigFont{6}{7.2}{\rmdefault}{\mddefault}{\updefault}{\color[rgb]{0,0,0}$4$}%
}}}}
\put(5551,13739){\makebox(0,0)[lb]{\smash{{\SetFigFont{6}{7.2}{\rmdefault}{\mddefault}{\updefault}{\color[rgb]{0,0,0}$3$}%
}}}}
\put(8176,11039){\makebox(0,0)[lb]{\smash{{\SetFigFont{6}{7.2}{\rmdefault}{\mddefault}{\updefault}{\color[rgb]{0,0,0}$3$}%
}}}}
\put(8176,11414){\makebox(0,0)[lb]{\smash{{\SetFigFont{6}{7.2}{\rmdefault}{\mddefault}{\updefault}{\color[rgb]{0,0,0}$2$}%
}}}}
\put(14776,11039){\makebox(0,0)[lb]{\smash{{\SetFigFont{6}{7.2}{\rmdefault}{\mddefault}{\updefault}{\color[rgb]{0,0,0}$3$}%
}}}}
\put(14776,11414){\makebox(0,0)[lb]{\smash{{\SetFigFont{6}{7.2}{\rmdefault}{\mddefault}{\updefault}{\color[rgb]{0,0,0}$2$}%
}}}}
\put(7726,11414){\makebox(0,0)[lb]{\smash{{\SetFigFont{6}{7.2}{\rmdefault}{\mddefault}{\updefault}{\color[rgb]{0,0,0}$1$}%
}}}}
\put(7726,11039){\makebox(0,0)[lb]{\smash{{\SetFigFont{6}{7.2}{\rmdefault}{\mddefault}{\updefault}{\color[rgb]{0,0,0}$4$}%
}}}}
\put(14326,11414){\makebox(0,0)[lb]{\smash{{\SetFigFont{6}{7.2}{\rmdefault}{\mddefault}{\updefault}{\color[rgb]{0,0,0}$1$}%
}}}}
\put(14326,11039){\makebox(0,0)[lb]{\smash{{\SetFigFont{6}{7.2}{\rmdefault}{\mddefault}{\updefault}{\color[rgb]{0,0,0}$4$}%
}}}}
\end{picture}%

%% file: draws/FirstInvolution.tex
\begin{picture}(0,0)%
\includegraphics{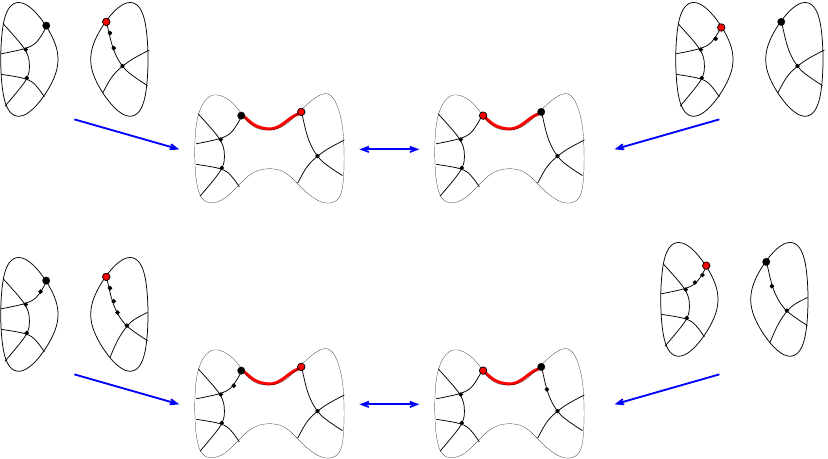}%
\end{picture}%
\setlength{\unitlength}{789sp}%
\begingroup\makeatletter\ifx\SetFigFont\undefined%
\gdef\SetFigFont#1#2#3#4#5{%
  \reset@font\fontsize{#1}{#2pt}%
  \fontfamily{#3}\fontseries{#4}\fontshape{#5}%
  \selectfont}%
\fi\endgroup%
\begin{picture}(33011,18266)(-12577,-19321)
\put(-11313,-14305){\makebox(0,0)[lb]{\smash{{\SetFigFont{5}{6.0}{\rmdefault}{\mddefault}{\updefault}{\color[rgb]{0,0,0}$1$}%
}}}}
\put(-11373,-13465){\makebox(0,0)[lb]{\smash{{\SetFigFont{5}{6.0}{\rmdefault}{\mddefault}{\updefault}{\color[rgb]{0,0,0}$2$}%
}}}}
\put(-11924,-13561){\makebox(0,0)[lb]{\smash{{\SetFigFont{5}{6.0}{\rmdefault}{\mddefault}{\updefault}{\color[rgb]{0,0,0}$3$}%
}}}}
\put(-12149,-14611){\makebox(0,0)[lb]{\smash{{\SetFigFont{5}{6.0}{\rmdefault}{\mddefault}{\updefault}{\color[rgb]{0,0,0}$3$}%
}}}}
\put(-11699,-14836){\makebox(0,0)[lb]{\smash{{\SetFigFont{5}{6.0}{\rmdefault}{\mddefault}{\updefault}{\color[rgb]{0,0,0}$2$}%
}}}}
\put(-8699,-12736){\makebox(0,0)[lb]{\smash{{\SetFigFont{5}{6.0}{\rmdefault}{\mddefault}{\updefault}{\color[rgb]{0,0,0}$1$}%
}}}}
\put(-11699,-12961){\makebox(0,0)[lb]{\smash{{\SetFigFont{5}{6.0}{\rmdefault}{\mddefault}{\updefault}{\color[rgb]{0,0,0}$1$}%
}}}}
\put(-12224,-13111){\makebox(0,0)[lb]{\smash{{\SetFigFont{5}{6.0}{\rmdefault}{\mddefault}{\updefault}{\color[rgb]{0,0,0}$4$}%
}}}}
\put(-11924,-14161){\makebox(0,0)[lb]{\smash{{\SetFigFont{5}{6.0}{\rmdefault}{\mddefault}{\updefault}{\color[rgb]{0,0,0}$4$}%
}}}}
\put(-10949,-13036){\makebox(0,0)[lb]{\smash{{\SetFigFont{5}{6.0}{\rmdefault}{\mddefault}{\updefault}{\color[rgb]{0,0,0}$3$}%
}}}}
\put(-8549,-13186){\makebox(0,0)[lb]{\smash{{\SetFigFont{5}{6.0}{\rmdefault}{\mddefault}{\updefault}{\color[rgb]{0,0,0}$2$}%
}}}}
\put(-8399,-13711){\makebox(0,0)[lb]{\smash{{\SetFigFont{5}{6.0}{\rmdefault}{\mddefault}{\updefault}{\color[rgb]{0,0,0}$3$}%
}}}}
\put(-7349,-14236){\makebox(0,0)[lb]{\smash{{\SetFigFont{5}{6.0}{\rmdefault}{\mddefault}{\updefault}{\color[rgb]{0,0,0}$2$}%
}}}}
\put(-7649,-14611){\makebox(0,0)[lb]{\smash{{\SetFigFont{5}{6.0}{\rmdefault}{\mddefault}{\updefault}{\color[rgb]{0,0,0}$3$}%
}}}}
\put(-7724,-13861){\makebox(0,0)[lb]{\smash{{\SetFigFont{5}{6.0}{\rmdefault}{\mddefault}{\updefault}{\color[rgb]{0,0,0}$1$}%
}}}}
\put(-8099,-14161){\makebox(0,0)[lb]{\smash{{\SetFigFont{5}{6.0}{\rmdefault}{\mddefault}{\updefault}{\color[rgb]{0,0,0}$4$}%
}}}}
\put(-3513,-7705){\makebox(0,0)[lb]{\smash{{\SetFigFont{5}{6.0}{\rmdefault}{\mddefault}{\updefault}{\color[rgb]{0,0,0}$1$}%
}}}}
\put(-3573,-6865){\makebox(0,0)[lb]{\smash{{\SetFigFont{5}{6.0}{\rmdefault}{\mddefault}{\updefault}{\color[rgb]{0,0,0}$2$}%
}}}}
\put(-573,-7335){\makebox(0,0)[lb]{\smash{{\SetFigFont{5}{6.0}{\rmdefault}{\mddefault}{\updefault}{\color[rgb]{0,0,0}$3$}%
}}}}
\put(-149,-7861){\makebox(0,0)[lb]{\smash{{\SetFigFont{5}{6.0}{\rmdefault}{\mddefault}{\updefault}{\color[rgb]{0,0,0}$2$}%
}}}}
\put(451,-7411){\makebox(0,0)[lb]{\smash{{\SetFigFont{5}{6.0}{\rmdefault}{\mddefault}{\updefault}{\color[rgb]{0,0,0}$1$}%
}}}}
\put(  1,-6886){\makebox(0,0)[lb]{\smash{{\SetFigFont{5}{6.0}{\rmdefault}{\mddefault}{\updefault}{\color[rgb]{0,0,0}$4$}%
}}}}
\put(-3974,-6361){\makebox(0,0)[lb]{\smash{{\SetFigFont{5}{6.0}{\rmdefault}{\mddefault}{\updefault}{\color[rgb]{0,0,0}$1$}%
}}}}
\put(-4574,-6436){\makebox(0,0)[lb]{\smash{{\SetFigFont{5}{6.0}{\rmdefault}{\mddefault}{\updefault}{\color[rgb]{0,0,0}$4$}%
}}}}
\put(-4124,-6961){\makebox(0,0)[lb]{\smash{{\SetFigFont{5}{6.0}{\rmdefault}{\mddefault}{\updefault}{\color[rgb]{0,0,0}$3$}%
}}}}
\put(-4199,-7561){\makebox(0,0)[lb]{\smash{{\SetFigFont{5}{6.0}{\rmdefault}{\mddefault}{\updefault}{\color[rgb]{0,0,0}$4$}%
}}}}
\put(-4349,-8011){\makebox(0,0)[lb]{\smash{{\SetFigFont{5}{6.0}{\rmdefault}{\mddefault}{\updefault}{\color[rgb]{0,0,0}$3$}%
}}}}
\put(-3899,-8236){\makebox(0,0)[lb]{\smash{{\SetFigFont{5}{6.0}{\rmdefault}{\mddefault}{\updefault}{\color[rgb]{0,0,0}$2$}%
}}}}
\put(-3074,-16711){\makebox(0,0)[lb]{\smash{{\SetFigFont{5}{6.0}{\rmdefault}{\mddefault}{\updefault}{\color[rgb]{0,0,0}$3$}%
}}}}
\put(8776,-16711){\makebox(0,0)[lb]{\smash{{\SetFigFont{5}{6.0}{\rmdefault}{\mddefault}{\updefault}{\color[rgb]{0,0,0}$3$}%
}}}}
\put(6087,-17905){\makebox(0,0)[lb]{\smash{{\SetFigFont{5}{6.0}{\rmdefault}{\mddefault}{\updefault}{\color[rgb]{0,0,0}$1$}%
}}}}
\put(6027,-17065){\makebox(0,0)[lb]{\smash{{\SetFigFont{5}{6.0}{\rmdefault}{\mddefault}{\updefault}{\color[rgb]{0,0,0}$2$}%
}}}}
\put(5626,-16561){\makebox(0,0)[lb]{\smash{{\SetFigFont{5}{6.0}{\rmdefault}{\mddefault}{\updefault}{\color[rgb]{0,0,0}$1$}%
}}}}
\put(5026,-16636){\makebox(0,0)[lb]{\smash{{\SetFigFont{5}{6.0}{\rmdefault}{\mddefault}{\updefault}{\color[rgb]{0,0,0}$4$}%
}}}}
\put(5476,-17161){\makebox(0,0)[lb]{\smash{{\SetFigFont{5}{6.0}{\rmdefault}{\mddefault}{\updefault}{\color[rgb]{0,0,0}$3$}%
}}}}
\put(5401,-17761){\makebox(0,0)[lb]{\smash{{\SetFigFont{5}{6.0}{\rmdefault}{\mddefault}{\updefault}{\color[rgb]{0,0,0}$4$}%
}}}}
\put(5251,-18211){\makebox(0,0)[lb]{\smash{{\SetFigFont{5}{6.0}{\rmdefault}{\mddefault}{\updefault}{\color[rgb]{0,0,0}$3$}%
}}}}
\put(5701,-18436){\makebox(0,0)[lb]{\smash{{\SetFigFont{5}{6.0}{\rmdefault}{\mddefault}{\updefault}{\color[rgb]{0,0,0}$2$}%
}}}}
\put(-2174,-16111){\makebox(0,0)[lb]{\smash{{\SetFigFont{5}{6.0}{\rmdefault}{\mddefault}{\updefault}{\color[rgb]{1,0,0}$S$}%
}}}}
\put(7426,-16111){\makebox(0,0)[lb]{\smash{{\SetFigFont{5}{6.0}{\rmdefault}{\mddefault}{\updefault}{\color[rgb]{1,0,0}$S$}%
}}}}
\put(9076,-17536){\makebox(0,0)[lb]{\smash{{\SetFigFont{5}{6.0}{\rmdefault}{\mddefault}{\updefault}{\color[rgb]{0,0,0}$4$}%
}}}}
\put(9601,-17236){\makebox(0,0)[lb]{\smash{{\SetFigFont{5}{6.0}{\rmdefault}{\mddefault}{\updefault}{\color[rgb]{0,0,0}$1$}%
}}}}
\put(9976,-17686){\makebox(0,0)[lb]{\smash{{\SetFigFont{5}{6.0}{\rmdefault}{\mddefault}{\updefault}{\color[rgb]{0,0,0}$2$}%
}}}}
\put(9451,-17986){\makebox(0,0)[lb]{\smash{{\SetFigFont{5}{6.0}{\rmdefault}{\mddefault}{\updefault}{\color[rgb]{0,0,0}$3$}%
}}}}
\put(6087,-7705){\makebox(0,0)[lb]{\smash{{\SetFigFont{5}{6.0}{\rmdefault}{\mddefault}{\updefault}{\color[rgb]{0,0,0}$1$}%
}}}}
\put(6027,-6865){\makebox(0,0)[lb]{\smash{{\SetFigFont{5}{6.0}{\rmdefault}{\mddefault}{\updefault}{\color[rgb]{0,0,0}$2$}%
}}}}
\put(9027,-7335){\makebox(0,0)[lb]{\smash{{\SetFigFont{5}{6.0}{\rmdefault}{\mddefault}{\updefault}{\color[rgb]{0,0,0}$3$}%
}}}}
\put(9451,-7861){\makebox(0,0)[lb]{\smash{{\SetFigFont{5}{6.0}{\rmdefault}{\mddefault}{\updefault}{\color[rgb]{0,0,0}$2$}%
}}}}
\put(10051,-7411){\makebox(0,0)[lb]{\smash{{\SetFigFont{5}{6.0}{\rmdefault}{\mddefault}{\updefault}{\color[rgb]{0,0,0}$1$}%
}}}}
\put(9601,-6886){\makebox(0,0)[lb]{\smash{{\SetFigFont{5}{6.0}{\rmdefault}{\mddefault}{\updefault}{\color[rgb]{0,0,0}$4$}%
}}}}
\put(5626,-6361){\makebox(0,0)[lb]{\smash{{\SetFigFont{5}{6.0}{\rmdefault}{\mddefault}{\updefault}{\color[rgb]{0,0,0}$1$}%
}}}}
\put(5026,-6436){\makebox(0,0)[lb]{\smash{{\SetFigFont{5}{6.0}{\rmdefault}{\mddefault}{\updefault}{\color[rgb]{0,0,0}$4$}%
}}}}
\put(5476,-6961){\makebox(0,0)[lb]{\smash{{\SetFigFont{5}{6.0}{\rmdefault}{\mddefault}{\updefault}{\color[rgb]{0,0,0}$3$}%
}}}}
\put(5401,-7561){\makebox(0,0)[lb]{\smash{{\SetFigFont{5}{6.0}{\rmdefault}{\mddefault}{\updefault}{\color[rgb]{0,0,0}$4$}%
}}}}
\put(5251,-8011){\makebox(0,0)[lb]{\smash{{\SetFigFont{5}{6.0}{\rmdefault}{\mddefault}{\updefault}{\color[rgb]{0,0,0}$3$}%
}}}}
\put(5701,-8236){\makebox(0,0)[lb]{\smash{{\SetFigFont{5}{6.0}{\rmdefault}{\mddefault}{\updefault}{\color[rgb]{0,0,0}$2$}%
}}}}
\put(15687,-4105){\makebox(0,0)[lb]{\smash{{\SetFigFont{5}{6.0}{\rmdefault}{\mddefault}{\updefault}{\color[rgb]{0,0,0}$1$}%
}}}}
\put(15627,-3265){\makebox(0,0)[lb]{\smash{{\SetFigFont{5}{6.0}{\rmdefault}{\mddefault}{\updefault}{\color[rgb]{0,0,0}$2$}%
}}}}
\put(15226,-2761){\makebox(0,0)[lb]{\smash{{\SetFigFont{5}{6.0}{\rmdefault}{\mddefault}{\updefault}{\color[rgb]{0,0,0}$1$}%
}}}}
\put(14626,-2836){\makebox(0,0)[lb]{\smash{{\SetFigFont{5}{6.0}{\rmdefault}{\mddefault}{\updefault}{\color[rgb]{0,0,0}$4$}%
}}}}
\put(15076,-3361){\makebox(0,0)[lb]{\smash{{\SetFigFont{5}{6.0}{\rmdefault}{\mddefault}{\updefault}{\color[rgb]{0,0,0}$3$}%
}}}}
\put(15001,-3961){\makebox(0,0)[lb]{\smash{{\SetFigFont{5}{6.0}{\rmdefault}{\mddefault}{\updefault}{\color[rgb]{0,0,0}$4$}%
}}}}
\put(14851,-4411){\makebox(0,0)[lb]{\smash{{\SetFigFont{5}{6.0}{\rmdefault}{\mddefault}{\updefault}{\color[rgb]{0,0,0}$3$}%
}}}}
\put(15301,-4636){\makebox(0,0)[lb]{\smash{{\SetFigFont{5}{6.0}{\rmdefault}{\mddefault}{\updefault}{\color[rgb]{0,0,0}$2$}%
}}}}
\put(16051,-2911){\makebox(0,0)[lb]{\smash{{\SetFigFont{5}{6.0}{\rmdefault}{\mddefault}{\updefault}{\color[rgb]{0,0,0}$3$}%
}}}}
\put(-2174,-5911){\makebox(0,0)[lb]{\smash{{\SetFigFont{5}{6.0}{\rmdefault}{\mddefault}{\updefault}{\color[rgb]{1,0,0}$S$}%
}}}}
\put(7426,-5911){\makebox(0,0)[lb]{\smash{{\SetFigFont{5}{6.0}{\rmdefault}{\mddefault}{\updefault}{\color[rgb]{1,0,0}$S$}%
}}}}
\put(-11313,-4105){\makebox(0,0)[lb]{\smash{{\SetFigFont{5}{6.0}{\rmdefault}{\mddefault}{\updefault}{\color[rgb]{0,0,0}$1$}%
}}}}
\put(-11373,-3265){\makebox(0,0)[lb]{\smash{{\SetFigFont{5}{6.0}{\rmdefault}{\mddefault}{\updefault}{\color[rgb]{0,0,0}$2$}%
}}}}
\put(-8549,-3061){\makebox(0,0)[lb]{\smash{{\SetFigFont{5}{6.0}{\rmdefault}{\mddefault}{\updefault}{\color[rgb]{0,0,0}$2$}%
}}}}
\put(-7949,-4261){\makebox(0,0)[lb]{\smash{{\SetFigFont{5}{6.0}{\rmdefault}{\mddefault}{\updefault}{\color[rgb]{0,0,0}$2$}%
}}}}
\put(-7349,-3811){\makebox(0,0)[lb]{\smash{{\SetFigFont{5}{6.0}{\rmdefault}{\mddefault}{\updefault}{\color[rgb]{0,0,0}$1$}%
}}}}
\put(-7799,-3286){\makebox(0,0)[lb]{\smash{{\SetFigFont{5}{6.0}{\rmdefault}{\mddefault}{\updefault}{\color[rgb]{0,0,0}$4$}%
}}}}
\put(-11924,-3361){\makebox(0,0)[lb]{\smash{{\SetFigFont{5}{6.0}{\rmdefault}{\mddefault}{\updefault}{\color[rgb]{0,0,0}$3$}%
}}}}
\put(-12149,-4411){\makebox(0,0)[lb]{\smash{{\SetFigFont{5}{6.0}{\rmdefault}{\mddefault}{\updefault}{\color[rgb]{0,0,0}$3$}%
}}}}
\put(-11699,-4636){\makebox(0,0)[lb]{\smash{{\SetFigFont{5}{6.0}{\rmdefault}{\mddefault}{\updefault}{\color[rgb]{0,0,0}$2$}%
}}}}
\put(-8699,-2536){\makebox(0,0)[lb]{\smash{{\SetFigFont{5}{6.0}{\rmdefault}{\mddefault}{\updefault}{\color[rgb]{0,0,0}$1$}%
}}}}
\put(-11699,-2761){\makebox(0,0)[lb]{\smash{{\SetFigFont{5}{6.0}{\rmdefault}{\mddefault}{\updefault}{\color[rgb]{0,0,0}$1$}%
}}}}
\put(-12224,-2911){\makebox(0,0)[lb]{\smash{{\SetFigFont{5}{6.0}{\rmdefault}{\mddefault}{\updefault}{\color[rgb]{0,0,0}$4$}%
}}}}
\put(-11924,-3961){\makebox(0,0)[lb]{\smash{{\SetFigFont{5}{6.0}{\rmdefault}{\mddefault}{\updefault}{\color[rgb]{0,0,0}$4$}%
}}}}
\put(-8174,-3736){\makebox(0,0)[lb]{\smash{{\SetFigFont{5}{6.0}{\rmdefault}{\mddefault}{\updefault}{\color[rgb]{0,0,0}$3$}%
}}}}
\put(18751,-3736){\makebox(0,0)[lb]{\smash{{\SetFigFont{5}{6.0}{\rmdefault}{\mddefault}{\updefault}{\color[rgb]{0,0,0}$3$}%
}}}}
\put(19126,-3361){\makebox(0,0)[lb]{\smash{{\SetFigFont{5}{6.0}{\rmdefault}{\mddefault}{\updefault}{\color[rgb]{0,0,0}$4$}%
}}}}
\put(19126,-4186){\makebox(0,0)[lb]{\smash{{\SetFigFont{5}{6.0}{\rmdefault}{\mddefault}{\updefault}{\color[rgb]{0,0,0}$2$}%
}}}}
\put(19501,-3811){\makebox(0,0)[lb]{\smash{{\SetFigFont{5}{6.0}{\rmdefault}{\mddefault}{\updefault}{\color[rgb]{0,0,0}$1$}%
}}}}
\put(15087,-13705){\makebox(0,0)[lb]{\smash{{\SetFigFont{5}{6.0}{\rmdefault}{\mddefault}{\updefault}{\color[rgb]{0,0,0}$1$}%
}}}}
\put(15027,-12865){\makebox(0,0)[lb]{\smash{{\SetFigFont{5}{6.0}{\rmdefault}{\mddefault}{\updefault}{\color[rgb]{0,0,0}$2$}%
}}}}
\put(14476,-12961){\makebox(0,0)[lb]{\smash{{\SetFigFont{5}{6.0}{\rmdefault}{\mddefault}{\updefault}{\color[rgb]{0,0,0}$3$}%
}}}}
\put(14251,-14011){\makebox(0,0)[lb]{\smash{{\SetFigFont{5}{6.0}{\rmdefault}{\mddefault}{\updefault}{\color[rgb]{0,0,0}$3$}%
}}}}
\put(14701,-14236){\makebox(0,0)[lb]{\smash{{\SetFigFont{5}{6.0}{\rmdefault}{\mddefault}{\updefault}{\color[rgb]{0,0,0}$2$}%
}}}}
\put(14701,-12361){\makebox(0,0)[lb]{\smash{{\SetFigFont{5}{6.0}{\rmdefault}{\mddefault}{\updefault}{\color[rgb]{0,0,0}$1$}%
}}}}
\put(14176,-12511){\makebox(0,0)[lb]{\smash{{\SetFigFont{5}{6.0}{\rmdefault}{\mddefault}{\updefault}{\color[rgb]{0,0,0}$4$}%
}}}}
\put(14476,-13561){\makebox(0,0)[lb]{\smash{{\SetFigFont{5}{6.0}{\rmdefault}{\mddefault}{\updefault}{\color[rgb]{0,0,0}$4$}%
}}}}
\put(19051,-13636){\makebox(0,0)[lb]{\smash{{\SetFigFont{5}{6.0}{\rmdefault}{\mddefault}{\updefault}{\color[rgb]{0,0,0}$2$}%
}}}}
\put(18751,-14011){\makebox(0,0)[lb]{\smash{{\SetFigFont{5}{6.0}{\rmdefault}{\mddefault}{\updefault}{\color[rgb]{0,0,0}$3$}%
}}}}
\put(18676,-13261){\makebox(0,0)[lb]{\smash{{\SetFigFont{5}{6.0}{\rmdefault}{\mddefault}{\updefault}{\color[rgb]{0,0,0}$1$}%
}}}}
\put(18301,-13561){\makebox(0,0)[lb]{\smash{{\SetFigFont{5}{6.0}{\rmdefault}{\mddefault}{\updefault}{\color[rgb]{0,0,0}$4$}%
}}}}
\put(15226,-12661){\makebox(0,0)[lb]{\smash{{\SetFigFont{5}{6.0}{\rmdefault}{\mddefault}{\updefault}{\color[rgb]{0,0,0}$3$}%
}}}}
\put(15526,-12361){\makebox(0,0)[lb]{\smash{{\SetFigFont{5}{6.0}{\rmdefault}{\mddefault}{\updefault}{\color[rgb]{0,0,0}$4$}%
}}}}
\put(17776,-12511){\makebox(0,0)[lb]{\smash{{\SetFigFont{5}{6.0}{\rmdefault}{\mddefault}{\updefault}{\color[rgb]{0,0,0}$3$}%
}}}}
\put(-3513,-17905){\makebox(0,0)[lb]{\smash{{\SetFigFont{5}{6.0}{\rmdefault}{\mddefault}{\updefault}{\color[rgb]{0,0,0}$1$}%
}}}}
\put(-3573,-17065){\makebox(0,0)[lb]{\smash{{\SetFigFont{5}{6.0}{\rmdefault}{\mddefault}{\updefault}{\color[rgb]{0,0,0}$2$}%
}}}}
\put(-3974,-16561){\makebox(0,0)[lb]{\smash{{\SetFigFont{5}{6.0}{\rmdefault}{\mddefault}{\updefault}{\color[rgb]{0,0,0}$1$}%
}}}}
\put(-4574,-16636){\makebox(0,0)[lb]{\smash{{\SetFigFont{5}{6.0}{\rmdefault}{\mddefault}{\updefault}{\color[rgb]{0,0,0}$4$}%
}}}}
\put(-4124,-17161){\makebox(0,0)[lb]{\smash{{\SetFigFont{5}{6.0}{\rmdefault}{\mddefault}{\updefault}{\color[rgb]{0,0,0}$3$}%
}}}}
\put(-4199,-17761){\makebox(0,0)[lb]{\smash{{\SetFigFont{5}{6.0}{\rmdefault}{\mddefault}{\updefault}{\color[rgb]{0,0,0}$4$}%
}}}}
\put(-4349,-18211){\makebox(0,0)[lb]{\smash{{\SetFigFont{5}{6.0}{\rmdefault}{\mddefault}{\updefault}{\color[rgb]{0,0,0}$3$}%
}}}}
\put(-3899,-18436){\makebox(0,0)[lb]{\smash{{\SetFigFont{5}{6.0}{\rmdefault}{\mddefault}{\updefault}{\color[rgb]{0,0,0}$2$}%
}}}}
\put(-524,-17536){\makebox(0,0)[lb]{\smash{{\SetFigFont{5}{6.0}{\rmdefault}{\mddefault}{\updefault}{\color[rgb]{0,0,0}$4$}%
}}}}
\put(  1,-17161){\makebox(0,0)[lb]{\smash{{\SetFigFont{5}{6.0}{\rmdefault}{\mddefault}{\updefault}{\color[rgb]{0,0,0}$1$}%
}}}}
\put(301,-17536){\makebox(0,0)[lb]{\smash{{\SetFigFont{5}{6.0}{\rmdefault}{\mddefault}{\updefault}{\color[rgb]{0,0,0}$2$}%
}}}}
\put(-149,-17911){\makebox(0,0)[lb]{\smash{{\SetFigFont{5}{6.0}{\rmdefault}{\mddefault}{\updefault}{\color[rgb]{0,0,0}$3$}%
}}}}
\end{picture}%

%% file: draws/LeftRightC.tex
\begin{picture}(0,0)%
\includegraphics{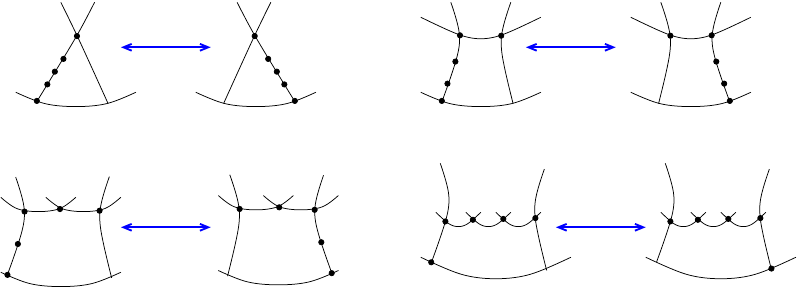}%
\end{picture}%
\setlength{\unitlength}{789sp}%
\begingroup\makeatletter\ifx\SetFigFont\undefined%
\gdef\SetFigFont#1#2#3#4#5{%
  \reset@font\fontsize{#1}{#2pt}%
  \fontfamily{#3}\fontseries{#4}\fontshape{#5}%
  \selectfont}%
\fi\endgroup%
\begin{picture}(31866,11438)(-32,-14766)
\put(1981,-6961){\makebox(0,0)[lb]{\smash{{\SetFigFont{5}{6.0}{\rmdefault}{\mddefault}{\updefault}{\color[rgb]{0,0,0}$1$}%
}}}}
\put(2281,-6466){\makebox(0,0)[lb]{\smash{{\SetFigFont{5}{6.0}{\rmdefault}{\mddefault}{\updefault}{\color[rgb]{0,0,0}$2$}%
}}}}
\put(2506,-5971){\makebox(0,0)[lb]{\smash{{\SetFigFont{5}{6.0}{\rmdefault}{\mddefault}{\updefault}{\color[rgb]{0,0,0}$3$}%
}}}}
\put(2731,-5341){\makebox(0,0)[lb]{\smash{{\SetFigFont{5}{6.0}{\rmdefault}{\mddefault}{\updefault}{\color[rgb]{0,0,0}$4$}%
}}}}
\put(11221,-6961){\rotatebox{360.0}{\makebox(0,0)[rb]{\smash{{\SetFigFont{5}{6.0}{\rmdefault}{\mddefault}{\updefault}{\color[rgb]{0,0,0}$3$}%
}}}}}
\put(10921,-6466){\rotatebox{360.0}{\makebox(0,0)[rb]{\smash{{\SetFigFont{5}{6.0}{\rmdefault}{\mddefault}{\updefault}{\color[rgb]{0,0,0}$2$}%
}}}}}
\put(10696,-5971){\rotatebox{360.0}{\makebox(0,0)[rb]{\smash{{\SetFigFont{5}{6.0}{\rmdefault}{\mddefault}{\updefault}{\color[rgb]{0,0,0}$1$}%
}}}}}
\put(10471,-5341){\rotatebox{360.0}{\makebox(0,0)[rb]{\smash{{\SetFigFont{5}{6.0}{\rmdefault}{\mddefault}{\updefault}{\color[rgb]{0,0,0}$4$}%
}}}}}
\put(965,-12116){\makebox(0,0)[lb]{\smash{{\SetFigFont{5}{6.0}{\rmdefault}{\mddefault}{\updefault}{\color[rgb]{0,0,0}$2$}%
}}}}
\put(2090,-12104){\makebox(0,0)[lb]{\smash{{\SetFigFont{5}{6.0}{\rmdefault}{\mddefault}{\updefault}{\color[rgb]{0,0,0}$3$}%
}}}}
\put(3315,-12107){\makebox(0,0)[lb]{\smash{{\SetFigFont{5}{6.0}{\rmdefault}{\mddefault}{\updefault}{\color[rgb]{0,0,0}$4$}%
}}}}
\put(898,-13377){\makebox(0,0)[lb]{\smash{{\SetFigFont{5}{6.0}{\rmdefault}{\mddefault}{\updefault}{\color[rgb]{0,0,0}$1$}%
}}}}
\put(17798,-12677){\makebox(0,0)[lb]{\smash{{\SetFigFont{5}{6.0}{\rmdefault}{\mddefault}{\updefault}{\color[rgb]{0,0,0}$1$}%
}}}}
\put(18748,-12583){\makebox(0,0)[lb]{\smash{{\SetFigFont{5}{6.0}{\rmdefault}{\mddefault}{\updefault}{\color[rgb]{0,0,0}$2$}%
}}}}
\put(19874,-12570){\makebox(0,0)[lb]{\smash{{\SetFigFont{5}{6.0}{\rmdefault}{\mddefault}{\updefault}{\color[rgb]{0,0,0}$3$}%
}}}}
\put(20865,-12607){\makebox(0,0)[lb]{\smash{{\SetFigFont{5}{6.0}{\rmdefault}{\mddefault}{\updefault}{\color[rgb]{0,0,0}$4$}%
}}}}
\put(12538,-12046){\rotatebox{360.0}{\makebox(0,0)[rb]{\smash{{\SetFigFont{5}{6.0}{\rmdefault}{\mddefault}{\updefault}{\color[rgb]{0,0,0}$4$}%
}}}}}
\put(11413,-12034){\rotatebox{360.0}{\makebox(0,0)[rb]{\smash{{\SetFigFont{5}{6.0}{\rmdefault}{\mddefault}{\updefault}{\color[rgb]{0,0,0}$3$}%
}}}}}
\put(10188,-12037){\rotatebox{360.0}{\makebox(0,0)[rb]{\smash{{\SetFigFont{5}{6.0}{\rmdefault}{\mddefault}{\updefault}{\color[rgb]{0,0,0}$2$}%
}}}}}
\put(12605,-13307){\rotatebox{360.0}{\makebox(0,0)[rb]{\smash{{\SetFigFont{5}{6.0}{\rmdefault}{\mddefault}{\updefault}{\color[rgb]{0,0,0}$1$}%
}}}}}
\put(18440,-5188){\makebox(0,0)[lb]{\smash{{\SetFigFont{5}{6.0}{\rmdefault}{\mddefault}{\updefault}{\color[rgb]{0,0,0}$3$}%
}}}}
\put(19398,-5174){\makebox(0,0)[lb]{\smash{{\SetFigFont{5}{6.0}{\rmdefault}{\mddefault}{\updefault}{\color[rgb]{0,0,0}$4$}%
}}}}
\put(18315,-6016){\makebox(0,0)[lb]{\smash{{\SetFigFont{5}{6.0}{\rmdefault}{\mddefault}{\updefault}{\color[rgb]{0,0,0}$2$}%
}}}}
\put(18048,-6928){\makebox(0,0)[lb]{\smash{{\SetFigFont{5}{6.0}{\rmdefault}{\mddefault}{\updefault}{\color[rgb]{0,0,0}$1$}%
}}}}
\put(28362,-5188){\makebox(0,0)[rb]{\smash{{\SetFigFont{5}{6.0}{\rmdefault}{\mddefault}{\updefault}{\color[rgb]{0,0,0}$4$}%
}}}}
\put(27404,-5174){\makebox(0,0)[rb]{\smash{{\SetFigFont{5}{6.0}{\rmdefault}{\mddefault}{\updefault}{\color[rgb]{0,0,0}$3$}%
}}}}
\put(28487,-6016){\makebox(0,0)[rb]{\smash{{\SetFigFont{5}{6.0}{\rmdefault}{\mddefault}{\updefault}{\color[rgb]{0,0,0}$1$}%
}}}}
\put(28754,-6928){\makebox(0,0)[rb]{\smash{{\SetFigFont{5}{6.0}{\rmdefault}{\mddefault}{\updefault}{\color[rgb]{0,0,0}$2$}%
}}}}
\put(26798,-12677){\makebox(0,0)[lb]{\smash{{\SetFigFont{5}{6.0}{\rmdefault}{\mddefault}{\updefault}{\color[rgb]{0,0,0}$1$}%
}}}}
\put(27748,-12583){\makebox(0,0)[lb]{\smash{{\SetFigFont{5}{6.0}{\rmdefault}{\mddefault}{\updefault}{\color[rgb]{0,0,0}$2$}%
}}}}
\put(28874,-12570){\makebox(0,0)[lb]{\smash{{\SetFigFont{5}{6.0}{\rmdefault}{\mddefault}{\updefault}{\color[rgb]{0,0,0}$3$}%
}}}}
\put(29865,-12607){\makebox(0,0)[lb]{\smash{{\SetFigFont{5}{6.0}{\rmdefault}{\mddefault}{\updefault}{\color[rgb]{0,0,0}$4$}%
}}}}
\end{picture}%

%% file: draws/2Rm.tex
\begin{picture}(0,0)%
\includegraphics{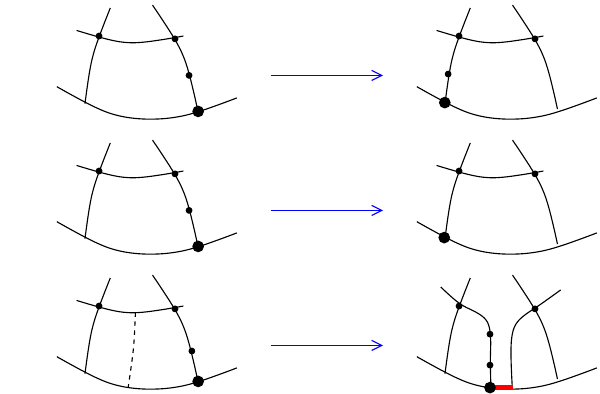}%
\end{picture}%
\setlength{\unitlength}{1184sp}%
\begingroup\makeatletter\ifx\SetFigFont\undefined%
\gdef\SetFigFont#1#2#3#4#5{%
  \reset@font\fontsize{#1}{#2pt}%
  \fontfamily{#3}\fontseries{#4}\fontshape{#5}%
  \selectfont}%
\fi\endgroup%
\begin{picture}(15948,10446)(-2414,1052)
\put(176,10144){\makebox(0,0)[lb]{\smash{{\SetFigFont{7}{8.4}{\rmdefault}{\mddefault}{\updefault}{\color[rgb]{0,0,0}$2$}%
}}}}
\put(9776,10144){\makebox(0,0)[lb]{\smash{{\SetFigFont{7}{8.4}{\rmdefault}{\mddefault}{\updefault}{\color[rgb]{0,0,0}$2$}%
}}}}
\put(11363,10117){\makebox(0,0)[lb]{\smash{{\SetFigFont{7}{8.4}{\rmdefault}{\mddefault}{\updefault}{\color[rgb]{0,0,0}$3$}%
}}}}
\put(9601,9239){\makebox(0,0)[lb]{\smash{{\SetFigFont{7}{8.4}{\rmdefault}{\mddefault}{\updefault}{\color[rgb]{0,0,0}$1$}%
}}}}
\put(1876,10064){\makebox(0,0)[lb]{\smash{{\SetFigFont{7}{8.4}{\rmdefault}{\mddefault}{\updefault}{\color[rgb]{0,0,0}$3$}%
}}}}
\put(2101,9314){\makebox(0,0)[lb]{\smash{{\SetFigFont{7}{8.4}{\rmdefault}{\mddefault}{\updefault}{\color[rgb]{0,0,0}$4$}%
}}}}
\put(176,6544){\makebox(0,0)[lb]{\smash{{\SetFigFont{7}{8.4}{\rmdefault}{\mddefault}{\updefault}{\color[rgb]{0,0,0}$2$}%
}}}}
\put(1763,6517){\makebox(0,0)[lb]{\smash{{\SetFigFont{7}{8.4}{\rmdefault}{\mddefault}{\updefault}{\color[rgb]{0,0,0}$3$}%
}}}}
\put(9776,6544){\makebox(0,0)[lb]{\smash{{\SetFigFont{7}{8.4}{\rmdefault}{\mddefault}{\updefault}{\color[rgb]{0,0,0}$2$}%
}}}}
\put(11363,6517){\makebox(0,0)[lb]{\smash{{\SetFigFont{7}{8.4}{\rmdefault}{\mddefault}{\updefault}{\color[rgb]{0,0,0}$3$}%
}}}}
\put(2101,5789){\makebox(0,0)[lb]{\smash{{\SetFigFont{7}{8.4}{\rmdefault}{\mddefault}{\updefault}{\color[rgb]{0,0,0}$4$}%
}}}}
\put(1763,2917){\makebox(0,0)[lb]{\smash{{\SetFigFont{7}{8.4}{\rmdefault}{\mddefault}{\updefault}{\color[rgb]{0,0,0}$3$}%
}}}}
\put(1196,1729){\makebox(0,0)[lb]{\smash{{\SetFigFont{7}{8.4}{\rmdefault}{\mddefault}{\updefault}{\color[rgb]{0,0,0}$A$}%
}}}}
\put(11483,2722){\makebox(0,0)[lb]{\smash{{\SetFigFont{7}{8.4}{\rmdefault}{\mddefault}{\updefault}{\color[rgb]{0,0,0}$3$}%
}}}}
\put(2176,2039){\makebox(0,0)[lb]{\smash{{\SetFigFont{7}{8.4}{\rmdefault}{\mddefault}{\updefault}{\color[rgb]{0,0,0}$4$}%
}}}}
\put(10126,2489){\makebox(0,0)[lb]{\smash{{\SetFigFont{7}{8.4}{\rmdefault}{\mddefault}{\updefault}{\color[rgb]{0,0,0}$3$}%
}}}}
\put(10126,1664){\makebox(0,0)[lb]{\smash{{\SetFigFont{7}{8.4}{\rmdefault}{\mddefault}{\updefault}{\color[rgb]{0,0,0}$4$}%
}}}}
\put(9826,2939){\makebox(0,0)[lb]{\smash{{\SetFigFont{7}{8.4}{\rmdefault}{\mddefault}{\updefault}{\color[rgb]{0,0,0}$2$}%
}}}}
\put(226,2939){\makebox(0,0)[lb]{\smash{{\SetFigFont{7}{8.4}{\rmdefault}{\mddefault}{\updefault}{\color[rgb]{0,0,0}$2$}%
}}}}
\put(12601,11039){\makebox(0,0)[lb]{\smash{{\SetFigFont{11}{13.2}{\rmdefault}{\mddefault}{\updefault}{\color[rgb]{0,0,0}$(2L+)$}%
}}}}
\put(13201,9539){\makebox(0,0)[lb]{\smash{{\SetFigFont{11}{13.2}{\rmdefault}{\mddefault}{\updefault}{\color[rgb]{0,0,0}$\cdot\rho_{234}$}%
}}}}
\put(13201,5939){\makebox(0,0)[lb]{\smash{{\SetFigFont{11}{13.2}{\rmdefault}{\mddefault}{\updefault}{\color[rgb]{0,0,0}$\cdot\rho_{234}$}%
}}}}
\put(12601,7439){\makebox(0,0)[lb]{\smash{{\SetFigFont{11}{13.2}{\rmdefault}{\mddefault}{\updefault}{\color[rgb]{0,0,0}$(2C+)^0$}%
}}}}
\put(12601,3239){\makebox(0,0)[lb]{\smash{{\SetFigFont{11}{13.2}{\rmdefault}{\mddefault}{\updefault}{\color[rgb]{0,0,0}$[RR+]$}%
}}}}
\put(-2399,11039){\makebox(0,0)[lb]{\smash{{\SetFigFont{11}{13.2}{\rmdefault}{\mddefault}{\updefault}{\color[rgb]{0,0,0}$(2R-)$}%
}}}}
\put(-2399,7439){\makebox(0,0)[lb]{\smash{{\SetFigFont{11}{13.2}{\rmdefault}{\mddefault}{\updefault}{\color[rgb]{0,0,0}$(2R-)$}%
}}}}
\put(-2399,3839){\makebox(0,0)[lb]{\smash{{\SetFigFont{11}{13.2}{\rmdefault}{\mddefault}{\updefault}{\color[rgb]{0,0,0}$(2R-)$}%
}}}}
\put(-1799,9839){\makebox(0,0)[lb]{\smash{{\SetFigFont{11}{13.2}{\rmdefault}{\mddefault}{\updefault}{\color[rgb]{0,0,0}$\rho_{123}\cdot $}%
}}}}
\put(-1799,6239){\makebox(0,0)[lb]{\smash{{\SetFigFont{11}{13.2}{\rmdefault}{\mddefault}{\updefault}{\color[rgb]{0,0,0}$\rho_{23}\cdot $}%
}}}}
\put(-1799,2639){\makebox(0,0)[lb]{\smash{{\SetFigFont{11}{13.2}{\rmdefault}{\mddefault}{\updefault}{\color[rgb]{0,0,0}$\rho_{3}\cdot $}%
}}}}
\end{picture}%

%% file: draws/2Cm.tex
\begin{picture}(0,0)%
\includegraphics{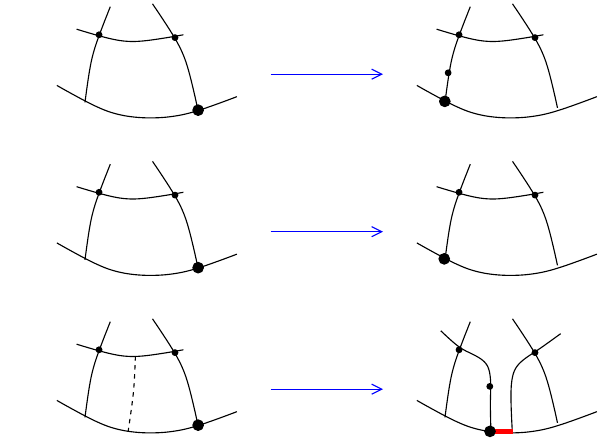}%
\end{picture}%
\setlength{\unitlength}{1184sp}%
\begingroup\makeatletter\ifx\SetFigFont\undefined%
\gdef\SetFigFont#1#2#3#4#5{%
  \reset@font\fontsize{#1}{#2pt}%
  \fontfamily{#3}\fontseries{#4}\fontshape{#5}%
  \selectfont}%
\fi\endgroup%
\begin{picture}(15948,11646)(-2414,-148)
\put(1763,10117){\makebox(0,0)[lb]{\smash{{\SetFigFont{7}{8.4}{\rmdefault}{\mddefault}{\updefault}{\color[rgb]{0,0,0}$3$}%
}}}}
\put(9601,9239){\makebox(0,0)[lb]{\smash{{\SetFigFont{7}{8.4}{\rmdefault}{\mddefault}{\updefault}{\color[rgb]{0,0,0}$1$}%
}}}}
\put(11463,10117){\makebox(0,0)[lb]{\smash{{\SetFigFont{7}{8.4}{\rmdefault}{\mddefault}{\updefault}{\color[rgb]{0,0,0}$3$}%
}}}}
\put(9801,10119){\makebox(0,0)[lb]{\smash{{\SetFigFont{7}{8.4}{\rmdefault}{\mddefault}{\updefault}{\color[rgb]{0,0,0}$2$}%
}}}}
\put(226,10119){\makebox(0,0)[lb]{\smash{{\SetFigFont{7}{8.4}{\rmdefault}{\mddefault}{\updefault}{\color[rgb]{0,0,0}$2$}%
}}}}
\put(176,5944){\makebox(0,0)[lb]{\smash{{\SetFigFont{7}{8.4}{\rmdefault}{\mddefault}{\updefault}{\color[rgb]{0,0,0}$2$}%
}}}}
\put(1763,5917){\makebox(0,0)[lb]{\smash{{\SetFigFont{7}{8.4}{\rmdefault}{\mddefault}{\updefault}{\color[rgb]{0,0,0}$3$}%
}}}}
\put(9776,5944){\makebox(0,0)[lb]{\smash{{\SetFigFont{7}{8.4}{\rmdefault}{\mddefault}{\updefault}{\color[rgb]{0,0,0}$2$}%
}}}}
\put(11438,5892){\makebox(0,0)[lb]{\smash{{\SetFigFont{7}{8.4}{\rmdefault}{\mddefault}{\updefault}{\color[rgb]{0,0,0}$3$}%
}}}}
\put(1763,1717){\makebox(0,0)[lb]{\smash{{\SetFigFont{7}{8.4}{\rmdefault}{\mddefault}{\updefault}{\color[rgb]{0,0,0}$3$}%
}}}}
\put(9776,1744){\makebox(0,0)[lb]{\smash{{\SetFigFont{7}{8.4}{\rmdefault}{\mddefault}{\updefault}{\color[rgb]{0,0,0}$2$}%
}}}}
\put(1196,529){\makebox(0,0)[lb]{\smash{{\SetFigFont{7}{8.4}{\rmdefault}{\mddefault}{\updefault}{\color[rgb]{0,0,0}$A$}%
}}}}
\put(11558,1597){\makebox(0,0)[lb]{\smash{{\SetFigFont{7}{8.4}{\rmdefault}{\mddefault}{\updefault}{\color[rgb]{0,0,0}$3$}%
}}}}
\put(226,1719){\makebox(0,0)[lb]{\smash{{\SetFigFont{7}{8.4}{\rmdefault}{\mddefault}{\updefault}{\color[rgb]{0,0,0}$2$}%
}}}}
\put(10051,1026){\makebox(0,0)[lb]{\smash{{\SetFigFont{7}{8.4}{\rmdefault}{\mddefault}{\updefault}{\color[rgb]{0,0,0}$3$}%
}}}}
\put(13201,9539){\makebox(0,0)[lb]{\smash{{\SetFigFont{11}{13.2}{\rmdefault}{\mddefault}{\updefault}{\color[rgb]{0,0,0}$\cdot\rho_{23}$}%
}}}}
\put(12601,11039){\makebox(0,0)[lb]{\smash{{\SetFigFont{11}{13.2}{\rmdefault}{\mddefault}{\updefault}{\color[rgb]{0,0,0}$(2L+)$}%
}}}}
\put(13201,5339){\makebox(0,0)[lb]{\smash{{\SetFigFont{11}{13.2}{\rmdefault}{\mddefault}{\updefault}{\color[rgb]{0,0,0}$\cdot\rho_{23}$}%
}}}}
\put(12601,2639){\makebox(0,0)[lb]{\smash{{\SetFigFont{11}{13.2}{\rmdefault}{\mddefault}{\updefault}{\color[rgb]{0,0,0}$[RR+]$}%
}}}}
\put(12601,6839){\makebox(0,0)[lb]{\smash{{\SetFigFont{11}{13.2}{\rmdefault}{\mddefault}{\updefault}{\color[rgb]{0,0,0}$(2C+)^0$}%
}}}}
\put(-2399,11039){\makebox(0,0)[lb]{\smash{{\SetFigFont{11}{13.2}{\rmdefault}{\mddefault}{\updefault}{\color[rgb]{0,0,0}$(2C-)^0$}%
}}}}
\put(-2399,6839){\makebox(0,0)[lb]{\smash{{\SetFigFont{11}{13.2}{\rmdefault}{\mddefault}{\updefault}{\color[rgb]{0,0,0}$(2C-)^0$}%
}}}}
\put(-2399,2639){\makebox(0,0)[lb]{\smash{{\SetFigFont{11}{13.2}{\rmdefault}{\mddefault}{\updefault}{\color[rgb]{0,0,0}$(2C-)^0$}%
}}}}
\put(-1799,9839){\makebox(0,0)[lb]{\smash{{\SetFigFont{11}{13.2}{\rmdefault}{\mddefault}{\updefault}{\color[rgb]{0,0,0}$\rho_{123}\cdot $}%
}}}}
\put(-1799,5639){\makebox(0,0)[lb]{\smash{{\SetFigFont{11}{13.2}{\rmdefault}{\mddefault}{\updefault}{\color[rgb]{0,0,0}$\rho_{23}\cdot $}%
}}}}
\put(-1799,1439){\makebox(0,0)[lb]{\smash{{\SetFigFont{11}{13.2}{\rmdefault}{\mddefault}{\updefault}{\color[rgb]{0,0,0}$\rho_{3}\cdot $}%
}}}}
\end{picture}%

%% file: draws/LLm.tex
\begin{picture}(0,0)%
\includegraphics{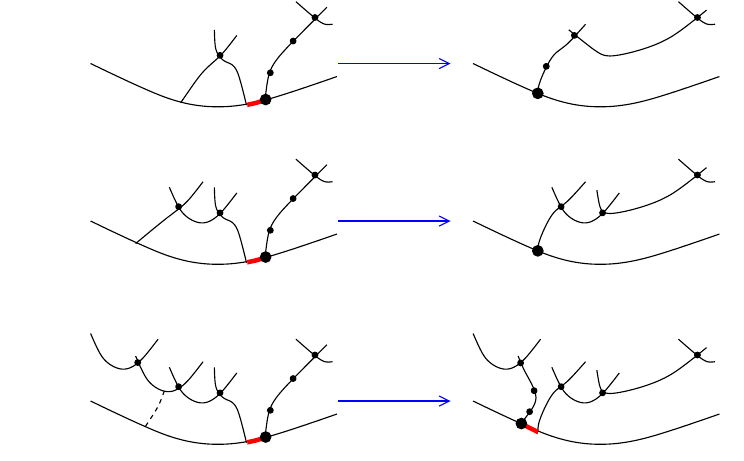}%
\end{picture}%
\setlength{\unitlength}{1184sp}%
\begingroup\makeatletter\ifx\SetFigFont\undefined%
\gdef\SetFigFont#1#2#3#4#5{%
  \reset@font\fontsize{#1}{#2pt}%
  \fontfamily{#3}\fontseries{#4}\fontshape{#5}%
  \selectfont}%
\fi\endgroup%
\begin{picture}(19830,12268)(-4214,-6146)
\put(14193,-3783){\makebox(0,0)[lb]{\smash{{\SetFigFont{7}{8.4}{\rmdefault}{\mddefault}{\updefault}{\color[rgb]{0,0,0}$4$}%
}}}}
\put(10673,-4628){\makebox(0,0)[lb]{\smash{{\SetFigFont{7}{8.4}{\rmdefault}{\mddefault}{\updefault}{\color[rgb]{0,0,0}$2$}%
}}}}
\put(11663,-4793){\makebox(0,0)[lb]{\smash{{\SetFigFont{7}{8.4}{\rmdefault}{\mddefault}{\updefault}{\color[rgb]{0,0,0}$3$}%
}}}}
\put(9451,-5761){\makebox(0,0)[lb]{\smash{{\SetFigFont{7}{8.4}{\rmdefault}{\mddefault}{\updefault}{\color[rgb]{1,0,0}$S$}%
}}}}
\put(9371,-4001){\makebox(0,0)[lb]{\smash{{\SetFigFont{7}{8.4}{\rmdefault}{\mddefault}{\updefault}{\color[rgb]{0,0,0}$1$}%
}}}}
\put(9496,-4511){\makebox(0,0)[lb]{\smash{{\SetFigFont{7}{8.4}{\rmdefault}{\mddefault}{\updefault}{\color[rgb]{0,0,0}$2$}%
}}}}
\put(9366,-4981){\makebox(0,0)[lb]{\smash{{\SetFigFont{7}{8.4}{\rmdefault}{\mddefault}{\updefault}{\color[rgb]{0,0,0}$3$}%
}}}}
\put(3993,5217){\makebox(0,0)[lb]{\smash{{\SetFigFont{7}{8.4}{\rmdefault}{\mddefault}{\updefault}{\color[rgb]{0,0,0}$4$}%
}}}}
\put(3068,4117){\makebox(0,0)[lb]{\smash{{\SetFigFont{7}{8.4}{\rmdefault}{\mddefault}{\updefault}{\color[rgb]{0,0,0}$2$}%
}}}}
\put(1463,4207){\makebox(0,0)[lb]{\smash{{\SetFigFont{7}{8.4}{\rmdefault}{\mddefault}{\updefault}{\color[rgb]{0,0,0}$3$}%
}}}}
\put(3601,4739){\makebox(0,0)[lb]{\smash{{\SetFigFont{7}{8.4}{\rmdefault}{\mddefault}{\updefault}{\color[rgb]{0,0,0}$3$}%
}}}}
\put(2488,2977){\makebox(0,0)[lb]{\smash{{\SetFigFont{7}{8.4}{\rmdefault}{\mddefault}{\updefault}{\color[rgb]{1,0,0}$S$}%
}}}}
\put(14193,5217){\makebox(0,0)[lb]{\smash{{\SetFigFont{7}{8.4}{\rmdefault}{\mddefault}{\updefault}{\color[rgb]{0,0,0}$4$}%
}}}}
\put(10576,4064){\makebox(0,0)[lb]{\smash{{\SetFigFont{7}{8.4}{\rmdefault}{\mddefault}{\updefault}{\color[rgb]{0,0,0}$2$}%
}}}}
\put(10951,4664){\makebox(0,0)[lb]{\smash{{\SetFigFont{7}{8.4}{\rmdefault}{\mddefault}{\updefault}{\color[rgb]{0,0,0}$3$}%
}}}}
\put(3993,1017){\makebox(0,0)[lb]{\smash{{\SetFigFont{7}{8.4}{\rmdefault}{\mddefault}{\updefault}{\color[rgb]{0,0,0}$4$}%
}}}}
\put(3068,-83){\makebox(0,0)[lb]{\smash{{\SetFigFont{7}{8.4}{\rmdefault}{\mddefault}{\updefault}{\color[rgb]{0,0,0}$2$}%
}}}}
\put(473,172){\makebox(0,0)[lb]{\smash{{\SetFigFont{7}{8.4}{\rmdefault}{\mddefault}{\updefault}{\color[rgb]{0,0,0}$2$}%
}}}}
\put(1463,  7){\makebox(0,0)[lb]{\smash{{\SetFigFont{7}{8.4}{\rmdefault}{\mddefault}{\updefault}{\color[rgb]{0,0,0}$3$}%
}}}}
\put(3601,539){\makebox(0,0)[lb]{\smash{{\SetFigFont{7}{8.4}{\rmdefault}{\mddefault}{\updefault}{\color[rgb]{0,0,0}$3$}%
}}}}
\put(2488,-1223){\makebox(0,0)[lb]{\smash{{\SetFigFont{7}{8.4}{\rmdefault}{\mddefault}{\updefault}{\color[rgb]{1,0,0}$S$}%
}}}}
\put(14193,1017){\makebox(0,0)[lb]{\smash{{\SetFigFont{7}{8.4}{\rmdefault}{\mddefault}{\updefault}{\color[rgb]{0,0,0}$4$}%
}}}}
\put(10673,172){\makebox(0,0)[lb]{\smash{{\SetFigFont{7}{8.4}{\rmdefault}{\mddefault}{\updefault}{\color[rgb]{0,0,0}$2$}%
}}}}
\put(11663,  7){\makebox(0,0)[lb]{\smash{{\SetFigFont{7}{8.4}{\rmdefault}{\mddefault}{\updefault}{\color[rgb]{0,0,0}$3$}%
}}}}
\put(3993,-3783){\makebox(0,0)[lb]{\smash{{\SetFigFont{7}{8.4}{\rmdefault}{\mddefault}{\updefault}{\color[rgb]{0,0,0}$4$}%
}}}}
\put(3068,-4883){\makebox(0,0)[lb]{\smash{{\SetFigFont{7}{8.4}{\rmdefault}{\mddefault}{\updefault}{\color[rgb]{0,0,0}$2$}%
}}}}
\put(-82,-5123){\makebox(0,0)[lb]{\smash{{\SetFigFont{7}{8.4}{\rmdefault}{\mddefault}{\updefault}{\color[rgb]{0,0,0}$A$}%
}}}}
\put(473,-4628){\makebox(0,0)[lb]{\smash{{\SetFigFont{7}{8.4}{\rmdefault}{\mddefault}{\updefault}{\color[rgb]{0,0,0}$2$}%
}}}}
\put(1463,-4793){\makebox(0,0)[lb]{\smash{{\SetFigFont{7}{8.4}{\rmdefault}{\mddefault}{\updefault}{\color[rgb]{0,0,0}$3$}%
}}}}
\put(3601,-4261){\makebox(0,0)[lb]{\smash{{\SetFigFont{7}{8.4}{\rmdefault}{\mddefault}{\updefault}{\color[rgb]{0,0,0}$3$}%
}}}}
\put(-829,-4021){\makebox(0,0)[lb]{\smash{{\SetFigFont{7}{8.4}{\rmdefault}{\mddefault}{\updefault}{\color[rgb]{0,0,0}$1$}%
}}}}
\put(2488,-6023){\makebox(0,0)[lb]{\smash{{\SetFigFont{7}{8.4}{\rmdefault}{\mddefault}{\updefault}{\color[rgb]{1,0,0}$S$}%
}}}}
\put(15601,4289){\makebox(0,0)[lb]{\smash{{\SetFigFont{11}{13.2}{\rmdefault}{\mddefault}{\updefault}{\color[rgb]{0,0,0}$\cdot \rho_3$}%
}}}}
\put(15601,539){\makebox(0,0)[lb]{\smash{{\SetFigFont{11}{13.2}{\rmdefault}{\mddefault}{\updefault}{\color[rgb]{0,0,0}$\cdot \rho_{23}$}%
}}}}
\put(15601,5639){\makebox(0,0)[lb]{\smash{{\SetFigFont{11}{13.2}{\rmdefault}{\mddefault}{\updefault}{\color[rgb]{0,0,0}$(2L+)$}%
}}}}
\put(15601,-2761){\makebox(0,0)[lb]{\smash{{\SetFigFont{11}{13.2}{\rmdefault}{\mddefault}{\updefault}{\color[rgb]{0,0,0}$[RR+]$}%
}}}}
\put(15601,2039){\makebox(0,0)[lb]{\smash{{\SetFigFont{11}{13.2}{\rmdefault}{\mddefault}{\updefault}{\color[rgb]{0,0,0}$(2C+)^0$}%
}}}}
\put(-4199,5039){\makebox(0,0)[lb]{\smash{{\SetFigFont{11}{13.2}{\rmdefault}{\mddefault}{\updefault}{\color[rgb]{0,0,0}$[LL-]$}%
}}}}
\put(-4199,1439){\makebox(0,0)[lb]{\smash{{\SetFigFont{11}{13.2}{\rmdefault}{\mddefault}{\updefault}{\color[rgb]{0,0,0}$[LL-]$}%
}}}}
\put(-4199,-2761){\makebox(0,0)[lb]{\smash{{\SetFigFont{11}{13.2}{\rmdefault}{\mddefault}{\updefault}{\color[rgb]{0,0,0}$[LL-]$}%
}}}}
\end{picture}%

%% file: draws/LLRRex.tex
\begin{picture}(0,0)%
\includegraphics{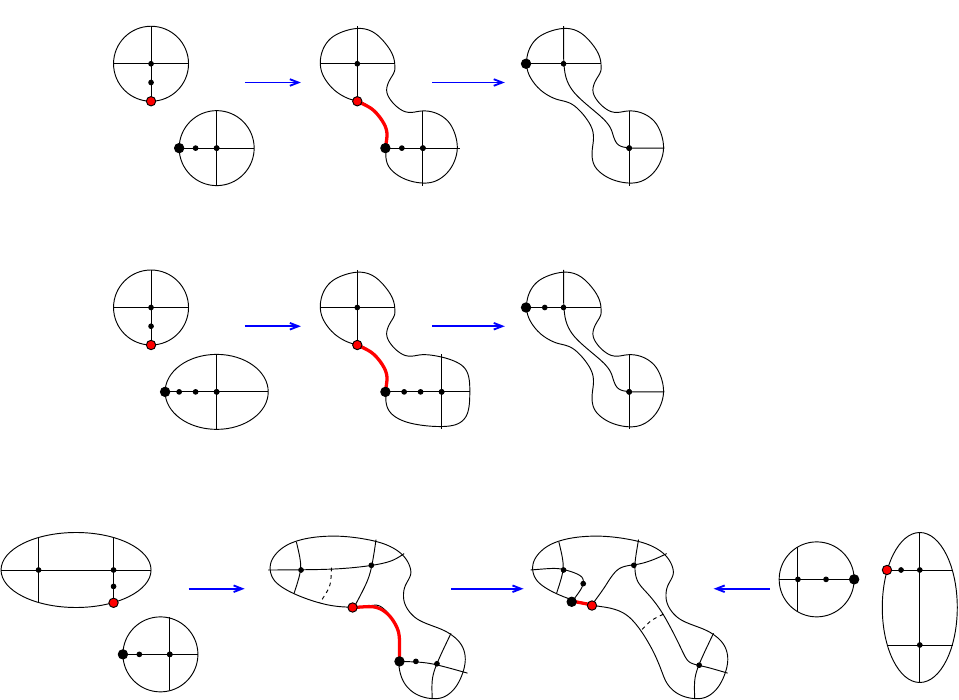}%
\end{picture}%
\setlength{\unitlength}{987sp}%
\begingroup\makeatletter\ifx\SetFigFont\undefined%
\gdef\SetFigFont#1#2#3#4#5{%
  \reset@font\fontsize{#1}{#2pt}%
  \fontfamily{#3}\fontseries{#4}\fontshape{#5}%
  \selectfont}%
\fi\endgroup%
\begin{picture}(30656,22344)(1168,-6319)
\put(7951,-2461){\makebox(0,0)[lb]{\smash{{\SetFigFont{8}{9.6}{\rmdefault}{\mddefault}{\updefault}{\color[rgb]{0,0,1}$\#$}%
}}}}
\put(16351,-2461){\makebox(0,0)[lb]{\smash{{\SetFigFont{8}{9.6}{\rmdefault}{\mddefault}{\updefault}{\color[rgb]{0,0,1}$M$}%
}}}}
\put(9751,5939){\makebox(0,0)[lb]{\smash{{\SetFigFont{8}{9.6}{\rmdefault}{\mddefault}{\updefault}{\color[rgb]{0,0,1}$\#$}%
}}}}
\put(15751,5939){\makebox(0,0)[lb]{\smash{{\SetFigFont{8}{9.6}{\rmdefault}{\mddefault}{\updefault}{\color[rgb]{0,0,1}$M$}%
}}}}
\put(15376,3239){\makebox(0,0)[lb]{\smash{{\SetFigFont{6}{7.2}{\rmdefault}{\mddefault}{\updefault}{\color[rgb]{0,0,0}$3$}%
}}}}
\put(15376,3614){\makebox(0,0)[lb]{\smash{{\SetFigFont{6}{7.2}{\rmdefault}{\mddefault}{\updefault}{\color[rgb]{0,0,0}$2$}%
}}}}
\put(14926,3614){\makebox(0,0)[lb]{\smash{{\SetFigFont{6}{7.2}{\rmdefault}{\mddefault}{\updefault}{\color[rgb]{0,0,0}$1$}%
}}}}
\put(14926,3239){\makebox(0,0)[lb]{\smash{{\SetFigFont{6}{7.2}{\rmdefault}{\mddefault}{\updefault}{\color[rgb]{0,0,0}$4$}%
}}}}
\put(7201,3239){\makebox(0,0)[lb]{\smash{{\SetFigFont{6}{7.2}{\rmdefault}{\mddefault}{\updefault}{\color[rgb]{0,0,0}$3$}%
}}}}
\put(6601,3239){\makebox(0,0)[lb]{\smash{{\SetFigFont{6}{7.2}{\rmdefault}{\mddefault}{\updefault}{\color[rgb]{0,0,0}$2$}%
}}}}
\put(14401,3239){\makebox(0,0)[lb]{\smash{{\SetFigFont{6}{7.2}{\rmdefault}{\mddefault}{\updefault}{\color[rgb]{0,0,0}$3$}%
}}}}
\put(13801,3239){\makebox(0,0)[lb]{\smash{{\SetFigFont{6}{7.2}{\rmdefault}{\mddefault}{\updefault}{\color[rgb]{0,0,0}$2$}%
}}}}
\put(15751,13739){\makebox(0,0)[lb]{\smash{{\SetFigFont{8}{9.6}{\rmdefault}{\mddefault}{\updefault}{\color[rgb]{0,0,1}$M$}%
}}}}
\put(10876,-2461){\makebox(0,0)[lb]{\smash{{\SetFigFont{6}{7.2}{\rmdefault}{\mddefault}{\updefault}{\color[rgb]{0,0,0}$2$}%
}}}}
\put(10801,-2086){\makebox(0,0)[lb]{\smash{{\SetFigFont{6}{7.2}{\rmdefault}{\mddefault}{\updefault}{\color[rgb]{0,0,0}$1$}%
}}}}
\put(13126,-2311){\makebox(0,0)[lb]{\smash{{\SetFigFont{6}{7.2}{\rmdefault}{\mddefault}{\updefault}{\color[rgb]{0,0,0}$2$}%
}}}}
\put(13051,-1936){\makebox(0,0)[lb]{\smash{{\SetFigFont{6}{7.2}{\rmdefault}{\mddefault}{\updefault}{\color[rgb]{0,0,0}$1$}%
}}}}
\put(19166,-2521){\makebox(0,0)[lb]{\smash{{\SetFigFont{6}{7.2}{\rmdefault}{\mddefault}{\updefault}{\color[rgb]{0,0,0}$2$}%
}}}}
\put(19351,-2771){\makebox(0,0)[lb]{\smash{{\SetFigFont{6}{7.2}{\rmdefault}{\mddefault}{\updefault}{\color[rgb]{0,0,0}$3$}%
}}}}
\put(21526,-2311){\makebox(0,0)[lb]{\smash{{\SetFigFont{6}{7.2}{\rmdefault}{\mddefault}{\updefault}{\color[rgb]{0,0,0}$2$}%
}}}}
\put(29701,-2161){\makebox(0,0)[lb]{\smash{{\SetFigFont{6}{7.2}{\rmdefault}{\mddefault}{\updefault}{\color[rgb]{0,0,0}$1$}%
}}}}
\put(26776,-2761){\makebox(0,0)[lb]{\smash{{\SetFigFont{6}{7.2}{\rmdefault}{\mddefault}{\updefault}{\color[rgb]{0,0,0}$2$}%
}}}}
\put(27451,-2911){\makebox(0,0)[lb]{\smash{{\SetFigFont{6}{7.2}{\rmdefault}{\mddefault}{\updefault}{\color[rgb]{0,0,0}$3$}%
}}}}
\put(30676,-2461){\makebox(0,0)[lb]{\smash{{\SetFigFont{6}{7.2}{\rmdefault}{\mddefault}{\updefault}{\color[rgb]{0,0,0}$2$}%
}}}}
\put(15226,-5161){\makebox(0,0)[lb]{\smash{{\SetFigFont{6}{7.2}{\rmdefault}{\mddefault}{\updefault}{\color[rgb]{0,0,0}$2$}%
}}}}
\put(19201,-2086){\makebox(0,0)[lb]{\smash{{\SetFigFont{6}{7.2}{\rmdefault}{\mddefault}{\updefault}{\color[rgb]{0,0,0}$1$}%
}}}}
\put(21526,-1861){\makebox(0,0)[lb]{\smash{{\SetFigFont{6}{7.2}{\rmdefault}{\mddefault}{\updefault}{\color[rgb]{0,0,0}$1$}%
}}}}
\put(23551,-5611){\makebox(0,0)[lb]{\smash{{\SetFigFont{6}{7.2}{\rmdefault}{\mddefault}{\updefault}{\color[rgb]{0,0,0}$3$}%
}}}}
\put(23626,-5161){\makebox(0,0)[lb]{\smash{{\SetFigFont{6}{7.2}{\rmdefault}{\mddefault}{\updefault}{\color[rgb]{0,0,0}$2$}%
}}}}
\put(26776,-2386){\makebox(0,0)[lb]{\smash{{\SetFigFont{6}{7.2}{\rmdefault}{\mddefault}{\updefault}{\color[rgb]{0,0,0}$1$}%
}}}}
\put(30676,-2086){\makebox(0,0)[lb]{\smash{{\SetFigFont{6}{7.2}{\rmdefault}{\mddefault}{\updefault}{\color[rgb]{0,0,0}$1$}%
}}}}
\put(30226,-2461){\makebox(0,0)[lb]{\smash{{\SetFigFont{6}{7.2}{\rmdefault}{\mddefault}{\updefault}{\color[rgb]{0,0,0}$3$}%
}}}}
\put(30226,-2086){\makebox(0,0)[lb]{\smash{{\SetFigFont{6}{7.2}{\rmdefault}{\mddefault}{\updefault}{\color[rgb]{0,0,0}$4$}%
}}}}
\put(26326,-2386){\makebox(0,0)[lb]{\smash{{\SetFigFont{6}{7.2}{\rmdefault}{\mddefault}{\updefault}{\color[rgb]{0,0,0}$4$}%
}}}}
\put(26326,-2761){\makebox(0,0)[lb]{\smash{{\SetFigFont{6}{7.2}{\rmdefault}{\mddefault}{\updefault}{\color[rgb]{0,0,0}$3$}%
}}}}
\put(12526,-1936){\makebox(0,0)[lb]{\smash{{\SetFigFont{6}{7.2}{\rmdefault}{\mddefault}{\updefault}{\color[rgb]{0,0,0}$4$}%
}}}}
\put(12526,-2311){\makebox(0,0)[lb]{\smash{{\SetFigFont{6}{7.2}{\rmdefault}{\mddefault}{\updefault}{\color[rgb]{0,0,0}$3$}%
}}}}
\put(10351,-2461){\makebox(0,0)[lb]{\smash{{\SetFigFont{6}{7.2}{\rmdefault}{\mddefault}{\updefault}{\color[rgb]{0,0,0}$3$}%
}}}}
\put(15151,-5686){\makebox(0,0)[lb]{\smash{{\SetFigFont{6}{7.2}{\rmdefault}{\mddefault}{\updefault}{\color[rgb]{0,0,0}$3$}%
}}}}
\put(22876,-5536){\makebox(0,0)[lb]{\smash{{\SetFigFont{6}{7.2}{\rmdefault}{\mddefault}{\updefault}{\color[rgb]{0,0,0}$4$}%
}}}}
\put(23101,-5011){\makebox(0,0)[lb]{\smash{{\SetFigFont{6}{7.2}{\rmdefault}{\mddefault}{\updefault}{\color[rgb]{0,0,0}$1$}%
}}}}
\put(30751,-4861){\makebox(0,0)[lb]{\smash{{\SetFigFont{6}{7.2}{\rmdefault}{\mddefault}{\updefault}{\color[rgb]{0,0,0}$2$}%
}}}}
\put(30751,-4486){\makebox(0,0)[lb]{\smash{{\SetFigFont{6}{7.2}{\rmdefault}{\mddefault}{\updefault}{\color[rgb]{0,0,0}$1$}%
}}}}
\put(10351,-2086){\makebox(0,0)[lb]{\smash{{\SetFigFont{6}{7.2}{\rmdefault}{\mddefault}{\updefault}{\color[rgb]{0,0,0}$4$}%
}}}}
\put(18751,-2086){\makebox(0,0)[lb]{\smash{{\SetFigFont{6}{7.2}{\rmdefault}{\mddefault}{\updefault}{\color[rgb]{0,0,0}$4$}%
}}}}
\put(18751,-2461){\makebox(0,0)[lb]{\smash{{\SetFigFont{6}{7.2}{\rmdefault}{\mddefault}{\updefault}{\color[rgb]{0,0,0}$3$}%
}}}}
\put(20926,-1936){\makebox(0,0)[lb]{\smash{{\SetFigFont{6}{7.2}{\rmdefault}{\mddefault}{\updefault}{\color[rgb]{0,0,0}$4$}%
}}}}
\put(21001,-2386){\makebox(0,0)[lb]{\smash{{\SetFigFont{6}{7.2}{\rmdefault}{\mddefault}{\updefault}{\color[rgb]{0,0,0}$3$}%
}}}}
\put(30151,-4486){\makebox(0,0)[lb]{\smash{{\SetFigFont{6}{7.2}{\rmdefault}{\mddefault}{\updefault}{\color[rgb]{0,0,0}$4$}%
}}}}
\put(30151,-4861){\makebox(0,0)[lb]{\smash{{\SetFigFont{6}{7.2}{\rmdefault}{\mddefault}{\updefault}{\color[rgb]{0,0,0}$3$}%
}}}}
\put(4876,-2836){\makebox(0,0)[lb]{\smash{{\SetFigFont{6}{7.2}{\rmdefault}{\mddefault}{\updefault}{\color[rgb]{0,0,0}$1$}%
}}}}
\put(2476,-2461){\makebox(0,0)[lb]{\smash{{\SetFigFont{6}{7.2}{\rmdefault}{\mddefault}{\updefault}{\color[rgb]{0,0,0}$2$}%
}}}}
\put(4876,-2461){\makebox(0,0)[lb]{\smash{{\SetFigFont{6}{7.2}{\rmdefault}{\mddefault}{\updefault}{\color[rgb]{0,0,0}$2$}%
}}}}
\put(4876,-2086){\makebox(0,0)[lb]{\smash{{\SetFigFont{6}{7.2}{\rmdefault}{\mddefault}{\updefault}{\color[rgb]{0,0,0}$1$}%
}}}}
\put(6676,-5161){\makebox(0,0)[lb]{\smash{{\SetFigFont{6}{7.2}{\rmdefault}{\mddefault}{\updefault}{\color[rgb]{0,0,0}$3$}%
}}}}
\put(6676,-4786){\makebox(0,0)[lb]{\smash{{\SetFigFont{6}{7.2}{\rmdefault}{\mddefault}{\updefault}{\color[rgb]{0,0,0}$2$}%
}}}}
\put(6151,-5161){\makebox(0,0)[lb]{\smash{{\SetFigFont{6}{7.2}{\rmdefault}{\mddefault}{\updefault}{\color[rgb]{0,0,0}$4$}%
}}}}
\put(6151,-4786){\makebox(0,0)[lb]{\smash{{\SetFigFont{6}{7.2}{\rmdefault}{\mddefault}{\updefault}{\color[rgb]{0,0,0}$1$}%
}}}}
\put(5401,-5236){\makebox(0,0)[lb]{\smash{{\SetFigFont{6}{7.2}{\rmdefault}{\mddefault}{\updefault}{\color[rgb]{0,0,0}$3$}%
}}}}
\put(2476,-2086){\makebox(0,0)[lb]{\smash{{\SetFigFont{6}{7.2}{\rmdefault}{\mddefault}{\updefault}{\color[rgb]{0,0,0}$1$}%
}}}}
\put(2026,-2086){\makebox(0,0)[lb]{\smash{{\SetFigFont{6}{7.2}{\rmdefault}{\mddefault}{\updefault}{\color[rgb]{0,0,0}$4$}%
}}}}
\put(2026,-2461){\makebox(0,0)[lb]{\smash{{\SetFigFont{6}{7.2}{\rmdefault}{\mddefault}{\updefault}{\color[rgb]{0,0,0}$3$}%
}}}}
\put(4426,-2086){\makebox(0,0)[lb]{\smash{{\SetFigFont{6}{7.2}{\rmdefault}{\mddefault}{\updefault}{\color[rgb]{0,0,0}$4$}%
}}}}
\put(4426,-2461){\makebox(0,0)[lb]{\smash{{\SetFigFont{6}{7.2}{\rmdefault}{\mddefault}{\updefault}{\color[rgb]{0,0,0}$3$}%
}}}}
\put(24751,-2461){\makebox(0,0)[lb]{\smash{{\SetFigFont{8}{9.6}{\rmdefault}{\mddefault}{\updefault}{\color[rgb]{0,0,1}$\#$}%
}}}}
\put(11401,-661){\makebox(0,0)[lb]{\smash{{\SetFigFont{8}{9.6}{\rmdefault}{\mddefault}{\updefault}{\color[rgb]{0,0,0}$[LL-]$}%
}}}}
\put(19801,-661){\makebox(0,0)[lb]{\smash{{\SetFigFont{8}{9.6}{\rmdefault}{\mddefault}{\updefault}{\color[rgb]{0,0,0}$[RR+]$}%
}}}}
\put(19276,5939){\makebox(0,0)[lb]{\smash{{\SetFigFont{6}{7.2}{\rmdefault}{\mddefault}{\updefault}{\color[rgb]{0,0,0}$2$}%
}}}}
\put(19276,6314){\makebox(0,0)[lb]{\smash{{\SetFigFont{6}{7.2}{\rmdefault}{\mddefault}{\updefault}{\color[rgb]{0,0,0}$1$}%
}}}}
\put(12376,7664){\makebox(0,0)[lb]{\smash{{\SetFigFont{8}{9.6}{\rmdefault}{\mddefault}{\updefault}{\color[rgb]{0,0,0}$[LL-]$}%
}}}}
\put(18301,7814){\makebox(0,0)[lb]{\smash{{\SetFigFont{8}{9.6}{\rmdefault}{\mddefault}{\updefault}{\color[rgb]{0,0,0}$(2L+)$}%
}}}}
\put(6076,5939){\makebox(0,0)[lb]{\smash{{\SetFigFont{6}{7.2}{\rmdefault}{\mddefault}{\updefault}{\color[rgb]{0,0,0}$2$}%
}}}}
\put(6076,6314){\makebox(0,0)[lb]{\smash{{\SetFigFont{6}{7.2}{\rmdefault}{\mddefault}{\updefault}{\color[rgb]{0,0,0}$1$}%
}}}}
\put(12676,5939){\makebox(0,0)[lb]{\smash{{\SetFigFont{6}{7.2}{\rmdefault}{\mddefault}{\updefault}{\color[rgb]{0,0,0}$2$}%
}}}}
\put(12676,6314){\makebox(0,0)[lb]{\smash{{\SetFigFont{6}{7.2}{\rmdefault}{\mddefault}{\updefault}{\color[rgb]{0,0,0}$1$}%
}}}}
\put(6076,5489){\makebox(0,0)[lb]{\smash{{\SetFigFont{6}{7.2}{\rmdefault}{\mddefault}{\updefault}{\color[rgb]{0,0,0}$1$}%
}}}}
\put(21376,3239){\makebox(0,0)[lb]{\smash{{\SetFigFont{6}{7.2}{\rmdefault}{\mddefault}{\updefault}{\color[rgb]{0,0,0}$3$}%
}}}}
\put(21376,3614){\makebox(0,0)[lb]{\smash{{\SetFigFont{6}{7.2}{\rmdefault}{\mddefault}{\updefault}{\color[rgb]{0,0,0}$2$}%
}}}}
\put(20851,3689){\makebox(0,0)[lb]{\smash{{\SetFigFont{6}{7.2}{\rmdefault}{\mddefault}{\updefault}{\color[rgb]{0,0,0}$1$}%
}}}}
\put(20926,3239){\makebox(0,0)[lb]{\smash{{\SetFigFont{6}{7.2}{\rmdefault}{\mddefault}{\updefault}{\color[rgb]{0,0,0}$4$}%
}}}}
\put(20926,6689){\makebox(0,0)[lb]{\smash{{\SetFigFont{8}{9.6}{\rmdefault}{\mddefault}{\updefault}{\color[rgb]{0,0,0}$\cdot \rho_3$}%
}}}}
\put(12226,6314){\makebox(0,0)[lb]{\smash{{\SetFigFont{6}{7.2}{\rmdefault}{\mddefault}{\updefault}{\color[rgb]{0,0,0}$4$}%
}}}}
\put(12226,5939){\makebox(0,0)[lb]{\smash{{\SetFigFont{6}{7.2}{\rmdefault}{\mddefault}{\updefault}{\color[rgb]{0,0,0}$3$}%
}}}}
\put(5551,6314){\makebox(0,0)[lb]{\smash{{\SetFigFont{6}{7.2}{\rmdefault}{\mddefault}{\updefault}{\color[rgb]{0,0,0}$4$}%
}}}}
\put(5551,5939){\makebox(0,0)[lb]{\smash{{\SetFigFont{6}{7.2}{\rmdefault}{\mddefault}{\updefault}{\color[rgb]{0,0,0}$3$}%
}}}}
\put(8176,3239){\makebox(0,0)[lb]{\smash{{\SetFigFont{6}{7.2}{\rmdefault}{\mddefault}{\updefault}{\color[rgb]{0,0,0}$3$}%
}}}}
\put(8176,3614){\makebox(0,0)[lb]{\smash{{\SetFigFont{6}{7.2}{\rmdefault}{\mddefault}{\updefault}{\color[rgb]{0,0,0}$2$}%
}}}}
\put(7726,3614){\makebox(0,0)[lb]{\smash{{\SetFigFont{6}{7.2}{\rmdefault}{\mddefault}{\updefault}{\color[rgb]{0,0,0}$1$}%
}}}}
\put(7726,3239){\makebox(0,0)[lb]{\smash{{\SetFigFont{6}{7.2}{\rmdefault}{\mddefault}{\updefault}{\color[rgb]{0,0,0}$4$}%
}}}}
\put(18826,6314){\makebox(0,0)[lb]{\smash{{\SetFigFont{6}{7.2}{\rmdefault}{\mddefault}{\updefault}{\color[rgb]{0,0,0}$4$}%
}}}}
\put(18826,5939){\makebox(0,0)[lb]{\smash{{\SetFigFont{6}{7.2}{\rmdefault}{\mddefault}{\updefault}{\color[rgb]{0,0,0}$3$}%
}}}}
\put(18301,5939){\makebox(0,0)[lb]{\smash{{\SetFigFont{6}{7.2}{\rmdefault}{\mddefault}{\updefault}{\color[rgb]{0,0,0}$2$}%
}}}}
\put(14626,-5461){\makebox(0,0)[lb]{\smash{{\SetFigFont{6}{7.2}{\rmdefault}{\mddefault}{\updefault}{\color[rgb]{0,0,0}$4$}%
}}}}
\put(14776,-5086){\makebox(0,0)[lb]{\smash{{\SetFigFont{6}{7.2}{\rmdefault}{\mddefault}{\updefault}{\color[rgb]{0,0,0}$1$}%
}}}}
\put(14251,-5461){\makebox(0,0)[lb]{\smash{{\SetFigFont{6}{7.2}{\rmdefault}{\mddefault}{\updefault}{\color[rgb]{0,0,0}$3$}%
}}}}
\put(19276,13739){\makebox(0,0)[lb]{\smash{{\SetFigFont{6}{7.2}{\rmdefault}{\mddefault}{\updefault}{\color[rgb]{0,0,0}$2$}%
}}}}
\put(19276,14114){\makebox(0,0)[lb]{\smash{{\SetFigFont{6}{7.2}{\rmdefault}{\mddefault}{\updefault}{\color[rgb]{0,0,0}$1$}%
}}}}
\put(12376,15464){\makebox(0,0)[lb]{\smash{{\SetFigFont{8}{9.6}{\rmdefault}{\mddefault}{\updefault}{\color[rgb]{0,0,0}$[LL-]$}%
}}}}
\put(18301,15614){\makebox(0,0)[lb]{\smash{{\SetFigFont{8}{9.6}{\rmdefault}{\mddefault}{\updefault}{\color[rgb]{0,0,0}$(2L+)$}%
}}}}
\put(6076,13739){\makebox(0,0)[lb]{\smash{{\SetFigFont{6}{7.2}{\rmdefault}{\mddefault}{\updefault}{\color[rgb]{0,0,0}$2$}%
}}}}
\put(6076,14114){\makebox(0,0)[lb]{\smash{{\SetFigFont{6}{7.2}{\rmdefault}{\mddefault}{\updefault}{\color[rgb]{0,0,0}$1$}%
}}}}
\put(12676,13739){\makebox(0,0)[lb]{\smash{{\SetFigFont{6}{7.2}{\rmdefault}{\mddefault}{\updefault}{\color[rgb]{0,0,0}$2$}%
}}}}
\put(12676,14114){\makebox(0,0)[lb]{\smash{{\SetFigFont{6}{7.2}{\rmdefault}{\mddefault}{\updefault}{\color[rgb]{0,0,0}$1$}%
}}}}
\put(6076,13289){\makebox(0,0)[lb]{\smash{{\SetFigFont{6}{7.2}{\rmdefault}{\mddefault}{\updefault}{\color[rgb]{0,0,0}$1$}%
}}}}
\put(21376,11039){\makebox(0,0)[lb]{\smash{{\SetFigFont{6}{7.2}{\rmdefault}{\mddefault}{\updefault}{\color[rgb]{0,0,0}$3$}%
}}}}
\put(21376,11414){\makebox(0,0)[lb]{\smash{{\SetFigFont{6}{7.2}{\rmdefault}{\mddefault}{\updefault}{\color[rgb]{0,0,0}$2$}%
}}}}
\put(20851,11489){\makebox(0,0)[lb]{\smash{{\SetFigFont{6}{7.2}{\rmdefault}{\mddefault}{\updefault}{\color[rgb]{0,0,0}$1$}%
}}}}
\put(20926,11039){\makebox(0,0)[lb]{\smash{{\SetFigFont{6}{7.2}{\rmdefault}{\mddefault}{\updefault}{\color[rgb]{0,0,0}$4$}%
}}}}
\put(20926,14489){\makebox(0,0)[lb]{\smash{{\SetFigFont{8}{9.6}{\rmdefault}{\mddefault}{\updefault}{\color[rgb]{0,0,0}$\cdot \rho_3$}%
}}}}
\put(12226,14114){\makebox(0,0)[lb]{\smash{{\SetFigFont{6}{7.2}{\rmdefault}{\mddefault}{\updefault}{\color[rgb]{0,0,0}$4$}%
}}}}
\put(12226,13739){\makebox(0,0)[lb]{\smash{{\SetFigFont{6}{7.2}{\rmdefault}{\mddefault}{\updefault}{\color[rgb]{0,0,0}$3$}%
}}}}
\put(5551,14114){\makebox(0,0)[lb]{\smash{{\SetFigFont{6}{7.2}{\rmdefault}{\mddefault}{\updefault}{\color[rgb]{0,0,0}$4$}%
}}}}
\put(5551,13739){\makebox(0,0)[lb]{\smash{{\SetFigFont{6}{7.2}{\rmdefault}{\mddefault}{\updefault}{\color[rgb]{0,0,0}$3$}%
}}}}
\put(8176,11039){\makebox(0,0)[lb]{\smash{{\SetFigFont{6}{7.2}{\rmdefault}{\mddefault}{\updefault}{\color[rgb]{0,0,0}$3$}%
}}}}
\put(8176,11414){\makebox(0,0)[lb]{\smash{{\SetFigFont{6}{7.2}{\rmdefault}{\mddefault}{\updefault}{\color[rgb]{0,0,0}$2$}%
}}}}
\put(14776,11039){\makebox(0,0)[lb]{\smash{{\SetFigFont{6}{7.2}{\rmdefault}{\mddefault}{\updefault}{\color[rgb]{0,0,0}$3$}%
}}}}
\put(14776,11414){\makebox(0,0)[lb]{\smash{{\SetFigFont{6}{7.2}{\rmdefault}{\mddefault}{\updefault}{\color[rgb]{0,0,0}$2$}%
}}}}
\put(7726,11414){\makebox(0,0)[lb]{\smash{{\SetFigFont{6}{7.2}{\rmdefault}{\mddefault}{\updefault}{\color[rgb]{0,0,0}$1$}%
}}}}
\put(7726,11039){\makebox(0,0)[lb]{\smash{{\SetFigFont{6}{7.2}{\rmdefault}{\mddefault}{\updefault}{\color[rgb]{0,0,0}$4$}%
}}}}
\put(14326,11414){\makebox(0,0)[lb]{\smash{{\SetFigFont{6}{7.2}{\rmdefault}{\mddefault}{\updefault}{\color[rgb]{0,0,0}$1$}%
}}}}
\put(14326,11039){\makebox(0,0)[lb]{\smash{{\SetFigFont{6}{7.2}{\rmdefault}{\mddefault}{\updefault}{\color[rgb]{0,0,0}$4$}%
}}}}
\put(18826,14114){\makebox(0,0)[lb]{\smash{{\SetFigFont{6}{7.2}{\rmdefault}{\mddefault}{\updefault}{\color[rgb]{0,0,0}$4$}%
}}}}
\put(18826,13739){\makebox(0,0)[lb]{\smash{{\SetFigFont{6}{7.2}{\rmdefault}{\mddefault}{\updefault}{\color[rgb]{0,0,0}$3$}%
}}}}
\put(9751,13739){\makebox(0,0)[lb]{\smash{{\SetFigFont{8}{9.6}{\rmdefault}{\mddefault}{\updefault}{\color[rgb]{0,0,1}$\#$}%
}}}}
\put(7201,11039){\makebox(0,0)[lb]{\smash{{\SetFigFont{6}{7.2}{\rmdefault}{\mddefault}{\updefault}{\color[rgb]{0,0,0}$3$}%
}}}}
\put(13801,10964){\makebox(0,0)[lb]{\smash{{\SetFigFont{6}{7.2}{\rmdefault}{\mddefault}{\updefault}{\color[rgb]{0,0,0}$3$}%
}}}}
\end{picture}%

%% file: gradings.tex
\section{Gradings}\label{sec:grading}
This section is devoted to the gradings on $\MAlg$. As in the $\HFa$
case, the algebra $\MAlg$ is graded by a non-commutative group. (See
Section~\ref{sec:W-gradings} for a discussion of gradings of weighted
$\Ainf$-algebras by non-commutative groups, and~\cite[Section
2.5]{LOT1} for a more leisurely discussion of gradings of
$\Ainf$-algebras by non-commutative groups.) Also as in the $\HFa$
case, there are three different groups that can be used to grade
$\MAlg$. Consider the pointed matched circle $(Z,\CircPts,M,z)$ for the
torus (Figure~\ref{fig:torus}). The largest of the grading groups,
denoted $\bigGroup$, is a central extension
\begin{equation}\label{eq:big-group-extension}
0\longrightarrow\ZZ\longrightarrow \bigGroup\longrightarrow H_1(Z,\CircPts)\cong \ZZ^4\longrightarrow 0.
\end{equation}
(This is analogous to $\bigGroup(4)$ from~\cite[Section 3.3.1]{LOT1}.)
The smallest of the grading groups, denoted $\smallGroup(\mathbb{T})$, is a
central extension
\begin{equation}\label{eq:small-group-extension}
0\longrightarrow\ZZ\longrightarrow \smallGroup(\mathbb{T})\longrightarrow H_1(T^2)\cong \ZZ^2\longrightarrow 0.
\end{equation}
(This is analogous to $\smallGroup(\PMC)$ from~\cite[Section 3.3.2]{LOT1}.)
An intermediate grading group $\smallGroup$ is described in
Section~\ref{sec:Intermediate} (analogous to~\cite[Section 11.1]{LOT1}).
In Section~\ref{sec:unique}, we will primarily work with this
intermediate group.

The gradings by $\smallGroup$ and $\smallGroup(\mathbb{T})$ are
inherited from the grading by $\bigGroup$, but through different
processes. The grading by $\smallGroup$ is induced by a homomorphism
$\bigGroup\onto\smallGroup$ (Section~\ref{sec:Intermediate}). The
grading by $\smallGroup(\mathbb{T})$ is defined using grading
refinement data (Section~\ref{sec:small-gr}).

We also define two other gradings, the winding number grading and
total weight grading, in Section~\ref{sec:win-num}, and a (rather boring) mod-2
grading coming from the grading by $\smallGroup(\mathbb{T})$ with
respect to suitable choices, in Section~\ref{sec:mod-2-gr}.

\subsection{The big grading group}\label{sec:big-group}
Elements of $H_1(Z,\CircPts)$ are linear combinations of connected
components of $Z\setminus\CircPts$. There is a map $m\co
H_1(Z,\CircPts)\otimes H_0(\CircPts)\to \ZZ$ given by defining, for $I$
a component of $Z\setminus\CircPts$ and $p\in\CircPts$, a multiplicity
\[
  m(I,p) =
  \begin{cases}
    1/2 & \text{if $p$ is the terminal endpoint of $I$}\\
    -1/2 & \text{if $p$ is the initial endpoint of $I$}\\
    0 & \text{if $p$ is disjoint from the closure of $I$}
  \end{cases}
\]
and extending bilinearly. We can use $m$ to define
a linking pairing $L\co H_1(Z,\CircPts)\otimes H_1(Z,\CircPts)\to \OneHalf\ZZ$
by setting $L(\alpha_1,\alpha_2)=m(\alpha_2,\bdy\alpha_1) = -m(\alpha_1,\bdy\alpha_2)$.

The big grading group $\bigGroup$ is the central extension as in
Formula~\eqref{eq:big-group-extension} with commutation relation
$gh=\lambda^{2L([g],[h])}hg$ where $[g]$ denotes the image of $g$ in $\ZZ^4$ and
$\lambda$ is a generator of the central $\ZZ$.
Explicitly, consider the set of quintuples
$(m;a,b,c,d)\in\left((\OneHalf\ZZ)\times\ZZ^4\right)$. Define a multiplication by
setting
\begin{multline*}
(m;a,b,c,d)\cdot (m';a',b',c',d')
\\=
\left(m+m'
+\OneHalf\left|
  \begin{smallmatrix}
    a & b\\
    a' & b'
  \end{smallmatrix}
  \right|
+\OneHalf\left|
  \begin{smallmatrix}
    b & c\\
    b' & c'
  \end{smallmatrix}
  \right|
+\OneHalf\left|
  \begin{smallmatrix}
    c & d\\
    c' & d'
  \end{smallmatrix}
  \right|
+\OneHalf\left|
  \begin{smallmatrix}
    d & a\\
    d' & a'
  \end{smallmatrix}
  \right|,a+a',b+b',c+c',d+d'\right)
\end{multline*}

\begin{lemma}\label{lem:def-big-group}
  This operation makes $(\OneHalf\ZZ)\times\ZZ^4$ into a
  group. Further, the elements $(-1/2;1,0,0,0)$, $(-1/2;0,1,0,0)$,
  $(-1/2;0,0,1,0)$ and $(-1/2;0,0,0,1)$ and $\lambda=(1;0,0,0,0)$
  generate an index $2$ subgroup isomorphic to $\bigGroup$.
\end{lemma}
\begin{proof}
  Straightforward; see also~\cite[Proposition 3.37]{LOT1} (which also makes the index $2$ subgroup explicit).
\end{proof}

Given an element $g=(m;a,b,c,d)\in\bigGroup$ we refer to $m$ as 
\emph{the Maslov component} of $g$ and $(a,b,c,d)$ as \emph{the $\SpinC$
component} of $g$.

To define the grading of $\MAlg$ by $\bigGroup$,
recall from Section~\ref{sec:AsUnDefAlg} that the algebra
$\AsUnDefAlg$ has an $\Field$-basis given by idempotents $\iota_0$,
$\iota_1$ and
chords $\rho_{i,i+1,\dots,i+n}$ for some $i\in\{1,2,3,4\}$ and
$n\geq 0$. Each $\rho_{i,\dots,i+n}$ has a support $[\rho_{i,\dots,i+n}]\in
H_1(Z,\CircPts)$.

Define $\grb(\iota_0)=\grb(\iota_1)=0$.  Define the grading of
$\rho_{i,\dots,i+n}$ to be
\begin{equation}\label{eq:gr-of-rho}
\grb(\rho_{i,\dots,i+n})=
\begin{cases}
(-\frac{n+1}{4};[\rho_{i,\dots,i+n}]) & 4 \mid n+1\\
(-1/2-\left\lfloor\frac{n+1}{4}\right\rfloor;[\rho_{i,\dots,i+n}]) & 4\nmid n+1.
\end{cases}
\end{equation}

For example:
\begin{align*}
  \grb(\rho_1)&=(-1/2;1,0,0,0) & \grb(\rho_2)&=(-1/2;0,1,0,0) \\
  \grb(\rho_3)&=(-1/2;0,0,1,0) & \grb(\rho_4)&=(-1/2;0,0,0,1)\\
  \grb(\rho_{12})&=(-1/2;1,1,0,0) & \grb(\rho_{23})&=(-1/2;0,1,1,0)\\
  \grb(\rho_{34})&=(-1/2;0,0,1,1)& \grb(\rho_{41})&=(-1/2;1,0,0,1)\\
  \grb(\rho_{123})&=(-1/2;1,1,1,0) & \grb(\rho_{234})&=(-1/2;0,1,1,1) & \dots \\
  \grb(\rho_{1234})&=(-1;1,1,1,1) & \grb(\rho_{2341})&=(-1;1,1,1,1) & \dots\\
  \grb(\rho_{12341})&=(-3/2;2,1,1,1) & \dots.
\end{align*}

\begin{lemma}\label{lem:gr-on-AsUnDef}
  Formula~\eqref{eq:gr-of-rho} defines a grading on the associative
  algebra $\AsUnDefAlg$ by $\bigGroup$.
\end{lemma}
\begin{proof}
  First, note that $\grb(\rho_{i,i+1,i+2,i+3})=(-1;1,1,1,1)$ is central, and
  the grading satisfies
  \[
    \grb(\rho_{i,\dots,i+3})\grb(\rho_{i,\dots,i+n})=\grb(\rho_{i,\dots,i+n+4}).
  \]
  So, it
  suffices to verify that if $n,m\leq 2$ and $\rho_{i,\dots,i+n}\cdot
  \rho_{j,\dots,j+m}\neq 0$ then
  \[
    \grb(\rho_{i,\dots,i+n})\grb(\rho_{j,\dots,j+m})
    =\grb(\rho_{i,\dots,i+n}\cdot\rho_{j,\dots,j+m}).
  \]
  By symmetry, we may assume $i=1$. Further, it suffices to
  check the cases $n=1$ or $m=1$, as we can factor any chord into
  length $1$ chords. So, we check:
  \begin{align*}
    \grb(\rho_{1})\grb(\rho_{2})&=(-1/2;1,0,0,0)(-1/2;0,1,0,0)=(-1/2;1,1,0,0)=\grb(\rho_{12})\\
    \grb(\rho_{12})\grb(\rho_{3})&=(-1/2;1,1,0,0)(-1/2;0,0,1,0)=(-1/2;1,1,1,0)=\grb(\rho_{123})\\
    \grb(\rho_{123})\grb(\rho_{4})&=(-1/2;1,1,1,0)(-1/2;0,0,0,1)=(-1;1,1,1,1)=\grb(\rho_{1234})\\
    \grb(\rho_{1})\grb(\rho_{23})&=(-1/2;1,0,0,0)(-1/2;0,1,1,0)=(-1/2;1,1,1,0)=\grb(\rho_{123})\\
    \grb(\rho_{1})\grb(\rho_{234})&=(-1/2;1,0,0,0)(-1/2;0,1,1,1)=(-1;1,1,1,1)=\grb(\rho_{1234}).
  \end{align*}
  This proves the result.
\end{proof}

Next, to grade $\UnDefAlg$, define 
\begin{equation}\label{eq:gr-of-U}
  \grb(U)=(-1;1,1,1,1).
\end{equation}

\begin{proposition}\label{prop:gr-on-UnDef}
  Formulas~\eqref{eq:gr-of-rho} and~\eqref{eq:gr-of-U} define a
  grading on the $\Ainf$-algebra $\UnDefAlg$ by $\bigGroup$ with
  \begin{equation}\label{eq:lambdad-is}
    \lambda_d = \lambda = (1;0,0,0,0).
  \end{equation}
\end{proposition}
\begin{proof}
  Since $\mu_n$ is $U$-equivariant and $\grb(U)$ is central,
  it suffices to prove that for any $n$, basic elements
  $a_1,\dots,a_n$
  and term $b\in \mu_n(a_1,\dots,a_n)$,
  \begin{equation}
    \label{eq:GradedOperation}
  \grb(b)=\lambda^{n-2}\grb(a_1)\cdots\grb(a_n).
  \end{equation}

  Lemma~\ref{lem:gr-on-AsUnDef} implies Equation~\eqref{eq:GradedOperation}
  when $n=2$.

  The
  operation $\mu_3$ vanishes identically. For $\mu_4$, note that
  \begin{align*}
  \grb(\mu_4(\rho_4,\rho_3,\rho_2,\rho_1))&=\grb(U)=(-1;1,1,1,1)\\
  \lambda^2\grb(\rho_4)\grb(\rho_3)\grb(\rho_2)\grb(\rho_1)&=\lambda^2(-3/2;0,0,1,1)(-3/2;1,1,0,0)\\
  &=(2;0,0,0,0)(-3;1,1,1,1)=(-1;1,1,1,1).
  \end{align*}
  Similar computations hold for cyclic permutations of the indices.

  We prove Equation~\eqref{eq:GradedOperation} in general
  by induction on
  the total length of the inputs $L=\sum_{i=1}^n |a_i|$. 

  If $L>4$, there is some $i$ so that $a_i$ factors as a product $a_i=a_i'\cdot a_i''$ of
  Reeb elements. This is obvious  when $n=4$ and it follows from
  Lemma~\ref{lem:alg-property-factor} when $n>4$. Consider the
  $\Ainf$ relation with inputs
  $(a_1,\dots,a_{i-1},a_i',a_i'',a_{i+1},\dots,a_n)$.
  One term in the relation is
  $\mu_n(a_1,\dots,\mu_2(a_i',a_i''),\dots,a_n)=b$.
  By Lemma~\ref{lem:alg-property-nonmult}, the only other non-zero
  terms have the form $\mu_{n-k+2}(a_1,\dots,
  a'_i,\mu_{k}(a''_i,\dots,a_{i+k}),\dots,a_n)$ or
  $\mu_{n-k+2}(a_1,\dots,\mu_k(a_{i-k+1},\dots,a'_i),a''_i,\dots,a_n)$,
  for some $k>2$; and there must be a non-zero term of at least one of
  these two forms. (In the language of
  Lemma~\ref{lem:ProductCancellations}, in the unweighted case, we
  cancel $\Graphs_1$ against patterns of type $\mathfrak{P}$, $\mathfrak{L}$,
  and~$\mathfrak{R}$, all of which are of the above form.)
  So,
  in view of Lemma~\ref{lem:WeightGrading}, the
  inductive hypothesis ensures that
  \[
  \grb(U^\ell)=\lambda^{n-k}\lambda^{k-2}\grb(a_1)\cdots\grb(a'_i)\grb(a''_i)\cdots\grb(a_n)
  =\lambda^{n-2}\grb(a_1)\cdots\grb(a_n),
  \]
  as desired.
\end{proof}

Finally, to define a grading on $\MAlg$, we need to specify an element
$\lambda_w$, so that $\mu_m^k$ has degree $\lambda_d^{m-2}\lambda_w^k$
(or equivalently, the formal variable $t$ has grading
$\lambda_w^{-1}$); see Section~\ref{sec:W-gradings}. Define:
\begin{equation}
  \label{eq:gr-of-weight}
  \lambda_w=(1;1,1,1,1).
\end{equation}

\begin{theorem}\label{thm:grading}
  Formulas~\eqref{eq:gr-of-rho},~\eqref{eq:gr-of-U},~\eqref{eq:lambdad-is},
  and~\eqref{eq:gr-of-weight} define a grading on the weighted $\Ainf$-algebra
  $\MAlg$ by $\bigGroup$.
\end{theorem}
\begin{proof}
  The proof is by induction on the weight. The case of weight $0$ is
  Proposition~\ref{prop:gr-on-UnDef}. Next, suppose that we know the
  statement for weight $k-1$, and that $\mu_n^k(a_1,\dots,a_n)=b\neq
  0$. By Lemma~\ref{lem:alg-diet}, we can find an $i$ and a
  factorization $a_i=a'_ia''_i$ so that
  $\mu_{n+2}^{k-1}(a_1,\dots,a'_i,\mu_0^1,a''_i,\dots,a_n)$ has $b$ as
  a term. Then
  \begin{align*}
  \grb(b)&=\lambda_w^{k-1}\lambda_d^{n}\grb(a_1)\cdots\grb(a'_i)\grb(\mu_0^1)\grb(a''_i)\cdots\grb(a_n)\\
  &=\lambda_w^{k-1}\lambda_d^{n-1}\grb(\mu_0^1)\grb(a_1)\cdots\grb(a_n).\\
  \shortintertext{So, the result follows from the fact that }
  \lambda_w^{k-1}\lambda_d^{n-1}\grb(\mu_0^1)&=(k-1;k-1,k-1,k-1,k-1)(n;0,0,0,0)(-1;1,1,1,1)\\
  &=(k+n-2;k,k,k,k)=\lambda_w^k\lambda_d^{n-2}.\qedhere
  \end{align*}
\end{proof}


\subsection{The intermediate grading group}\label{sec:Intermediate}
As in the case of bordered $\HFa$ with torus (but not higher genus)
boundary~\cite[Section 11.1]{LOT1}, there is a grading group $\smallGroup$
between $\bigGroup$ and $\smallGroup(\mathbb{T})$ which admits a homomorphism
from $\bigGroup$.
That is, let the \emph{intermediate grading group} be
\begin{equation}
  \smallGroup=\biggl\{(m;a,b)\in (\OneHalf\ZZ)^3\biggm\mid a+b\in\ZZ, m+\frac{(2a+1)(a+b+1)+1}{2}\in\ZZ\biggr\}.
\end{equation}
The second arithmetic condition for $m$, $a$, and $b$ is equivalent to:
$m\in\ZZ$ if and only if $a,b\in\ZZ$ and $a\equiv b\pmod{2}$.
The multiplication on $\smallGroup$ is 
\begin{equation}
  (m;a,b)\cdot(n;c,d)=(m+n+ad-bc;a+c,b+d).
\end{equation}
The homological grading element is
\[
  \lambda_d=(1;0,0).
\]

\begin{lemma}
  This operation makes $\smallGroup$ into a group.
\end{lemma}
\begin{proof}
  The only nontrivial part is verifying that $\smallGroup$ is closed under
  multiplication. The condition that $a+b\in\ZZ$ is certainly closed under
  multiplication. If we let
  \[
    f(m;a,b)=m+\frac{(2a+1)(a+b+1)+1}{2}
  \]
  then
  \[
    f((m;a,b)\cdot(n;c,d))-f(m;a,b)-f(m;c,d)=ad-bc+\frac{4ac+2ad+2bc-2}{2}\equiv 2a(c+d)\pmod{1}
  \]
  which is an integer, since $c+d$ is.
\end{proof}

There is a homomorphism $\bigGroup\to\smallGroup$ defined by
\begin{equation}\label{eq:big-to-small}
(j;a,b,c,d)\mapsto \left(j-d;\frac{a+b-c-d}{2},\frac{-a+b+c-d}{2}\right).
\end{equation}
(See also~\cite[Section 11.1]{LOT1}.)
Composing with this homomorphism allows us to turn the grading by $\bigGroup$ into a grading by $\smallGroup$, with
\begin{align*}
  \gr(\rho_1)&=(-1/2;1/2,-1/2) & \gr(\rho_2)&=(-1/2;1/2,1/2) \\
  \gr(\rho_3)&=(-1/2;-1/2,1/2) & \gr(\rho_4)&=(-3/2;-1/2,-1/2)\\
  \gr(\rho_{12})&=(-1/2;1,0) & \gr(\rho_{23})&=(-1/2;0,1)\\
  \gr(\rho_{34})&=(-3/2;-1,0)& \gr(\rho_{41})&=(-3/2;0,-1)\\
  \gr(\rho_{123})&=(-1/2;1/2,1/2) & \gr(\rho_{234})&=(-3/2;-1/2,1/2) & \dots \\
  \gr(\rho_{1234})&=(-2;0,0) & \gr(\rho_{2341})&=(-2;0,0) & \dots\\
  \gr(\rho_{12341})&=(-5/2;1/2,-1/2) & \dots\\
  \gr(U)&=(-2;0,0) & \lambda_w&=(0;0,0).
\end{align*}
In particular, having the term $j-d$ instead of just $j$ in
Formula~\eqref{eq:big-to-small} ensured that $\lambda_w=(0;0,0)$.

\subsection{Grading refinements and the small grading group}\label{sec:small-gr}
The small grading group $\smallGroup(\mathbb{T})$ is a central extension as
in Formula~\eqref{eq:small-group-extension} with commutation relation
\[
  gh=\lambda^{2[g]\cdot[h]}hg
\]
where $[g]$ denotes the image of $g$ in $H_1(F)$ and $\cdot$ is the intersection
pairing; that is, the central extension corresponding to the 2-cocycle on
$H_1(T^2)$ given by the intersection form (not twice it).
We have several explicit models for this group:
\begin{lemma}\label{lem:G-is-subgroup}
  The group $\smallGroup(\mathbb{T})$ is isomorphic to the following:
  \begin{enumerate}[label=(SG-\arabic*)]
  \item\label{item:real-small-1} The subgroup
    $\{(m;a,b,c,d)\in\bigGroup\mid b=a+c,\ d=0\}\subset\bigGroup$.
  \item\label{item:real-small-2} The subquotient
    $\{(m;a,b,c,d)\in\bigGroup/ (1;1,1,1,1)\mid a+c=b+d\}\subset
    \bigGroup/\langle (1;1,1,1,1)\rangle$.
  \item\label{item:real-small-3} The subgroup
    $\{(m;a,b)\in \smallGroup\mid a,b\in\ZZ\}\subset \smallGroup$.
  \end{enumerate}
\end{lemma}
\begin{proof}
  Model~\ref{item:real-small-1} for the small grading group is the one
  given in our first paper and is identified with the abstract
  definition there~\cite[Section 3.3.2]{LOT1}. The isomorphism between
  model~\ref{item:real-small-1} and~\ref{item:real-small-2} is
  clear. The isomorphism between model~\ref{item:real-small-1}
  and~\ref{item:real-small-3} is given by
  $(m;a,b)\mapsto (m;a,a+b,b,0)$.
\end{proof}
We will most often use the third of these ways of realizing
$\smallGroup(\mathbb{T})$, viewing elements of
$\smallGroup(\mathbb{T})$ as triples 
$(m;a,b)\in\OneHalf \ZZ \times \ZZ\times \ZZ$ with $m+\frac{a+b}{2}\in\ZZ$, equipped with group law
\[
  (m;a,b)\cdot (m';a',b')=\bigl(m+m'+ \left|\begin{smallmatrix}
      a & b \\
      a' & b'
    \end{smallmatrix}\right|;a+a',b+b'\bigr).
\]
The grading elements are, again,
\begin{align*}
  \lambda_d&=\lambda=(1;0,0) & \lambda_w&=(0;0,0).
\end{align*}

Note that the surjection $\bigGroup\to\smallGroup$ does not send
$\smallGroup(\mathbb{T})\subset\bigGroup$ to
$\smallGroup(\mathbb{T})\subset\smallGroup$.

To define a grading on $\MAlg$ by $\smallGroup$, we use the notion of
grading refinement data, which we review first.

\subsubsection{Grading refinement data in general}
The following is an abstract reformulation of material from our
earlier papers~\cite[Section~3.3.2]{LOT1},~\cite[Section~3.2.1]{LOT2}.

Recall that a groupoid is a category in which every morphism is invertible. All
groupoids relevant later will also be \emph{connected}, i.e.,
$\Hom(i,j)\neq\emptyset$ for all objects $i,j$.
Since we will think of a groupoid as a group with many
objects, given a groupoid $\mathcal{G}$ and morphisms
$g\in\Hom_{\mathcal{G}}(x,y)$, $h\in\Hom_{\mathcal{G}}(y,z)$, let
\[
  g\cdot h = h\circ g.
\]
From here on, we will drop the subscript $\mathcal{G}$ from
$\Hom_{\mathcal{G}}$ when it will not cause confusion.

\begin{definition}\label{def:groupoid-central}
  A \emph{central element} of $\mathcal{G}$ consists of an element
  $\lambda_x\in \Hom(x,x)$ for each object $x\in\Obj(\mathcal{G})$ so
  that for any $g\in\Hom(x,y)$, $\lambda_x\cdot g =
  g\cdot\lambda_y$. We will typically denote the collection of
  elements $\lambda_x$ by $\lambda$.
\end{definition}

Let $\Ground=\bigoplus_{x\in X}\Field$ for some finite set $X$ and let
$\Alg$ be a strictly unital weighted algebra over $\Ground$. So, each
$\x\in X$ corresponds to some idempotent $I_\x = \mu_2^0(\x,1)\in
\Alg$.

\begin{definition}\label{def:groupoid-gr}
  Let $\mathcal{G}$ be a groupoid and $\lambda_d$ and $\lambda_w$ central elements of
  $G$. A \emph{grading on $\Alg$ by $\mathcal{G}$} consists of a map
  of sets $\pi\co X \to \Obj(\mathcal{G})$ and, for each $\x,\y\in X$,
  a decomposition 
  \[
    I_\x \cdot \Alg \cdot I_\y =\bigoplus_{g\in \Hom(\pi(\x),\pi(\y))} (I_\x \cdot \Alg \cdot I_\y)_g
  \]
  satisfying the following property. For each pair of integers $w,m\geq 0$, $(w,m)\neq(0,0)$,
  sequence $\x_0,\dots,\x_m\in X$, elements $g_i\in
  \Hom(\pi(\x_{i-1}),\pi(\x_i))$ for $i=1,\dots,m$, and elements
  $a_i\in (I_{\x_{i-1}}\cdot\Alg\cdot I_{\x_i})_{g_i}$, we have
  \[
    \mu_m^w(a_1,\dots,a_m)\in (I_{\x_0}\cdot\Alg\cdot
    I_{\x_m})_{\lambda_d^{m-2}\lambda_w^wg_1\cdots g_m}.
  \]
\end{definition}

This generalizes the notion of group-valued gradings.  If $G'$ is a
group, there is an associated groupoid with one object. Any
$G'$-graded algebra $\Alg$ can be viewed as graded by this associated
groupoid.

\begin{definition}\label{def:groupoid-refinement-data}
  With notation as in Definition~\ref{def:groupoid-gr}, \emph{grading
    refinement data} consists of an element $\x_0\in X$, called the
  \emph{base idempotent}, and for every object $\x\in X$ an element
  $\psi(\x)\in\Hom(\x_0,\x)$.
\end{definition}

\begin{definition}\label{def:groupoid-refined-gr}
  Let $\Alg$ be a $\mathcal{G}$-graded algebra, and let
  $G=\Hom_{\mathcal{G}}(\x_0,\x_0)$.  Given grading refinement data
  $\{\psi\}$, there is an induced $G$-grading on $\Alg$, denoted
  $\Alg_\psi$, specified by
  \[
    \gr_\psi(I_\x\cdot a \cdot I_\y)=\psi(\x)\cdot \gr(a)\cdot\psi(\y)^{-1}
  \]
  for each $\mathcal{G}$-homogeneous element $a$ with
  $a=I_\x\cdot a \cdot I_\y$. We call this the $G$-valued grading the
  \emph{refined grading with respect to $\psi$.}
\end{definition}

We can also un-refine gradings:
\begin{lemma}\label{lem:unrefine}
  Let $\Alg$ be a weighted $\Ainf$-algebra over $\Ground$,
  $(\mathcal{G},\lambda_d,\lambda_w)$ be a groupoid with distinguished central
  elements, and $(\x_0,\{\psi_\x\})$ be grading refinement data. Suppose $\Alg$
  is graded by $(G=\Hom_{\mathcal{G}}(\x_0,\x_0),\lambda_d,\lambda_w)$.  Let
  $\gr_\psi$ denote this $G$-valued grading. Then setting
  \[
    \gr(a)=\psi(\iota_i)^{-1}\gr_\psi(a)\psi(\iota_j)
    \qquad\text{if }\iota_ia\iota_j=a
  \]
  defines a grading on $\Alg$ by $\mathcal{G}$.
\end{lemma}
\begin{proof}
  This is immediate from the definitions.
\end{proof}

\subsubsection{Grading refinement data for the torus algebra}\label{sec:gr-refine}
Consider the groupoid $\mathcal{G}$ with two objects, $0$ and $1$, and
\[
  \Hom(i,j)=\biggl\{(m;a,b)\in \smallGroup\biggm\mid a+\frac{i+j}{2}\in\ZZ\biggr\}.
\]
\begin{lemma}
  Multiplication in $\smallGroup$ makes $\mathcal{G}$ into a
  groupoid.
\end{lemma}
\begin{proof}
  It is straightforward to verify that $\mathcal{G}$ is closed under multiplication
  (composition) and inverses.
\end{proof}

Observe that for $i\in\{0,1\}$, we have that
$\Hom(i,i)=\smallGroup(\Torus)$, the smallest of the three grading
groups.

The element $\lambda\in \smallGroup$ induces an element
$\lambda\in\Hom_{\mathcal{G}}(i,i)$, $i\in\{0,1\}$, forming a central
element of $\mathcal{G}$.

\begin{lemma}
  The $\smallGroup$-grading on $\MAlg$ induces a $\mathcal{G}$-grading
  on $\MAlg$.
\end{lemma}
\begin{proof}
  It suffices to verify that the grading of each algebra element $a$
  with $\iota_ia\iota_j=a$ ($i,j\in\{0,1\}$) lies in $\Hom(i,j)$. To
  see this, it suffices to check the result for $\rho_1$, $\rho_2$,
  $\rho_3$, and $\rho_4$, which is straightforward.
\end{proof}

We define grading refinement data for the torus algebra. Choose the
base idempotent $x_0=\iota_0$, and let $\psi(\iota_0)=e$ and
$\psi(\iota_1)=\gr(\rho_1)$. This induces a grading $\gr_\psi$ on
$\MAlg$ by $\smallGroup(\mathbb{T})$. Explicitly, we have:
\begin{align*}
  \gr_\psi(\rho_1)&=(0;0,0) & \gr_\psi(\rho_2)&=(-1/2;1,0) \\
  \gr_\psi(\rho_3)&=(0;-1,1) & \gr_\psi(\rho_4)&=(-5/2;0,-1)\\
  \gr_\psi(\rho_{12})&=(-1/2;1,0) & \gr_\psi(\rho_{23})&=(1/2;0,1)\\
  \gr_\psi(\rho_{34})&=(-3/2;-1,0)& \gr_\psi(\rho_{41})&=(-5/2;0,-1)\\
  \gr_\psi(\rho_{123})&=(1/2;0,1) & \gr_\psi(\rho_{234})&=(-2;0,0) & \dots \\
  \gr_\psi(\rho_{1234})&=(-2;0,0) & \gr_\psi(\rho_{2341})&=(-2;0,0) & \dots\\
  \gr_\psi(\rho_{12341})&=(-2;0,0) &   \gr_\psi(\rho_{23412})&=(-5/2;1,0) &\dots\\
  \gr_\psi(U)&=(-2;0,0) & \lambda_w&=(0;0,0).
\end{align*}

\subsection{The winding number and total weight}\label{sec:win-num}
The algebra $\UnDefAlg$ has two other gradings: the {\em length} and the {\em multiplicity at the chord
$\rho_4$}. The length was defined in
Section~\ref{sec:AsUnDefAlg}; recall in particular that $|U|=4$.
We will think of the multiplicity at $\rho_4$ as the winding number,
and denote it by $\wingr$. Specifically, define
$\wingr(\rho_1)=\wingr(\rho_2)=\wingr(\rho_3)=0$ and
$\wingr(\rho_4)=\wingr(U)=1$, and extend $\wingr$ to all of $\UnDefAlg$ by
$\wingr(ab)=\wingr(a)+\wingr(b)$. Equivalently, there is a homomorphism
$\bigGroup\to\ZZ$ by $(m;a,b,c,d)\mapsto d$ and $\wingr$ is the
composition of $\grb$ with
this projection. In particular, $\wingr$ sends $\lambda_w$ to $1$.

Sometimes, it is convenient to combine the above gradings.
In Section~\ref{sec:unique}, we will formulate a uniqueness result for
the algebra, phrased in terms of a grading
\begin{equation}
  \gamma=\gr\times\wingr
\end{equation}
with values in the group $\Gamma=\smallGroup\times\ZZ$.

\subsection{A mod-2 grading}\label{sec:mod-2-gr}
We conclude with the rather dull mod-2 grading:

\begin{lemma}
  The map $\epsilon\co \smallGroup(\mathbb{T})\to \Zmod{2}$, defined
  by $\epsilon(m;a,b)\equiv m+\frac{a-b}{2}+ a b\pmod{2}$ is a
  homomorphism. Further, the map $\epsilon$ sends $\lambda$ to $1$ and
  the grading $\gr_\psi(a)$ of every homogeneous algebra element $a$
  to $0$.
\end{lemma}
\begin{proof}
  To verify that $\epsilon$ is a homomorphism, observe that
  \begin{align*}
    \epsilon(m;a,b)+\epsilon(n;c,d)&=m+n+\frac{a-b+c-d}{2}+ab+cd\\
                                   &\equiv m+n+ad-bc+\frac{a+c-(b+d)}{2}+(a+c)(b+d)\\
                                   &=\epsilon(m+n+ad-bc;a+c,b+d).
  \end{align*}
  To verify the statement about the gradings of algebra elements, it
  suffices to compute
  $\epsilon(\gr_\psi(\rho_1))=\epsilon(\gr_\psi(\rho_2))=\epsilon(\gr_\psi(\rho_3))=\epsilon(\gr_\psi(\rho_4))=\epsilon(\gr_\psi(U))=0$. Similarly,
  $\epsilon(\lambda)=1+\frac{0-0}{2}+0=1$. 
\end{proof}

\begin{remark}
  A homomorphism $\smallGroup(\mathbb{T})\to \Zmod{2}$ sending
  $\lambda$ to $0$ is an element of $H^1(\mathbb{T};\Zmod{2})$. The
  homomorphisms sending $\lambda$ to $1$ are obtained from $\epsilon$
  by adding one of these four maps.

  Note that since $\lambda$ is a commutator in $\bigGroup$, there is no
  homomorphism $\bigGroup\to\Zmod{2}$ sending $\lambda$ to $1$.
\end{remark}

\begin{remark}
  For the case of $\HFa$, modulo-2 gradings on the bordered
  algebras and modules have been constructed previously by
  Petkova~\cite{Petkova18:decat} and
  Hom-Lidman-Watson~\cite{HLW17:decat}, for surfaces of arbitrary
  genus.
\end{remark}


%% file: unique.tex
\newcommand\dcob{\delta^{\Cobarop}}
\newcommand\gammaCobar{\gamma^{\Cobarop}}

\newcommand\doverline[1]{\overline{\overline{#1}}}

\newcommand\dtilde[1]{\widetilde{\widetilde{#1}}}
\newcommand\dtildew[1]{\doverline{#1}}

\section{Abstract approach to existence and uniqueness of the torus
  algebra}\label{sec:unique}
The goal of this section is to provide a more algebraic
characterization of the weighted algebra $\MAlg$. 

\subsection{\texorpdfstring{$\Ainf$}{A-infinity} deformations and Hochschild cohomology}
\label{sec:A-inf-deform}
In this section, we show that $\Ainf$ deformations of an associative
algebra are controlled by Hochschild cohomology.
This material is well-known, but we recall it here for
the reader's convenience. See, for instance,~\cite[Section
3a]{Seidel15:quartic},~\cite{LekiliPerutz11:torus},~\cite[Section 2.3]{Sheridan:CY-hyp} and the references
therein. In particular, the group-graded context of~\cite[Section 2.3]{Sheridan:CY-hyp} is close to the setting of interest here (cf.~Section~\ref{sec:Fukaya}). 

Fix a commutative ring $\Ground$; for us, $\Ground$ will be a finite direct
sum of copies of $\Field$ or $\Field[U]$. Unless otherwise specified tensor
products are over $\Ground$.

By an \emph{$A_n$-algebra} we mean a projective $\Ground$-module $A$ together with
$\Ground$-linear maps $\{\mu_i\co A^{\otimes i}\to A\}_{i=2}^n$ satisfying those $\Ainf$-algebra
relations with at most $n+1$ inputs. (So, in this section, we are only
considering $A_n$-algebras with trivial differential.)
A \emph{homomorphism} of  $A_n$-algebras $f\co\Alg\to\Blg$ consists of maps
$\{f_i\co A^{\otimes i}\to B\}_{i=1}^{n-1}$, satisfying
the relations for an $A_{\infty}$-algebra homomorphism with at most $n$ inputs.

Fix an associative algebra $A$ and an augmentation $\epsilon\co
A\to\Ground$. Let $A_+=\ker(\epsilon)$ be the augmentation ideal
and let $\Pi\co A \to A_+$ denote the projection to the
augmentation ideal.

We are interested in $\Ainf$-algebra structures on $A$
so that $\mu_1=0$, $\mu_2$ is the given multiplication on $A$, and the
operations are strictly unital and $\Ground$-multilinear.
(Unitality
of the $\Ainf$-algebra can be formulated as the condition that the
operations~$\mu_i$
for $i>2$ satisfy
$\mu_i=\mu_i\circ (\overbrace{\Pi\otimes\dots\otimes\Pi}^i)$.)
We will
define such structures inductively. The obstructions to extending, and
different choices of extensions, will be given in terms of the
Hochschild cochains.

\begin{definition}\label{def:HH-assoc}
  The \emph{(reduced) bar complex} of an augmented associative algebra
  $A$ over $\Ground$ is given by
  \[
  \Barop(A)=A\otimes A \longleftarrow A\otimes A_+\otimes A\longleftarrow A\otimes
  A_+\otimes A_+\otimes A\longleftarrow\cdots
  \]
  where the differential $A\otimes A_+^{\otimes n}\otimes A\to
  A\otimes A_+^{\otimes (n-1)}\otimes A$ is specified by
  \begin{equation}
    \label{eq:BarDifferential}
    a_0\otimes\cdots\otimes a_{n+1}\mapsto \sum_{i=0}^{n}
    a_0\otimes\cdots\otimes a_ia_{i+1}\otimes a_{i+2}\otimes\cdots
  \otimes a_{n+1}.
  \end{equation}
  This is a chain complex of $(A,A)$-bimodules, and in
  fact is a resolution of $A$.
  
  The \emph{Hochschild cochain complex} of $A$ is given by 
  \begin{equation}
    \label{eq:DefHC}
  \HC^*(A)=\Hom_{A\otimes A^\op}(\Barop(A),A).
  \end{equation}
  The grading is chosen so that $\Hom(A\otimes A,A)$ lies in grading
  $0$.
\end{definition}

Note that we can absorb the $A$ factors on the left and right into the $\Hom$, to get
\begin{equation}\label{eq:HC-Hom-over-ground}
\HC^n(A)=\Hom_{\Ground\otimes\Ground}(A_+^{\otimes n},A).
\end{equation}
Let $\delta$ denote the differential on this model for the Hochschild complex.

\begin{remark}
  In the terminology of~\cite{LOT2},
  $\HC^*(A)=\Mor(\lsup{A}[\Id]_A,\lsup{A}[\Id]_A)$, the complex of
  type \DA\ bimodule morphisms from the identity \DA\ bimodule of $A$
  to itself.
\end{remark}

\begin{definition}
  Given an associative algebra $A$ (with no differential), an $A_n$-algebra
  structure $\Alg$ on $A$ with $\mu_2$ the given multiplication on $A$ (and
  $\mu_1=0$) is an \emph{$A_n$ deformation of $A$}.
\end{definition}

\begin{proposition}\label{prop:deform-assoc-alg}
  Let $A$ be an associative algebra and let $\Alg$ be an $A_{n-1}$-algebra that
  is a deformation of $A$. Then there is a Hochschild cochain obstruction class
  $\obstr_{n}\in \HC^{n+1}(A)$ so that:
  \begin{enumerate}[label=(A$\obstr$-\arabic*)]
  \item
    \label{o:Cocycle}
    $\obstr_{n}$ is a cocycle.
  \item $\obstr_{n}$
    \label{o:Extension}
    is a coboundary if and only if there is an
    operation $\mu_{n}$ making $\Alg$ into an $A_{n}$-algebra;
    indeed, the operation $\mu_n\in\HC^{n}(A)$
    is a cochain with $\delta(\mu_n)=\obstr_{n}$.
  \item\label{item:coboundary} If $\mu_n$ and $\mu_n'$ are cochains with
    $\delta(\mu_n)=\delta(\mu_n')=\obstr_{n}$ then 
    $\mu_n-\mu_n'$ is itself a coboundary if and only if there is an
    $A_{n}$-homomorphism $f$ between the corresponding structures
     with $f_1=\Id$ and $f_j=0$ for $1<j<n-1$.
  \item\label{item:naturality}
    If $\Alg$ and $\Alg'$ are two $A_{n}$ deformations
    with $\mu_i=\mu_i'$ for all $i<n$, and $\mu_n-\mu_n'$ is a coboundary,
    then their respective obstruction cocycles
    $\obstr_{n+1}, \obstr_{n+1}'\in\HC^{n+2}(A)$ are cohomologous.
  \end{enumerate}
  
  There are analogous statements for maps. In particular, given
  $A_{n}$ deformations $\Alg$ and $\Alg'$ of $A$ and an
  $A_{n-1}$-homomorphism $f\co \Alg\to \Alg'$ so that $f_1\co A\to A$ is
  the identity map, there is an obstruction class $\fobstr_n\in
  \HC^{n}(A)$ so that:
  \begin{enumerate}[label=(A$\fobstr$-\arabic*)]
  \item
    \label{item:mapcocycle} $\fobstr_n$ is a cocycle.
  \item
    \label{item:coboundaryfobstr}
      $\fobstr_n$ is a coboundary if and only if there is an
    $A_{n}$-homomorphism extending $f$. 
  \end{enumerate}
\end{proposition}

\begin{proof}
  Given a map $f\co A_+^{\otimes i}\to A$, we will also let $f$ denote
  the extension $f\circ \Pi^{\otimes i}\co A^{\otimes i}\to A$. 
  To define the obstruction cocycle $\obstr_n$, we use a composition map
  \[ \Hom(A_+^{\otimes i},A)\otimes  \Hom(A^{\otimes j}_+,A)\to \Hom(A_+^{\otimes i+j-1},A),\]
  which we denote by $\star$, defined by
  \begin{equation}
    \label{eq:DefStar}
    f_i\star g_j = 
    \sum_{\ell=0}^{i-j+1} f_i(\Id_{A_+^{\otimes\ell}}\otimes \circ g_j \otimes \Id_{A_+^{\otimes(i-\ell-1)}})
    =
    \mathcenter{\begin{tikzpicture}
      \node at (-1,0) (tl) {};
      \node at (0,0) (tc) {};
      \node at (1,0) (tr) {};
      \node at (0,-1) (m1) {$g$};
      \node at (0,-2) (m2) {$f$};
      \node at (0,-3) (bc) {};
      \draw[taa] (tl) to (m2);
      \draw[taa] (tr) to (m2);
      \draw[taa] (tc) to (m1);
      \draw[alga] (m1) to (m2);
      \draw[alga] (m2) to (bc);
    \end{tikzpicture}
    }.
  \end{equation}
  (This uses the extension of $f$, via pre-composition with $\Pi$.)
  Let 
  \begin{equation}\label{eq:DefDObstr}
    \obstr_{n}=\sum_{\substack{i,j\geq 3\\ i+j=n+2}} \mu_i\star \mu_j=
    \sum_{\substack{i,j\geq 3\\i+j=n+2}}\mathcenter{
      \begin{tikzpicture}
        \node at (-1,0) (tl) {};
        \node at (0,0) (tc) {};
        \node at (1,0) (tr) {};
        \node at (0,-1) (m1) {$\mu_i$};
        \node at (0,-2) (m2) {$\mu_j$};
        \node at (0,-3) (bc) {};
        \draw[taa] (tl) to (m2);
        \draw[taa] (tr) to (m2);
        \draw[taa] (tc) to (m1);
        \draw[alga] (m1) to (m2);
        \draw[alga] (m2) to (bc);
      \end{tikzpicture}
    }.
    \end{equation}
    The $\Ainf$ relation with $n+1$ inputs is the condition that
    \begin{equation}
      \label{eq:AinfHochschild}
    \delta \mu_{n} = \sum_{\substack{i, j\geq 3 \\ i+j=n+2}} \mu_i\star \mu_j;
    \end{equation}
    i.e., $\delta\mu_n=\obstr_n$. Property~\ref{o:Extension} follows.

    To verify that $\obstr_n$ is a cocycle, it helps to have the following
    easily verified identity:
    for any $f_i\co A_+^{\otimes i}\to A$ and $g_j\co A_+^{\otimes j}\to A$,
    \begin{equation}
      \label{eq:LeibnizStar}
      \delta (f_i\star g_j)=(\delta f_i)\star g_j
      + f_i\star (\delta g_j)+ \mu_2(f_i,g_j)+\mu_2(g_j,f_i).
    \end{equation}
    Using this identity and Equation~\eqref{eq:AinfHochschild},
    we see that
  \begin{align*}
    \delta \obstr_n &= \sum_{\substack{i,k\geq 3\\i+k=n+2}} (\delta \mu_i)\star \mu_k + \mu_i \star (\delta \mu_k)
    + \mu_2(\mu_i,\mu_k)+\mu_2(\mu_k,\mu_i) \\
    &= \left(\sum_{\substack{i,j,k\geq 3 \\ i+j+k=n+4}}  (\mu_i \star \mu_j)\star \mu_k\right) + 
    \left(\sum_{\substack{i,j,k\geq 3\\ i+j+k=n+4}}  \mu_i \star (\mu_j\star \mu_k)\right).
  \end{align*}
  Note that $\star$ is not associative; rather,
  \begin{equation}
    \label{eq:Associator}
    (a \star b)\star c + a \star(b\star c)= a\circ (\Id \otimes b \otimes \Id\otimes c\otimes \Id)\circ\Delta^5 +
    a\circ (\Id \otimes c \otimes \Id\otimes b\otimes \Id)\circ
  \Delta^5, 
  \end{equation}
  where
  \begin{equation}
    \label{eq:Comultiplication}
    \Delta^m \co \Tensor^*(A_+)\to \overbrace{(\Tensor^*(A_+)\otimes\dots\otimes \Tensor^*(A_+))}^m 
  \end{equation}
  denotes the comultiplication map on $\Tensor^*(A_+)$ applied $m-1$ times.
  Property~\ref{o:Cocycle} follows.
  
  To verify Property~\ref{item:coboundary}, we argue as follows. Let
  ${\mathcal A}$ and ${\mathcal A}'$ be the $A_{n}$ algebras such that
  $\mu_i=\mu_i'$ for all $i<n$, but with possibly different $\mu_n$ and
  $\mu_n'$. The $\Ainf$ relation with $n$ inputs for a map $f\co
  \mathcal{A}\to \mathcal{A}'$ with $f_1=\Id$ and $f_j=0$ for $1\leq
  j< n-1$ is precisely the condition that
  \[
    \delta(f_{n-1})=\mu_n-\mu_n'.
  \]

  To verify Property~\ref{item:naturality},
  observe that the hypothesis that $\mu_i=\mu_i'$ for all $i<n$ ensures that
  \[ \obstr_{n+1}-\obstr_{n+1}'=\mu_n\star \mu_3 + \mu_3\star\mu_n + 
  \mu_n'\star \mu_3 + \mu_3\star\mu_n'.\] Our hypotheses also give us
  a $c_{n-1}$ with 
  \[ \delta c_{n-1}=\mu_n-\mu_n'.\]
  Now, using Equation~\eqref{eq:LeibnizStar} (and using the fact that
  $\delta\mu_3=0$), we see that
  \[
    \obstr_{n+1}-\obstr_{n+1}'
    = \delta(c_{n-1}\star \mu_3 + \mu_3\star c_{n-1}).
  \]
  
  We now turn to the statements for maps.  Fix $A_n$ deformations
  $\Alg$ and $\Alg'$ of $A$ and an $A_{n-1}$-homomorphism $f\co
  \Alg\to\Alg'$ with $f_1$ the identity map.  To define the
  obstruction cocycle $\fobstr_n$, we introduce some notation. Extend
  $\{f_i\co A_+^{\otimes i}\to A\}_{i\leq n-2}$ to a map $F_{n}\co
  A^{\otimes n}\to \Tensor^*A$ by summing over all ways of
  parenthesizing $A^{\otimes n}$ into strings of $\leq n-2$ elements
  and applying the appropriate $f_i$ to each string. For example, for
  $n=4$,
  \begin{multline*}
    F_4(a_1\otimes a_2\otimes a_3\otimes a_4)
    \\=f_1(a_1)\otimes f_1(a_2)\otimes f_1(a_3)\otimes f_1(a_4)
    +f_2(a_1\otimes a_2)\otimes f_1(a_3)\otimes f_1(a_4)+
    f_1(a_1)\otimes f_2(a_2\otimes a_3)\otimes f_1(a_4)\\+
    f_1(a_1)\otimes f_1(a_2)\otimes f_2(a_3\otimes a_4)+
    f_2(a_1\otimes a_2)\otimes f_2(a_3\otimes a_4).
  \end{multline*}

  Define the obstruction class by
  \begin{equation}\label{eq:DefFObstr}
    \fobstr_n= \Bigl(\sum_{\substack{j\geq 3 \\ i+j=n+1}} f_i\star
    \mu_j\Bigr)+\Bigl(\mu\circ F_{n}\Bigr)\co A_+^{\otimes n}\to A,
  \end{equation}
  or graphically by
  \[
  \fobstr_n=\mathcenter{
    \begin{tikzpicture}
      \node at (-1,0) (tl) {};
      \node at (0,0) (tc) {};
      \node at (1,0) (tr) {};
      \node at (0,-1) (m) {$\mu_{\geq 3}$};
      \node at (0,-2) (f) {$f_{\leq n-2}$};
      \node at (0,-3) (bc) {};
      \draw[taa] (tl) to (f);
      \draw[taa] (tr) to (f);
      \draw[taa] (tc) to (m);
      \draw[alga] (m) to (f);
      \draw[alga] (f) to (bc);
    \end{tikzpicture}
  }
  +
  \mathcenter{
    \begin{tikzpicture}
      \node at (-1,0) (tl) {};
      \node at (0,0) (tdots) {$\cdots$};
      \node at (1,0) (tr) {};
      \node at (-1,-1) (lf) {$f_{\leq {n-2}}$};
      \node at (1,-1) (rf) {$f_{\leq {n-2}}$};
      \node at (0,-1) (moredots) {$\cdots$};
      \node at (0, -2) (mu) {$\mu_{\geq 2}$};
      \node at (0, -3) (bc) {};
      \draw[taa] (tl) to (lf);
      \draw[taa] (tr) to (rf);
      \draw[alga] (lf) to (mu);
      \draw[alga] (rf) to (mu);
      \draw[alga] (mu) to (bc);
    \end{tikzpicture}
  }.
  \]
  The condition on a map $f_{n-1}\co A_+^{\otimes n-1}\to A$ that
  \[
    \delta f_{n-1}=\fobstr_n
  \]
  is precisely the $\Ainf$ relation with $n$ inputs for the map
  $\{f_i\}_{i=1}^{n-1}$, so Part~\ref{item:coboundaryfobstr} holds.

  To verify that $\fobstr$ is a cocycle, we introduce a little more notation.
  Given $g_j\co A_+^{\otimes j} \to A$, let
  $\widetilde{g_j}\co \Tensor^*(A_+)\to \Tensor^*(A)$
  be the map defined by
  \begin{equation}
    \label{eq:deftilde}
  \widetilde{g_j}(a_1\otimes\dots\otimes a_i)=\sum_{m=0}^{i-j}
  a_1\otimes \dots\otimes a_m \otimes g_j(a_{m+1},\dots,a_{m+j})\otimes
  a_{m+j+1}\otimes \dots\otimes a_i.
  \end{equation}
  In particular, $f\star g= f\circ \widetilde{g_j}$. (Again, recall
  that we are abusing notation so $f=f\circ\Pi^{\otimes n}$.)
  There is a differential
  $D \co \Hom(\Tensor^*(A),\Tensor^*(A))\to \Hom(\Tensor^*(A),\Tensor^*(A))$
  defined by
  \begin{equation}
    \label{eq:DefD}
    D(\Phi)=\widetilde{\mu_2}\circ \Phi + \Phi\circ \widetilde{\mu_2}. 
  \end{equation}
  The fact that $\mu_2$ is associative implies that $D^2=0$.
  
  Given maps $\{\phi_i\in \Hom_{\Ground\otimes\Ground}(A^{\otimes
    j}_+,A)\}_{j=1}^{m}$ with $\phi_1=\Id$, let 
  $\dtilde{\phi}\co \Tensor^*(A_+) \to \Tensor^*(A)$ denote the map
  defined by
  \begin{equation}
    \label{eq:deftilde2}
    \dtilde{\phi}(a_1,\dots,a_t) = 
    \sum_{i_1+\dots+i_k=t} \phi_{i_1}(a_1,\dots,a_{i_1})\otimes
    \phi_{i_2}(a_{i_1+1},\dots,a_{i_1+i_2})\otimes \dots\otimes
    \phi_{i_k}(a_{m-i_k+1},\dots a_{t}).
  \end{equation}
  Given maps $\{ f_i\in \Hom(A^{\otimes i},A)\}_{i=1}^m$, we will let
  $f = \sum_{i=1}^m f_i \co \Tensor^*(A) \to A$, and similarly let
  $f_{\le n} = \sum_{i=1}^n f_i$, and so forth.
  Conversely, let $(g)_m$ denote the $m$-input component of $g$.
  Observe that
  \begin{equation}
    \label{eq:SecondLeibniz}
    \delta((f\circ \dtilde{\phi})_m)= (\delta(f)\circ \dtilde{\phi})_{m+1}
    + (f \circ D \dtilde{\phi})_{m+1} + \sum_{\substack{i,j\\ i+j=m+1, i>1}}
    \mu_2(\phi_{i},(f\circ \dtilde{\phi})_j)+
    \mu_2((f\circ \dtilde{\phi})_j,\phi_{i}).
  \end{equation}

  In this notation,
  \[
    \fobstr_n=f_{\leq n-2}\star {\mu_{\geq 3}}+ \mu_{\geq 2}\circ {\widetilde{\widetilde{f_{\leq n-2}}}}=
    f_{\leq n-2}\star\mu_{\geq 3} + \mu_{\geq 3}\circ \widetilde{\widetilde{f_{\leq n-2}}} +
    \sum_{\substack{i>1,\ j>1\\  i+j=n+1}} \mu_2(f_i,f_j).
  \]

  By hypothesis, we have maps $\{f_i\co A^{\otimes i}\to A\}_{i=1}^{n-2}$, that satisfy the 
  $\Ainf$ relation with $k$ inputs for all $k\leq n-1$. As noted earlier,
  the 
  $\Ainf$ relation on these maps with $k$ inputs can be formulated as
  \[ \delta (f_{k-1})= f_{\leq k-2}\star \mu_{\geq 3}
  + \mu_{\geq 2}\circ \dtilde{f_{\leq k-2}}
  = \left(\sum_{\substack{i+j=k+1\\ j\geq 3}} f_i\star\mu_j + \mu_j\circ {\dtilde{f}_{\leq k}} \right)
  + \sum_{\substack{i>1, j>1 \\ i+j=k}} \mu_2(f_i,f_j),\]
  which we abbreviate
  \begin{equation}
    \label{eq:AinfMap1}
    \delta(f_{k-1}) =(f_{\leq k-2}\star\mu_{\geq 3})_k + (\mu_{\geq 3}\circ \dtilde{f_{\leq k-2}})_k + \mu_2(f_{>1},f_{>1})_k.
  \end{equation}
  
  We will also use a reformulation of the $A_n$-algebra homomorphism
  relations, stated in terms of the map $D$ from
  Equation~\eqref{eq:DefD}. Abusing notation, given a tensor product
  $g^1\otimes\cdots\otimes g^k$ of maps $\Tensor^*(A)\to A$ (such as
  $\dtilde{f_{\leq n-1}}$ or $D(\dtilde{f_{\leq n-1}})$), let
  $(g^1\otimes\cdots\otimes g^k)_{\leq n}$ denote the restriction that
  no individual $g^i$ has more than $n$ inputs. Then the reformulation is
  \begin{equation}
    \label{eq:AinfMap2}
    \bigl(D(\dtilde{f_{\leq n-1}})\bigr)_{\leq n}=\bigl({\dtilde{f_{\leq n-2}}}\circ {\widetilde \mu_{\geq 3}} + \widetilde{\mu_{\geq 3}}\circ \dtilde{f_{\leq n-2}}\bigr)_{\leq n}
  \end{equation}

  Another identity we shall use is that
  \[ \delta(\mu_2(f,g))=\mu_2(\delta(f),g)+\mu_2(f,\delta(g)).\]

  With these preliminaries in hand, we compute
  \begin{equation}  \label{eq1}
    \begin{split}
  \delta \bigl((f_{\leq n-2}\star \mu_{\geq 3})_n\bigr)
  & =\bigl(\delta(f_{\leq n-2})\star \mu_{\geq 3}
  + f_{\leq n-3} \star \delta(\mu_{\geq 3}) + \mu_2(f_{>1},\mu_{\geq 3})+
  \mu_2(\mu_{\geq 3},f_{>1})\bigr)_{n+1} \\
    &= \begin{aligned}[t]
      &\bigl(\mu_{\geq 3} \circ {\widetilde{\widetilde{f_{\leq n-3}}}}\circ {\widetilde {\mu_{\geq 3}}}
      + \mu_2(f_{>1}\star \mu_{\geq 3},f_{>1})
      +\mu_2(f_{>1},f_{>1}\star \mu_{\geq 3})\\
      &\qquad + \mu_2(f_{>1},\mu_{\geq 3})+
      \mu_2(\mu_{\geq 3},f_{>1})\bigr)_{n+1}
    \end{aligned} \\
  &= \bigl(\mu_{\geq 3} \circ {\dtilde{f_{\leq n-3}}}\circ {\widetilde {\mu_{\geq 3}}}
  + \mu_2(f\star \mu_{\geq 3},f_{>1})+\mu_2(f_{>1},f\star \mu_{\geq 3})\bigr)_{n+1}.
  \end{split}
  \end{equation}
  In going from the first to the second line above, we are using a
  cancellation of $(f_{\leq n-1}\star \mu_{\geq 3})\star \mu_{\geq 3}$
  against $f_{\leq n-2}\star (\mu_{\geq 3}\star \mu_{\geq 3})$, which
  uses Equation~\eqref{eq:Associator}.

Next, we compute
  \begin{align}
    \delta\bigl( (\mu_{\geq 3}\circ \dtilde{f_{\leq n-2}})_n\bigr)
    &= \bigl(\delta(\mu_{\geq 3})\circ \dtilde{f_{\leq n-2}}
    + \mu_{\geq 3} \circ D(\dtilde{f_{\leq n-2}})  + \mu_2(\mu_{\geq 3}\circ \dtilde{f},f_{>1})
    + \mu_2(f_{>1},\mu_{\geq 3}\circ \dtilde{f})\bigr)_{n+1} \nonumber \\
    &=\bigl(\mu_{\geq 3} \circ \dtilde{f_{\leq n-3}}\circ {\widetilde {\mu_{\geq 3}}} 
        +\mu_2(\mu_{\geq 3}\circ \dtilde{f},f_{>1})
        + \mu_2(f_{>1},\mu_{\geq 3}\circ \dtilde{f})\bigr)_{n+1},
        \label{eq2}
  \end{align}
and also
  \begin{equation}
    \label{eq3}
    \delta\bigl(\mu_2(f_{>1},f_{>1})_n\bigr)
    = \mu_2(\delta f_{>1},f_{>1})_{n+1}+\mu_2(f_{>1},\delta f_{>1})_{n+1}
    = \mu_2(\delta f,f_{>1})_{n+1}+\mu_2(f_{>1},\delta f)_{n+1}.
  \end{equation}
  Adding Equations~\eqref{eq1},~\eqref{eq2}, and~\eqref{eq3},
  and using associativity of $\mu_2$, and once again using Equation~\eqref{eq:AinfMap1},
  we see that
  \begin{align*}
    \delta (\fobstr_n)&=\mu_2(\delta f + f\star \mu_{\geq 3} + \mu_{\geq 3} \circ\dtilde{f},  f_{>1})_{n+1}+
    \mu_2(f_{>1}\delta f + f\star \mu_{\geq 3} + \mu_{\geq 3} \circ \dtilde{f})_{n+1} \\
    &=
      \begin{aligned}[t]
        &\mu_2(\delta f + f\star \mu_{\geq 3} + \mu_{\geq 3}\circ \dtilde{f}+\mu_2(f_{>1},f_{>1}),  f_{>1})_{n+1}+\mu_2(f_{>1},\delta f + f\star \mu_{\geq 3} + \mu_{\geq 3} \circ \dtilde{f}\\
        &\qquad +\mu_2(f_{>1},f_{>1}))_{n+1} \\
      \end{aligned}\\
    &=0,
  \end{align*}
  verifying Property~\ref{item:mapcocycle}.
\end{proof}

\begin{corollary}\label{cor:deform-assoc-alg}
  Let $\Alg$ be an $A_n$-structure on the associative algebra $A$.
  If $\HH^{m+2}(A)=0$ for all $m\geq n$ then $\Alg$ extends to an
  $\Ainf$-algebra structure on $A$. If in addition $\HH^{m+1}(A)=0$
  for all $m\geq n$ then this extension is unique up to isomorphism.
\end{corollary}
\begin{proof}
  Suppose that $\HH^{m+2}(A)=0$ for all $m\geq n$.  By
  Properties~\ref{o:Cocycle} we can inductively find the requisite
  sequence of elements $\mu_k\in \HC^{k}$ with
  $\delta(\mu_k)=\obstr_k$ for all $k\geq m+1$,
  giving an extension of $\Alg$ to an $\Ainf$ algebra.

  Suppose that $\HH^{m+1}(A)=0$ for all $m\geq n$, and let $\Alg$ and
  $\Alg'$ be two $\Ainf$ deformations of $A$ that agree as $A_n$
  algebras.  Choose $f_1=\Id$ and $f_k=0$ for all $k=1,\dots,n$.  By
  hypothesis, these are the components of an $A_n$-homomorphism from
  $\Alg$ to $\Alg'$. For the inductive step, suppose that we have
  components $\{f_i\co A^{\otimes i}\to A\}_{i=1}^k$ for $n<k$ of
  an $A_k$ homomorphism, the obstruction $\fobstr_k$ to extending it
  to an $A_{k+1}$-homomorphism lies in $\HH^{k}(A)$ by
  Properties~\ref{item:mapcocycle}
  and~\ref{item:coboundaryfobstr}.
  This map is an isomorphism since $f_1$ is invertible.
\end{proof}


We will actually be interested in deforming a $(G,\lambda)$-graded associative
algebra to a $(G,\lambda)$-graded $\Ainf$-algebra (so that $\mu_n$ has grading
$\lambda^{n-2}$). (For $G$ abelian, this case was studied by
Sheridan~\cite[Section 2.3]{Sheridan:CY-hyp}.) We will assume further that the distinguished
central element $\lambda\in G$ has infinite order.

In this setting, define a grading on the bar complex $\Barop(A)$
by viewing the $n\th$
term as $A\otimes(A_+[1])^{\otimes n}\otimes A$, i.e.,
$\gr(a_0\otimes a_1\otimes\cdots\otimes
a_n\otimes a_{n+1})=\lambda^n\gr(a_0)\cdots\gr(a_{n+1})$. Then, the
differential has grading $\lambda^{-1}$.
If we are interested in deformations which preserve this grading---and
we are---then we are interested in the subcomplex $\HC^*_G(A)\subset \HC^*(A)$ of
morphisms which respect the grading by $G/\langle \lambda\rangle$,
i.e., maps $f\co \Barop(A)\to A[1]$ so that $\gr_{A[1]}(f(x))=\lambda^{k}\gr_{\Barop(A)}(x)$,
or equivalently $\gr_A(f(x))=\lambda^{k-1}\gr_{\Barop(A)}(x)$,
for some $k\in\ZZ$.
In addition to the grading by~$n$, this complex has an obvious
$\ZZ$-grading (by~$k$), and the
differential decreases this grading by~$1$.
 A graded
$\Ainf$ operation lies in grading $-1$.

Let
$\HC^{i,j}_G(A)$ denote the part $\HC^*_G(A)$ consisting of
$(\Ground\otimes\Ground)$-module maps
$A_+[1]^{\otimes i}\to A[1]$ of grading~$j$. Explicitly, if $f\in
\HC^{i,j}_G(A)$, then
\[
  \gr(f(x_1,\dots,x_{i})) = \lambda^{i+j-1}\gr(x_1)\cdots\gr(x_{i}).
\]
The differential has the property that
\[\delta\co \HC^{r,d}_G(A)\to \HC^{r+1,d-1}_G(A).\]

Proposition~\ref{prop:deform-assoc-alg} has the following $(G,\lambda)$-graded analogue:

\begin{proposition}
  \label{prop:deform-assoc-alg-G}
  Let $A$ be a $(G,\lambda)$-graded associatve algebra and let $\Alg$
  be a $(G,\lambda)$-graded $A_{n-1}$ deformation of $A$.  Then, the
  bigradings of the obstruction classes are given by
  $\obstr_n\in\HC^{n+1,-2}_G(A)$ and $\fobstr_n\in \HC^{n,-1}_G(A)$,
  and an operation $\mu_n$ defining a $(G,\lambda)$-graded
  $A_n$-deformation lies in $\HC^{n,-1}_G(A)$.  Moreover, $\obstr_n$
  is the obstruction to extending the $(G,\lambda)$-graded $A_{n-1}$
  deformation to a $(G,\lambda)$-graded $A_n$ deformation. If $\Alg$
  and $\Alg'$ are two $(G,\lambda)$-graded $A_n$ deformations of $A$,
  and $f$ is a $(G,\lambda)$-graded homomorphism of the underlying
  $A_{n-1}$ deformations, then $\fobstr_n$ is the obstruction to
  extending $f$ to a $(G,\lambda)$-graded $A_{n}$ homomorphism.
\end{proposition}

\begin{proof}
  This is a straightforward adaptation of the proof of Proposition~\ref{prop:deform-assoc-alg}.
\end{proof}

\begin{corollary}\label{cor:deform-assoc-alg-G}
  Let $A$ be a $(G,\lambda)$-graded associative algebra and $\Alg$ a
  $(G,\lambda)$-graded $A_{n-1}$ deformation of $A$. If
  $\HH^{m+1,-2}_{G}(A)=0$ for all $m\geq n$ then $\Alg$ extends to a
  $(G,\lambda)$-graded $\Ainf$-algebra structure on $A$. If
  $\HH^{m,-1}_{G}(A)=0$ for all $m\geq n$ then any two
  $(G,\lambda)$-graded $\Ainf$-algebra structures on $A$ extending
  $\Alg$ are $\Ainf$-isomorphic.
\end{corollary}

\begin{proof}
  This follows from Proposition~\ref{prop:deform-assoc-alg-G} exactly
  as Corollary~\ref{cor:deform-assoc-alg} follows from
  Proposition~\ref{prop:deform-assoc-alg}.
\end{proof}

\begin{remark}
  The discussion in our previous paper~\cite[Section 2.5.3.]{LOT2}
  gives a grading on the space of morphisms from $\Barop(A)$ to $A$ by
  $G\times_{G\times G}G$, which is the set of conjugacy classes in
  $G$. The definition of $\HC^{*,*}_G(A)$ then restricts to morphisms
  lying over the conjugacy class $\{\lambda^n\}$ for $n\in\ZZ$.
\end{remark}

\subsection{The cobar complex of the torus algebra}
\label{subsec:CobarComplex}
Let $\Ground$ be a finite direct sum of copies of $\Field$,
$A$ be an augmented associative $\Ground$-algebra
and $A_+$ be the augmentation ideal. Note that the dual space
$\Hom(A_+,\Field)$ is a $\Ground$-bimodule.

\begin{definition}
  The \emph{reduced cobar algebra $\Cobarop(A)$} is the dual chain complex to
  $\Barop(A)$, that is, the direct sum over $n$ of the dual of
  $(A_+)^{\otimes n}$. The multiplication on $\Cobarop(A)$ is
  the transpose of the comultiplication $\Delta$ on $\Barop(A)$, which
  in turn is defined by
  \[
    \Delta(a_1\otimes\cdots\otimes
    a_n)=\sum_{i=0}^n(a_1\otimes\cdots\otimes
    a_i)\otimes(a_{i+1}\otimes\cdots\otimes a_n)\in \Barop(A)\otimes\Barop(A).
  \]
  The differential on $\Cobarop(A)$ is the transpose of the map
  \[
    \sum_{m,n\geq 0}\Id^{\otimes m}\otimes \mu_2\otimes\Id^{\otimes n}
    \co\Barop(A)\to\Barop(A).
  \]
\end{definition}

Under modest finiteness assumptions, we can describe $\Cobarop(A)$ more
explicitly. The easiest case is if $A$ is finite-dimensional over
$\Field$, in which case $\Cobarop(A)$ is the tensor algebra on
$A_+^*$, with differential given by
\[
  \dcob(a_1^*\otimes\dots\otimes a_n^*)=\sum_{i=1}^{n} a_1^*\otimes\dots\otimes \mu_2^*(a_i^*)\otimes\dots\otimes a_n^*,
\]
where $\mu_2^*$ is the dual of the multiplication $\mu_2$ on
$\Alg$. (Here and below, undecorated tensor products are over $\Ground$.)
We use the convention that 
\[
  (a_n\otimes \dots \otimes a_1)^*=a_1^*\otimes\dots\otimes a_n^*.
\]
For example, when $A=\AsUnDefAlg$, we have that
$\mu_2^*(\rho_{12}^*)=\rho_2^*\otimes \rho_1^*$.  

More generally, suppose that $A$  is filtered by subspaces $F_0A\subset
F_1A\subset\cdots\subset A$ with $A=\bigcup_i F_iA$, and so that each
$F_iA$ is finite-dimensional, $\Ground\subset F_0A$, and
$\mu_2(F_iA,F_jA)\subset F_{i+j}A$. There is an induced filtration on
$\Barop(A)$, which we also denote $F_i$, and
\begin{align*}
  \Cobarop(A)_n&=(A_+^{\otimes n})^*\cong \varprojlim (F_i(A_+^{\otimes n}))^*\\
  \Cobarop(A)&=\bigoplus_{n=1}^{\infty} \Cobarop(A)_n.
\end{align*}

Explicitly, given an $\Field$-basis $\{e_i\}$ for $A_+$, let $A_+^\dagger$ be the
subspace of $A_+^*$ spanned by the dual basis elements $e_i^*$. Then
$\Cobarop(A)_n$ is the completion of $(A_+^\dagger)^{\otimes n}$
with respect to the (descending) filtration dual to $F_i$.

Turning to the example of interest to us, the algebra $\AsUnDefAlg$ from
Section~\ref{sec:AsUnDefAlg} has a filtration by the winding number
grading, $\wingr$, so that each filtration level is
finite-dimensional. Hence, the cobar algebra $\Cobarop(\AsUnDefAlg)$
is the completion of the tensor algebra on $(\AsUnDefAlg)_+^\dagger$ with
respect to this filtration.

The algebra $\AsUnDefAlg$
also has a grading by $\Gamma=\smallGroup\times \ZZ$, with grading given by
$\gamma=\gr\times\wingr$, as defined in Section~\ref{sec:grading}.
This induces a $\Gamma$-grading on $\Cobarop(\AsUnDefAlg)$ by
the formula
\begin{equation}
  \label{eq:GradeCobar}
  \gammaCobar(a_1^*\otimes\dots\otimes a_n^*)=\lambda^{-n}\gamma(a_1)^{-1}\cdots \gamma(a_n)^{-1}.
\end{equation}
(Since the cobar algebra is the completion of the tensor algebra on
$(\AsUnDefAlg)_+^\dagger$, not every element is a finite linear combination of
homogeneous elements.)

There is an automorphism $\alpha\co \Gamma\to\Gamma$ defined by
\begin{equation}
  \label{eq:DefAlpha}
  \alpha((j;a,b)\times i)=(j+2i;-a,-b)\times (-i).
\end{equation}

The following relationship between $\AsUnDefAlg$ and its cobar algebra
can be seen as a kind of Koszul self-duality.
\begin{lemma}\label{lem:Koszul}
  There is a quasi-isomorphism of $\Gamma$-graded differential algebras
  \[
    \phi\co (\Cobarop(\AsUnDefAlg),\gammaCobar)\to (\AsUnDefAlg,\alpha\circ\gamma)
  \]
  specified by $\phi(\iota_0)=\iota_1$, $\phi(\iota_1)=\iota_0$,
  $\phi(\rho_i^*)=[\rho_i]$ for $i=1,\dots,4$, and $\phi(a^*)=0$ if
  $|a|>1$.
\end{lemma}

(When thinking of $\AsUnDefAlg$ as a bimodule quasi-isomorphic to
$\Cobarop(\AsUnDefAlg)$, we will write elements of $\AsUnDefAlg$
inside brackets.)

\begin{proof}
  By construction, $\phi$ is a ring homomorphism.
  Direct computation shows that if $\rho_i\cdots\rho_\ell\neq 0$ then
  \[
    \lambda^{-\ell+i-1}\cdot \gamma(\rho_i)^{-1}\cdots \gamma(\rho_\ell)^{-1}
    = \alpha(\gamma(\rho_i)\cdots\gamma(\rho_\ell)).
  \]
  It follows that $\phi$ respects the grading.

  Consider the homomorphism $j\co \AsUnDefAlg\to \Cobarop(\AsUnDefAlg)$
  specified by $j(\iota_0)=\iota_1$, $j(\iota_1)=\iota_0$,
  $j(\rho_i)=\rho_i^*$ for $i=1,\dots,4$.
  Clearly,
  \begin{equation}
    \label{eq:phij}
    \phi\circ j=\Id.
  \end{equation}

  The image of $j$ is spanned by elements of the form
  $\rho_i^* \otimes \rho_{i+1}^*\otimes\dots\otimes \rho_\ell^*$.
  Any element of $\Cobarop(\AsUnDefAlg)$ can be written as a series in elements in the image of $j$
  and further elements of the form
  \begin{equation}
    \label{eq:Generic}
    \rho_i^*\otimes \rho_{i+1}^*\otimes\dots\otimes \rho_\ell^*
    \otimes a_1^*\otimes\dots\otimes a_m^*,
  \end{equation}
  where the $a_i$ are basic algebra elements and $a_1$ is the first
  element in the sequence with length greater than one or for which
  $a_1=\rho_{\ell-1}$. Given such an element, let $k=0$ if the element
  is in the image of $j$ and $k=\ell-i+1$ otherwise, and consider the homotopy operator
  \begin{equation}
    \label{eq:DefH}
  H(\overbrace{\rho_i^*\otimes \rho_{i+1}^*\otimes\dots\otimes \rho_\ell^*}^k
  \otimes\, a_1^*\otimes\dots\otimes a_m^*)=
  \begin{cases}
    \rho_i^*\otimes\rho_{i+1}^*\otimes\dots\otimes \rho_{\ell-1}^*
  \otimes (a_1\cdot \rho_\ell)^*\otimes a_2^*\otimes\dots\otimes a_m^* &
  \text{if $k>0$} \\
  0 &\text{if $k=0$.}
  \end{cases}
  \end{equation}
  Since the total winding number of the output of $H$ and the input of
  $H$ are the same, $H$ extends (continuously) to all of
  $\Cobarop(\AsUnDefAlg)$.
  
  We claim that the following formula holds:
  \begin{equation}
    \label{eq:HomotopyFormula}
    \dcob\circ H + H\circ \dcob = \Id + j\circ \phi.
  \end{equation}
  It suffices to verify Equation~\eqref{eq:HomotopyFormula} for each
  element
  $\xi=\rho_i^*\otimes \rho_{i+1}^*\otimes\dots\otimes \rho_\ell^*\otimes
  a_1^*\otimes\dots\otimes a_m^*$. There are three cases:
  \begin{itemize}
  \item If $\xi\in \Image(j)$, $H(\xi)=0$, $\dcob(\xi)=0$, and
    $j\circ\phi (\xi)=\xi$, so Equation~\eqref{eq:HomotopyFormula} is
    immediate.
  \item  Suppose that $k>0$ and $a_1\cdot \rho_\ell=0$; or alternatively,
    that $k=0$.  Then, $H(\xi)=0$, $\phi(\xi)=0$ and the only non-zero
    term in $H(\dcob(\xi))$ is $\xi$ itself, corresponding to the
    factorization $a_1 = \rho_{\ell+1} \cdot a_1'$ for some $a_1'$.
  \item Suppose that $k>0$ and $a_1\cdot \rho_\ell\neq 0$.  In that case,
    $\xi$ is the term in $\dcob\circ H$ corresponding to the
    factorization of $a_1\cdot \rho_\ell$ as $a_1$ times $\rho_\ell$.  All
    other terms in $\dcob \circ H$ cancel corresponding terms in
    $H\circ \dcob$.
  \end{itemize}
  
  Together, Equations~\eqref{eq:phij} and~\eqref{eq:HomotopyFormula}
  ensure that $\phi$ is a quasi-isomorphism, as claimed.
\end{proof}

The cobar algebra is of interest to us because of its relationship
with the Hochschild complex.  Suppose that $A$ is an augmented
associative $\Ground$-algebra (such as $\AsUnDefAlg$). Assume that $A$ is
endowed with an increasing filtration so that each $F_iA$
finite-dimensional.
The tensor product $A\otimes_{\Ground\otimes\Ground} \Cobarop(A)$ inherits a
decreasing filtration from the filtration on $\Cobarop(A)$ (which does
not use the filtration on the $A$-factor). Let
$A\hotimes_{\Ground\otimes\Ground} \Cobarop(A)$
denote the completion with respect to this filtration on each $A\otimes_{\Ground\otimes\Ground} \Cobarop(A)_n$.
Equivalently, $A\hotimes_{\Ground\otimes\Ground} \Cobarop(A)$ is
the direct sum over~$n$ of the completion of $A\otimes_{\Ground\otimes\Ground}(A_+^\dagger)^{\otimes n}$.
Given a (potentially
infinite) basis $\{e_i\}$ for $A_+$, define a differential on
$A\hotimes_{\Ground\otimes\Ground}
\Cobarop(A)$ by
\begin{equation}\label{eq:d-HC-complete}
  \partial(b\otimes \xi)=b\otimes \bigl(\delta^{\Cobarop}(\xi)\bigr)+\sum_i e_ib\otimes \bigl(\xi\otimes e_i^*\bigr)+be_i\otimes \bigl(e_i^*\otimes\xi\bigr)
\end{equation}
and then extending linearly to the completion (which contains infinite
sums of elements of the form $b\otimes[\xi]$). The assumption that
each $F_iA$ is finite-dimensional implies this gives a well-defined
map. (The last two terms come from partly dualizing the operation
$\mu_2\co A\otimes A\to A$ to maps $A\to A\otimes A^*$ and
$A\to A^*\otimes A$.)

\begin{lemma}\label{lem:CobarHC-noU}
  If $A$ is a filtered algebra, and each $F_iA$ is finite-dimensional,
  there is an isomorphism of chain complexes
\begin{equation}
  \label{eq:CobarHC-noU}
  A\hotimes_{\Ground\otimes\Ground} \Cobarop(A)\cong \HC^*(A)
\end{equation}
where on the left side we use the differential from
Equation~\eqref{eq:d-HC-complete}.
\end{lemma}
\begin{proof}
  This is straightforward from the definitions.
\end{proof}

In the application to $\AsUnDefAlg$, we actually want to extend
scalars from $\FF_2$ to $\FF_2[U]$. So, given an algebra $A$ as above,
let $A[U]=A\otimes_\Field \Field[U]$. Consider
$A[U]\otimes_{\Ground\otimes\Ground} \Cobarop(A)$; note that the
$U$-variable appears only on the $A[U]$-factor.  The increasing
filtration on $A$ induces a decreasing filtration on
$A[U]\hotimes_{\Ground\otimes\Ground} \Cobarop(A)_n$ and we again have
a completed tensor product
$A[U]\hotimes_{\Ground\otimes\Ground} \Cobarop(A)$ and, given a basis
$\{e_i\}$ for $A_+$, a differential induced by Equation~\eqref{eq:d-HC-complete}.
We have the following analogue of Lemma~\ref{lem:CobarHC-noU},
reformulating the Hochschild cochain complex of $A[U]$ over $\Ground[U]$:
\begin{lemma}\label{lem:CobarHC}
  If $A$ is a filtered algebra, and each $F_iA$ is finite-dimensional,
  there is an isomorphism of chain complexes
  \begin{equation}
    \label{eq:CobarHC}
    A[U]\hotimes_{\Ground\otimes\Ground} \Cobarop(A)\cong \HC^*(A[U]).
  \end{equation}
  On the left side of Equation~\eqref{eq:CobarHC}, the differential is
  given by Equation~\eqref{eq:d-HC-complete}. On the right side of
  Equation~\eqref{eq:CobarHC}, $U$ is viewed as an element of the ground
  ring.
\end{lemma}
\begin{proof}
  Again, this is straightforward from the definitions.
\end{proof}

\begin{example}
  We consider the torus algebra. Since (as proved below) $\mu_3=0$, it
  follows that the
  obstruction class from Proposition~\ref{prop:deform-assoc-alg} satisfies
  $\obstr_4=0$. (See Eq.~\eqref{eq:DefDObstr}.)
  There is a map $\mu_4\colon (\AsUnDefAlg)^{\otimes 4}\to \AsUnDefAlg[U]$
  constructed in Section~\ref{sec:AsUnDefAlg}, whose non-trivial operations are
  \[ \mu_4(\rho_4,\rho_3,\rho_2,\rho_{1}\cdot a)=U a,\qquad 
    \mu_4(b\cdot \rho_4,\rho_3,\rho_2,\rho_{1})=Ub,\]
  and the
  additional operations obtained by cyclically permuting
  $\rho_4,\rho_3,\rho_2,\rho_1$.  The $5$-input $A_\infty$-relation
  (which holds by a very easy special case of
  Theorem~\ref{thm:AinftyAlgebra}) is equivalent to the statement
  that $\delta\mu_4=0$; i.e., $\mu_4\in\HC^*(A)$ is a Hochschild
  cocycle.
  
  Under the isomorphism from Equation~\eqref{eq:CobarHC}, this cocycle
  $\mu_4$ corresponds to the element of
  $A[U]\hotimes_{\Ground\otimes\Ground} \Cobarop(A)$ specified by
  \begin{multline*}
    U\otimes(\rho_1^*\otimes\rho_2^*\otimes\rho_3^*\otimes\rho_4^*)+U\rho_2\otimes(\rho_{12}^*\otimes\rho_2^*\otimes\rho_3^*\otimes\rho_4^*)+
    U\rho_{23}\otimes(\rho_{123}^*\otimes\rho_2^*\otimes\rho_3^*\otimes\rho_4^*)+\cdots\\
    +U\rho_3\otimes(\rho_{1}^*\otimes\rho_2^*\otimes\rho_3^*\otimes\rho_{34}^*)
    +U\rho_{23}\otimes(\rho_{1}^*\otimes\rho_2^*\otimes\rho_3^*\otimes\rho_{234}^*)
    +\cdots+\cdots
  \end{multline*}
  (where the last $\cdots$ corresponds to cyclically permuting the set
  $1,2,3,4$). For this sum to make sense, we need the completed tensor
  product
  $A[U]\hotimes_{\Ground\otimes\Ground}
  \Cobarop(A)$.
\end{example}

We use Lemma~\ref{lem:Koszul} to obtain a useful small model for the Hochschild
complex (compare~\cite{LOTHomPair}), which we describe after introducing some notation.

\begin{definition}
  \label{def:SmallModel}
  The \emph{small model Hochschild complex $C^*$} is
  defined as follows.  As a vector space, $C^*$ is generated by
  $a\otimes [b]$ with $a\in \AsUnDefAlg$ and $[b]\in\AsUnDefAlg$ are
  basic elements
  with the property that $i\cdot a \cdot j=a$ and
  $[j'\cdot b \cdot i']=[b]$, for some idempotents
  $i,j\in\{\iota_0,\iota_1\}$ and complementary idempotents
  $i',j'$. (That is, if $i=\iota_0$ then $i'=\iota_1$.)  We endow $C^*$ with the following
  further structure:
  \begin{itemize}
  \item  a $\ZZ$-grading, the {\em length grading} specified by
    $|a\otimes[b]|=|b|$  (when $b$ is a basic algebra element).
  \item a $G\times \ZZ$-grading,
    specified by 
    \[ \gamma(a\otimes [b])=\lambda\cdot \gamma(a)\cdot \alpha(\gamma(b)).\]
  \item a differential
    \begin{equation}
      \label{eq:DefCdiff}
      \partial (a\otimes [b])=\sum_{i=1}^4 \left(\rho_i\cdot a \otimes [b \cdot \rho_i]+ a\cdot \rho_i \otimes [\rho_i\cdot b]]\right). 
    \end{equation}
  \end{itemize}
  We let $C^*_\Gamma\subset C^*$ be the portion in grading
  $0\times \ZZ\subset G\times \ZZ$, i.e., generated by $a\otimes [b]$
  with the property that
  \[ \gamma(a)\cdot \alpha(\gamma(b))=\lambda^{k-1},\]
  for some integer $k$, called the {\em homological grading} of $a\otimes [b]$. 
  Let $C^{n,k}_\Gamma\subset C^{n,k}_\Gamma$ denote the portion with
  length grading $n$ and homological grading $k$. The differential
  sends $C^{n,k}$ to $C^{n+1,k-1}$.
\end{definition}

For example, $\rho_{2341}\otimes [\rho_{1234}]\in C^{4,-1}_\Gamma$.
(The element $U\iota_1\otimes [\rho_{1234}]$ also lies in this bigrading.)

To see that $C^*$ is a chain complex, note that for any
$i,j\in\{1,\dots,4\}$, at least one of $\rho_i\rho_j$ and
$\rho_j\rho_i$ vanishes.

\begin{proposition}
  \label{prop:SmallerModel}
  The chain complex $C^*_{\Gamma}$ is quasi-isomorphic to 
  the complex $\HC^*_\Gamma(\AsUnDefAlg[U])$; in particular
  $H^{n,k}(C_\Gamma)\cong \HH^{n,k}(\AsUnDefAlg[U])$.
\end{proposition}

\begin{proof}
  Recall that elements of $\HC^{n,k}_\Gamma(\AsUnDefAlg)\subset
  \AsUnDefAlg\otimes_{\Ground\otimes\Ground} (\AsUnDefAlg)_+ ^{\otimes
    n}$ are series with terms of the form $a_0\otimes
  (a_1^*\otimes\dots\otimes a_n^*)$, where the $a_i$ are all basic
  algebra elements, with the property that the right idempotent of
  $a_1$ (which is the left idempotent of $a_1^*$) agrees with the
  right idempotent of $a_0$, and the left idempotent of $a_n$ (which
  is the right idempotent of $a_n^*$) agrees with the left idempotent
  of $a_0$.  The $\Gamma$-grading is computed by
  $\lambda\cdot\gamma(a_0)\cdot \gammaCobar(a_1^*\otimes \dots\otimes
  a_n^*)= \gamma(a_0)\cdot \lambda^{1-n} (\gamma(a_1)^{-1}\cdots
  \gamma(a_n)^{-1})$.

  Consider the map $\HC^*(\AsUnDefAlg)\to C^*$ induced by
  \[
    \Id\otimes \phi\co \AsUnDefAlg\otimes_{\Ground\otimes\Ground}\Cobarop(\AsUnDefAlg)
    \to \AsUnDefAlg\otimes_{\Ground\otimes\Ground} \AsUnDefAlg,
  \]
  where $\phi$ is the map from Lemma~\ref{lem:Koszul}.
  Since $\phi$ sends any element of $\Cobarop(\AsUnDefAlg)_n$ with
  filtration greater than $n/4+1$ to $0$, $\Id\otimes\phi$ indeed induces a map
  from
  $A[U]\hotimes_{\Ground\otimes\Ground} \Cobarop(A)\cong
  \HC^*(A[U])$ to $C^*$.

  Comparing the differential on $\HC^*(\AsUnDefAlg)$ from
  Equation~\eqref{eq:d-HC-complete} (with basic algebra elements as
  the basis)
  with Equation~\eqref{eq:DefCdiff}, we see that
  $\Id\otimes\phi$ is a chain map.

  Observe that $\HC^{*,*}_\Gamma(\AsUnDefAlg)$ is a direct summand of
  $\HC^*(\AsUnDefAlg)$, $C^{*,*}_\Gamma$ is a direct summand of $C^*$,
  and $\Id\otimes \phi$ takes $\HC^{*,*}_\Gamma(\AsUnDefAlg)$ to
  $C^{*,*}_\Gamma$. 

  If we filter $\HC^{*,*}_\Gamma(\AsUnDefAlg)$ by the sum of the
  lengths of the input elements and filter $C^{*,*}_\Gamma$ by the
  length grading from Definition~\ref{def:SmallModel} then
  $\Id\otimes\phi$ is a filtered chain map.
  The induced map at the $E_1$-page of the associated
  spectral sequence is $\Id\otimes \phi_*$ where $\phi_*$ is the
  isomorphism $H_*(\Cobarop(\AsUnDefAlg))\to \AsUnDefAlg$ from
  Lemma~\ref{lem:Koszul} (or, rather, its restriction to the
  $\Gamma$-graded part). It follows that $\Id\otimes\phi$ is a
  quasi-isomorphism from the completion of
  $\HC^{*,*}_\Gamma(\AsUnDefAlg)$ to the completion of
  $C^{*,*}_\Gamma$. However, for each fixed $\Gamma$-grading on
  $\HC^{*,*}_\Gamma(\AsUnDefAlg)$, for any element
  $b\otimes[\xi]\in\HC^{n,k}_\Gamma(\AsUnDefAlg)$ there is a bound on
  the difference between the length of $b$ and four times the winding number
  of~$\xi$.  Hence, for each pair of integers $(n,k)$,
  $\HC^{n,k}_\Gamma(\AsUnDefAlg)$ is
  already complete with respect to the length filtration on
  $\AsUnDefAlg$. Similarly, for $C^{*,*}_\Gamma$, there are finitely
  many elements in each grading (see the proof of
  Proposition~\ref{prop:AsUnDefHoch} below) so $C^{*,*}_\Gamma$ is
  also already complete. Hence, $\Id\otimes\phi$ is a
  quasi-isomorphism $\HC^*_\Gamma(\AsUnDefAlg[U])\to
  C^{*,*}_\Gamma$, as desired.
\end{proof}

\subsection{Uniqueness of \texorpdfstring{$\UnDefAlg$}{the un-weighted algebra}}\label{sec:unweighted-unique}

\begin{theorem}\label{thm:UnDefAlg-unique}
  Up to isomorphism, there is a unique $\Ainf$ deformation of
  $\AsUnDefAlg$ over $\Field[U]$ satisfying the following conditions:
  \begin{enumerate}
  \item\label{item:Uuniqe-graded} The deformation is $\Gamma=G\times \ZZ$-graded, where the gradings of the
    chords $\rho_i$ is defined by $\gamma(\rho_i)=\gr(\rho_i)\times \wingr(\rho_i)$.
    (The gradings $\gr$ and $\wingr$ are defined in Section~\ref{sec:grading}.)
  \item\label{item:Uuniqe-mu-4} The operations satisfy
    $\mu_4(\rho_4,\rho_3,\rho_2,\rho_1)=U\iota_1$ and
    $\mu_4(\rho_3,\rho_2,\rho_1,\rho_4)=U\iota_0$.
  \end{enumerate}
\end{theorem}

The conditions of the theorem immediately imply that
$\gamma(U)=(-2;0,0)\times 1$. Also, the relation
$\mu_4(\rho_4,\rho_3,\rho_2,\rho_1)=U\iota_1$ implies the relation
$\mu_4(\rho_3,\rho_2,\rho_1,\rho_4)=U\iota_0$, by considering the
$5$-input $\Ainf$-relations.

A key step in the proof is a computation of (part of) the Hochschild cohomology
of $\AsUnDefAlg$:
\begin{proposition}\label{prop:AsUnDefHoch}
  The graded Hochschild cohomology
  $\HH^{*,*}_\Gamma(\AsUnDefAlg[U])$ of
  $\AsUnDefAlg[U]$ over $\Ground[U]$
  satisfies
  \begin{align*}
  \HH^{n,-1}_\Gamma(\AsUnDefAlg[U])&=
  \begin{cases}
   \Field & n=4\\
    0 &\text{otherwise}
  \end{cases}\\
  \HH^{n,-2}_\Gamma(\AsUnDefAlg[U])&=
  \begin{cases}
   \Field & n=5\\
    0 &\text{otherwise}
  \end{cases}
  \end{align*}
  Moreover, suppose $\xi \in
  \HC^{4,-1}_\Gamma(\AsUnDefAlg[U])$ is a cycle and
  $\xi(\rho_4\otimes\rho_3\otimes\rho_2\otimes\rho_1)=U$. Then $\xi$
  represents a generator of
  $\HH^{4,-1}_\Gamma(\AsUnDefAlg[U])$.
\end{proposition}
\begin{proof}
  Proposition~\ref{prop:SmallerModel} supplies a smaller quasi-isomorphic model for this complex $C$, where
  \[ C^{n,k}_\Gamma\subset\AsUnDefAlg\otimes_{\Ground\otimes\Ground}\AsUnDefAlg,\]
  is generated by elements of the form
  $a\otimes [b]$ with $a,b\in\AsUnDefAlg$ for which:
  \begin{enumerate}[label=($HC$-\arabic*)]
  \item 
    \label{c:Idemp}
    the right idempotent of $b$ is complementary to the left idempotent of $a$
    and the left idempotent of $b$ is complementary to the right idempotent of $a$.
  \item 
    \label{c:Grading} the gradings satisfy
    \[
      \gamma(a \otimes [b]) = \lambda\cdot \gamma(a)\cdot \alpha(\gamma(b)) =\lambda^{k},
      \]
    where $\alpha$ is as in Equation~\eqref{eq:DefAlpha}
    and $|b|=n$. 
  \end{enumerate}

(We will typically suppress the $\otimes$ symbol from $a\otimes[b]$.)
  
The above conditions ensure that any such element $a[b]$ must be one of:
  \begin{enumerate}[label=($C$-\arabic*)]
  \item the following elements $a\otimes[b]$
    \[ \rho_1  [\rho_1 ], \qquad \rho_{123}[\rho_{123}],
    \qquad \iota_0[\iota_1], 
\qquad \iota_1[\iota_0];
\]
\item \label{Cenlarge} any of the elements 
  obtained by multiplying the above $b$ by some further
  element $b'$ with $|b'|=4s$, and while also multiplying $a$ by
  some further element $a'$ with $|a'|=4s$;
\item any element obtained by adding some $i\in \Zmod{4}$ to all the indices
  in any of the above obtained elements.
\end{enumerate}
(In particular, the elements
$\rho_{1234}[\rho_{4123}]$, $\rho_{1234}[\rho_{2341}]$, and $U[\rho_{1234}]$
are all obtained from $a[b]=\iota_1[\iota_0]$ by multiplying both $a$ and $b$
by length four algebra elements.)

In a little more detail, suppose that $a[b]=U^n\rho_{i,\dots,j+1}[\rho_{\ell,\dots,m+1}]$
is such an element; so that
$|a|=j-i+4n$ and $|b|=\ell-m$. Then, the
$G/\langle \lambda\rangle $ factor of Condition~\ref{c:Grading} ensures that
$|a|\equiv |b|\mod{4}$; furthermore, if $|a|\not\equiv 0\pmod{4}$, then
$i\equiv \ell\pmod{4}$ and $j\equiv m\pmod{4}$.
Condition~\ref{c:Idemp} now excludes the possibility that $|a|\equiv
2\pmod{4}$.  Finally, the $\ZZ$ factor of Condition~\ref{c:Grading} now ensures
that
$\wingr(a) = \wingr(b)$ and so
$\abs{a}=\abs{b}$. The above classification follows.

Direct computation of the homological grading $k$ gives:
\[ k(\rho_1[\rho_1])=0,\qquad k(\rho_{123}[\rho_{123}])=0, \qquad
k(\iota_0[\iota_1])=k(\iota_1[\iota_0])=1,\]
and the usual symmetry obtained by adding $i\in \Zmod{4}$ to all the subscripts
in the first two equations above.
All other homological gradings are determined by the property that
\[ k(a\cdot a'[b\cdot b'])=k(a[b])-2s,\]
if $|a'|=|b'|=4s$.
In particular,
\[ k(\rho_{1234}[\rho_{4123}])=k(U[\rho_{1234}])=k(\rho_{1234}[\rho_{2341}])=-1.\]

It follows that $\HH^{n,-1}=0$ unless $n=4$, and $\HC^{n,-2}=\HH^{n,-2}=0$
unless $n=5,7$.

We now compute the differentials of elements involving terms with $k=-1$ and $k=-2$:
  \begin{align*}
    \partial(\rho_{123}[\rho_{123}])&=
    \rho_{1234}[\rho_{4123}]+
    \rho_{4123}[\rho_{1234}]\\
    \partial(\rho_{1234}[\rho_{4123}])&=
    \rho_{41234}[\rho_{41234}]\\
    \partial(\rho_{4123}[\rho_{1234}])&=
    \rho_{41234}[\rho_{41234}]\\
    \partial(U[\rho_{1234}])&=
    U\rho_4[\rho_{41234}]+
    U\rho_1[\rho_{12341}]\\
    \partial(\rho_{12341}[\rho_{12341}])&=
    0\\
    \partial(U\rho_1 [\rho_{1234123}])&=0\\
    \partial(\rho_{1234123}[\rho_{1234123}])&=
    \rho_{12341234}[\rho_{41234123}]+
    \rho_{41234123}[\rho_{12341234}] \\
   \partial(U\rho_{123}[\rho_{1234123}])&=
    U\rho_{1234}[\rho_{41234123}]+
    U\rho_{4123}[\rho_{12341234}]
  \end{align*}
  All other such terms are obtained by 
  adding $i\in\ZZ/4\ZZ$ to all of the indices in any of these
  expressions.

For $n=4$, there are two kinds of cycles,
\[ \rho_{1234}[\rho_{4123}]+\rho_{4123}[\rho_{1234}]=\partial(\rho_{123}[\rho_{123}])\]
  (there are four cycles of this form), and
  \begin{equation}\label{eq:HH-4-class}
  U[\rho_{1234}]+
  U[\rho_{2341}]+
  U[\rho_{3412}]+
  U[\rho_{4123}]
  \end{equation}
  (there is a unique cycle of this form). This proves the claim about $\HH^{n,-1}_\Gamma$.

  Turning to $\HH^{n,-2}_\Gamma$, for $n=5$, the single homology class is 
  \[
  U\rho_1 [\rho_{12341}]\sim
  U\rho_2 [\rho_{23412}]\sim
  U\rho_3 [\rho_{34123}]\sim
  U\rho_4 [\rho_{41234}].
  \]
  Finally, for $n=7$ there are no cycles at all.
\end{proof}

\begin{proof}[Proof of Theorem~\ref{thm:UnDefAlg-unique}]
  Throughout this proof, by ``deformation'' we mean ``$\Gamma$-graded
  deformation''.

  It is immediate from Proposition~\ref{prop:deform-assoc-alg-G} and
  Proposition~\ref{prop:AsUnDefHoch} that there is no nontrivial $A_3$
  deformation of $\AsUnDefAlg$. Thus, taking $\mu_3=0$, $[\obstr_4]=0$
  since taking $\mu_4=0$ defines an $A_4$ algebra. (In fact, $\obstr_4$
  vanishes as a chain.) Thus, the choices
  of $A_4$ deformation of $\AsUnDefAlg$ correspond to
  $\HH^{4,-1}_{\Gamma}(\AsUnDefAlg)\cong\Field$. Thus, there is a single
  nontrivial $A_4$ deformation. Moreover, from the description of the
  generator of $\HH^{4,-1}_{\Gamma}(\AsUnDefAlg)\cong\Field$ in
  Proposition~\ref{prop:AsUnDefHoch}, this deformation satisfies and
  is characterized by Property~(\ref{item:Uuniqe-mu-4}) of the
  statement of
  Theorem~\ref{thm:UnDefAlg-unique}. Now, again by
  Proposition~\ref{prop:AsUnDefHoch},
  $\HH^{m+1,-2}_{\Gamma}(\AsUnDefAlg)=\HH^{m,-1}_{\Gamma}(\AsUnDefAlg)=0$ for all
  $m>4$, so by Corollary~\ref{cor:deform-assoc-alg-G}, this
  deformation extends uniquely to an $\Ainf$ deformation of $\AsUnDefAlg$.
\end{proof}

\begin{remark}
  \label{rem:WhyWeHaveAinf}
  The non-trivial deformation from
  Theorem~\ref{thm:UnDefAlg-unique} appears in
  bordered Floer homology.  For example, consider the $A_\infty$
  module for the solid torus, as
  in~\cite[Figure~\ref{TM:fig:cable-HD}]{LOT:torus-mod}.  This has a
  single generator $a$ with $m_2(a,\iota_1)=a$, and actions
  $m_3(a,\rho_2,\rho_1)=a$, $m_3(a,\rho_4,\rho_3)=U\cdot a$. Composing
  these two actions gives a non-zero term in the $A_\infty$ relation
  with input sequence $(a,\rho_4,\rho_3,\rho_2,\rho_1)$. Since
  $m_1=0$, the only possible term that can cancel this sequence is
  $m_2(a,\mu_4(\rho_4,\rho_3,\rho_2,\rho_1))$, forcing
  $\mu_4(\rho_4,\rho_3,\rho_2,\rho_1)$ as in the theorem. This is a
  formalization of the more geometric observation: composing the
  holomorphic disks giving the $m_3$ operations, we obtain a
  one-dimensional moduli space whose other end consists of a curve
  that covers $T^2$ once, with boundary asymptotics as given by
  $\rho_4,\rho_3,\rho_2,\rho_1$; compare also
  Section~\ref{subsec:Immersions}; and
  see~\cite[Figure~\ref{TM:fig:simple-bdy-degen-eg}]{LOT:torus-mod}.
  (This deformation also appears from the wrapped Fukaya
  category, as discussed in Section~\ref{sec:Fukaya}.)
\end{remark}

\subsection{Weighted algebras and Hochschild cohomology}
\label{sec:weighted-deformation}
Next, we discuss deforming $\Ainf$-algebras into weighted
algebras. This is similar to the discussion in~\cite[Section
3b]{Seidel15:quartic} and~\cite[Section 2.4]{Sheridan:CY-hyp}; again,
the group-graded setting of~\cite[Section 2.4]{Sheridan:CY-hyp} is
particularly relevant.

Fix an augmented $\Ainf$-algebra $\Alg^0=(A,\{\mu_m\})$
over $\Ground$ with underlying
vector space $A$ and augmentation ideal $A_+\subset A$. By a
\emph{weighted deformation} of $\Alg^0$ we mean a weighted
$\Ainf$-algebra $(A,\{\mu_m^k\})$ with the same underlying vector
space as $\Alg^0$ and whose weight-zero operations are the same as for
$\Alg^0$: i.e., $\mu_m^0=\mu_m$ for all $m\geq 0$. Suppose that $\Alg$ and $\Blg$ 
are both weighted deformations of the same undeformed $\Ainf$ algebra.
A \emph{homomorphism of deformations}
from $\Alg$ to $\Blg$ is a sequence of maps 
$f^{\bullet}=\{f^W\co \Tensor^*(A_+)\to B\}_{W=0}^{\infty}$
satisfying the weighted $\Ainf$ homomorphism relations
\begin{equation}\label{eq:WeightedAinftyHomomorphism}
  \sum_{a+b=W} f^a \circ (\Id\otimes \mu^b\otimes \Id)\circ \Delta^3 
  + \sum_{a+w_1+\dots+w_m=W} \mu^a\circ (f^{w_1}\otimes\dots\otimes f^{w_m})\circ
  \Delta^m = 0
\end{equation}
for each $W\geq 0$.
In words, the second sum expresses the sum of all ways
of parenthesizing the tensor product into disjoint bundles and
applying some $f^v$ to each bundle, and then channeling the outputs into
a $\mu^a$ so that the total weight of the
$f$'s plus the weight $a$ is $W$.

Like $\Ainf$ deformations, we will build weighted deformations
step-by-step. By a \emph{$W$-truncated weighted $\Ainf$-algebra} we
mean a vector space $A$ and operations $\mu_m^w\co A^{\otimes m}\to A$
for $m\geq 0$ and $0\leq w\leq W$, $(m,w)\neq (0,0)$, satisfying the weighted
$\Ainf$-algebra relations up to weight $W$ (i.e., the ones only
involving the operations defined). A \emph{$W$-truncated weighted
  deformation} of an $\Ainf$-algebra $\Alg^0$ is a $W$-truncated
weighted $\Ainf$-algebra $\Alg^{W}$ whose undeformed (unweighted)
$\Ainf$-algebra is $\Alg^0$.

Let $\Alg^W$ and $\Blg^W$ be $W$-truncated weighted deformations of
$\Alg^0$. By a \emph{homomorphism of $W$-truncated weighted deformations} from $\Alg^W$ to
$\Blg^W$ we mean maps $f_m^w\co A^{\otimes m}\to A$ for $0\leq w\leq
W$ and all $m\geq 0$ with $(m,w)\neq(0,0)$, such that:
\begin{itemize}
\item $f_1^0=\Id$,
\item $f_m^0=0$ for $m\neq 1$, and
\item the $f_m^w$ satisfy the weighted $\Ainf$-algebra homomorphism
  relations in Equation~\eqref{eq:WeightedAinftyHomomorphism} up to weight~$W$.
\end{itemize}
(The first two conditions specify that $f^0$ is the identity map of
$\Ainf$-algebras.) An \emph{isomorphism} is an invertible homomorphism;
by the proof of Lemma~\ref{lem:bij-hom-is-iso}, every homomorphism of
$W$-truncated weighted deformations is an isomorphism.

\begin{definition}\label{def:HH-ainf}
  Let $\Alg^0$ be an augmented $\Ainf$-algebra over $\Ground$, with underlying vector
  space $A$ and augmentation ideal $A_+\subset A$.

  Let $\lsupv{\Alg^0}[\Id]_{\Alg^0}$ denote the identity type \DA\ bimodule over $\Alg^0$
  (see~\cite{LOT2}). The \emph{Hochschild cochain complex} of $\Alg^0$
  is given by
  \[
  \HC^*(\Alg^0)=\Mor(\lsupv{\Alg^0}[\Id]_{\Alg^0},\lsupv{\Alg^0}[\Id]_{\Alg^0})
  \]
  of strictly unital type \DA\ bimodule morphisms from
  $\lsupv{\Alg^0}[\Id]_{\Alg^0}$ to itself. $\HH(\Alg^0)$ is the homology of
  this complex.
\end{definition}
Explicitly, as a vector space,
\[
\HC^*(\Alg^0)=\prod_{n=0}^\infty \Hom_{\Ground\otimes\Ground}(\Ground\otimes (A_+)^{\otimes
  n}, A\otimes \Ground)=\prod_{n=0}^\infty \Hom_{\Ground\otimes\Ground}((A_+)^{\otimes
  n}, A).
\]
The differential is given as follows. Recall the operation $\star$
from Equation~\eqref{eq:DefStar}. Let
$\mu^0=\sum_i\mu^0_i\in \prod_{i=1}^\infty \Hom(A_+^{\otimes i},A_+)$. Then
the differential of $f$ is given by
\[
  \delta(f)=\mu^0\star f+f\star\mu^0 =
      \mathcenter{\begin{tikzpicture}
      \node at (-1,0) (tl) {};
      \node at (0,0) (tc) {};
      \node at (1,0) (tr) {};
      \node at (0,-1) (m1) {$f$};
      \node at (0,-2) (m2) {$\mu^0$};
      \node at (0,-3) (bc) {};
      \draw[taa] (tl) to (m2);
      \draw[taa] (tr) to (m2);
      \draw[taa] (tc) to (m1);
      \draw[alga] (m1) to (m2);
      \draw[alga] (m2) to (bc);
    \end{tikzpicture}
  }
+    \mathcenter{\begin{tikzpicture}
      \node at (-1,0) (tl) {};
      \node at (0,0) (tc) {};
      \node at (1,0) (tr) {};
      \node at (0,-1) (m1) {$\mu^0$};
      \node at (0,-2) (m2) {$f$};
      \node at (0,-3) (bc) {};
      \draw[taa] (tl) to (m2);
      \draw[taa] (tr) to (m2);
      \draw[taa] (tc) to (m1);
      \draw[alga] (m1) to (m2);
      \draw[alga] (m2) to (bc);
    \end{tikzpicture}
    }.
\]

Unlike the associative case (Definition~\ref{def:HH-assoc}), in the
$\Ainf$-setting the Hochschild cohomology is not graded unless $\Alg^0$
is graded.

\begin{proposition}\label{prop:deform-Ainf-alg}
  Let $\Alg^0$ be an $\Ainf$-algebra and $\Alg^{W-1}$ a $(W-1)$-truncated
  weighted deformation of $\Alg^0$. Then there is a Hochschild
  cochain obstruction class $\obstr^W\in \HC^*(\Alg^0)$ so that:
  \begin{enumerate}[label=($\infty\obstr$-\arabic*)]
  \item 
    \label{wo:Cocycle} 
    $\obstr^W$ is a cocycle.
  \item 
    \label{wo:Extension}
    $\obstr^W$ is a coboundary if and only if there are
    operations $\mu^{W}=\{\mu_m^{W}\}_{m=0}^{\infty}$ making $\Alg^W$ into a $W$-truncated
    weighted $\Ainf$ algebra;
    indeed, the operation $\mu^{W}$ is a cochain with $\delta(\mu^{W})=\obstr^W$.
  \item 
        \label{wo:coboundary}
        If $\mu^{W}$ and ${\overline \mu}^{W}$ are cochains with
        $\delta(\mu^{W})=\delta(\overline{\mu}^{W})=\obstr^W$, then
        $\mu^{W}-{\overline\mu}^{W}$ is itself a coboundary if and only if 
        there is a homomorphism $f$ of $W$-truncated deformations 
        between the structures induced by $\mu^W$ and $\overline{\mu}^W$ with $f^w=0$ for 
        all $0<w<W$.
  \item\label{wo:naturality}
    Suppose $\Alg$ and ${\overline\Alg}$ are two $W$-truncated
    deformations, $W>0$,
    with $\mu^w={\overline\mu}^w$ for all $w<W$. If
    $\mu^W-{\overline\mu}^{W}$ is a coboundary
    then their respective obstruction cocycles
    $\obstr^{W+1}$ and ${\overline\obstr}^{W+1}$ are cohomologous.
  \end{enumerate}

  There are analogous statements for maps. In particular given
  $W$-truncated deformations $\Alg$ and ${\overline\Alg}$ of
  $\Alg^0$, for some $W>0$, and a homomorphism 
  $f^{\leq W-1}\co \Alg^{\leq W-1}\to {\overline\Alg}^{\leq W-1}$
  of $W-1$-truncated weighted deformations, there is an obstruction class
  $\fobstr^W\in \HC^*(\Alg^0)$ so that:
  \begin{enumerate}[label=($\infty\fobstr$-\arabic*)]
  \item 
    \label{w:mapcocycle}
    $\fobstr^W$ is a cocycle.
  \item 
    \label{w:coboundary}
    $\fobstr^W$ is a coboundary if and only if there is a
    homomorphism $\Alg^{W}\to{\overline\Alg}^{W}$ of $W$-truncated
    weighted deformations extending $f$. 
  \end{enumerate}
\end{proposition}
\begin{proof}
  Let $\mu^v_*\co \Tensor^*A_+\to A_+$ be the direct sum of the maps
  $\mu^v_n\co A_+^{\otimes n}\to A_+$.

  We think of the operation $\star$ from Equation~\eqref{eq:DefStar}
  as follows. Given $f\co \Tensor^*(A_+)\to A_+$ and $g\co
  \Tensor^*(A_+)\to A_+$,
  \[ f\star g = g\circ (\Id\otimes f\otimes \Id)\circ \Delta^3,\]
  where $\Delta^3$ is as in Equation~\eqref{eq:Comultiplication}.
  We extend the operation $\star$ to sequences of maps
  $f^\bullet=\{f^W\co \Tensor^*(A_+)\to A_+\}_{W=0}^{\infty}$ and $g^\bullet=\{g^W\co \Tensor^*(A_+)\to A^+\}_{W=0}^\infty$, letting
  $f^\bullet\star g^\bullet=\{(f\star g)^W\co \Tensor^*(A_+)\to A_+\}_{W=0}^{\infty}$ be the sequence of maps
  whose components are given by
  \[ (f^\bullet\star g^\bullet)^W = \sum_{a+b=W} f^a \circ (\Id\otimes g^b\otimes \Id)\circ \Delta^3.\]

  Given a sequence of maps $f^\bullet=\{f^W \co \Tensor^*(A_+)\to A_+\}_{W=0}^{\infty}$, let 
  $f^{\bullet\geq 1}$ denote the sequence of maps $\phi^{\bullet}=\{\phi^W\co
  \Tensor^*(A_+)\to A_+\}_{W=0}^{\infty}$ with $\phi^0=0$ and
  $\phi^W=f^W$ for all $W>0$.

  The obstruction class $\obstr^W$ is defined by
  \[
    \obstr^W=(\mu^{\bullet \geq 1}\star\mu^{\bullet \geq 1})^W;
  \]
  i.e.,  $\obstr^W \co \Tensor^*(A_+)\to A_+$ is the map
  \[ \obstr^W=
  \sum_{\substack{a+b=W \\ 1\leq a\leq W-1}}
    \mu^a \star\mu^b, \]
    whose  components $\obstr^W_n\co \overbrace{A_+\otimes\dots\otimes A_+}^n\to A_+$ are given by
  \[ \obstr^W_n=
  \sum_{\substack{a+b=W \\ 1\leq a\leq W-1}}\,\,
  \sum_{\substack{i+j-1=n \\
      1\leq i\leq n+1}}
    \mu_i^a \star\mu_j^b.\]
  Graphically,
  \[
  \obstr^W=\mathcenter{
    \begin{tikzpicture}
      \node at (-1,0) (tl) {};
      \node at (0,0) (tc) {};
      \node at (1,0) (tr) {};
      \node at (0,-1) (m1) {$\mu^{\geq 1}_*$};
      \node at (0,-2) (m2) {$\mu^{\geq 1}_*$};
      \node at (0,-3) (bc) {};
      \draw[taa] (tl) to (m2);
      \draw[taa] (tr) to (m2);
      \draw[taa] (tc) to (m1);
      \draw[alga] (m1) to (m2);
      \draw[alga] (m2) to (bc);
    \end{tikzpicture}
  }.
  \]
  
  Given $f^\bullet=\{f^w\co \Tensor^*(A_+)\to A_+\}_{w=0}^\infty$,
  we define a sequence of maps
  $\dtildew{f}^\bullet=\{\dtildew{f}^W\co\Tensor^*(A_+)\to\Tensor^*(A_+)\}_{W=0}^{\infty}$ whose components $\dtildew{f}^W$ are defined by
  \begin{equation}
    \label{eq:DefdTildew}
  \dtildew{f}^W = \sum_{m=1}^\infty \sum_{ w_1+\dots+w_m=W} (f^{w_1}\otimes\dots\otimes f^{w_m})\circ \Delta^{m}. 
  \end{equation}
  Given $g^\bullet=\{g^w\co\Tensor^*(A_+)\to A_+\}_{w=0}^{\infty}$ and $\phi^\bullet=\{\phi^w\co\Tensor^*(A_+)\to\Tensor^*(A_+)\}_{w=0}^{\infty}$,
  we can define their weighted composition 
  $g^\bullet\circ \phi^\bullet=\{(g^\bullet\circ\phi^\bullet)^w\co \Tensor^*(A_+)\to A_+\}_{w=0}^{\infty}$ by
  \[ (g^\bullet\circ\phi^\bullet)^W=\sum_{a+b=W} g^a\circ \phi^b.\]
  In this notation, the $W\th$ weighted
  $\Ainf$-homomorphism relation is
  \[
    (f^\bullet \star \mu^\bullet)^W
  +(\mu^\bullet\circ \dtildew{f}^\bullet)^W=0
  \]

  Now, suppose that $f$ is only a $(W-1)$-truncated homomorphism,
  with  components $\{f^v\co \Tensor^* A_+\to A_+\}_{v=0}^{W-1}$.
  The obstruction class $\fobstr^W$ is defined to be
  \[
  \fobstr^W=
  (f^\bullet\star \mu^{\bullet\geq 1})^W +
  (\mu^{\bullet}\circ \dtildew{f}^\bullet)^W;
\]
which, in turn,  is shorthand for
\[ 
  \fobstr^W=
\sum_{\substack{a+b=W\\ 0\leq a\leq W-1}}f^a \star \mu^{b}
+ \sum_{\substack{a+b=W\\0\leq a\leq W}}
\mu^a\circ F^b.
  \]
  We represent this equation graphically by
  \[
    \fobstr^W=
  \mathcenter{
    \begin{tikzpicture}
      \node at (-1,0) (tl) {};
      \node at (0,0) (tc) {};
      \node at (1,0) (tr) {};
      \node at (0,-1) (m) {$\mu^{\geq 1}$};
      \node at (0,-2) (f) {$f^{\leq W-1}$};
      \node at (0,-3) (bc) {};
      \draw[taa] (tl) to (f);
      \draw[taa] (tr) to (f);
      \draw[taa] (tc) to (m);
      \draw[alga] (m) to (f);
      \draw[alga] (f) to (bc);
    \end{tikzpicture}
  }+\mathcenter{
    \begin{tikzpicture}
      \node at (-1,0) (tl) {};
      \node at (0,0) (tdots) {$\cdots$};
      \node at (1,0) (tr) {};
      \node at (-1,-1) (lf) {$f^{\leq W-1}$};
      \node at (1,-1) (rf) {$f^{\leq W-1}$};
      \node at (0,-1) (moredots) {$\cdots$};
      \node at (0, -2) (mu) {$\mu^{\geq 0}$};
      \node at (0, -3) (bc) {};
      \draw[taa] (tl) to (lf);
      \draw[taa] (tr) to (rf);
      \draw[alga] (lf) to (mu);
      \draw[alga] (rf) to (mu);
      \draw[alga] (mu) to (bc);
    \end{tikzpicture}
  }.
  \]
  (We write here $f^{\leq W-1}$ to bear in mind that there are no terms $f^i$ with $i\geq W$.)
  
  Having defined $\obstr^W$ and $\fobstr^W$, we now check they satisfy the requisite
  properties. 

  The weight $W$ $\Ainf$ relation for $\mu^w$ can be written
  \begin{equation}
    \label{eq:WeightedAinftyRelation}
    \delta \mu^{W} =
    (\mu^{\bullet\geq 1}\star
    \mu^{\bullet\geq 1})^W=\obstr^W,
  \end{equation}
  which is Property~\ref{wo:Extension}.

  We verify Property~\ref{wo:Cocycle} after introducing some notation.

Equation~\eqref{eq:LeibnizStar} generalizes, as follows.  Given
 $f\co \Tensor^*(A_+)\to A_+$ and $g\co \Tensor^*(A_+)\to A_+$, let
  \[ \eta(f,g)\co \Tensor^*(A_+)\to A_+ \]
  be the map
  \begin{equation}
    \label{eq:DefEta}
    \eta(f,g)=\mu^0\circ (\Id\otimes f\otimes\Id\otimes g\otimes\Id)\circ\Delta^5
    +\mu^0\circ (\Id\otimes g\otimes\Id\otimes f\otimes\Id)\circ\Delta^5.
  \end{equation}
  Graphically
  \[   \eta(f,g)=
  \mathcenter{
    \begin{tikzpicture}
      \node at (-1.5,0) (tll) {};
      \node at (-.75,0) (tl) {};
      \node at (0,0) (tm) {};
      \node at (1.5,0) (trr) {};
      \node at (.75,0) (tr) {};
      \node at (-.375,-1) (f) {$f$};
      \node at (.375,-1) (g) {$g$};
      \node at (0, -2) (mu) {$\mu^0$};
      \node at (0, -3) (bc) {};
      \draw[taa] (tl) to (f);
      \draw[taa] (tr) to (g);
      \draw[alga,bend right=10] (f) to (mu);
      \draw[taa] (tm) to (mu);
      \draw[taa,bend left=15] (trr) to (mu);
      \draw[taa,bend right=15] (tll) to (mu);
      \draw[alga,bend left=10] (g) to (mu);
      \draw[alga] (mu) to (bc);
    \end{tikzpicture}}
+   
\mathcenter{\begin{tikzpicture}
      \node at (-1.5,0) (tll) {};
      \node at (-.75,0) (tl) {};
      \node at (0,0) (tm) {};
      \node at (1.5,0) (trr) {};
      \node at (.75,0) (tr) {};
      \node at (-.375,-1) (f) {$g$};
      \node at (.375,-1) (g) {$f$};
      \node at (0, -2) (mu) {$\mu^0$};
      \node at (0, -3) (bc) {};
      \draw[taa] (tl) to (f);
      \draw[taa] (tr) to (g);
      \draw[alga,bend right=10] (f) to (mu);
      \draw[taa] (tm) to (mu);
      \draw[taa,bend left=15] (trr) to (mu);
      \draw[taa,bend right=15] (tll) to (mu);
      \draw[alga,bend left=10] (g) to (mu);
      \draw[alga] (mu) to (bc);
    \end{tikzpicture}
  }.\]
  (So, for an associative algebra, $\eta(f,g)=\mu_2(f,g)+\mu_2(g,f)$.)
  Equation~\eqref{eq:LeibnizStar} readily generalizes to 
  \begin{equation}
    \label{eq:FirstLeibnizW}
    \delta (f\star g)=(\delta f)\star g
    + f\star (\delta g)+ \eta(f,g).
  \end{equation}

  Extend $\eta$ to sequences of maps $f^\bullet$ and $g^\bullet$ as usual, letting 
  $\eta^\bullet(f^\bullet,g^\bullet)=\{\eta^W(f^\bullet,g^\bullet)\co\Tensor^*(A_+)\to A_+\}_{W=0}^{\infty}$ be the sequence of maps 
  whose $W^{th}$ component is given by 
  \[ \eta^W(f^{\bullet},g^{\bullet})=\sum_{a+b=W} \eta(f^a,g^b).\]
  Equation~\eqref{eq:FirstLeibnizW} generalizes to
  \begin{equation}
    \label{eq:FirstLeibnizW2}
    \delta(f^\bullet\star g^\bullet)=(\delta f^\bullet)\star g^\bullet+f^{\bullet}\star (\delta g^\bullet) +\eta^\bullet(f^\bullet,g^\bullet).
  \end{equation}

  To verify Property~\ref{wo:Cocycle}, observe that
  \begin{align*}
    \delta \obstr^W= \delta(\mu^{\bullet\geq 1} \star \mu^{\bullet\geq 1})^W
    &= (\delta(\mu^{\bullet\geq 1})\star \mu^{\bullet\geq 1})^W + 
     (\mu^{\bullet\geq 1}\star \delta(\mu^{\bullet\geq 1}))^W
    + \eta^{\bullet}(\mu^{\bullet\geq 1},\mu^{\bullet^\geq 1}) \\
    &= ((\mu^{\bullet \geq 1}\star\mu^{\bullet\geq 1})\star \mu^{\bullet \geq 1})^W +
    (\mu^{\bullet \geq 1}\star (\mu^{\bullet \geq 1}\star\mu^{\bullet\geq 1}))^W\\ &= 0
  \end{align*}
  (Here, the last step uses the analogue of Equation~\eqref{eq:Associator}.)
  
  Let
  ${\mathcal A}$ and ${\overline{\mathcal A}}$ be the $W$-truncated deformations
  with weighted operations $\mu^w$ and ${\overline \mu}^w$ respectively.
  The weight $W$ $\Ainf$ relation for a 
  $W$-weighted map $f\co  \mathcal{A}\to \overline{\mathcal{A}}$ 
  with $f^0=\Id$ and $f^w=0$ for all $0<w<W$ 
  is precisely  the condition 
  \[
    \delta(f^W)=\mu^W-{\overline \mu}^W;
  \]
  and this is Property~\ref{wo:coboundary}.

  To verify Property~\ref{wo:naturality}, observe that
  if $\Alg$ and ${\overline\Alg}$ have $\mu^w={\overline\mu}^w$ for all $w<W$, and
  \[ \mu^W-{\overline\mu}^{W}=\delta c\]
  for some $c\in \HC^*(\Alg^0)$,
  then 
  \begin{align*}
    \obstr^{W+1}-{\overline \obstr}^{W+1}&= \mu^W\star \mu^1 + \mu^1\star\mu^W - {\overline\mu}^W\star \mu^1
  -\mu^1\star{\overline \mu}^W \\
  &= (\mu^W-{\overline \mu}^W)\star \mu^1 + \mu^1\star(\mu^W -{\overline \mu}^W) \\
  & = \delta(c\star \mu^1 + \mu^1\star c).
  \end{align*}
  (The last line uses
  Equation~\eqref{eq:FirstLeibnizW}, 
  $\eta(\mu^W-{\overline\mu}^W,\mu^1)+\eta(\mu^1,\mu^W-{\overline\mu}^W)=0$,
  and the $\Ainf$ relation that guarantees that $\delta\mu^1=0$.)
  
  Next, we consider the case for maps.
  Fix two $W$-truncated deformations
  $\Alg$ and ${\overline \Alg}$ of $\Alg^0$, and a homomorphism 
  of the underlying $(W-1)$-truncated deformations
  \[ f^{\leq W-1}\co \Alg^{\leq W-1} \to {\overline\Alg}^{\leq W-1}, \]
  with $f^0_1\co A\to A$ the identity map and $f^0_n=0$ for $n>1$.
  We wish to extend this to a homomorphism of $W$-truncated deformations
  by introducing a new component $f^W\co \Tensor^*(A_+)\to A_+$. 
  The weight $W$ $\Ainf$ relation has the form
  \[ \delta f^W = \fobstr^W;\]
  so Property~\ref{w:coboundary} follows.

  To verify Property~\ref{w:mapcocycle}, we introduce some more notation.

  Given $\phi\co \Tensor^*(A_+)\to A_+$, let
  ${\overline \phi}\co \Tensor^*(A_+)\to \Tensor^*(A_+)$ be the induced map
  \[ {\overline \phi}=(\Id_{\Tensor^*(A_+)}\otimes \phi\otimes \Id_{\Tensor^*(A_+)})\circ \Delta^3.\]
  There is a differential
  \[ D\co\Hom(\Tensor^*(A_+),\Tensor^*(A_+))\to \Hom(\Tensor^*(A_+),\Tensor^*(A_+)) \]
  defined by
  \[ D(\Phi)={\overline{\mu^0}}\circ \Phi+\Phi\circ {\overline{\mu^0}}.\]

  Equation~\eqref{eq:SecondLeibniz} has the following  analogue:
  \begin{equation}
    \label{eq:SecondLeibnizW}
    \delta(f^\bullet\circ \dtildew{\phi}^\bullet)= (\delta(f^\bullet)\circ \dtildew{\phi}^\bullet)
    + (f^\bullet \circ D \dtildew{\phi}^\bullet) 
    + (\mu^0\star f^\bullet)\circ \dtildew{\phi}^\bullet
    + \mu^0\star (f^\bullet\circ \dtildew{\phi}^\bullet).
  \end{equation}
  (The last term is a $\star$ of a single map $\mu^0\in\Mor(\Tensor^*(A_+),A_+)$
  with a sequence of maps $g^\bullet=f^\bullet\circ\dtildew{\phi}^\bullet$.
  This is to be interpreted as a sequence of maps whose $W^{\th}$ component is 
  $\mu^0\star g^W$.)
  Equation~\eqref{eq:SecondLeibnizW} follows from the following identities:
  \begin{align*}
    \delta(f^\bullet\circ \dtildew{\phi}^\bullet)&=
    \mu^0\star (f^\bullet\circ \dtildew{\phi}^\bullet) + (f^\bullet\circ \dtildew{\phi}^\bullet\circ {\overline\mu}^0) \\
    f^\bullet\circ D{\dtildew{\phi}^\bullet}&=
    (f^\bullet\circ {\overline\mu}^0 \circ {\dtildew{\phi}}^\bullet) + (f^\bullet \circ {\dtildew{\phi}}^\bullet\circ {\overline\mu}^0) \\
    (\delta f^\bullet)\circ \dtildew{\phi}\bullet
    &= (\mu^0\star f^\bullet)\circ \dtildew{\phi}^\bullet+ (f^\bullet\star \mu^0)\circ \dtildew{\phi}^\bullet.
  \end{align*}

  The maps $f^\bullet=\{f^w\co \Tensor^*(A_+)\to A_+\}_{w=0}^{\infty}$ form the components of a weighted $\Ainf$ homomorphism if 
  \begin{equation}
    \label{eq:AinfMapFirstFormulation}
    \mu^{\bullet\geq 0}\circ \dtildew{f}^\bullet + \dtildew{f}^\bullet\circ \mu^{\bullet\geq 0}=0.
  \end{equation}
  When $f^\bullet$ is merely a homomorphism of $W$-truncated
  deformations, i.e.  $f^\bullet=\{f^w\co \Tensor^*(A_+)\to
  A_+\}_{w=0}^{W}$, we require that
  Equation~\eqref{eq:AinfMapFirstFormulation} holds for the components
  indexed by $w=0,\dots,W$.

  The $w^{\th}$ component of Equation~\eqref{eq:AinfMapFirstFormulation} has two alternative
  formulations. We begin with the analgue of Equation~\eqref{eq:AinfMap1}, which states that 
  if $f^\bullet$ is a homomorphism of $W$-truncated deformations, then for each $w=0,\dots,W$,
  \begin{equation}
    \label{eq:AinfMap1w}
    \delta f^w = (f^\bullet\star {\overline{\mu^{\bullet\geq 1}}})^w
    +  (\mu^{\bullet \geq 1} \circ {\dtildew f}^\bullet)^w
    + \mu^0\circ \dtildew{f}^w + \mu^0\star f^w;
  \end{equation}
  i.e.,
  \[ 
    \delta f^w = \sum_{\substack{a+b=w\\b\geq 1}} \left(f^a\star \mu^b +  
    \mu^b \circ {\dtildew f}^b\right) + 
    \sum_{w_1+\dots+w_m=w} \mu^0\circ(f^{w_1}\otimes\dots\otimes f^{w_m})\circ \Delta^m
    + \mu^0\star f^w.\]
    Equation~\eqref{eq:AinfMap1w} follows immediately from the homomorphism relation
    \[ \mu^0 \circ (\dtildew{f}^{w})+ (f^w \star \mu^0) = (f^\bullet \star\mu^{\bullet\geq 1} + \mu^{\bullet\geq 1}\star f^\bullet)^w, \]
    together with the definition
    \[ \delta(f^w)=\mu^0\star f^w + f^w\star \mu^0.\]
    Since $f^0=\Id$, 
    the terms in $\mu^0\circ \dtildew{f}^w$ that involve the $f^w$ component cancel against
    the terms in $\mu^0\star f^w$; thus, 
    we can rewrite Equation~\eqref{eq:AinfMap1w} as:
    \begin{equation}
    \label{eq:AinfMap1w2}
    \delta f^w = (f^\bullet\star \mu^{\bullet\geq 1})^w
    +  (\mu^{\bullet \geq 1} \circ {\dtildew f}^\bullet)^w + 
    \mu^0\circ (\dtildew{f}^{\bullet<w})^w.
    \end{equation}
    Here, 
    $\mu^0\circ (\dtildew{f}^{\bullet<w})^w$ is the map which, given an element of $\Tensor^*(A_+)$, splits the element into $m$ tensor factors, 
    applies $f^{w_i}$ to the $i^{\th}$ tensor factor, for all choices of $(w_1,\dots,w_m)$ with $0\leq w_i\leq w-1$ and $\sum_{i=1}^{m} w_i =m$,
    and then applies $\mu^0_m$ to the outputs.

    The analogue of Equation~\eqref{eq:AinfMap2} is
    \begin{equation}
      \label{eq:AinfMap2w}
      D({\dtildew f})^W=({\overline\mu^{\bullet\geq 1}}\circ \dtildew{f})^W+(\dtildew{f}\circ {\overline\mu^{\bullet\geq 1}})^W.
    \end{equation}

    Using the $\Ainf$ homomorphism relation for all $w<W$, we get the following version:
    \begin{align}
      \label{eq:AinfMap2w2}
      D ({\dtildew f^{\bullet<W}})^W&= 
    ({\overline{\mu^{\bullet\geq 1}}}\circ {\dtildew f}^\bullet)^W 
    + ({\dtildew f}\circ \overline{\mu^{\bullet\geq 1}})^W \nonumber \\
    &
    + {\overline {(\mu^0\circ \dtildew{f^{\bullet<W}})^W}} +
    {\overline{(\mu^{\bullet\geq 1}\circ \dtildew{f^{\bullet<W}})^W}}
    - {\overline{(f^{\bullet<W}\star \mu^{\bullet\geq 1})^W}}
  \end{align}
  
  With this notational background in place, we turn to the verification of Property~\ref{w:mapcocycle}.
  By hypothesis, Equations~\eqref{eq:AinfMap1w2}
  holds.
  In particular, although we cannot assume that Equation~\eqref{eq:AinfMap1w} holds for $w=W$, 
  we 
  do have that
  \begin{equation}
    \label{eq:Map1wStarMu}
    (\delta f^\bullet\star \mu^{\bullet\geq 1})^W = 
  \left(\left((f^\bullet\star \mu^{\bullet\geq 1})
    +  (\mu^{\bullet \geq 1} \circ {\dtildew f}^\bullet) 
    + (\mu^0\circ \dtildew{f^{\bullet}})
    + (\mu^0\star f^{\bullet})\right)
    \star\mu^{\bullet\geq 1}\right)^W,
  \end{equation}
    using Equation~\eqref{eq:AinfMap1w2} in the components $w=0,\dots,W-1$.
    Thus, by Equations~\eqref{eq:FirstLeibnizW} and~\eqref{eq:Map1wStarMu},
  \begin{align}
    \delta(f^\bullet\star \mu^{\bullet\geq 1})^W&=
    ((\delta f^\bullet)\star \mu^{\bullet\geq 1})^W
    + (f^\bullet\star \delta(\mu^{\bullet \geq 1}))^W+ 
    \eta^W(f^\bullet,\mu^{\bullet \geq 1}) \nonumber \\
    &=((f^\bullet\star \mu^{\bullet \geq 1})\star \mu^{\bullet \geq 1})^W+ 
    ((\mu^{\bullet \geq 1}\circ \dtildew{f}^\bullet)\star \mu^{\bullet \geq 1})^W
    + ((\mu^0\circ  \dtildew{f^{\bullet}} 
    + \mu^0\star f^\bullet) \star \mu^{\bullet\geq 1})^W \nonumber 
    \\ &
    \qquad
    + (f^\bullet\star (\mu^{\bullet \geq 1}\star\mu^{\bullet \geq 1}))^W 
    + \eta^W(f^\bullet,\mu^{\bullet \geq 1}) \nonumber \\
    &= ((\mu^{\bullet \geq 1}\circ \dtildew{f}^\bullet)\star \mu^{\bullet \geq 1})^W
    + ((\mu^0\circ  \dtildew{f^{\bullet}}) \star \mu^{\bullet\geq 1})^W
    + ((\mu^0\star  f^{\bullet}) \star \mu^{\bullet\geq 1})^W
    +\eta^W(f^\bullet,\mu^{\bullet \geq 1}). \label{eq:Part1} 
  \end{align}
  Note that we cancelled above the terms $(f^\bullet\star \mu^{\bullet
    \geq 1})\star \mu^{\bullet \geq 1}$ and $ f^\bullet\star
  (\mu^{\bullet \geq 1}\star \mu^{\bullet \geq 1})$ as in
  Equation~\eqref{eq:Associator}.

  Similarly, using Equation~\eqref{eq:AinfMap2w} for $w<W$, we see
  that
  \begin{equation}
    \label{eq:MuCircAinfMap2w}
    (\mu^{\bullet\geq 1}\circ D\dtildew{f}^\bullet)^W=\left(\mu^{\bullet\geq 1}\circ (\overline{\mu^{\bullet\geq 1}}\circ \dtildew{f}^\bullet) 
    + \mu^{\bullet\geq 1}\circ (\dtildew{f}^\bullet\circ \overline{\mu^{\bullet\geq 1}})\right)^W
  \end{equation}
  Thus, by Equations~\eqref{eq:SecondLeibnizW},~\eqref{eq:MuCircAinfMap2w}, and \eqref{eq:WeightedAinftyRelation}
  \begin{align}
    \delta(\mu^{\bullet \geq 1}\circ \dtildew{f}^\bullet)^W&=
        \Big(\delta(\mu^{\bullet \geq 1})\circ \dtildew{f}^\bullet
        + \mu^{\bullet \geq 1}\circ (D \dtildew{f}^\bullet)+(\mu^0\star\mu^{\bullet \geq 1})\circ \dtildew{f}^\bullet+\mu^0\star(\overline{\mu^{\bullet \geq 1}}\circ\dtildew{f}^\bullet)\Big)^W
        \nonumber \\
        &=\Big((\mu^{\bullet \geq 1}\star \mu^{\bullet \geq 1})\circ {\dtildew{f}^\bullet} 
        + \mu^{\bullet \geq 1}\circ (\overline{\mu^{\bullet \geq 1}}\circ \dtildew{f}^\bullet)
        + \mu^{\bullet \geq 1}\circ (\dtildew{f}^\bullet \circ\overline{\mu^{\bullet \geq 1}})\nonumber \\
        &\qquad + (\mu^0\circ(\overline{\mu^{\bullet \geq 1}}\circ \dtildew{f}^\bullet)+\mu^0\star(\mu^{\bullet \geq 1}\circ \dtildew{f}^\bullet)\Big)^W 
        \nonumber \\
        &=\Big(\mu^{\bullet \geq 1}\circ (\dtildew{f}^\bullet \circ\overline{\mu^{\bullet \geq 1}})+ (\mu^0\circ(\overline{\mu^{\bullet \geq 1}}\circ \dtildew{f}^\bullet)+\mu^0\star(\mu^{\bullet \geq 1}\circ \dtildew{f}^\bullet)\Big)^W.
        \label{eq:Part2}   
  \end{align}
  Above, we
  have used that $\mu^{\bullet \geq 1}\star \mu^{\bullet \geq
    1}=\mu^{\bullet \geq 1}\circ {\overline\mu^{\bullet \geq 1}}$.

  By Equation~\eqref{eq:SecondLeibnizW}, and using the identities
  $\delta(\mu^0)=0$ and $\mu^0\star\mu^0=0$, we see that
  \begin{align}
    \delta(\mu^0\circ (\dtildew{f^{\bullet<W}})^W)&=(\delta\mu^0)\circ \dtildew{f^{\bullet<W}}^W
    + \mu^0\circ (D\dtildew{f^{\bullet<W}})^W + (\mu^0\star\mu^0)\circ \dtildew{f^{\bullet<W}}^W
    + \mu^0\star(\mu^0\circ \dtildew{f^{\bullet<W}})^W  \nonumber \\
    &= \mu^0\circ (D\dtildew{f^{\bullet<W}})^W 
    + \mu^0\star(\mu^0\circ \dtildew{f^{\bullet<W}})^W.
    \label{eq:Part3}
  \end{align}

  Applying $\mu^0$ to Equation~\eqref{eq:AinfMap2w2}, we find that
    \begin{align}
      \mu^0\circ D (\dtildew{f^{\bullet<W}})^W&= 
    \mu^0\circ({\overline{\mu^{\bullet\geq 1}}}\circ {\dtildew f}^\bullet)^W 
    + \mu^0\circ({\dtildew f}^\bullet\circ \overline{\mu^{\bullet\geq 1}})^W \nonumber \\
    &
    + \mu^0\star (\mu^0\circ \dtildew{f^{\bullet<W}})^W +
    \mu^0\star (\mu^{\bullet\geq 1}\circ \dtildew{f^{\bullet<W}})^W
    + \mu^0\star(f^{\bullet<W}\star \mu^{\bullet\geq 1})^W
    \label{eq:Part4}
  \end{align}
  
  Adding up Equations~\eqref{eq:Part1},~\eqref{eq:Part2},\eqref{eq:Part3}, and \eqref{eq:Part4}, together with 
  the following case of Equation~\eqref{eq:Associator}
  \[ \mu^0\star (f^{\bullet<W}\star \mu^{\bullet\geq 1}))+
  (\mu^0\star f^{\bullet<W})\star \mu^{\bullet\geq 1}
  +\eta^W(f^{\bullet},\mu^{\bullet\geq 1})=0,\]
  we find that
  \[
    \delta(\fobstr^W)= \delta(f^\bullet\star \mu^{\bullet \geq 1}+\mu^{\bullet \geq 1}\circ \dtildew{f}^\bullet + \mu^0\circ \dtildew{f^{\bullet<W}})^W = 0,
  \]
  verifying Property~\ref{w:mapcocycle}.
\end{proof}

We turn next to the graded case. Fix a group $\Gamma$ and central elements
$\lambda_d$ and $\lambda_w$. Assume further that $\lambda_d$
and $\lambda_w$ generate a subgroup isomorphic to $\ZZ^2$.
(We will be working here with $\Gamma=\smallGroup\times \ZZ$,
where $\ZZ$ is winding number grading, as in Section~\ref{sec:win-num}.)

If $\Alg^0$ is a $(\Gamma,\lambda_d)$-graded $\Ainf$ algebra then we can
consider the complex $\HC_\Gamma^{W,*}(\Alg^0)$ generated by elements of
$\Mor([\Id],[\Id])$ which shift the grading by
$\lambda_d^{\ell}\lambda_w^W$ for some~$\ell$. 

More explicitly, the $\Gamma$-grading on $\HC^*(\Alg^0)$ is specified by
\[
\gamma((a_1\otimes\cdots\otimes a_n)\mapsto
b)\coloneqq
\lambda_d^{1-n} \gamma(b)\cdot \gamma(a_n)^{-1}\cdots\gamma(a_1)^{-1}.
\] 
Let $\HC_\Gamma^{W,\ell}$ be the portion with 
\[
  \gamma((a_1\otimes\cdots\otimes a_n)\mapsto b)=\lambda_w^{W}\lambda_d^{\ell}.
\]
Then $\delta$ maps $\HC_\Gamma^{W,\ell}$ to $\HC_\Gamma^{W,\ell-1}$.

With these remarks in place, we have the following graded version of Proposition~\ref{prop:deform-Ainf-alg} 
(compare Proposition~\ref{prop:deform-assoc-alg-G}, as well
as~\cite[Section 2.4]{Sheridan:CY-hyp}):

\begin{proposition}
  \label{prop:deform-Ainf-alg-G}
  Given a $(\Gamma,\lambda_d,\lambda_w)$-graded $(W-1)$-truncated, weighted
  $\Ainf$ deformation $\Alg^{W-1}$ of $\Alg^0$, the obstruction class
  $\obstr^W\in\HC_\Gamma^{W,-2}(\Alg^0)$
  is the obstruction to extending $\Alg^{W-1}$ to a $W$-truncated, weighted $\Ainf$ deformation of $\Alg^0$.
  Given two $(\Gamma,\lambda_d,\lambda_w)$-graded $W$-truncated deformations $\Alg$ and $\Alg'$
  of $\Alg^0$ and a $(\Gamma,\lambda_d,\lambda_w)$-graded homomorphism $f$ between their
  underlying $(W-1)$-truncated parts, the class 
  $\fobstr^W\in \HH_\Gamma^{W,-1}(\Alg^0)$ is the obstruction to extending $f$ to a $(\Gamma,\lambda_d,\lambda_w)$-graded
  $W$-truncated homomorphism.
\end{proposition}

\begin{corollary}
  \label{cor:deform-Ainf-alg-G}
  Let $\Alg^0$ be a $(\Gamma,\lambda_d)$-graded $\Ainf$-algebra and
  $\Alg^W$ a $(\Gamma,\lambda_d,\lambda_w)$-graded $W$-truncated weighted
  deformation of $\Alg^0$. If $\HH_\Gamma^{w,-2}(\Alg^0)=0$ for all $w>W$
  then $\Alg^W$ extends to a $(\Gamma,\lambda_d,\lambda_w)$-graded weighted
  $\Ainf$-algebra structure on $A$. If $\HH_\Gamma^{w,-1}(\Alg^0)=0$ for all
  $w>W$ then any two $(\Gamma,\lambda_d,\lambda_w)$-graded weighted
  $\Ainf$-algebra structures on $A$ extending $\Alg^W$ are isomorphic.
\end{corollary}

\begin{proof}
  This follows readily from Proposition~\ref{prop:deform-Ainf-alg-G};
  cf.\ the proof of Corollary~\ref{cor:deform-assoc-alg}.
\end{proof}

\subsection{Uniqueness of \texorpdfstring{$\MAlg$}{the torus
    algebra}}\label{sec:MAlg-unique}
In this section, we view the ground ring for $\UnDefAlg$ as $\Ground=\Field\oplus
\Field$, not $\Ground[U]$. So, our augmentation is a map
$\UnDefAlg\to \Ground$, and there is a corresponding augmentation ideal.

\begin{theorem}\label{thm:MAlg-unique}
  Up to isomorphism, there is a unique weighted deformation $\MAlg$ of $\UnDefAlg$ such
  that:
  \begin{enumerate}
  \item $\MAlg$ is $\Gamma=G\times \ZZ$-graded and
  \item\label{item:MAlg-m1} $\mu^1_0=\rho_{1234}+\rho_{2341}+\rho_{3412}+\rho_{4123}$.
  \end{enumerate}
\end{theorem}
It follows that the distinguished central
elements in $G\times \ZZ$ are
\begin{align}
\lambda_d=\lambda\coloneqq (1;0,0)\times 0\label{eq:lambda-d-is-lambda}\\
\lambda_w=(0;0,0)\times 1
\end{align}
(as in Section~\ref{sec:Intermediate}).

Like Theorem~\ref{thm:UnDefAlg-unique}, Theorem~\ref{thm:MAlg-unique}
follows from a computation of certain Hochschild cohomology groups.
Before giving that computation, we adapt the material from
Section~\ref{subsec:CobarComplex} to $\UnDefAlg$.

Under mild assumptions on an augmented $\Ainf$ algebra
$\Alg^0$,
we can construct its cobar algebra $\Cobarop(\Alg^0)$, as follows.
Given an augmented $\Ainf$-algebra
$\Alg^0=(A,\{\mu_n\co A_+^{\otimes n}\to A\}_{n=1}^{\infty})$, let 
$\Cobarop(\Alg^0)$ be the component-wise dual of the bar complex,
\[
  \Cobarop(\Alg^0)=\bigoplus_{n=0}^{\infty} \Hom(A_+^{\otimes n},\Field)
\]
with multiplication induced by the comultiplication on the bar
complex.
Let
\[
  (\mu^0)^*\co A_+^*\to \Hom\biggl(\bigoplus_{n=1}^{\infty} A_+^{\otimes n},\Field\biggr)\cong \prod_{n=1}^{\infty} \bigl(A_+^{\otimes n}\bigr)^*
\]
be dual to the operations $\mu^0_n\co A_+^{\otimes n}\to A_+$. Call
$\Alg^0$ \emph{bounded enough for cobar} if the image of $(\mu^0)^*$
lies in
\[
  \Cobarop(\Alg^0)=\bigoplus_{n=1}^{\infty} \Hom(A_+^{\otimes n},\Field)\subset \prod_{n=1}^{\infty} \Hom(A_+^{\otimes n},\Field).
\]
Under this assumption, the differential on the bar complex dualizes to
a differential on $\Cobarop(\Alg^0)$.

As in Section~\ref{subsec:CobarComplex}, we will make the differential
more explicit with the help of a filtration. Assume that $A$ is
filtered by subspaces $F_0A\subset
F_1A\subset\cdots\subset A$ with $A=\bigcup_i F_iA$, and so that each
$F_iA$ is finite-dimensional, $\Ground\subset F_0A$, and
$\mu_n(F_{i_1}A,F_{i_2}A,\cdots,F_{i_n}A)\subset F_{i_1+\cdots+i_n}A$.
Then the $n$-input part of $\Cobarop(\Alg^0)$ is given by
$\Cobarop(\Alg^0)_n=\varprojlim (F_i(A_+^{\otimes n}))^*$. Pick a
basis $\{e_i\}$ for $A_+$ and define
$A_+^\dagger \subset (A_+)^*$ to be the subspace spanned by the
elements~$e_i^*$. Then
$\Cobarop(\Alg^0)_n$ is the completion of $(A_+^\dagger)^{\otimes n}$. The differential on $\Cobarop(\Alg^0)$ is
induced by the formula
\[
  \dcob(a_1^*\otimes\dots\otimes a_k^*)=\sum_{i=1}^{k} a_1^*\otimes\dots\otimes (\mu^0)^*(a_i^*)\otimes\dots\otimes a_k^*.
\]
There is also a grading on $\Cobarop(\Alg^0)$ as defined in
Equation~\eqref{eq:GradeCobar}, though $\Cobarop(\Alg^0)$ is not the
direct sum of its graded pieces.

In the application to the torus algebra, a suitable filtration $F_i$ is
given, for instance, by the winding number $\wingr$.
We will always take the elements $U^m\rho_{i,\dots,j}$ and $U^m\iota_i$ as the basis $\{e_i\}$. Call elements of this basis \emph{$\Field$-basic elements}.

\begin{example}
  For $\Cobarop(\UnDefAlg)$, the differential of $U^*$ is given by
  \[
    \dcob(U^*)=\rho_1^*\otimes\rho_2^*\otimes\rho_3^*\otimes\rho_4^*
    +\rho_2^*\otimes\rho_3^*\otimes\rho_4^*\otimes\rho_1^*
    +\rho_3^*\otimes\rho_4^*\otimes\rho_1^*\otimes\rho_2^*
    +\rho_4^*\otimes\rho_1^*\otimes\rho_2^*\otimes\rho_3^*.
  \]
  The differential of $(U\rho_1)^*$ is given by
  \[
   \dcob\bigl((U\rho_1)^*\bigr)= U^*\otimes\rho_1^*+\rho_1^*\otimes U^*
    +\rho_3^*\otimes\rho_4^*\otimes\rho_1^*\otimes\rho_{12}^*
    +\rho_{41}^*\otimes\rho_1^*\otimes\rho_2^*\otimes\rho_3^*.
  \]
  The differential of $(U^2)^*$ is given by
  \begin{align*}
    \dcob\bigl((U^2)^*\bigr)=U^*\otimes U^*
    &+(U\rho_1)^*\otimes\rho_2^*\otimes\rho_3^*\otimes\rho_4^*
    +\rho_1^*\otimes(U\rho_2)^*\otimes\rho_3^*\otimes\rho_4^*
    +\rho_1^*\otimes\rho_2^*\otimes(U\rho_3)^*\otimes\rho_4^*\\
    &+\rho_1^*\otimes\rho_2^*\otimes\rho_3^*\otimes(U\rho_4)^*
    +(\text{cyclic permutations of indices in these terms})\\
    &+\rho_{41}^*\otimes\rho_1^*\otimes\rho_2^*\otimes\rho_{23}^*\otimes\rho_3^*\otimes\rho_4^*
    +(\text{cyclic permutations of these 6 chords})\\
    &+\rho_{12}^*\otimes\rho_2^*\otimes\rho_3^*\otimes\rho_{34}^*\otimes\rho_4^*\otimes\rho_1^*
    +(\text{cyclic permutations of these 6 chords}).
  \end{align*}
\end{example}

The following lemma will quickly lead to a proof that $\UnDefAlg$ is bounded enough for cobar:
\begin{lemma}\label{lem:UnDefAlg-bounded-ish}
  Let $\UnDefAlg$ be a $\Gamma$-graded deformation of $\AsUnDefAlg$.  Given
  $h=(j;{\mathfrak a},{\mathfrak b})\times m\in \Gamma$, there is an upper bound
  on $n$ for which there is a sequence $(a_1,\dots,a_n)$ of $\Field$-basic elements in the augmentation ideal $A_+$ of
  $\AsUnDefAlg[U]$, with the following properties:
  \begin{enumerate}[label=(c-\arabic*),ref=(c-\arabic*)]
  \item\label{item:NonZero}\label{item:First} The tensor product
    $(a_n\otimes\cdots\otimes a_1)\neq 0\in A_+^{\otimes n}$, or equivalently
    $(a_1^*\otimes\cdots\otimes a_n^*)\neq 0$ in $(A_+^*)^{\otimes n}$.
  \item\label{item:Last} The gradings satisfy
    $\lambda^n\gamma(a_n)\cdots\gamma(a_1)=h^{-1}$, or equivalently
    $\gammaCobar(a_1^*\otimes\dots\otimes a_n^*)=h$.
  \end{enumerate}
\end{lemma}

\begin{proof}
  Recall that a $\Field$-basic element  $a_i$ of $A_+$ is of the form
  $U^{\ell_i} \rho^i$, where $\ell_i\geq 0$, or $U^\ell\iota_k$ where $\ell>0$.

  Let $a$ be a basic algebra element.
  For $i=1,\dots,4$, let $\wingr_i(a)$ denote the multiplicity of an element  of $\UnDefAlg$
  at $\rho_i$; i.e., for a chord~$\rho$, $\grb(\rho) = (-1/2;\wingr_1(\rho),\wingr_2(\rho),\wingr_3(\rho),\wingr_4(\rho))$, and $\wingr_4$ was the winding number grading $\wingr$ of Section~\ref{sec:win-num}.
  We extend this to $(a_1^*\otimes\dots\otimes a_n^*)$ by
  \[
    \wingr_i(a_1^*\otimes\dots\otimes a_n^*)=-\sum_{j=1}^n \wingr_i(a_j).
  \]

  Perusing the gradings from Section~\ref{sec:Intermediate}, we see that
  if $\gr(\rho)=(j;\alpha,\beta)$, then
  $\alpha+\beta=\wingr_2(\rho)-\wingr_4(\rho)$. Thus,
  ${\mathfrak a}+{\mathfrak b}=-\wingr_2+\wingr_4$.
  Since $\wingr_4=m$, $\wingr_2$ is now determined by $(j;{\mathfrak a},{\mathfrak b})$.

  Suppose that $x$ of the $a_i$ have the form $U^{\ell_i}\iota_k$, and let
  $y=n-x$.  Property~\ref{item:NonZero} ensures that
  \begin{equation}
    \label{eq:YisBounded}
    -\wingr_2-\wingr_4\geq \lfloor y/2\rfloor,
  \end{equation}
  for if some element $a_i^*$ has $\wingr_2(a_i^*)+\wingr_4(a_i^*)=0$
  but $\wingr_1(a_i^*)+\wingr_3(a_i^*)\neq 0$,
  then the next $j>i$ so that $a_j^*\neq (U^{n_j}\iota_k)^*$ has the property that
  $\wingr_2(a_j^*)+\wingr_4(a_j^*)\neq 0$.

  Now, there are only finitely many sequences 
  $b_1^*\otimes\dots\otimes b_y^*$ with the following properties:
  \begin{itemize}
  \item The length of the sequence is fixed. (It is $y$, as above.)
  \item Each $b_i$ is a Reeb element
  \item $\wingr_4(b_1^*\otimes\dots\otimes b_y^*)$  is fixed. (It is
    $m$, as in the statement.)
  \end{itemize}
  Thus, for all sequences as above, there is a constant $c$ with the property that
  the Maslov component of
  $\gr(b_1)^{-1}\cdots\gr(b_y)^{-1}$ is bounded below by $c$.

  Next, $\gammaCobar((U^\ell\iota_k)^*)=\lambda^{2\ell}$.  It
  follows that the Maslov component $j$ of
  $\lambda^{-n}\gr(a_1)^{-1}\cdots\gr(a_n)^{-1}$ satisfies
  \begin{equation}
    \label{eq:XisBounded}
    j\geq x+c.
  \end{equation}
  The upper bounds from Equation~\eqref{eq:YisBounded} and~\eqref{eq:XisBounded}
  give the desired upper bound in terms of $h$ on $n$ for which there
  is a sequence $(a_1,\dots,a_n)$ with Properties~\ref{item:First}
  and~\ref{item:Last}:
  $n\leq j - c + 2\mathfrak{a} + 2\mathfrak{b} - 4m + 1$.
\end{proof}

\begin{corollary}\label{cor:BoundedEnoughForCobar}
  Any $\Gamma$-graded deformation of $\AsUnDefAlg$ is bounded enough for cobar.
\end{corollary}
\begin{proof}
  The $\Gamma$-grading hypothesis ensures that
  \[
    \gammaCobar((\mu^0_n)^*(b^*))=\lambda^{-1}\cdot \gammaCobar(b^*).
  \]
  So, this is immediate from Lemma~\ref{lem:UnDefAlg-bounded-ish}.
\end{proof}

\begin{remark}
  For the particular $\Gamma$-graded deformation $\UnDefAlg$ of
  $\AsUnDefAlg$ constructed geometrically in Section~\ref{sec:def},
  Lemma~\ref{lem:WeightGrading} (applied to operations with $w=0$)
  immediately implies that $\UnDefAlg$ is bounded enough for cobar.
\end{remark}

Lemma~\ref{lem:Koszul} identifies the cobar algebra of $\AsUnDefAlg$.
We promote this to an identification of the cobar algebra of $\UnDefAlg$, as follows.
Consider
$\SmallCobar=\AsUnDefAlg[h]/(h^2)$, equipped 
with the differential $\partial'$ that vanishes on $\AsUnDefAlg\subset\SmallCobar$ and satisfies
\[
 \partial' h = \rho_{1234}+\rho_{2341}+\rho_{3412}+\rho_{4123}.\]
A $\Gamma$-grading on $\SmallCobar$ is specified by
\[ \gamma'(h)=(-1;0,0)\times 1\]
and the condition that the natural inclusion map 
$i\co \AsUnDefAlg\to \SmallCobar$ (with image the elements without factors of~$h$)
satisfies
\[
  \gamma'(i([a]))=\alpha(\gamma(a)),
\]
where $\alpha$ is as in Equation~\eqref{eq:DefAlpha}.

\begin{lemma}
  \label{lem:Koszul2}
  There is a quasi-isomorphism of $\Gamma$-graded algebras
  $\phi'\co \Cobarop(\UnDefAlg)\to \SmallCobar$ satisfying
  $\phi'(\iota_0)=\iota_1$, $\phi'(\iota_1)=\iota_0$,
  $\phi'(\rho_i^*)=[\rho_i]$ for $i=1,\dots,4$, $\phi'(U^*)=[h]$, and
  $\phi'(a^*)=0$ for all $\Field$-basic elements~$a$ with $a\neq U$ and
  $|a|>1$.
\end{lemma}
(Here, by quasi-isomorphism, we mean
a ring homomorphism that induces an
isomorphism from each graded part of the homology of
$\Cobarop(\UnDefAlg)$ to the corresponding graded part of
$\SmallCobar$; recall that $\Cobarop(\UnDefAlg)$ is not the direct sum
of its homogeneous pieces.)
\begin{proof}
  The map $\phi'$ extends continuously to $A_+^*\subset \Cobarop(\UnDefAlg)$ because all elements with
  sufficiently large $\wingr_4$ are in the kernel. It is
  well-defined on $\Cobarop(\UnDefAlg)$
  because we defined
  $\Cobarop(\UnDefAlg)$ using a direct sum instead of a direct product.
  The fact that $\phi'$ is a chain map and respects the gradings is straightforward.

  As a chain complex, $\Cobarop(\UnDefAlg)$ is generated by elements of the form
  \[
    \rho_i^*\otimes\rho_{i+1}^*\otimes\dots\otimes \rho_j^*\otimes a_1^*\otimes\dots \otimes a_m^*,
  \]
  $m\geq 0$, where $a_1$ has length greater than one, or
  $a_1=\rho_{j-1}$. (Recall that the length of $U$ is defined to be
  $4$.)

  Consider the chain map
  \[
    j'\co \SmallCobar\to \Cobarop(\UnDefAlg)
  \]
  defined by
  $j'(\iota_0) = \iota_1$, $j'(\iota_1) = \iota_0$ and
  \begin{align*} 
    j'([h])&= U^* \\
    j'([\overbrace{\rho_i\dots \rho_j}^k])&=\rho_i^*\otimes\dots\otimes \rho_j^* & k\geq 1\\
    j'([h \overbrace{\rho_i\dots \rho_j}^k])&=\rho_i^*\otimes\dots\otimes \rho_j^*\otimes U^* 
    +\rho_i^*\otimes\dots\otimes \rho_{j-1}^*\otimes (\rho_{j-1}\rho_j)^* \otimes \rho_{j}^* \otimes \rho_{j+1}^*\otimes\rho_{j+2}^* & k\geq 1.
  \end{align*}
  (The map $j'$ is a chain map, but it is not a ring homomorphism.)
  We will prove that $j'$ and $\phi'$ are homotopy inverses.
  It is easy to see that $\phi' \circ j' = \Id$ as chain maps, so we focus on the
  reverse direction.

  Consider a pure generator $b_1^*\otimes\dots\otimes b_m^*$, and suppose that
  there is some $j\in 1,\dots,m$ so that $U$ divides $b_j$. 
  Let
  $\ell$ be the minimum $i$ so that $U$ divides $b_j$.  We then
  associate to the pure generator the following integer:
  \[
    v=\sum_{i=\ell}^m|b_i|.
  \]
  In cases where there is no $j$ so that $U$ divides $b_j$, let $v=0$.
  Let $\Filt^{v_0}\subset \Cobarop(\UnDefAlg)$ denote the subset
  generated by all generators with $v\leq v_0$.  Clearly,
  $\Filt^{v_0}$ is a subcomplex of $\Cobarop(\UnDefAlg)$.  The
  image of $j'$ is contained in $\Filt^{4}$.
  (Not every element of $\Cobarop(\UnDefAlg)$ lies in some
  $\Filt^{v}$, but $\Cobarop_n(\UnDefAlg)$ is the completion of
  $\bigcup_v (\Filt^v\cap \Cobarop_n(\UnDefAlg))$ with respect to the
  winding number filtration.)

  Consider the map $H'\co \Cobarop(\UnDefAlg)\to \Cobarop(\UnDefAlg)$ defined by
  \[
    H'(b_1^*\otimes\dots\otimes b_m^*)=
    \begin{cases}
      0 & \text{if $\forall j$, $U$ does not divide $b_j$.}\\
      0 &{\text{if $b_\ell \neq U$ or $b_\ell=U$ and $\ell=m$.}} \\
    b_1^*\otimes\dots\otimes b_{\ell-1}^*\otimes (b_{\ell+1}\cdot U)^*\otimes b_{\ell+2}^*\otimes \dots\otimes b_m^*& {\text{if $b_\ell=U$ and $\ell<m$.}} 
    \end{cases}
  \]
  and extending continuously to the completion.
  
  Obviously,  $H'(\Filt^{v_0})\subset \Filt^{v_0}$.
  Moreover, if $v>4$, then we claim that
  \begin{equation}
    \label{eq:Containment}
    (\Id + \dcob\circ H'+H'\circ \dcob)(\Filt^{v})\subset  \Filt^{v-1}.
  \end{equation}
  We check this for each
  element of the form
  \[ x=b_1^*\otimes\dots\otimes b_m^*, \]
  as follows.
  \begin{itemize}
  \item If there $U$ does not divide any of the $b_j$ then
    $x\in\Filt^0$ and $H'(x)=H'(\dcob(x))=0$, so
    $(\Id + \dcob\circ H'+H'\circ \dcob)(x)=x$ still lies in
    $\Filt^0\subset \Filt^4$.
    \item  Suppose that $b_\ell\neq U$. In this case,
      $H'(x)=0$, and $H'\circ \dcob(x)=x$; in particular,
      \[ \Id+\dcob\circ H'+H'\circ \dcob=0.\]
    \item Suppose that $b_\ell=U$ and $m>\ell$.
      Terms in $\dcob\circ H'$ arising from $\dcob(b_k^*)$ for
      $k>\ell+1$
      cancel with corresponding terms in $H'\circ \dcob$.
      Terms in $\dcob \circ H'$ arising from $\dcob((b_{\ell+1}\cdot U)^*)$
      are of the following types:
      \begin{itemize}
      \item 
        those (using the part of $\dcob$ adjoint to $\mu_2^*$) that cancel against the terms in $H'\circ \dcob$
        coming from $\dcob(b_{\ell+1})$;
      \item 
        the term $x$;
      \item 
        terms arising from $\mu_i^*$ with $i>2$; these automatically lie in 
        $\Filt^{v-4}$; and
      \item additional terms corresponding to factorizations of
        $b_{\ell+1}\cdot U= a_1\cdot a_2$
        (which dualize to $a_2^*\otimes a_1^*$), 
        where $U$ does not divide $a_2$. These lie in $\Filt^{v-1}$.
      \end{itemize}
      \end{itemize}
      On the other hand, on $\Filt^4$, $H'$ vanishes, so
      $\dcob\circ H'+H'\circ \dcob=0$.

      Define
      \[
        G=\lim_{n\to\infty}(\Id + \dcob\circ H'+H'\circ \dcob)^n\co \Cobarop(\UnDefAlg)\to \Filt^4.
      \]
      To see this limit exists, note that by
      Lemma~\ref{lem:UnDefAlg-bounded-ish}, for any particular grading
      and winding number filtration level, every basis element lies in
      $\Filt^v$ for some bounded~$v$. Since $\Id + \dcob\circ
      H'+H'\circ \dcob$ decreases the filtration level until we get to
      $\Filt^4$, where it is the identity, $G$ is a well-defined
      continuous map.
      For the same reason, the
      map $G$ is homotopic to the identity map. So,
      $\Cobarop(\UnDefAlg)$ is quasi-isomorphic to its subcomplex $\Filt^4$.
      
      Next, let $H$ be the homotopy from the proof of
      Lemma~\ref{lem:Koszul}. Consider the homotopy operator
      $H''\co \Filt^{4}\to \Filt^{4}$ defined by
      \[
      H''(x)=\begin{cases}
      H(x) &{\text{if $x\in \Filt^0$}} \\
      H(x_0)\otimes U^* &{\text{if $x=x_0\otimes U^*$ with $x_0\in \Filt^0$}}
      \end{cases}
    \]
    and extending continuously to all of $\Filt^4$. (In particular,
    $H''(U)=0$.)  We claim that
    \begin{equation}
      \label{eq:HomotopyFormulaPrime}
      \bigl(\dcob \circ H''+H''\circ \dcob +\Id\bigr)|_{\Filt^{4}} = j'\circ \phi'.
    \end{equation}
    Specifically, for $x_0\in \Filt^0$ of the form
    $x_0=a_1^*\otimes\dots\otimes a_n^*$ with $|a_i|\neq 1$ for some
    $i$ or $a_{i+1}\cdot a_i\neq 0$, the proof of
    Equation~\eqref{eq:HomotopyFormula} proves that
    $(\dcob\circ H''+H''\circ\dcob)(x)=x$.  This argument readily
    adapts also to $x\in \Filt^4$ of the form $x=x_0\otimes U^*$ with
    $x_0$ as above.  For the remaining generators $x$ of the form
    $\rho_i^*\otimes \rho_{i+1}^*\otimes\dots\otimes\rho_j^*$ or
    $\rho_i^*\otimes \rho_{i+1}^*\otimes\dots\otimes\rho_j^*\otimes
    U^*$ for which $H''$ vanishes,
    Equation~\eqref{eq:HomotopyFormulaPrime} is a straightforward
    verification.  Equation~\eqref{eq:HomotopyFormulaPrime} finishes
    the proof
    that $\phi'$ and $j'$ are inverse
    quasi-isomorphisms of chain complexes.
\end{proof}

Our interest in the cobar algebra comes from its relation to the
Hochschild complex. (Compare Equation~\eqref{eq:CobarHC}.) First,
\begin{equation}
  \label{eq:CobarHC2}
  \Alg^0\hotimes_{\Ground\otimes\Ground} \Cobarop(\Alg^0)\cong \HC(\Alg^0),
\end{equation}
where $\hotimes$ denotes the completed tensor product with
respect to the filtration on
$\bigoplus_n \Alg^0\otimes \Hom(A_+^{\otimes n},\Field)$ induced by
the length filtration on~$A_+^{\otimes n}$.
(This completion also completes the direct sum over~$n$ in
the definition of $\Cobarop(\Alg^0)$ to a direct product.)
Note that for each grading on
$\Alg^0\hotimes \Cobarop(\Alg^0)$, by
Lemma~\ref{lem:UnDefAlg-bounded-ish}, it is equivalent to
complete by the length filtration on the first, $\Alg^0$, factor.

Define a differential on $\Alg^0\hotimes
\Cobarop(\Alg^0)$ by setting,
for $a \in \Alg^0$ and $x^* \in \Cobarop(\Alg^0)$,
\begin{equation}\label{eq:HC-diff-cob}
  \partial(a\otimes x^*)=a\otimes
  \dcob(x^*)+\sum_{\mu^0_n(b_1,\dots,b_k,a,c_1,\dots,c_\ell)=d}
  d\otimes c_\ell^*\otimes\cdots\otimes c_1^*\otimes x^*\otimes
  b_k^*\otimes\cdots\otimes b_1^*
\end{equation}
(with the $b_i$ and $c_i$ in the sum being
basic algebra elements)
and extending continuously. The fact that this sum
converges follows from the fact that each filtration level on $\Alg^0$ is finite-dimensional.
With respect to this differential, Equation~\eqref{eq:CobarHC2} is an
isomorphism of chain complexes.

Lemma~\ref{lem:Koszul2} leads to a smaller model of the Hochschild
complex, as follows.  Define an analogue~$\mathfrak{C}^*$ of the
complex $C^*$ from Definition~\ref{def:SmallModel}.  As a vector
space,
$\mathfrak{C}^*=\UnDefAlg\hotimes_{\Ground\otimes\Ground}
\SmallCobar$, where $\hotimes$ indicates we are completing
with respect to the winding number filtration on $\SmallCobar$.
(Again by
Lemma~\ref{lem:UnDefAlg-bounded-ish}, in each grading this is
equivalent to completing
with respect to either the winding number filtration or the length
filtration on $\UnDefAlg$.) Write an element
$a\otimes b\in \UnDefAlg\otimes\SmallCobar$ of $\mathfrak{C}^*$ as
$a[b]$. The fact that the tensor product is over
$\Ground\otimes\Ground$ means that the right idempotent of $b$ is
complementary to the left idempotent of $a$, and the left idempotent
of $b$ is complementary to the right idempotent of $a$. An arbitrary element of $\mathfrak{C}^*$ is a possibly
infinite linear combination of elements of the form $a[b]$, with
finitely many terms where $b$ has winding number $<k$.

The $\Gamma$-grading on $\mathfrak{C}^*$ is specified by
$\gamma(a[b])=\lambda\cdot \gamma(a)\cdot \gamma'(b)$.
Define the differential on basic elements by
\begin{multline}
  \partial (a [b])=a [\partial b]+\sum_{i=1}^4
  \left(\rho_i\cdot a  [b \cdot \rho_i]+ a\cdot \rho_i 
    [\rho_i\cdot b]\right) \\
      +\sum_{i=1}^4
      \bigl(\mu_4(\rho_{i+3},\rho_{i+2},\rho_{i+1}, a) [b \cdot\rho_{i+1,i+2,i+3}]
        +\mu_4(\rho_{i+2},\rho_{i+1},a,\rho_{i-1}) [\rho_{i-1}\cdot b \cdot\rho_{i+1,i+2}]\\
        +\mu_4(\rho_{i+1},a,\rho_{i-1},\rho_{i-2}) [\rho_{i-2,i-1}\cdot b \cdot\rho_{i+1}]
        +\mu_4(a,\rho_{i-1},\rho_{i-2},\rho_{i-3})  [\rho_{i-3,i-2,i-1}\cdot b ]
\end{multline}
and extend continuously.

For example, 
\begin{align}
  \partial(\rho_{1234}[\iota_1])&=\rho_{12341}[\rho_1]+\rho_{41234}[\rho_4] 
  + U \rho_{123}[\rho_{123}]+U \rho_{234}[\rho_{234}]\label{eq:diff1} \\
  \partial(\iota_0[\iota_1])&=\rho_1[\rho_1]+\rho_2[\rho_2]+\rho_3[\rho_3]+\rho_4[\rho_4]
  =\partial(\iota_1[\iota_0]) \label{eq:diff2} \\
 \partial(\rho_1[\rho_1])&= U[\rho_{1234}]+U[\rho_{3412}]+ U[\rho_{4123}]+U[\rho_{2341}]. \label{eq:diff3}
\end{align}

Observe that the differential commutes with the $\FF_2[U]$-module structure.

\begin{lemma}\label{lem:C-is-cx}
  This differential makes $\mathfrak{C}^*$ into a chain complex.
\end{lemma}
\begin{proof}
  Write
  \begin{align*}
    \partial_1 (a [b])&=a [\partial b]\\
    \partial_2(a[b])&=\sum_{i=1}^4
                      \left(\rho_i\cdot a  [b \cdot \rho_i]+ a\cdot \rho_i 
                      [\rho_i\cdot b]\right) \\
    \partial_4(a[b])&=\sum_{i=1}^4
      \mu_4(\rho_{i+3},\rho_{i+2},\rho_{i+1}, a) [b \cdot\rho_{i+1,i+2,i+3}]
        +\mu_4(\rho_{i+2},\rho_{i+1},a,\rho_{i-1}) [\rho_{i-1}\cdot b \cdot\rho_{i+1,i+2}]\\       
        &\qquad+\mu_4(\rho_{i+1},a,\rho_{i-1},\rho_{i-2}) [\rho_{i-2,i-1}\cdot b \cdot\rho_{i+1}]
        +\mu_4(a,\rho_{i-1},\rho_{i-2},\rho_{i-3})  [\rho_{i-3,i-2,i-1}\cdot b ].
  \end{align*}
  It is immediate that
  \[
    \bdy_1^2=\bdy_1\bdy_2+\bdy_2\bdy_1=\bdy_1\bdy_4+\bdy_4\bdy_1=\bdy_2^2=0.
  \]
  To see that $\bdy_2\bdy_4+\bdy_4\bdy_2=0$ is a case check. The most interesting case is when $a$ has length $1$, say, $a=\rho_1$. In that case, we have
  \begin{align*}
    \bdy_2(\bdy_4(\rho_1[b]))
    &=\bdy_2\bigl(\mu_4(\rho_4,\rho_3,\rho_2,\rho_1)[b\cdot\rho_{234}]+
      \mu_4(\rho_3,\rho_2,\rho_1,\rho_4)[\rho_4\cdot b\cdot\rho_{23}]\\
    &\quad+
      \mu_4(\rho_2,\rho_1,\rho_4,\rho_3)[\rho_{34}\cdot b\cdot\rho_{2}]+
      \mu_4(\rho_1,\rho_4,\rho_3,\rho_2)[\rho_{234}\cdot b]\bigr)
    \\
    &=U\rho_1[b\cdot\rho_{2341}]+U\rho_2[\rho_2\cdot b\cdot \rho_{234}]+U\rho_4[\rho_4\cdot b\cdot \rho_{234}]\\
    &\quad+U\rho_3[\rho_{34}\cdot b\cdot \rho_{23}]
      +U\rho_4[\rho_4\cdot b\cdot \rho_{234}]+U\rho_2[\rho_{234}\cdot b\cdot\rho_2]+U\rho_3[\rho_{34}\cdot b\cdot \rho_{23}]\\
    &\quad+U\rho_1[\rho_{1234}\cdot b]+U\rho_2[\rho_{234}\cdot b\cdot \rho_2]+U\rho_4[\rho_{234}\cdot b\cdot \rho_4]\\
    &=U\rho_1[\rho_{1234}\cdot b]+U\rho_1[b\cdot \rho_{2341}]+U\rho_2[\rho_2\cdot b\cdot\rho_{234}]+U\rho_4[\rho_{234}\cdot b\cdot\rho_4]\\
    \bdy_4(\bdy_2(\rho_1[b]))
    &=\bdy_4\bigl(\rho_{12}[\rho_2\cdot b]+\rho_{41}[b\cdot \rho_4]\bigr)\\
    &=\mu_4(\rho_{12},\rho_1,\rho_4,\rho_3)[\rho_{3412}\cdot b]
      +\mu_4(\rho_4,\rho_3,\rho_2,\rho_{12})[\rho_2\cdot b\cdot \rho_{234}]\\
      &\quad+\mu_4(\rho_{41},\rho_4,\rho_3,\rho_2)[\rho_{234}\cdot b\cdot \rho_4]
        +\mu_4(\rho_3,\rho_2,\rho_1,\rho_{41})[b\cdot \rho_{4123}]\\
    &=U\rho_1[\rho_{3412}\cdot b]
      +U\rho_2[\rho_2\cdot b\cdot \rho_{234}]
      +U\rho_4[\rho_{234}\cdot b\cdot \rho_4]
        +U\rho_1[b\cdot \rho_{4123}].
  \end{align*}
  So,
  \[
    \bdy_2(\bdy_4(\rho_1[b]))+\bdy_4(\bdy_2(\rho_1[b]))=U\rho_1[(\rho_{1234}+\rho_{3412})\cdot b]+U\rho_1[b\cdot (\rho_{2341}+\rho_{4123})]=0.
  \]

  Finally, we check that $\bdy_4^2=0$. Suppose
  $
    a'[\rho_{i,\dots,j}\cdot b\cdot \rho_{k,\dots,\ell}]
  $
  appears in $\bdy_4^2(a[b])$. Then there is a composition of two
  $\mu_4$-operations taking
  $
  \rho_{\ell}\otimes\cdots\otimes\rho_k\otimes a\otimes\rho_j\otimes\cdots\otimes\rho_i
  $  
  to $a'$. It follows that the first operation must be left-extended
  or right-extended, and have $a$ (which must have length $>1$) as the
  left-most or right-most term, respectively. If the output of
  the first operation is not the right-most (respectively left-most) term of the
  second operation then the product
  in brackets vanishes:  if we write $a=\rho_{m,\dots,n}$ then in the
  left-extended case, say, the result of the first operation is
  $\rho_{m,\dots,n-1}[\rho_{n-3,n-2,n-1}b]$ and if
  $\rho_{m,\dots,n-1}$ is not the last input to the next operation
  then the coefficient of $b$ has a term $\rho_{n-2}\rho_{n-3,n-2,n-1}=0$.
  In the remaining case, the pair of operations must
  be either
  \begin{align*}
    \mu_4(\mu_4(\rho_{k+2},\rho_{k+1},\rho_{k},a),\rho_{\ell+2},\rho_{\ell+1},\rho_{\ell})\qquad\text{or}\qquad
    \mu_4(\rho_{k+2},\rho_{k+1},\rho_{k},\mu_4(a,\rho_{\ell_2},\rho_{\ell+1},\rho_{\ell})).
  \end{align*}
  These two cases cancel with each other.
\end{proof}

For integers $W,\ell$, let $\mathfrak{C}^{W,\ell}_\Gamma\subset \mathfrak{C}^*$ denote the
subspace spanned by elements $a[b]$ with
\begin{equation}
  \label{eq:SpecialGradings}
  \gamma(a[b])=\lambda_d\cdot \gamma(a)\cdot\gamma'(b)
  =\lambda_w^{W}\lambda_d^{\ell},
\end{equation}
and 
\[
  \mathfrak{C}_\Gamma=\bigoplus_{W,\ell} \mathfrak{C}^{W,\ell}.
\]

The following is an analogue of Proposition~\ref{prop:SmallerModel}:
\begin{proposition}
  \label{prop:SmallerModel2}
  The chain complex $\mathfrak{C}^*_{\Gamma}$ is quasi-isomorphic to 
  the complex $\HC^*_\Gamma(\UnDefAlg)$; in particular
  $H^{W,k}(\mathfrak{C}_\Gamma)\cong \HH^{W,k}_\Gamma(\UnDefAlg)$.
\end{proposition}

\begin{proof}
  Consider the map
  \[
    \Id\otimes \phi'\co \HC^*(\UnDefAlg)=\UnDefAlg\hotimes\Cobarop(\UnDefAlg)
    \to \UnDefAlg\otimes_{\Ground\otimes\Ground}\SmallCobar
  \]
  where $\phi'$ is as in Lemma~\ref{lem:Koszul2}. That is, define
  \[
    (\Id\otimes\phi')(\sum a_i\otimes x^*_i)=\sum a_i[\phi'(x^*_i)].
  \]
  Since $\phi'$ preserves the winding number, this induces a
  well-defined map of completed tensor products.
  
  It is immediate from the definitions that this respects the
  $\Gamma$-gradings.

  To see that $\Id\otimes\phi'$ is a chain map, write the differential
  on $\mathfrak{C}^*$ as $\bdy_1+\bdy_2+\bdy_4$ as in
  Lemma~\ref{lem:C-is-cx}. Given an element
  $a\otimes x^*\in\HC^*(\Alg^0)$, it is straightforward to check that
  $(\Id\otimes\phi')(a\otimes \dcob(x^*))=\bdy_1(a[x])$: both sides
  are only non-zero if $x^*=U^*$.
  Similarly, terms of the form
  $d\otimes c^*\otimes x^*$ (respectively $d\otimes x^*\otimes b^*$)
  in the differential on $\HC^*(\Alg^0)$ coming from operations
  $\mu_2^0(a,c)=d$ (respectively $\mu_2^0(b,a)=d$)
  (Equation~\eqref{eq:HC-diff-cob}) are mapped to zero by
  $\Id\otimes\phi'$ unless $c$ (respectively $b$) has length $1$, in
  which case these terms correspond to $\bdy_2$. Similarly, the only
  terms in the differential on $\HC^*(\Alg^0)$ coming from
  $\mu_4^0$-operations involving $a$ which are not mapped to zero by
  $\Id\otimes\phi'$ have the other three chords of length $1$, and
  these terms correspond to $\bdy_4$. Finally, by
  Lemma~\ref{lem:alg-property-factor}, every non-zero $\mu_n^0$ with
  $n>4$ has at least two inputs with length $>1$, and hence the
  corresponding terms are mapped to zero by $\Id\otimes\phi'$.

  To see that $\Id\otimes\phi'$ is a quasi-isomorphism, filter both
  $\HC^*_\Gamma(\UnDefAlg)$ and $\mathfrak{C}^*_\Gamma$ by the total
  length of the element of $\UnDefAlg$. Both are complete with respect
  to this filtration. By Lemma~\ref{lem:Koszul2}, the map
  $\Id\otimes\phi'$ induces an isomorphism of $E^1$-pages of the
  resulting spectral sequences, hence is a quasi-isomorphism.
\end{proof}

\begin{proposition}\label{prop:UnDefHoch} The Hochschild cohomology
  groups $\HH_{\Gamma}^{W,k}(\UnDefAlg)$, $W>0$, have
  \begin{align*}
  \HH^{W,-1}_\Gamma(\UnDefAlg)&=
  \begin{cases}
   \FieldSup{2} & W=1\\
    0 &\text{otherwise}
  \end{cases}
  \end{align*}
  and $\HH^{W,-2}_\Gamma(\UnDefAlg)$ is entirely supported in weight
  ($W$) grading $1$.
  Moreover, one can choose a basis for $\HH^{1,-1}_\Gamma(\UnDefAlg)$ so
  that one basis element sends $1\in\Ground$ to
  $\rho_{1234}+\rho_{2341}+\rho_{3412}+\rho_{4123}$ and the other
  sends $1\in\Ground$ to $U=U(\iota_0+\iota_1)$.
\end{proposition}
\begin{proof}
By Proposition~\ref{prop:SmallerModel2}, we can use ${\mathfrak C}^*_\Gamma$ to compute the
$\Gamma$-graded Hochschild cohomology. In turn, this complex is
generated by the
following elements $a[b]$:
  \begin{enumerate}[label=(${\mathfrak C}$-\arabic*)]
  \item the elements
    \[
      a[b]\in \{\rho_1  [\rho_1 ], \rho_{123}[\rho_{123}],\iota_0[\iota_1], \iota_1[\iota_0]\},\]
\item 
  \label{Cenlarge2}
  any of the elements 
  obtained by multiplying the above $b$ by some further
  element $b'$ with $|b'|=4s_1$, and also multiplying $a$ by
  some further element $a'$ with $|a'|=4s_2$. Here, $a'$ might contain
  factors of $U$, and $b'$ might contain (at most one factor of) $h$;
  with the understanding that $|U|=|h|=4$.
\item any element obtained by adding some $i\in \Zmod{4}$ to all the indices
  in any of the above elements.
\end{enumerate}

The proof is straightforward. Observe that in Case~\ref{Cenlarge2}, 
$s_1$ and $s_2$ can be distinct. (Compared with Case~\ref{Cenlarge},
we now have more flexibility in the gradings.)

Next, we compute the $(W,k)$ bigradings.
The elements  $\rho_1[\rho_1]$ and $\rho_{123}[\rho_{123}]$ have $(W,k)=(0,0)$,
as do the additional six elements obtained by cyclically permuting indices; moreover,
$\iota_0[\iota_1]$ and $\iota_1[\iota_0]$ have $(W,k)=(0,1)$.
All other bigradings are determined by the relations
\[
  W(a \cdot a'[b \cdot b'])=W(a[b])+s_1-s_2 \qquad{\text{and}}\quad
  k(a \cdot a'[b \cdot b'])=k(a[b])-2s_1
\]
where $|a'|=4s_1$ and $|b'|=4 s_2$, provided $b'$ is not divisible by $h$; and
\[
  W(a [b \cdot h])=W(a[b])-1 \qquad{\text{and}}\quad k(a [b \cdot h])=k(a[b])+1.
\]

Thus, provided that $W>0$, there are no non-zero elements with $k=0$ and the
following elements with $k=-1$
or $k=-2$:
\begin{itemize}
\item $k=-1$: $\rho_{1234}[\iota_1]$, $\rho_{2341}[\iota_0]$,
  $\rho_{3412}[\iota_1]$, $\rho_{4123}[\iota_0]$, $U\iota_0[\iota_1]$ and
  $U\iota_1[\iota_0]$.
\item $k=-2$: $\rho_{1234123}[\rho_{123}]$,
  $U\rho_{123}[\rho_{123}]$, $\rho_{12341}[\rho_1]$,
  $U\rho_1[\rho_1]$,  $\rho_{12341234}[h]$,
  $U\rho_{1234}[h]$,  $U^2\iota_0[h\iota_1]$ and
  $U^2\iota_1[h \iota_0]$ and $18$ more obtained by shifting the
  indices of all but the last two terms.
\end{itemize}
From the computation above, all of these elements have weight $W=1$,
proving the claim about $\HH^{W,-2}_\Gamma(\Alg^0_-)$.

Differentials  of terms with $k=-1$ are specified by
Equations~\eqref{eq:diff1} and~\eqref{eq:diff2}
(along with the usual symmetry of adding $i\in \Zmod{4}$ to all the subscripts
in Equation~\eqref{eq:diff1}).
It follows that $H^{>0,-1}({\mathfrak C})\cong \FieldSup{2}$, generated by
$\rho_{1234}[\iota_1]+\rho_{2341}[\iota_0]+\rho_{3412}[\iota_1]+\rho_{4123}[\iota_0]$ and
$U\iota_0[\iota_1]+U\iota_1[\iota_0]$, as desired.
\end{proof}

\begin{proof}[Proof of Theorem~\ref{thm:MAlg-unique}]
  Throughout this proof, by deformation we will mean a $\Gamma$-graded deformation.

  Since the trivial deformation ($\mu_{n}^w=0$ for all $w>0$) defines
  a weighted $\Ainf$-algebra, the class $\obstr^1$ must vanish. Thus,
  by Proposition~\ref{prop:deform-Ainf-alg}, the isomorphism classes
  of 1-truncated weighted deformations correspond to
  $\HH^{1,-1}_{\Gamma}(\UnDefAlg)$. By Proposition~\ref{prop:UnDefHoch}, this
  group is isomorphic to $\FieldSup{2}$, and there is a unique generator
  satisfying Condition~(\ref{item:MAlg-m1}) of the theorem. By
  Corollary~\ref{cor:deform-Ainf-alg-G} and the fact that
  $\HH^{W,-2}_\Gamma=\HH^{W,-1}_\Gamma=0$ for $W>1$ from
  Proposition~\ref{prop:UnDefHoch}, this deformation extends uniquely
  to a weighted $\Ainf$-algebra structure.
\end{proof}

Finally, we note that Theorems~\ref{thm:UnDefAlg-unique}
and~\ref{thm:MAlg-unique} also hold with the refined grading, by
$\smallGroup(T^2)$:
\begin{corollary}\label{cor:refined-unique}
  Up to isomorphism, there is a unique $\Ainf$-deformation of
  $\AsUnDefAlg$ over $\Field[U]$ satisfying the following conditions:
  \begin{enumerate}
  \item\label{item:Uuniqe-graded-refined} The deformation is
    $\smallGroup(T^2)\times \ZZ$-graded, where the gradings of
    the chords $\rho_i$ is defined by
    $\gamma(\rho_i)=\gr_\psi(\rho_i)\times \wingr(\rho_i)$.
  \item\label{item:Uuniqe-mu-4-refined} The operations satisfy
    $\mu_4(\rho_4,\rho_3,\rho_2,\rho_1)=U\iota_1$ and
    $\mu_4(\rho_3,\rho_2,\rho_1,\rho_4)=U\iota_0$.
  \end{enumerate}

  Similarly, up to isomorphism, there is a unique weighted deformation
  $\MAlg$ of $\UnDefAlg$ such that:
  \begin{enumerate}
  \item $\MAlg$ is
    $\smallGroup(T^2)\times \ZZ$-graded and
  \item\label{item:MAlg-m1-refined} $\mu^1_0=\rho_{1234}+\rho_{2341}+\rho_{3412}+\rho_{4123}$.
  \end{enumerate}
\end{corollary}
\begin{proof}
  For the unweighted case, by Lemma~\ref{lem:unrefine}, any such
  deformation of $\AsUnDefAlg$ induces a $\smallGroup\times\ZZ$-graded
  deformation.  So, the result follows from
  Theorem~\ref{thm:UnDefAlg-unique}. The weighted statement is
  similar, but using Theorem~\ref{thm:MAlg-unique}.
\end{proof}

\begin{remark}
  \label{rem:WhyWeHaveDef}
  The non-trivial weighted deformation from
  Theorem~\ref{thm:MAlg-unique} is the one which appears in bordered
  Floer homology. Weighted actions on the modules are counts of rigid
  pseudo-holomorphic curves, where the weights signify the total
  number of Reeb orbits on the curve.  In one-dimensional families,
  these Reeb orbits can wander off on the $\alpha$-side, limiting to
  Reeb orbits either on boundary degenerations or curves at east
  infinity. (See~\cite[Section~\ref{TM:sec:moduli}]{LOT:torus-mod}.)
  The term $\mu^1_0$ above comes from counting certain rigid curves at
  east infinity; see for
  example~\cite[Figure~\ref{TM:fig:orbit-curve-eg}]{LOT:torus-mod}.
\end{remark}


%% file: Fukaya.tex
\section{The torus algebra and the wrapped Fukaya category}
\label{sec:Fukaya}
Let $\UnDefAlg$ denote the undeformed $\Ainf$-algebra of $\MAlg$. In
this section we give an alternate interpretation of $\UnDefAlg$ in
terms of the wrapped Fukaya category of the punctured torus.

Recall~\cite{AbouzaidSeidel10:wrapped,Auroux:beginners} that the \emph{wrapped Fukaya
  category} $\WFuk(M)$ of a symplectic manifold $(M,\omega)$ with convex, conical
ends has objects Lagrangians $L\subset (M,\omega)$ which are conical
at infinity. Typically, these Lagrangians are also equipped with brane
data---gradings and $\Pin$-structures---but we will not use brane
data in this section. We will also later restrict to simply-connected
Lagrangians. To define the morphism spaces, one chooses a
Hamiltonian function $H$ so that on the conical end $[1,\infty)\times
Z$, $H(r,x)=r^2$. If $\phi^1$ denotes the time-1 flow of $H$
then $\Hom(L_0,L_1)=\CF(\phi^1(L_0),L_1)$. The $\Ainf$-composition
map
\[
  \Hom(L_0,L_1)\otimes\cdots\otimes\Hom(L_{n-1},L_n)\to \Hom(L_0,L_n)
\]
counts holomorphic polygons
with Maslov index $2-n$ with boundary on
\[
  \bigl(\phi^n(L_0),\phi^{n-1}(L_1),\phi^{n-2}(L_2),\dots,\phi^1(L_{n-1}),L_n\bigr)
\]
(in counterclockwise order)
to obtain an element of $\CF(\phi^n(L_0),L_n)$ and then uses a
rescaling trick and a continuation map to map $\CF(\phi^n(L_0),L_n)$
to $\CF(\phi^1(L_0),L_n)=\Hom(L_0,L_n)$ via a
quasi-isomorphism.

(We have followed Auroux's exposition; Abouzaid-Seidel's original
definition is more algebraic. Note that we are not using composition
order, but rather the order of ``morphisms'' which is more natural for
an algebra. See Figure~\ref{fig:polygon-or} for our convention on the
orientations of polygons.)

\begin{figure}
  \centering
  \includegraphics{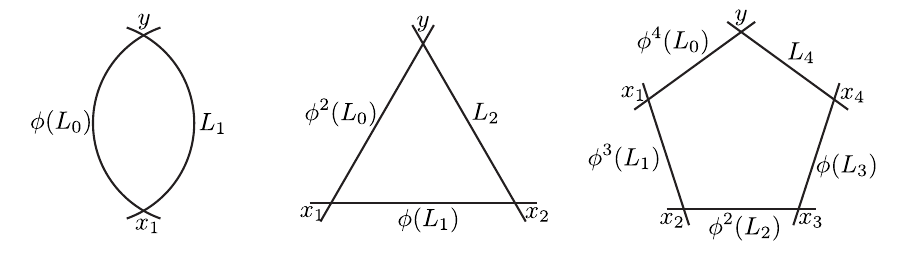}
  \caption{\textbf{Orientations of polygons.} Left: the model for a
    bigon contributing a term $y$ in $\mu_1(x)$. Center: the model for
    a triangle contributing a term $y$ in $\mu_2(x_1,x_2)$. Right: the
    model for a pentagon contributing a term $y$ in
    $\mu_4(x_1,x_2,x_3,x_4)$. This convention is consistent with
    Auroux's~\cite{Auroux:beginners} but is the opposite
    convention from~\cite{OS04:HolomorphicDisks}.}
  \label{fig:polygon-or}
\end{figure}

We are interested in the case that $M=T^2\setminus\{p\}$ is a
punctured torus, where the holomorphic curve counts are combinatorial
(by the Riemann mapping theorem) and the rescaling map
$\CF(\phi^n(L_0),L_n)\to
\CF(\phi^1(L_0),L_n)=\Hom(L_0,L_n)$ is induced by an obvious surface
diffeomorphism.

Let $\alpha_1$ and $\alpha_2$ be the arcs from Figure~\ref{fig:torus},
where the
puncture $p$ is $\alpha_1\cap\alpha_2$.  The following is a special case of a
result of Auroux~\cite[Theorem 4.7]{Auroux:beginners}:
\begin{theorem}\cite{Auroux10:ICM,Auroux10:Bordered,Auroux:beginners}
  The wrapped Fukaya category of the punctured torus,
  $\WFuk(T^2\setminus\{p\})$, is quasi-equivalent to
  $\End_{\WFuk}(\alpha_1\oplus\alpha_2)$.
\end{theorem}

We consider a relative version of the wrapped Fukaya category of
$T^2\setminus \{p\}$ by fixing a point $z\in T^2\setminus\{p\}$ not in the
conical end, considering Lagrangians disjoint from $z$, and a Hamiltonian
function $H$ which vanishes in a neighborhood of $z$. We can then enhance the
Floer complexes $\CF(\phi^1(L_0),L_n)$ to be modules over $\Field[U]$ by
counting holomorphic disks $u\co \bD^2\to T^2\setminus\{p\}$ with weight
$U^{n_z(u)}$, where $n_z(u)$ is the multiplicity of $u$ at $z$. The polygon
counts used to define the higher compositions on the wrapped Fukaya category
inherit a similar weight. Let $\WFuk_z(T^2\setminus\{p\})$ denote this
relative Fukaya category. (For a more general discussion of relative Fukaya
categories, see~\cite{Sheridan:CY-hyp}.)

With these definitions in hand, the goal of this section is to prove:
\begin{theorem}\label{thm:UnDef-is-Fuk}
  There is an $\Ainf$-quasi-isomorphism
  $
  \UnDefAlg\simeq \End_{\WFuk_z}(\alpha_1\oplus\alpha_2).
  $
\end{theorem}

\begin{corollary}
  There is an (ungraded) $\Ainf$-quasi-isomorphism
  $
  \UnDefAlg|_{U=1}\simeq \End_{\WFuk}(\alpha_1\oplus\alpha_2) \simeq \WFuk(T^2\setminus\{p\}).
  $
\end{corollary}
(In fact, the proof of Theorem~\ref{thm:UnDef-is-Fuk} also specifies a
grading on $\UnDefAlg|_{U=1}$ so that this becomes a graded
quasi-isomorphism.)

We will deduce Theorem~\ref{thm:UnDef-is-Fuk} from the uniqueness
theorem for $\UnDefAlg$, Theorem~\ref{thm:UnDefAlg-unique}. Since
gradings play an essential role in the uniqueness theorem, we
introduce a corresponding notion of gradings on
the wrapped Fukaya
category in Section~\ref{sec:Fuk-gr}, before returning to the model
computations needed to prove Theorem~\ref{thm:UnDef-is-Fuk} in
Section~\ref{sec:Fuk-proof}.

\subsection{Gradings on the wrapped Fukaya category}\label{sec:Fuk-gr}
In this section, we define a grading on the wrapped Fukaya
category using (a variant of) the notions of anchored
Lagrangians~\cite{FOOO10:anchored,Sheridan:CY-hyp}.  Fix a symplectic
manifold $(M,\omega)$ with $\pi_2(M)=0$; for instance, $(M,\omega)=(T^2\setminus
\{p\},\text{Area})$. Let $\LagGr{M}$ denote the space of Lagrangian
subspaces of $TM$. The space $\LagGr{M}$ fibers over $M$, with fiber
diffeomorphic to
the Grassmanian $\LagGras$ of Lagrangian planes in $\RR^{2n}$. Any
Lagrangian submanifold $L\subset M$ has a canonical lift $\tilde{L}\subset
\LagGr{M}$.

Recall that $\LagGras$ is path connected and
$\pi_1(\LagGras)\cong\ZZ$; the isomorphism is given by the Maslov
index. Thus, since $\pi_2(M)=0$, for basepoints $b \in M$ and
$\tilde{b} \in \LagGr{M}$, we have an exact sequence
\[
0\to \ZZ=\pi_1(\LagGras)\to \pi_1(\LagGr{M},\tilde{b})\to \pi_1(M,b)\to 0.
\]
\begin{lemma}\label{lem:act-trivial}
  The action of $\pi_1(M,b)$ on $\pi_1(\LagGras)$ is trivial.
\end{lemma}
\begin{proof}
  Recall that the action of $\gamma\in \pi_1(M,b)$ on
  $\eta\in\pi_1(\LagGras(T_bM),\tilde{b})$ is given by applying the homotopy
  lifting property to the diagram
  \[
    \xymatrix{
      [0,1]\times\{0\} \ar[r]^-{\eta} \ar@{^(->}[d]& \LagGr{M}\ar[d]\\
      [0,1]\times[0,1]\ar[r]_-{(s,t)\mapsto \gamma(t)} \ar@{-->}[ur] & M.
    }
  \]
  and restricting the dashed arrow to $[0,1]\times\{1\}$. (We may be
  negligent about basepoints since $\pi_1(\LagGras)$ is abelian.)
  Since $\pi_1(\LagGras)\cong\ZZ$ and $\gamma$ acts by a group automorphism on $\pi_1(\LagGras)$, the induced element
  $\gamma\cdot\eta$ has the form $\pm\eta$. So, it suffices to show that for each $\gamma$
  there is some nontrivial~$\eta$ with $\gamma\cdot\eta=\eta$.
  
  Fix an almost complex structure $J$ on $M$ compatible with $\omega$, making
  $TM$ into a Hermitian vector bundle. Consequently, $TM$ inherits an action of
  $S^1$ which takes Lagrangian subspaces to Lagrangian subspaces. Hence, we have
  an action $\cdot\co S^1\times \LagGr{M}\to \LagGr{M}$. Observe that the image
  of $\tilde{b}$ under the action of $S^1$ is $n=\dim(M)/2$ times a generator of
  $\pi_1(\LagGras(T_bM),\tilde{b})\cong\ZZ$. Let $\eta$ be this loop in
  $\pi_1(\LagGras(T_bM),\tilde{b})$. Then we can take the dashed arrow to be the
  map $(s,t)\mapsto e^{2\pi i t}\cdot\gamma(s)$, and the restriction to
  $\{1\}\times[0,1]$ is again $\eta$.
\end{proof}

Consequently, $\pi_1(\LagGr{M},\tilde{b})$ is a $\ZZ$ central extension of
$\pi_1(M,b)$.
Let $\lambda\in \pi_1(\LagGr{M},\tilde{b})$ be the image of the positive
generator of $\pi_1(\LagGras)$ (the one with $\mu=1$ or, equivalently, the direction induced by an almost complex structure; in the one-dimensional case, this is counterclockwise rotation).

Given a pair of connected Lagrangians $L,L'\subset M$, let $\GrSet(L,L')$ be the
set of homotopy classes of maps
$([0,1],\{0\},\{1\})\to(\LagGr{M},\tilde{L},\tilde{L}')$, i.e., the homotopy classes of
paths from $\tilde{L}$ to $\tilde{L}'$. The group $\pi_1(\LagGras)=\ZZ$ acts on
$\GrSet(L,L')$ by concatenating with loops in the fiber. In particular, we can
view $\lambda$ as an element of $\GrSet(L,L)$.

The following is a simple extension of Definition~\ref{def:groupoid-gr} to
$\Ainf$-categories (rather than algebras):
\begin{definition}\label{def:groupoid-grading}
  Given a groupoid $G$ with a central element $\lambda_s$ or $\lambda$
  (Definition~\ref{def:groupoid-central})  and an $\Ainf$-category
  $\Cat$, with composition
  operations denoted $\mu_n$, a \emph{grading of $\Cat$ by $G$} consists of:
  \begin{itemize}
  \item For each object $L\in\ob(\Cat)$, an object $s(L)\in\ob(G)$.
  \item For each pair of objects $L,L'$ of $\Cat$, a decomposition
    \[
      \Hom(L,L')=\bigoplus_{\gamma\in\Hom_G(s(L),s(L'))}\Hom(L,L';\gamma).
    \]
    An element
    $x\in\Hom_G(s(L),s(L');\gamma)$ is called \emph{homogeneous}, and we write
    $\gr(x)=\gamma$.
  \end{itemize}
  These data are required to satisfy the condition that for homogeneous elements
  $x_i\in\Hom(L_i,L_{i+1})$, $i=1,\dots,n$, the element
  $\mu_n(x_{1},x_{2},\dots,x_n)$ is homogeneous and
  \begin{equation}\label{eq:groupoid-gr-compat}
    \gr\bigl(\mu_n(x_{1},x_{2},\dots,x_n)\bigr)=\lambda^{n-2}\gr(x_{1})\gr(x_{2})\cdots\gr(x_n).
  \end{equation}
\end{definition}

If the Lagrangian submanifolds $L$ are all simply-connected then we
can form a groupoid $\GrGroupoid$
whose objects are Lagrangians in $M$, 
whose morphism sets are $\Hom(L,L')=\GrSet(L,L')$, and 
whose composition 
$\Hom(L,L')\times \Hom(L',L'')\to\Hom(L,L'')$ sends $(\gamma,\gamma')$ to the
concatenation $\gamma*\eta*\gamma'$, where $\eta$ is a path in $L'$ from
$\gamma(1)$ to $\gamma'(0)$. This concatenation is well-defined because $L'$ is
simply-connected. For any Lagrangian $L$, $\GrSet(L,L)\cong \pi_1(\LagGr{M})$
which, in particular, contains an element $\lambda_L$, and these elements
$\lambda$ are central in the sense described in
Definition~\ref{def:groupoid-grading}.

The grading on the Fukaya category by $\GrGroupoid$ is given as follows. Consider the following two kinds of paths in the Lagrangian Grassmanian: 
\begin{itemize}
\item
  Given Lagrangians $L$ and $L'$ with $L\pitchfork L'$ and $x\in L\cap L'$
  let $\gpath{x}{L}{L'}$ be the standard path in $\LagGras(T_xM)$ from $T_xL$ to
  $T_xL'$, corresponding to moving in the negative direction from $T_xL$ and
  remaining transverse to $T_xL$ for all positive time. (``Negative'' means opposite
  the direction induced by an almost complex structure, that is, negative means moving from $Jv$ towards
  $v$; compare~\cite{RobbinSalamon93:Maslov}. In the one-dimensional case, this is clockwise in the plane.) 
\item
  Given $m,n\in\RR$ and $x\in \phi^m(L)$ let $\hpath{x}{L}{m}{n}$ be the
  path in $\LagGr{M}$ from $T_x\phi^m(L)$ to $T_{\phi^{n-m}(x)}\phi^n(L)$
  induced by the Hamiltonian isotopy of $L$. (That is, the Hamiltonian induces a
  path of Lagrangians, hence a path in the Lagrangian Grassmanian bundle.)
\end{itemize}
Then, given Lagrangians $L_0$ and $L_1$ with $\phi^1(L_0)\pitchfork L_1$ and
an element $x\in \phi^1(L_0)\cap L_1$ the element $\gr(x)\in \GrSet(L_0,L_1)$
is the concatenation
\begin{equation}\label{eq:Fukaya-grading-1}
  \gr(x)=\hpath{\phi^{-1}(x)}{L_0}{0}{1}*\gpath{x}{\phi^1(L_0)}{L_1}
\end{equation}
of the path from $T_{\phi^{-1}(x)}L_0$ to $T_x\phi^1(L_0)$ induced by the
Hamiltonian isotopy and the positive path in $T_xM$ from $T_x\phi^1(L_0)$
to $T_xL_1$.

It is sometimes convenient to view all the elements $\gr(x)$ as having
the same endpoints, by
defining some additional paths:
\begin{itemize}
\item Given $x,y\in L$, let $\epath{x}{y}{L}$ be a path in $\tilde{L}$ from $T_xL$ to $T_yL$.
\end{itemize}
Then if we choose a point $q_0\in L_0$ and $q_1\in L_1$ we can also
write
\begin{equation}\label{eq:Fukaya-grading-2}
  \gr(x)=\epath{q_0}{\phi^{-1}(x)}{L_0}*\hpath{\phi^{-1}(x)}{L_0}{0}{1}*\gpath{x}{\phi^1(L_0)}{L_1}*\epath{x}{q_1}{L_1}.
\end{equation}

\begin{proposition}\label{prop:w-fuk-groupoid}
  Equation~\eqref{eq:Fukaya-grading-1} or~\eqref{eq:Fukaya-grading-2}
  defines a grading of the wrapped Fukaya category $\WFuk(M)$ of
  simply-connected
  Lagrangians in $M$ by the groupoid $\GrGroupoid$.
\end{proposition}
\begin{proof}
  We need to check the compatibility
  condition~\eqref{eq:groupoid-gr-compat}. The idea is that if $x_0$
  appears as a term in
  $\mu_n(x_1,\cdots, x_n)$ then there
  is a holomorphic disk connecting $\phi^{n-1}(x_1)$,
  $\phi^{n-2}(x_2),\dots,x_n$ and $x_0$ with Maslov index $2-n$.
  We will use the following lemma, which relates the gradings on
  points in $\phi^n(L_n) \cap L_0$ and $\phi^1(L_n) \cap L_0$, so the
  product of the gradings of the $x_i$ is equal to $\lambda^{n-2}$
  times the grading of $x_0$.
  
  \begin{lemma}\label{lem:dilate-grading}
    Let $x\in \phi^1(L_0)\cap L_n$ and let $x'$ be a point in
    $\phi^n(L_0)\cap L_n$ which corresponds to $x$ under dilation.
    Let $q_0\in L_0$ and $q_n\in L_n$ be any points. Then we have the
    following equality in $\GrSet(L_0,L_n)$:
    \[
      \epath{q_0}{\phi^{-1}(x)}{L_0}*\hpath{\phi^{-1}(x)}{L}{0}{1}*\gpath{x}{\phi^1(L_0)}{L_n}*\epath{x}{q_n}{L_n} =
      \epath{q_0}{\phi^{-n}(x')}{L_0}*\hpath{\phi^{-n}(x')}{L}{0}{n}*\gpath{x'}{\phi^n(L_0)}{L_n}*\epath{x'}{q_n}{L_n}.
    \]
  \end{lemma}
  \begin{proof}
    The statement is clearly independent of the choice of~$H$
    (defining~$\phi$),
    $q_0$, and $q_n$ and
    of compactly supported Hamiltonian isotopies of $L_0$ and~$L_n$.
    So, let $q_0=q_n\in L_0\pitchfork L_n$ outside the
    conical end and choose $H$ to vanish in a neighborhood of $q_0$. Then the
    two loops are homotopic to
    \[
      \epath{q_0}{x}{\phi^1(L_0)}*\gpath{x}{\phi^1(L_0)}{L_n}*\epath{x}{q_0}{L_n}
      \qquad\text{and}\qquad
      \epath{q_0}{x'}{\phi^n(L_0)}*\gpath{x'}{\phi^n(L_0)}{L_n}*\epath{x'}{q_0}{L_n}.
    \]
    So, we want to show that the loop
    \begin{multline*}
      \epath{q_0}{x}{\phi^1(L_0)}
      *\gpath{x}{\phi^1(L_0)}{L_n}*\epath{x}{q_0}{L_n}
        *\epath{q_0}{x'}{L_n}*\left(\gpath{x'}{\phi^n(L_0)}{L_n}\right)^{-1}*\epath{x'}{q_0}{\phi^n(L_0)}\\
      =
        \epath{q_0}{x}{\phi^1(L_0)}*\gpath{x}{\phi^1(L_0)}{L_n}\epath{x}{x'}{L_n}*\left(\gpath{x'}{\phi^n(L_0)}{L_n}\right)^{-1}*\epath{x'}{q_0}{\phi^n(L_0)}
    \end{multline*}
    is nullhomotopic. Applying the homotopy given by dilation (with varying dilation parameter) to the last three terms, this loop becomes
    \[
      \epath{q_0}{x}{\phi^1(L_0)}*\gpath{x}{\phi^1(L_0)}{L_n}*\epath{x}{x}{L_n}*\left(\gpath{x}{\phi^1(L_0)}{L_n}\right)^{-1}*\epath{x}{q_0}{\phi^1(L_0)}
    \]
    which is nullhomotopic, as desired.
  \end{proof}
  
  Returning to the proof of Proposition~\ref{prop:w-fuk-groupoid}, for $i=1,\dots,n$ let
  $x_i\in \phi^1(L_{i-1})\cap L_i$ and let
  $x_0\in \phi^1(L_0)\cap L_n$, and suppose that $x_0$ appears as a term in
  $\mu_n(x_1,\cdots, x_n)$. Since this composition map counts
  holomorphic disks with Maslov index $2-n$, we have
  \begin{equation}\label{eq:mu-is-2-n}
  \begin{split}
    \epath{x_0'}{\phi^{n-1}(x_1)}{\phi^n(L_0)}*
    \gpath{\phi^{n-1}(x_1)}{\phi^n(L_0)}{\phi^{n-1}(L_1)}
    &*\epath{\phi^{n-1}(x_1)}{\phi^{n-2}(x_2)}{\phi^{n-1}(L_1)}
    *\gpath{\phi^{n-2}(x_2)}{\phi^{n-1}(L_1)}{\phi^{n-2}(L_2)}
    *\epath{\phi^{n-2}(x_2)}{\phi^{n-3}(x_3)}{\phi^{n-2}(L_2)}
    \\
    &*\cdots
      *\gpath{x_n}{\phi^{1}(L_{n-1})}{L_n}
      *\epath{x_n}{x_0'}{L_n}
      *(\gpath{x_0'}{\phi^n(L_0)}{L_n})^{-1}
      =\lambda^{2-n}.
    \end{split}
  \end{equation}
  See, for instance,~\cite[Definition 1.8 and Formula (2.5)]{Auroux:beginners},
  and note in particular that the inverse of the negative path
  $\gpath{x_0}{\phi^n(L_0)}{L_n}$ at $x_0$ from $\phi^n(L_0)$ to $L_n$ is the
  positive path at $x_0$ from $L_n$ to $L_0$.

  On the other hand,
  \begin{align*}
    \gr(x_1)&\gr(x_2)\cdots\gr(x_n)\gr(x_0)^{-1}\\
    &=
      \hpath{\phi^{-1}(x_1)}{L_0}{0}{1}
      *\gpath{x_1}{\phi^1(L_0)}{L_1}
      *\epath{x_1}{\phi^{-1}(x_2)}{L_1}
      *\hpath{\phi^{-1}(x_2)}{L_1}{0}{1}
        *\gpath{x_2}{\phi^1(L_1)}{L_2}
        *\epath{x_2}{\phi^{-1}(x_3)}{L_2}\\
        &\qquad*\cdots*\hpath{\phi^{-1}(x_n)}{L_{n-1}}{0}{1}
          *\gpath{x_n}{\phi^1(L_{n-1})}{L_n}
          *\epath{x_n}{x_0}{L_n}
    *\left(\hpath{\phi^{-1}(x_0)}{L_{0}}{0}{1}
    *\gpath{x_0}{\phi^1(L_{0})}{L_n}\right)^{-1}
                      \\
                                   &=
                                 \hpath{\phi^{-1}(x_1)}{L_0}{0}{1}
      *\gpath{x_1}{\phi^1(L_0)}{L_1}
      *\epath{x_1}{\phi^{-1}(x_2)}{L_1}
      *\hpath{\phi^{-1}(x_2)}{L_1}{0}{1}
        *\gpath{x_2}{\phi^1(L_1)}{L_2}
        *\epath{x_2}{\phi^{-1}(x_3)}{L_2}
        \\
        &\qquad*\cdots*\hpath{\phi^{-1}(x_n)}{L_{n-1}}{0}{1}
          *\gpath{x_n}{\phi^1(L_{n-1})}{L_n}
          *\epath{x_n}{x_0'}{L_n}
    *\left(\hpath{\phi^{-n}(x_0')}{L_{0}}{0}{n}
    *\gpath{x_0'}{\phi^n(L_{0})}{L_n}\right)^{-1}
  \end{align*}
  where the second equality uses
  Lemma~\ref{lem:dilate-grading}.
  Since Hamiltonian isotopies take standard paths associated to Lagrangian
  intersections to standard paths associated to Lagrangian intersections and
  paths in Lagrangians to paths in Lagrangians,
  we have, for instance,
  \[
    \gpath{x_1}{\phi^1(L_0)}{L_1} * \epath{x_1}{\phi^{-1}(x_2)}{L_1} *
       \hpath{\phi^{-1}(x_2)}{L_1}{0}{1}
       = \hpath{x_1}{L_0}{1}{2} * \gpath{\phi(x_1)}{\phi^2(L_0)}{\phi(L_1)}
         * \epath{\phi(x_1)}{x_2}{\phi(L_1)}.
  \]
  Consequently, commuting
  the copies of $\eta$ to the left gives
  \begin{align*}
    \gr(x_1)\cdots\gr(x_n)\gr(x_0)^{-1}
    &=\hpath{\phi^{-n}(x_0)}{L_{0}}{0}{n}
      *\gpath{\phi^{n-1}(x_1)}{\phi^n(L_0)}{\phi^{n-1}(L_1)}
    *\epath{\phi^{n-1}(x_1)}{\phi^{n-2}(x_2)}{\phi^{n-1}(L_1)}\\
    &\qquad
      *\cdots* \gpath{x_n}{\phi^1(L_{n-1})}{L_n}
          *\epath{x_n}{x_0}{L_n}
      *\bigl(\gpath{x_0}{\phi^n(L_{0})}{L_n}\bigr)^{-1}
      *\bigl(\hpath{\phi^{-n}(x_0)}{L_{0}}{0}{n}\bigr)^{-1}.
  \end{align*}
  By Formula~\eqref{eq:mu-is-2-n}, since $\lambda$ is central, this is
  equal to $\lambda^{2-n}\in\GrSet(L_0,L_0)$, as desired.
\end{proof}

We can reduce this groupoid grading to a grading by a group,
as in Section~\ref{sec:gr-refine}, as follows. Fix a
basepoint $\tilde{b}\in \LagGr{M}$. For each Lagrangian $L$ choose a point
$\tilde{b}_L\in \tilde{L}$ and a path
$\eta_L$ in $\LagGr{M}$ from $\tilde{b}$
to $\tilde{b}_L$, i.e., an \emph{anchor} for $L$. Let
$G=\pi_1(\LagGr{M},\tilde{b})$. For each pair $L_0,L_1$ there is an induced
identification $\GrSet(L_0,L_1)\cong G$ which sends a path $\gamma$ from
$\tilde{L}_0$ to $\tilde{L}_1$ to the concatenation
$\eta_{L_0}*\nu_0*\gamma*\bar{\nu}_1*\eta_{L_1}^{-1}$ where $\nu_i$ is a path in
$\tilde{L}_i$ from $\tilde{b}_{L_i}$ to $\gamma(i)$. Since $L_i$ is simply
connected, this is independent of the choice of paths $\nu_0$ and
$\nu_1$. Further, this construction defines a homomorphism of groupoids from
$\GrGroupoid$ to $G$ sending $\lambda$ to $\lambda$. Hence, it induces a grading
of $\WFuk(M)$ by $G$.

We are interested in the case that $M=T^2\setminus\{p\}$ and the
Lagrangians are $\alpha_1$ and $\alpha_2$. In this case, the grading
group $\pi_1(\LagGr{T^2\setminus\{p\}},\tilde{b})$ is a $\ZZ$
central extension of the free group $F_2$. Indeed, the circle bundle
\[
  \LagGras\to \LagGr{T^2\setminus\{p\}}\to T^2\setminus\{p\}
\]
is trivial, since any surface bundle over a punctured surface is, so
we have $\pi_1(\LagGr{T^2\setminus\{p\}},\tilde{b})\cong \ZZ\times
F_2$. This isomorphism is not canonical; to fix one, let
$a=S^1\times\{q\}$ and $b=\{q\}\times S^1$ in $S^1\times S^1=T^2$,
viewed as oriented loops so that $a\cdot b=1$. The
curves $a,b$ are Lagrangian, so have canonical lifts
$\tilde{a},\tilde{b}\subset\LagGr{T^2\setminus\{p\}}$. Then
$\{\lambda,\tilde{a},\tilde{b}\}$ is a set of generators for
$\pi_1(\LagGr{T^2\setminus\{p\},\tilde{b}})$, giving an isomorphism 
$\pi_1(\LagGr{T^2\setminus\{p\},\tilde{b}})\cong \ZZ\times F_2$.

If we impose the relation
$\tilde{a}\tilde{b}\tilde{a}^{-1}\tilde{b}^{-1}=\lambda^{-2}$ we obtain the
group~$\smallGroup$.
Hence, we have a surjection
$\pi_1(\LagGr{T^2\setminus\{p\}},\tilde{b})\to \smallGroup$, and we
can consider the Fukaya category of $T^2\setminus\{p\}$ graded by
$\smallGroup$.

We will also exploit one further \emph{winding number grading}, coming from the fact that we are
considering a relative Fukaya category.
Fix a path $\gamma_z$ from $z$ to $p$ so that:
\begin{itemize}
\item $\gamma_z$ approaches $\{p\}$ from the region labeled $\rho_4$ in Figure~\ref{fig:torus}.
\item $\gamma_z$ is disjoint from $\alpha_1\cup\alpha_2$.
\end{itemize}
(See, e.g., Figure~\ref{fig:Fuk-gr}.)
Choose the Hamiltonian $H$ so that $\phi^1$ has a single fixed point $\iota_0$
on $\alpha_1$ and a single fixed point $\iota_1$ on $\alpha_2$, and
$H$ is constant in a neighborhood of $z$. Given a point
$\rho\in \phi^1(\alpha_i)\cap\alpha_j$, let $\wingr(\rho)$ be the algebraic
intersection number of the path in $\phi^1(\alpha_i)$ from $\iota_i$ to $\rho$ with
$\gamma_z$. More generally, given
$\rho\in \phi^m(\alpha_i)\cap\phi^n(\alpha_j)$, $m>n$, let $\eta_\rho$ (respectively
$\nu_\rho$) be the path in $\phi^m(\alpha_i)$ from $\iota_i$ to $\rho$
(respectively in $\phi^n(\alpha_j)$ from $\iota_j$ to $\rho$), and set
\[
  \wingr(\rho)=\gamma_z\cdot(\eta_\rho*\overline{\nu_\rho}),
\]
the difference of algebraic intersection numbers. Finally, define $\wingr(U)=1$.

\begin{lemma}
  If $\rho\in \phi^m(\alpha_i)\cap\phi^n(\alpha_j)$ and $\rho'$ is the
  corresponding point in $\phi^{m-n}(\alpha_i)\cap\alpha_j$ then
  $\wingr(\rho)=\wingr(\rho')$. Similarly, if
  $\rho\in \phi^n(\alpha_i)\cap\alpha_j$ and $\rho'$ is the point in
  $\phi^1(\alpha_i)\cap\alpha_j$ which corresponds to $\rho$ under rescaling
  then $\wingr(\rho)=\wingr(\rho')$.  
  Finally, the composition maps in $\End_{\WFuk_z}(\alpha_1\oplus\alpha_2)$
  preserve $\wingr$.
\end{lemma}
\begin{proof}
  The first statement follows from the observation that the loops
  $(\eta_\rho*\overline{\nu_\rho})$ and
  $(\eta_{\rho'}*\overline{\nu_\rho'})$ are isotopic in $T^2\setminus\{p,z\}$.
  The second statement is clear. The third statement follows from the
  fact that the intersection number of the boundary of a polygon with $\gamma_z$
  is the same as the multiplicity with which the polygon covers $z$.
\end{proof}

\begin{figure}
  \centering
  \includegraphics{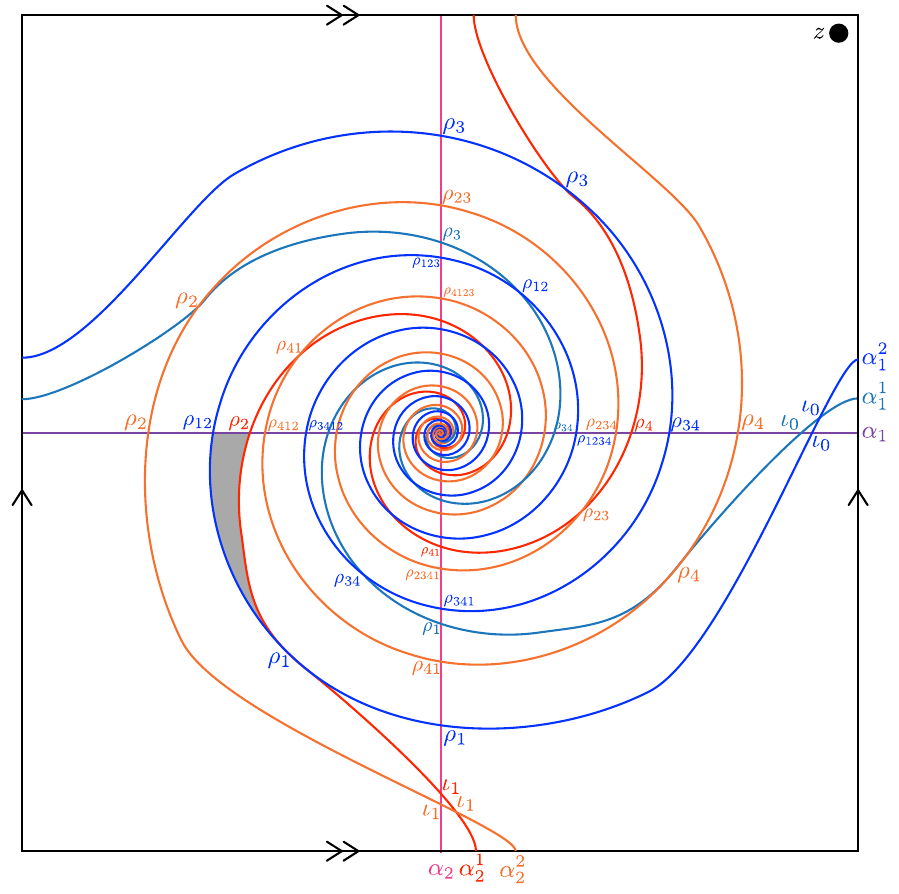}
  \caption{\textbf{Wrapped Fukaya category of the torus.}
    The puncture is in the middle of the picture.
    The shaded triangle corresponds to the operation
    $\mu_2(\rho_1,\rho_2)=\rho_{12}$. Here, we have perturbed the
    $\alpha$-curves slightly to resolve the triple intersections at
    the fixed points $\iota_0$ and $\iota_1$ (which are fixed points
    of $\phi$).}
  \label{fig:torus-fuk}
\end{figure}

\subsection{Model computations and the proof of Theorem~\ref{thm:UnDef-is-Fuk}}
\label{sec:Fuk-proof}
In this section, we prove Theorem~\ref{thm:UnDef-is-Fuk}, i.e., that
$\End_{\WFuk_z}(\alpha_1\oplus\alpha_2)$ is quasi-isomorphic to
$\UnDefAlg$. In fact, we will choose perturbations defining
$\End_{\WFuk_z}(\alpha_1\oplus\alpha_2)$ so that the two algebras are
isomorphic.

The identification of an $\FF_2$-basis for $\UnDefAlg$ and
$\End_{\WFuk_z}(\alpha_1\oplus\alpha_2)$ is shown in
Figure~\ref{fig:torus-fuk}.
In words, let $M = T^2 \setminus \{p\}$,
fix a cylindrical neighborhood $V$ of the
puncture, and choose $\alpha_1$ and $\alpha_2$ so that, in cylindrical
polar coordinates around the puncture,
$\alpha_1\cap \bdy V=\{(0,0),(0,\pi)\}$ and
$\alpha_2\cap \bdy V=\{(0,\pi/2),(0,3\pi/2)\}$.  Choose a perturbation
$\tilde{\alpha}_i$ of $\alpha_i$ so that the two intervals in
$\tilde{\alpha}_i\cap V$ are slightly counter-clockwise of
the $\alpha_i\cap V$ and so that
$\tilde{\alpha}_i$ intersects $\alpha_i$ in a single point.

Let $q_{0,0}^0=\alpha_1\cap\tilde{\alpha}_1$ and
$q_{1,1}^0=\alpha_2\cap \tilde{\alpha}_2$, and let
$\alpha_i^k=\phi^k(\tilde{\alpha}_i)$. Then, still using
cylindrical polar coordinates, we have
\begin{align*}
  \alpha_1\cap\alpha_1^1 &= \{(n-\epsilon,0)\mid n=1,2,\dots\}
                      \cup\{(n-\epsilon,\pi)\mid  n=1,2,\dots\}\cup\{q_{0,0}^0\}\\
  \alpha_2\cap\alpha_2^1 &= \{(n-\epsilon,\pi/2)\mid n=1,2,\dots\}
          \cup\{(n-\epsilon,3\pi/2)\mid  n=1,2,\dots\}\cup\{q_{1,1}^0\}\\
  \alpha_1\cap\alpha_2^1 &= \{(n+1/2,0)\mid n=0,1,\dots\}\cup
    \{(n+1/2,\pi)\mid n=0,1,\dots\}\\
  \alpha_2\cap\alpha_1^1 &= \{(n+1/2,\pi/2)\mid n=0,1,\dots\}\cup
    \{(n+1/2,3\pi/2)\mid n=0,1,\dots\},
\end{align*}
where the $\epsilon$s are small, unimportant positive real numbers.

The identification between these generators and $\UnDefAlg$ is given
by
\begin{align*}
  q_{0,0}^0&\leftrightarrow \iota_0 & q_{1,1}^0&\leftrightarrow \iota_1\\
  (n-\epsilon,0)&\leftrightarrow
    \begin{cases}
      \rho_{34}\rho_{1234}^{(n-1)/2} & n\text{ odd}\\
      \rho_{1234}^{n/2} & n \text{ even}
    \end{cases} &
    (n-\epsilon,\pi)&\leftrightarrow
    \begin{cases}
      \rho_{12}\rho_{3412}^{(n-1)/2} & n\text{ odd}\\
      \rho_{3412}^{n/2} & n \text{ even}
    \end{cases} \\
    (n-\epsilon,\pi/2) &\leftrightarrow
    \begin{cases}
      \rho_{23}\rho_{4123}^{(n-1)/2} & n\text{ odd}\\
      \rho_{4123}^{n/2} & n \text{ even}
    \end{cases} &
    (n-\epsilon,3\pi/2) &\leftrightarrow
    \begin{cases}
      \rho_{41}\rho_{2341}^{(n-1)/2} & n\text{ odd}\\
      \rho_{2341}^{n/2} & n \text{ even}
    \end{cases} \\
    (n+1/2,0) &\leftrightarrow
    \begin{cases}
      \rho_4(\rho_{1234})^{n/2} & n\text{ even}\\
      \rho_{234}(\rho_{2341})^{(n-1)/2} & n\text{ odd}
    \end{cases} &
    (n+1/2,\pi) & \leftrightarrow
    \begin{cases}
      \rho_2(\rho_{3412})^{n/2} & n\text{ even}\\
      \rho_{412}(\rho_{3412})^{(n-1)/2} & n\text{ odd}
    \end{cases}\\
    (n+1/2,\pi/2) &\leftrightarrow
    \begin{cases}
      \rho_3(\rho_{4123})^{n/2} & n\text{ even}\\
      \rho_{123}(\rho_{4123})^{(n-1)/2} & n\text{ odd}
    \end{cases} &
    (n+1/2,3\pi/2) &\leftrightarrow
    \begin{cases}
      \rho_1(\rho_{2341})^{n/2} & n\text{ even}\\
      \rho_{341}(\rho_{2341})^{(n-1)/2} & n\text{ odd}.
    \end{cases}
\end{align*}
%
We will henceforth abuse notation and use $\rho_I$ to refer to both
the element of $\UnDefAlg$ and the corresponding element of
$\End_{\WFuk_z}(\alpha_1\oplus\alpha_2)$.

\begin{lemma}\label{lem:Fuk-torus-grs}
  The above identification of basis elements intertwines the winding
  number gradings. Further, there is a homomorphism from
  $\ZZ\times F_2$ to $\smallGroup(T^2)$ so that if we let $\gr_F$
  denote the induced $\smallGroup(T^2)$-valued grading on the Fukaya
  category then the identification of basis elements intertwines the
  grading $\gr_F-(2\wingr;0,0)$ on the Fukaya category and the grading
  $\gr_\psi$ from Section~\ref{sec:gr-refine} on $\MAlg$.
\end{lemma}
\begin{proof}
  It is clear that the identification respects the winding number
  grading; see Figure~\ref{fig:Fuk-gr} on the left for the curve $\gamma_z$. 
  
  To specify the grading on the Fukaya category, choose anchors for
  the two Lagrangians as shown in Figure~\ref{fig:Fuk-gr}. There is
  then a loop of Lagrangian subspaces of the tangent space to $T^2$
  associated to each intersection point; an example is shown on the right in
  Figure~\ref{fig:Fuk-gr}.
  Trivialize the Lagrangian Grassmanian bundle using the
  section of lines at slope $\pi/4$ in the figure, so to extract an
  integer from the Maslov (fiber) component of the grading, one counts
  with sign how many times the loop of Lagrangian subspaces passes
  through slope $\pi/4$.

  We will construct the homomorphism $\ZZ\times F_2$ to
  $\smallGroup(T^2)$ below. To see that it intertwines the gradings of
  generators, it suffices to show that it intertwines the gradings of
  $\rho_1$, $\rho_2$, $\rho_3$, $\rho_4$, and $U$. 

  \begin{figure}
    \centering
    \includegraphics[scale=.5]{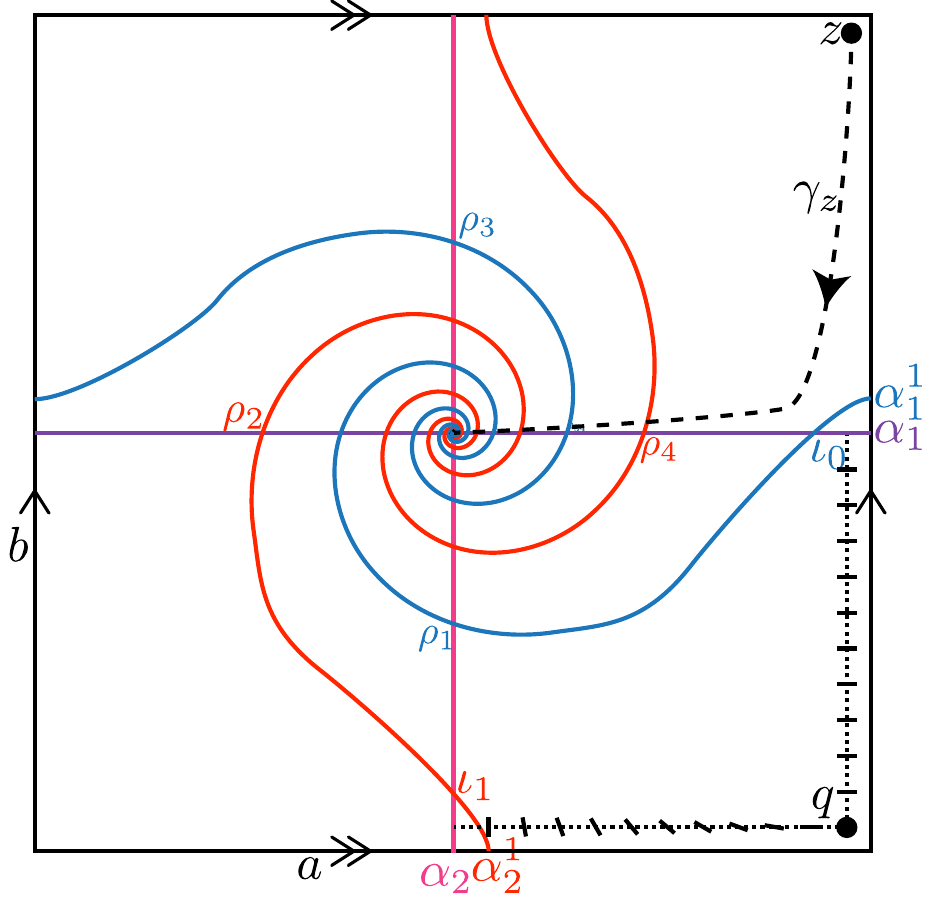} \quad 
    \includegraphics[scale=.5]{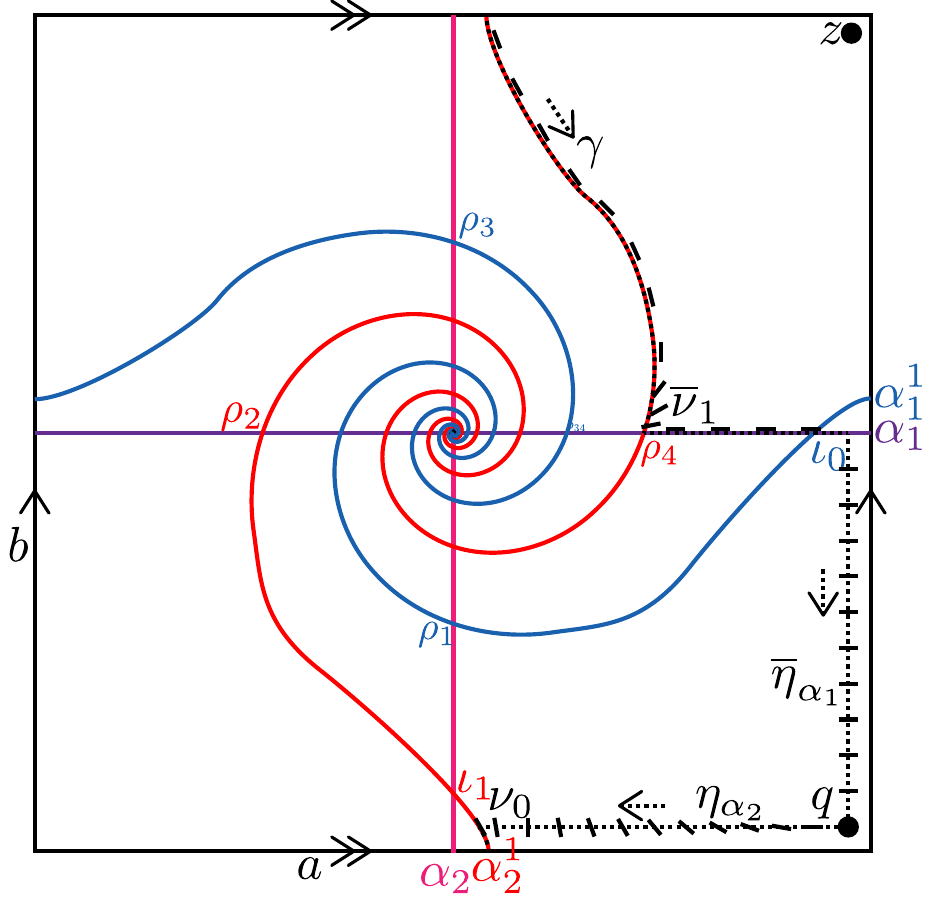}
    \caption{\textbf{Identifying the gradings on the Fukaya category
        of the torus.} Left: the base paths to the two Lagrangians
      used to identify the gradings (dotted). Right: the loop giving
      the grading of $\rho_4$. The short black segments indicate the
      lifts of the paths to the Lagrangian Grassmannian bundle. The
      paths $\nu_0$ and $\overline{\nu}_1$ correspond to the planes
      turning at the end; more precisely, these happen at the
      intersection points where $\gamma$ ends, but we have merged them
      slightly with $\eta_{\alpha_2}$ and $\gamma$ to make the drawing
      legible.}
    \label{fig:Fuk-gr}
  \end{figure}

  It is straightforward to compute that the $\ZZ\times F_2$-valued
  grading
  $\grt_F$ is given by
  \begin{align*}
    \grt_F(\rho_1)&=(0;\Id) &
    \grt_F(\rho_2)&=(-1;a^{-1})\\
    \grt_F(\rho_3)&=(0;b*a) &
    \grt_F(\rho_4)&=(-1;b^{-1}).
  \end{align*}
  (The final computation is illustrated on the right in 
  Figure~\ref{fig:Fuk-gr}.)  By construction, we have
  $\grt_F(U)=(0;0)$. The winding number grading is given by
  \begin{align*}
    \wingr(\rho_1)&=\wingr(\rho_2)=\wingr(\rho_3)=0 & \wingr(\rho_4)=\wingr(U)=1.
  \end{align*}

  Consider the homomorphism from
  $\ZZ\times F_2$ to $\smallGroup(T^2)$ defined by
  \begin{align*}
    (1;0)&\mapsto \lambda=(1;0,0) &
    (0;a)&\mapsto (-1/2;-1,0) & 
    (0;b)&\mapsto (-1/2;0,1).
  \end{align*}
  The image of $\grt_F$ under this map is
  \begin{align*}
    \gr_F(\rho_1) &= (0;0,0) &
    \gr_F(\rho_2) &= (-1;0,0)\cdot(1/2;1,0)=(-1/2;1,0) \\
    \gr_F(\rho_3)&= (-1/2;0,1)\cdot(-1/2;-1,0)=(0;-1,1) &
    \gr_F(\rho_4)&=(-1;0,0)\cdot(1/2;0,-1)=(-1/2;0,-1)\\
  \end{align*}

  To identify this with the grading on $\MAlg$ we subtract two times the
  winding number grading from the Maslov component of the grading, to get
  \begin{align*}
    \gr(\rho_1)&=(0;0,0) &
    \gr(\rho_2) &= (-1/2;1,0) \\
    \gr(\rho_3)&= (0;-1,1) &
    \gr(\rho_4)&=(-5/2;0,-1) &
    \gr(U)&=(-2;0,0).
  \end{align*}
  This agrees with the grading $\gr_\psi$ in
  Section~\ref{sec:gr-refine}, as claimed.
\end{proof}

\begin{figure}
  \centering
  \includegraphics[scale=.5]{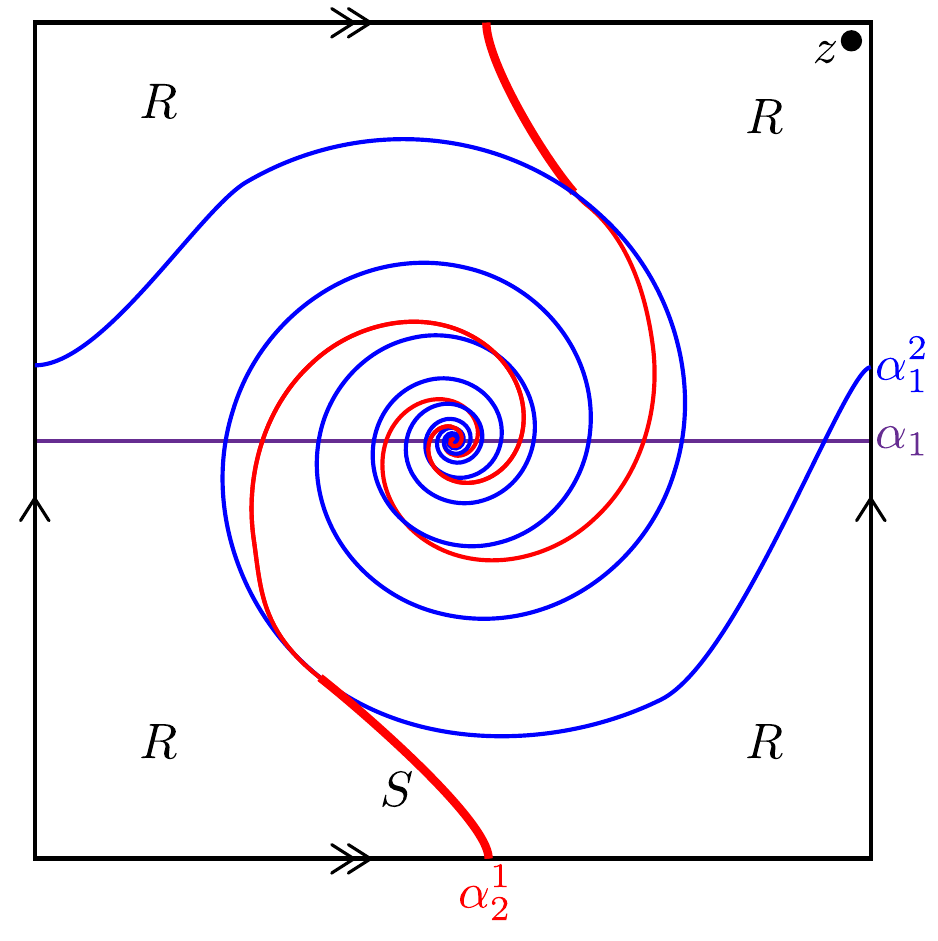}
  \caption{\textbf{No triangles cover $z$.} The segment $S$ on $\alpha_j$
    discussed in the proof of Lemma~\ref{lem:Fuk-torus-mu2} is in
    bold, and the region $R$ is indicated.}
  \label{fig:z-uncovered}
\end{figure}

\begin{lemma}\label{lem:Fuk-torus-mu2}
  The above identification of basis elements intertwines the
  operations $\mu_2$.
\end{lemma}
\begin{proof}
  This is similar to the computation for $T^*S^1$~\cite{Auroux:beginners}. For
  any of the operations $\rho_I\rho_J=\rho_{I\cup J}$ there is a
  unique small triangle
  contained in the conical end of $T^2\setminus\{p\}$ connecting
  $\phi^1(\rho_I)$, $\rho_J$, and the image of $\rho_{I\cup J}$ under
  rescaling. This is straightforward to verify, for instance, by considering
  lifts of the $\alpha$-arcs to the universal cover of the conical end of
  $T^2\setminus\{p\}$. We may take the lift of each component of
  $\alpha_i\cap [0,\infty)\times Z$ to be a straight line at slope $0$, and
  $\phi^1(\alpha_j)$ and $\phi^2(\alpha_i)$ to be straight lines at slopes
  $\pi$ and $2\pi$. The choice of a lift of $\alpha_i\cap [0,\infty)\times Z$
  and the points $\rho_{I\cup J}$ and $\rho_J$ fixes the other two lifts, and
  the three lines cut out a
  triangle. We just need to verify that the third vertex of the triangle is at
  the point $\phi^1(\rho_I)$. Writing $\widetilde{\phi}^1(x,y)=(x,y+\pi x)$ for the
  induced action on $[0,\infty)\times Z$, this is a straightforward computation.
  Triangles realizing the operations $\iota_i\rho_J=\rho_J$ and
  $\rho_J\iota_i=\rho_J$ can be found similarly.

  Further, we show that there is no immersed triangle
  in $T^2\setminus\{p\}$ with boundary on $\alpha_i$, $\phi^1(\alpha_j)$ and
  $\phi^2(\alpha_k)$ covers the region $R$ containing $z$. If all of
  the $\alpha_i$ that appear are the same, then $R$ is a non-trivial
  annulus that cannot be part of a disk. Otherwise, one $\alpha_i$
  appears one time; suppose it is $\phi^1(\alpha_2)$, as in
  Figure~\ref{fig:z-uncovered}, and consider the resulting segment~$S$,
  indicated there, which has $R$ on both sides. The
  boundary of such a triangle would then have to traverse $S$ an even
  number of times
  (since $R$ is on both sides of the segment). Since the boundary of
  the triangle is immersed, in fact $S$ must be traversed $0$
  times. Then the purported triangle contains a non-nullhomotopic
  loop, the
  preimage of the loop in $R$ passing through $S$ once, which is a contradiction.  Hence, terms of the
  form $U^m\rho_K$ for $m>0$ do not appear in $\rho_I\rho_J$. So, it follows
  from Lemma~\ref{lem:Fuk-torus-grs} that there are no more terms in $\rho_I\rho_J$.
\end{proof}

\begin{figure}
  \centering
  \includegraphics{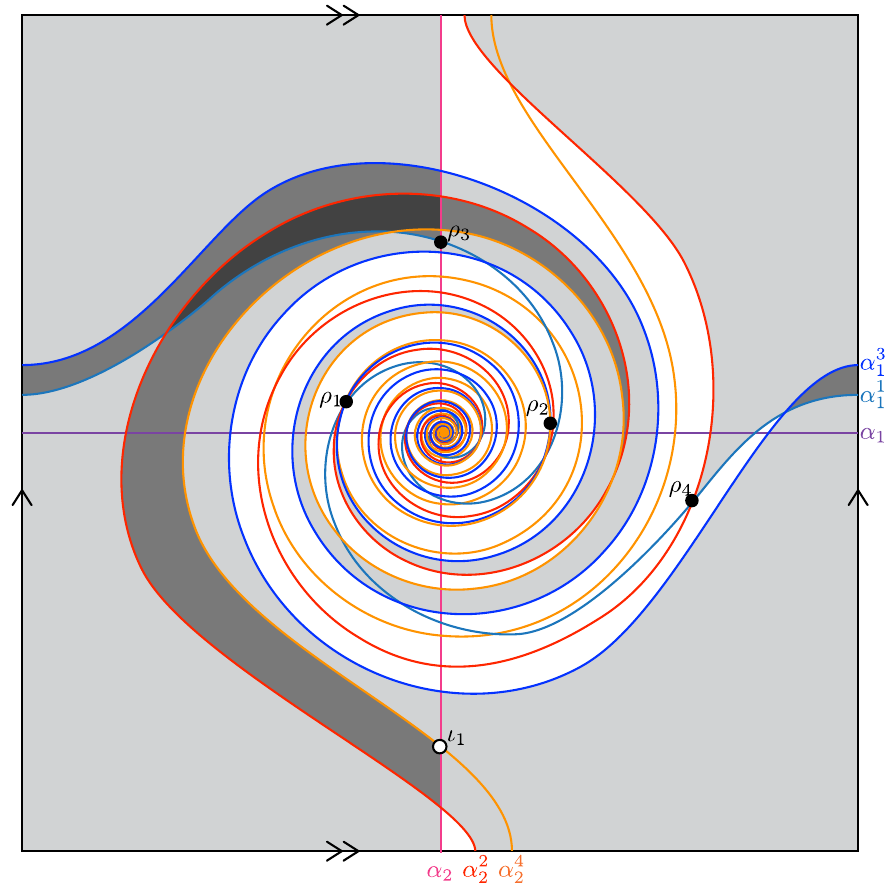}
  \caption{\textbf{The $\mu_4$-operation.} The shaded region is the
     image of the pentagon contributing
     $\mu_4(\rho_2,\rho_1,\rho_4,\rho_3)=U\iota_1$. Lightly shaded
     regions are covered once, medium-shading twice, and darkly-shaded
     regions are covered three times.}
  \label{fig:Fuk-mu4}
\end{figure}

\begin{lemma}\label{lem:Fuk-torus-mu4}
  The operation $\mu_4$ on $\End_{\WFuk_z}(\alpha_1\oplus\alpha_2)$
  satisfies
  \begin{align*}
    \mu_4(\rho_4,\rho_3,\rho_2,\rho_1)&=U\iota_1\\
    \mu_4(\rho_3,\rho_2,\rho_1,\rho_4)&=U\iota_0.
  \end{align*}
\end{lemma}
\begin{proof}
  By inspection, there is a holomorphic pentagon contributing each of
  these terms. See Figure~\ref{fig:Fuk-mu4}.
\end{proof}

\begin{proof}[Proof of Theorem~\ref{thm:UnDef-is-Fuk}]
  This is immediate from Corollary~\ref{cor:refined-unique} and
  Lemmas~\ref{lem:Fuk-torus-grs},~\ref{lem:Fuk-torus-mu2},
  and~\ref{lem:Fuk-torus-mu4}.
\end{proof}

\begin{remark}
  While $\UnDefAlg$ has a grading by $\ZZ\times F_2$, $\MAlg$ does not: with
  respect to this grading, the element
  $\mu^1_0=\rho_{1234}+\rho_{2341}+\rho_{3412}+\rho_{4123}$ does not have
  grading a central element of $\ZZ\times F_2$ (as required) and, in fact, is
  not even homogeneous. After quotienting by the relation
  $\tilde{a}\tilde{b}\tilde{a}^{-1}\tilde{b}^{-1}=\lambda^{-2}$, the grading of
  $\mu_0^1$ becomes well-defined and central.
\end{remark}


%% file: signs.tex
\newcommand\smu{\boldsymbol{\mu}}
\newcommand\smap{\boldsymbol{f}}
\newcommand\smapg{\boldsymbol{g}}
\section{Signs}
\label{sec:Signs}

In this section, we lift the discussion above from
$\FF_2$-coefficients to $\ZZ$-coefficients. Specifically, we lift the
weighted algebra $\MAlg$ to a weighted algebra over $\ZZ$, denoted
$\MAlgZ$, and extend the uniqueness theorems,
Theorems~\ref{thm:UnDefAlg-unique} and~\ref{thm:MAlg-unique}, to
$\ZZ$-coefficients. Note that for the case of $\HFa$, there is also
work of Knowles-Petkova lifting not just the bordered algebra but also
the modules over it (defined via nice diagrams) to $\ZZ$-coefficients~\cite{KP:integers}.

\subsection{Weighted algebras over \texorpdfstring{$\ZZ$}{the integers}}
So far, we have discussed group-graded algebras with characteristic $2$.  
We explain how to extend this to arbitrary characteristic.

In order to define $\Ainf$-algebras with characteristic $\neq 2$, 
the algebra needs to be equipped with a
$\Zmod{2}$-valued grading, which we denote by $\|\cdot \|$ here.
This grading is needed to  formulate
the Leibniz rule:
\[
  d(a\cdot b)=(da)\cdot b + (-1)^{\|a\|}a \cdot (db)
\]
or, more generally, the $n$-input $\Ainf$ relation
\begin{equation}
  \label{eq:AinfWithSign}
  \sum_{n=r+s+t} (-1)^{r+st} \mu_{r+1+t}\circ (\Id^{\otimes r} \otimes \mu_s \otimes \Id^{\otimes t})=0.
\end{equation}
In interpreting the above expression, we  follow conventions of Keller~\cite{AinftyAlg},
according to which
\begin{equation}
  \label{eq:SignConvention}
  (f\otimes g)(x\otimes y)=(-1)^{\|g\|\|x\|}f(x) \otimes g(y).
\end{equation}

The $\Ainf$-algebra homomorphism relation on $\{f_n\co A_+^{\otimes n}\to B_+\}$
takes the form 
\[
\sum _{n=r+s+t} (-1)^{r+st} f_{r+1+t}\circ (\Id^{\otimes r} \otimes \mu_s \otimes \Id^{\otimes t})= \sum_{n=i_1+\dots i_m} (-1)^{\sigma}\mu_m\circ(f_{i_1}\otimes\dots\otimes f_{i_m}),
\]
 where $\sigma=\sigma(i_1,\dots,i_m)$ is given by
 \begin{equation}
   \label{eq:DefSigma}
   \sigma=\sum_{k=1}^{m-1} (m-k)(i_k-1).
   \end{equation}

The $\Ainf$ relation has the obvious weighted generalization (for each $n,w$)
\begin{equation}
  \label{eq:WeightedAinf}
  \sum_{\substack{n=r+s+t\\w=u+v}} (-1)^{r+st} \mu^u_{r+1+t}\circ(\Id^{\otimes r} \otimes \mu^v_s \otimes \Id^{\otimes t})=0.
\end{equation}
(That is, the weight is treated as having even grading, so does not
contribute to the sign.) 
Similarly, the a weighted $\Ainf$ homomorphism relation (for each $n$
and $w$) is given by
\begin{equation}
  \label{eq:WeightedAinfHom}
\sum _{\substack{n=r+s+t\\w=w_1+w_2}} (-1)^{r+st} 
f^{w_1}_{r+1+t}\circ (\Id^{\otimes r} \otimes \mu_s^{w_2} \otimes \Id^{\otimes t})
\ = \hspace{-1.5em}
\sum_{\substack{n=i_1+\dots i_m\\
  w=w_0+w_1+\dots+w_m}}\hspace{-1.5em} (-1)^{\sigma}\mu_m^{w_0}\circ(f_{i_1}^{w_1}\otimes\dots\otimes f_{i_m}^{w_m}),
\end{equation}
with $\sigma$ as in Equation~\eqref{eq:DefSigma}.

Another convention we adhere to is the following: if $(C_*,\partial)$ and 
$(C'_*,\partial')$ are 
$\Zmod{2}$-graded complexes, then $\Mor(C_*,C'_*)$ is also $\Zmod{2}$-graded, with differential given by
\begin{equation}
  \label{eq:MorDifferential}
  \partial(f) = \partial\circ f - (-1)^{\|f\|} f\circ \partial.
\end{equation}

\subsection{Construction of the bordered algebras over \texorpdfstring{$\ZZ$}{the integers}}
Define $\AsUnDefAlgZ$ exactly as in Section~\ref{sec:AsUnDefAlg}, with
the understanding that the path algebra is taken over $\ZZ$, rather
than $\Field$. For the purposes of signs, we equip this algebra with
the $\Zmod{2}$ grading with the property that all of $\AsUnDefAlgZ$ is
supported in degree $0$. In other words, use the grading $\epsilon$
from Section~\ref{sec:mod-2-gr}.

Define the operations on $\MAlgZ$ as in Definition~\ref{def:MAlg}, by
counting tiling patterns; every tiling pattern counts
positively. Similarly, no signs appear in the definition of $\mu_0^1$
from Equation~\eqref{eq:mu-0-1}.
The operations $\mu^w_n$ are non-zero only when $n$ is even. This simplifies the sign
computations: the sign appearing in the weighted $\Ainf$ relation
(Equation~\eqref{eq:WeightedAinf}) is $(-1)^{r}$. Additionally, since
$\MAlgZ$ is supported in $\Zmod{2}$-grading $0$, the sign in
Equation~\eqref{eq:SignConvention} drops out.

A more careful look at the proof of Theorem~\ref{thm:AinftyAlgebra}
gives the following generalization:

\begin{theorem}
  \label{thm:AinftyAlgebraZ}
  The operations above give $\MAlgZ$ the
  structure of a weighted $\Ainf$-algebra over $\ZZ$.
\end{theorem}

\begin{proof}
  We must check that the cancelling terms in the proof of
  Theorem~\ref{thm:AinftyAlgebra} indeed come with opposite signs.

  We leave the case with $\leq 3$ inputs or $\leq 2$ inputs and weight
  $\geq 1$ to the reader.

  We consider the other cancellations in the order they are verified
  in Theorem~\ref{thm:AinftyAlgebra}, now with sign.  As in that proof,
  we can consider input sequences of basic algebra elements, with at
  most one element which is an idempotent (i.e., not Reeb-like); and
  the idempotent comes either as the first or last element in the
  sequence.  When the idempotent is the first element, we have terms
  of type $(2\wild-)$ and $(\wild2+)$. Evidently, the term of type $(2\wild-)$ has
  $r=1$ (and $t=0$), so it contributes to the $\Ainf$ relation with
  sign $-1$; while the term of type $(\wild2+)$ has $r=0$ (and $s=2$), so it
  contributes to the $\Ainf$ relation with sign $+1$. Thus these terms
  cancel with sign. Similarly, when the idempotent comes at the end,
  we obtain cancellation of terms of type $(2\wild+)$, which counts with sign $+1$ (as it has $r=0$ and even $s$),
  against terms of type $(\wild 2 -)$,
  which counts with sign $-1$ 
  (as it has  odd $r$ and $t=0$).

  Next, we consider the cancellation of terms as in
  Lemma~\ref{lem:ProductCancellations}. For the factorization of
  $\rho^i=\sigma_1\cdot \sigma_2$, we have a term of type $(\wild2\wild)$
  in the notation of the proof of Theorem~\ref{thm:AinftyAlgebra};
  or ${\mathfrak P}$, in the notation of Lemma~\ref{lem:ProductCancellations}. 
  The
  $\mu_2$ node has $r=i-1$. Consider the corresponding element of
  $\Graphs_1\cup{\mathfrak T}\cup{\mathfrak L}\cup{\mathfrak R}$ 
  under the correspondence from Lemma~\ref{lem:ProductCancellations}.
  If the term is of type ${\mathfrak T}$, the term has $r=i$.  
  If the terms is of type ${\mathfrak P}\cup{\mathfrak L}\cup{\mathfrak R}$
  we have a composition of two terms; and the higher one 
  can have $\rho_{i+1}$ as its first input, in which case $r=i$;
  or it can have $\rho_i$ as its last input, in which case $r=i-s$,
  where the second term has $s$ inputs.
  In any of these cases, by Lemma~\ref{lem:alg-even-operations},
  the composite contributes with a sign of $(-1)^i$.

  For the remaining cases---the cancellations coming from the
  bijection $M$ from the proof of Theorem~\ref{thm:AinftyAlgebra}--all
  cancellations are between a term of type $+$ (i.e., $(\wild0+)$,
  $(2\wild+)$, $(\wild2+)$, and $[\wild\wild+]$, where each $\wild\in\{L,C,R\}$) and a
  term of type $-$. By construction, terms of type $+$ contribute
  $+1$, and terms of type $-$ contribute $-1$, so the cancellation
  holds with signs, as well.
\end{proof}

\subsection{Section~\ref{sec:A-inf-deform} revisited}

There are two equivalent formulations of the $\Ainf$ relations, and we
will find it to convenient to move between them.  One formulation is
stated in Equation~\eqref{eq:AinfWithSign}. There is an alternative,
somewhat simpler formulation, as follows.

To set notation, let $A[1]$ denote $A$ with the following grading
shift: $A[1]_{g}=A_{g\cdot \lambda^{-1}}$. In particular, if $A$ is
$\ZZ$-graded, then $A[1]_n=A_{n-1}$. Thus, the identity map can be
viewed as a map $\sigma\co A\to A[1]$, of degree $+1$.  (Note
that we follow homological conventions, rather than the cohomological
conventions of~\cite{KellerRepTheory}.)

We now view the $\Ainf$ operations as maps $\smu_k\co A_+[1]^{\otimes
  k}\to A[1]$ of degree $-1$.
With this normalization, the $\Ainf$ relation takes the simpler form:
\begin{equation}
  \label{eq:AinfSigned2}
  \sum_{n=r+s+t} \smu_{r+1+t}\circ (\Id^{\otimes r} \otimes \smu_s
  \otimes \Id^{\otimes t})=0.
\end{equation}
(Our $\smu_n$ are the maps $b_n$
from~\cite{GetzlerJones,KellerRepTheory}.)
Formula~\eqref{eq:AinfSigned2} is equivalent to
Formula~\eqref{eq:AinfWithSign}, as follows.  The operations $\smu_n$
and $\mu_n$ are related by the following commutative diagram:
\[ \begin{CD}
  (A[1])^{\otimes n} @>{\smu_n}>> A[1] \\
  @AA{\sigma^{\otimes n}}A  @AA{\sigma}A \\
  A^{\otimes n} @>{\mu_n}>> A.
\end{CD}\]
Then, the relation $0=(\smu\circ \smu)_n\circ \sigma^{\otimes n}$ gives
\begin{align*}
  0 &=\sum_{\{r,s\mid r+1+s=n\}} \smu_{r+1+s} \circ (\Id_{{A[1]}^{\otimes r}}\otimes \smu_s \otimes \Id_{{A[1]}^{\otimes t}}) \circ (\sigma^{\otimes r}\otimes \sigma^{\otimes s}\otimes \sigma^{\otimes t}) \\
  &= \sum_{\{r,s\mid r+1+s=n\}} (-1)^{\|\smu_s\|\cdot \|\sigma^{\otimes r}\|} \smu_{r+1+s} \circ (\sigma^{\otimes (s+1)}\otimes \Id_{{A[1]}^{\otimes t}})
  \circ (\Id_{A_+^{\otimes r}}\otimes (\sigma^{-1}\circ \smu_s \circ \sigma^{\otimes s})
  \otimes \sigma^{\otimes t}) \\
  &= \sum_{\{r,s\mid r+1+s=n\}} (-1)^{\|\smu_s\|\cdot \|\sigma^{\otimes r}\|} \smu_{r+1+s} \circ (\sigma^{\otimes(s+1)}\otimes \Id_{{A[1]}^{\otimes t}})
  \circ (\Id_{A_+^{\otimes r}}\otimes \mu_s  \otimes \sigma^{\otimes t}) \\
  &= \sum_{\{r,s\mid r+1+s=n\}} (-1)^{\|\smu_s\|\cdot \|\sigma^{\otimes r}\|+\|\mu_s\|\cdot \|\sigma^{\otimes t}\|} \smu_{r+1+s} \circ (\sigma^{\otimes(s+1+t)})
  \circ (\Id_{A_+^{\otimes r}}\otimes \mu_s
  \otimes \Id_{A_+^{\otimes t}}) \\
  &=\sum_{\{r,s\mid r+1+s=n\}} (-1)^{r+st} \mu_{r+1+s}(\Id_{A_+^{\otimes r}}\otimes \mu_s\otimes \Id_{A_+^{\otimes t}}),
\end{align*}
since
$\mu_n=\sigma^{-1}\circ \smu_n\circ \sigma^{\otimes n}$, $\|\smu_s\|=-1$, $\|\sigma^{\otimes \ell}\|=-\ell$, so $\|\mu_s\|=s-2$.
We prefer to use the normalization $\smu_n$ in the proofs below, 
since the formulas are a
little simpler. 

We also normalize the $\Ainf$-homomorphism relation as follows.  For us,
an $\Ainf$-homomorphism is a collection of maps 
$\{\smap_n\co A_+[1]^n\to B_+[1]\}$ of degree $0$, satisfying
\begin{equation}
  \label{eq:AinfMor}
  \sum_{n=r+s+t} \smap_{r+1+t}\circ (\Id^{\otimes r} \otimes \smu_s
  \otimes \Id^{\otimes t}) = 
  \sum_{n=i_1+\dots i_m} \smu_m(\smap_{i_1}\otimes\dots\otimes \smap_{i_m}).
\end{equation}

The bar complex is defined by 
\[ \Barop(A)=\bigoplus_{n=0}^{\infty} A\otimes (A_+[1])^{\otimes n}\otimes A,\]
with differential
\[ \partial \co A\otimes (A_+[1])^{\otimes n} \otimes A\to
A\otimes (A_+[1])^{\otimes(n-1)}\otimes A \]
defined by 
\[ \sum_{r=0}^{n} \Id_{A^{\otimes r}[r-1]} \otimes \smu_{2}\otimes  \Id_{A^{\otimes(n-1-r)}[n-2-r]}.\]
(Compare Equation~\eqref{eq:BarDifferential}.)  
We can define the Hochschild
complex as in Definition~\ref{def:HH-assoc}, by 
\[ \HC^*(A)=\Hom_{A\otimes A^{\op}}(\Barop(A),A),\]
with its induced differential.

Equivalently, the analogue of Equation~\eqref{eq:HC-Hom-over-ground} gives
\[
  \HC^n(A)=\Hom_{\Ground\otimes\Ground}((A_+[1])^{\otimes n},A),
\]
with differential 
\begin{equation}\label{eq:signed-Hoch-diff}
  \delta  \smap_n
= \smu_2(\Id_{A_+[1]}\otimes \smap_n) -(-1)^{\|\smap_n\|} \smap_n\circ 
\left(\sum_{r=0}^{n-1} \Id_{A_+[1]^{\otimes r}}\otimes \smu_2\otimes \Id_{A_+[1]}^{n-r}\right) + \smu_2\circ (\smap_n\otimes \Id_{A_+[1]}).
\end{equation}
(The sign comes from Equation~\eqref{eq:MorDifferential}.)

The isomorphism of chain complexes
\[
  \Phi\co\bigoplus_n\Hom_{\Ground\otimes\Ground}(A_+^{\otimes n},A)
  \to \Hom_{A\otimes A^{\op}}(\Barop(A),A)
\]
is defined by
\[
  \Phi(\smap)=(\smu_2\circ (\Id_A\otimes \smu_2))\circ (\Id_{A} \otimes \smap \otimes \Id_{A}).
\]

In particular, if an $A_n$ algebra has only $\smu_2$ and $\smu_n$, the $\Ainf$ relation
with $n+1$ inputs is
\[ \delta \smu_n=0.\]

From its construction the Hochschild complex inherits a $\Zmod{2}$
grading.

\begin{proposition}
  \label{prop:deform-assoc-algZ}
  The statement of Proposition~\ref{prop:deform-assoc-alg} holds in arbitrary characteristic.
\end{proposition}
\begin{proof}
  Signs are incorporated into the definition of $\star$, as follows:
  \[ \smap_i\star \smapg_m = \smap_i\circ \left(\sum_{i=r+1+t} \Id^{\otimes r}\otimes \smapg_m \otimes \Id^{\otimes t}\right).\]
  Thus,
  \[ \delta \smap_n=\smu_2\star \smap_n -(-1)^{\|\smap_n\|} \smap_n\star \smu_2.\]

  With these sign conventions, define
  \[ 
    \obstr_{n}=-\sum_{\substack{i,j\geq 3\\ i+j=n+2}} \smu_i\star \smu_j.
  \]
  With this notation, using Equation~\eqref{eq:signed-Hoch-diff} and the fact that $\|\smu_n\|=1$, the $\Ainf$ relation with $n+1$ inputs  (Equation~\eqref{eq:AinfSigned2}) becomes
  \[
    \delta \smu_n + \sum_{\substack{i,j\geq 3 \\ i+j=n+2}}\smu_i\star\smu_j=0
  \]
  (compare Equation~\eqref{eq:AinfHochschild}).
  Thus, Property~\ref{o:Extension} holds in arbitrary characteristic.

    To verify Property~\ref{o:Cocycle}, we use the following refinement of
    Equation~\eqref{eq:LeibnizStar}:
    \begin{equation}\label{eq:LeibnizStar-signed}
      \delta (\smap_i\star \smapg_j)=(\delta \smap_i)\star \smapg_j
      + (-1)^{\|\smap_i\|} \smap_i\star (\delta \smapg_j) 
      -\smu_2(\smap_i,\smapg_j)
      -(-1)^{\|\smap_i\|\|\smapg_j\|}\smu_2(\smapg_j,\smap_i).
    \end{equation}
    
    The
    last two terms, for example, cancel against terms in $(\delta\smap_i)\star \smapg_j$, as follows.
    Note that $(\delta \smap_i)$ includes a term $\smu_2\circ (\Id\otimes \smap_i)$.
    When computing $\bigl(\smu_2 \circ (\Id\otimes \smap_i)\bigr)\star \smapg_j$, 
    the term with $r=0$ is given by
    \[
      (\Id\otimes \smap_i)\circ (\smapg_j\otimes \Id^{\otimes i})=
      (-1)^{\|\smap_i\| \|\smapg_j\|} \smapg_j\otimes \smap_i.
    \]
    Also, $(\delta \smap_i)$ includes another term, $\smu_2(\smap_i\otimes \Id)$.
    The $r=i$, $t=0$ term in
    $\bigl(\smu_2\circ (\smap_i\otimes \Id)\bigr)\star \smapg_j$ is
    \[ \smu_2\circ (\smap_i\otimes \Id)\circ (\Id\otimes \smapg_j)=\smu_2(\smap_i, \smapg_j).\]
    
    Thus, since $\|\smu_i\|=\|\smu_j\|=1$, we have that
    \begin{align*}
      \delta(\obstr_n)
      &=-\hspace{-.5em}\sum_{\substack{i,j\geq 3 \\ i+j=n+2}}\Bigl[(\delta\smu_i)\star\smu_j
      -\smu_i\star(\delta \smu_j) \Bigr]\\
      &=\hspace{-.5em}\sum_{\substack{i,j,k\geq 3\\i+j+k=n+2}}
      \Bigl[(\smu_i\star\smu_j)\star\smu_k - \smu_i\star(\smu_j\star\smu_k)\Bigr].
    \end{align*}
    
    To complete the verification of Property~\ref{o:Cocycle},
    we use the following signed version of Equation~\eqref{eq:Associator}:
    \begin{equation}
      \label{eq:AssociatorSign}
      (a \star b)\star c - a \star(b\star c)= 
      a\circ (\Id \otimes b \otimes \Id\otimes c\otimes \Id)\circ\Delta^5 
      +(-1)^{\|b\|\|c\|}
      a\circ (\Id \otimes c \otimes \Id\otimes b\otimes \Id)\circ
      \Delta^5.
  \end{equation}

    To verify Property~\ref{item:coboundary}, observe that the
    existence of a homomorphism $\smap\co \Alg\to \Alg'$ is equivalent
    to the existence of $c_{n-1}$ with
    $\delta(c_{n-1})=\mu_n-\mu_n'$.
    The proof of Property~\ref{item:naturality} from before applies
    similarly.

    Turning to the statements about maps, we define $\tilde \smap$ as in
    Equation~\eqref{eq:deftilde}, $D$ as in Equation~\eqref{eq:DefD},
    and $\dtilde \phi$ as in Equation~\eqref{eq:deftilde2}.  
    
    We will assume that the components of $\phi_i$ all have $\Zmod{2}$
    grading equal to zero. In particular, under the assumption that
    $\|\phi_i\|=0$, Equation~\eqref{eq:SecondLeibniz} is replaced by
    \[
      \delta((\smap\circ \dtilde{\phi})_m)= (\delta(\smap)\circ \dtilde{\phi})_{m+1}
      + (-1)^{\|f\|}(\smap \circ D \dtilde{\phi})_{m+1} - \sum_{\substack{i,j\\ i+j=m+1, i>1}}
      \left(\smu_2(\phi_{i},(\smap\circ \dtilde{\phi})_j)+
        \smu_2((\smap\circ \dtilde{\phi})_j,\phi_{i})\right).
    \]
    Summing the maps
    $\smu_k\colon A_+[1]^{\otimes k}\to A$
    over all $k\leq n$ gives a map
    \[ \smu_{\leq n}\colon \bigoplus_{k<n} A_+[1]^{\otimes k} \to A. \]
    Let
    \[
    \fobstr_n= \Bigl(\sum_{\substack{j\geq 3 \\ i+j=n+1}} \smap_i\star
    \smu_j\Bigr)-\Bigl(\smu_{\leq n}\circ \dtilde{\smap}\Bigr)_n\co A_+[1]^{\otimes n}\to A,
  \]
    where $(\cdot)_n$ denotes the restriction of a map from $\Tensor^* A_+[1]$ to the summand $A_+[1]^{\otimes n}$.
    Moreover, since the components of $\smap$ all have degree zero,
    the $A_{n-1}$-homomorphism relation is equivalent to 
    \[
      \delta \smap_{n-1} - (\smap_{\leq n-2}\star \smu_{\geq 3})_n+(\smu_{\geq 2}\circ \dtilde \smap)_n=0,
    \]
    i.e.,
    \[
      \delta \smap_{n-1}=\fobstr_n.
    \] 
    Property~\ref{item:coboundaryfobstr} follows in arbitrary characteristic.

    Analogous to Equation~\eqref{eq:AinfMap1}, we can formulate this
    $\Ainf$ relation for maps as
    \begin{equation}\label{eq:AinfMap1-signed}
      \delta(\smap_{n-1}) -(\smap_{\leq n-2}\star\smu_{\geq 3})_n + 
      (\smu_{\geq 3}\circ \dtilde{\smap_{\leq n-2}})_n -
      (\smu_2(\smap_{>1},\smap_{>1}))_n=0.
    \end{equation}
    If we define
    \[
      D(\Phi)=\widetilde{\smu_2}\circ \Phi-(-1)^{\|\Phi\|}\Phi\circ \widetilde{\smu_2},
    \]
  then the $\Ainf$ relation for maps can be reformulated (analogous to
  Equation~\eqref{eq:AinfMap2}) as
  \begin{equation}\label{eq:AinfMap2-signed}
    D({\dtilde{\smap_{\leq n-1}}})-\Bigl({\dtilde{\smap_{\leq n-2}}}\circ \widetilde{\smu_{\geq 3}} + \widetilde{\smu_{\geq 3}}\circ {\dtilde{\smap_{\leq n-2}}}\Bigr)_{n}=0
  \end{equation}

  Analogous to Equation~\eqref{eq1}, we find that
  \begin{align*}
  \delta \Bigl((\smap_{\leq n-2}\star \smu_{\geq 3})_n\Bigr)
  & =\Bigl(\delta(\smap_{\leq n-2})\star \smu_{\geq 3}
  + \smap_{\leq n-3} \star \delta(\smu_{\geq 3}) - \smu_2(\smap_{>1},\smu_{\geq 3})-
  \smu_2(\smu_{\geq 3},\smap_{>1})\Bigr)_{n+1}  \nonumber \\
  &=\Bigl( (\smap_{\leq n-3}\star \smu_{\geq 3})\star \smu_{\geq 3}- 
(\smu_{\geq 3} \circ {\dtilde{\smap_{\leq n-3}}})\circ \widetilde{\smu_{\geq 3}}
  - \smu_2(\smap_{>1}\star \smu_{\geq 3},\smap_{>1}) \\ 
  & \quad -\smu_2(\smap_{>1},\smap_{>1}\star \smu_{\geq 3}) -
  \smap_{\leq n-3} \star (\smu_{\geq 3} \star \smu_{\geq 3})
   - \smu_2(\smap_{>1},\smu_{\geq 3})-
  \smu_2(\smu_{\geq 3},\smap_{>1})\Bigr)_{n+1} \nonumber \\
  &= \Bigl(-(\smu_{\geq 3} \circ {\dtilde{\smap_{\leq n-3}}})\circ \widetilde {\smu_{\geq 3}}
  - \smu_2(\smap\star \smu_{\geq 3},\smap_{>1})-\smu_2(\smap_{>1},\smap\star \smu_{\geq 3})\Bigr)_{n+1}.
\end{align*}
Analogous to Equation~\eqref{eq2}:
  \begin{align*}
    \delta \Bigl((\smu_{\geq 3}\circ \dtilde{\smap_{\leq n-2}})_n\Bigr)
    &= \Bigl(\delta(\smu_{\geq 3})\circ \dtilde{\smap_{\leq n-2}}
    - \smu_{\geq 3} \circ D(\dtilde{\smap_{\leq n-2}})  - \smu_2(\smu_{\geq 3}\circ \dtilde{\smap},\smap_{>1})
    - \smu_2(\smap_{>1},\smu_{\geq 3}\circ \dtilde{\smap})\Bigr)_{n+1} \nonumber \\
    &=
    \Bigl(-(\smu_{\geq 3}\star \smu_{\geq 3})\circ \dtilde{\smap_{\leq n-2}}
    +\smu_{\geq 3}\circ(\widetilde{\smu_{\geq 3}}\circ \dtilde{\smap_{\leq n-2}}) \\
    &
    \qquad -\smu_{\geq 3} \circ (\dtilde{\smap_{\leq n-3}}\circ \widetilde{\smu_{\geq 3}}) 
        -\smu_2(\smu_{\geq 3}\circ \dtilde{\smap},\smap_{>1})
        - \smu_2(\smap_{>1},\smu_{\geq 3}\circ \dtilde{\smap})\Bigr)_{n+1} \\
        &=
        \Bigl(-\smu_{\geq 3} \circ (\dtilde{\smap_{\leq n-3}}\circ \widetilde{\smu_{\geq 3}}) 
        -\smu_2(\smu_{\geq 3}\circ \dtilde{\smap},\smap_{>1})
        - \smu_2(\smap_{>1},\smu_{\geq 3}\circ \dtilde{\smap})\Bigr)_{n+1}
  \end{align*}
  and analogous to Equation~\eqref{eq3}:
\[
  \delta\Bigl(\smu_2(\smap_{>1},\smap_{>1})_n\Bigr)
    = -\smu_2(\delta \smap_{>1},\smap_{>1})_{n+1}-\smu_2(\smap_{>1},\delta \smap_{>1})_{n+1}
    = -\smu_2(\delta \smap,\smap_{>1})_{n+1}-\smu_2(\smap_{>1},\delta \smap)_{n+1}.
    \]

    Adding these expressions up, we find that
    \begin{align*} \delta(\fobstr_n)&=\delta\Bigl((\smap_{\leq k-2}\star \smu_{\geq 3})_n\Bigr)-
      \delta\Bigl((\smu_{\geq 3}\circ \dtilde{\smap_{\leq n-2}})_n\Bigr)
      -\delta\Bigl((\smu_2(\smap_{>1},\smap_{>1}))_n\Bigr) \\
      &= 
      -\Bigl(\smu_2(\smap_{\leq n-2} \star \smu_{\geq 3}-\smu_{\geq 3}\circ {\dtilde \smap}
      - \delta \smap_{\leq n-2},\smap_{>1}) + 
\smu_2(\smap_{>1},\smap_{\leq n-2} \star \smu_{\geq 3}-\smu_{\geq 3}\circ {\dtilde \smap}
      - \delta \smap_{\leq n-2})\Bigr)_{n+1} \\
      &= 0.
      \end{align*}
    Property~\ref{item:mapcocycle} follows in arbitrary characteristic.
    \end{proof}

\subsection{Uniqueness of \texorpdfstring{$\UnDefAlgZ$}{the un-weighted algebra}
}\label{sec:unweighted-unique-Z}

The following analogue of Theorem~\ref{thm:UnDefAlg-unique} holds:

\begin{theorem}\label{thm:UnDefAlg-unique-Z}
  Up to isomorphism, there is a unique $\Ainf$-deformation of
  $\AsUnDefAlgZ$ over $\ZZ[U]$ satisfying the following conditions:
  \begin{enumerate}
  \item\label{item:Uuniqe-graded-Z} The deformation is $\Gamma=G\times \ZZ$-graded, where the gradings of the
    chords $\rho_i$ is defined by $\gamma(\rho_i)=\gr(\rho_i)\times \wingr(\rho_i)$.
    (The gradings $\gr$ and $\wingr$ are defined in Section~\ref{sec:grading}.)
  \item\label{item:Uuniqe-mu-4-Z} The operations satisfy
    $\mu_4(\rho_4,\rho_3,\rho_2,\rho_1)=U\iota_1$ and
    $\mu_4(\rho_3,\rho_2,\rho_1,\rho_4)=U\iota_0$.
  \end{enumerate}
\end{theorem}

This follows from a Hochchild homology computation, which is
facilitated by a signed analogue of Lemma~\ref{lem:Koszul}, which we
state after a few remarks.  Recall that the cobar algebra is obtained
by dualizing the tensor algebra over $A_+[1]$. We will use the trivial
$\Zmod{2}$ grading on the torus algebra (cf.\ Section~\ref{sec:mod-2-gr}). The cobar algebra inherits a $\Zmod{2}$
grading from this description. Concretely, the elements $\rho_i^*$ in
$\Cobarop(\AsUnDefAlgZ)$ all have odd $\Zmod{2}$ grading.

\begin{lemma}\label{lem:Koszul-Z}
  There is a quasi-isomorphism of $\Gamma$-graded algebras (over $\ZZ$)
  \[
    \phi\co (\Cobarop(\AsUnDefAlgZ),\gammaCobar)\to (\AsUnDefAlgZ,\alpha\circ\gamma)
  \]
  specified by $\phi(\iota_0)=\iota_1$, $\phi(\iota_1)=\iota_0$,
  $\phi(\rho_i^*)=[\rho_i]$ for $i=1,\dots,4$, and $\phi(a^*)=0$ if
  $|a|>1$.  The map $\phi$ shifts $\Zmod{2}$ gradings: it identifies
  $\Cobarop(\AsUnDefAlgZ)$ with the above $\Zmod{2}$ grading with the
  algebra $(\AsUnDefAlgZ,\alpha\circ\gamma)$, with a $\Zmod{2}$ grading
  which is the mod-2 reduction of the length grading.
\end{lemma}
\begin{proof}
  The proof of Lemma~\ref{lem:Koszul} can be adapted, noting that 
  the element $\rho_i^*\otimes\rho_{i+1}^*\dots\otimes \rho_{\ell-1}\otimes \rho_\ell^*\in\Cobarop(\AsUnDefAlgZ)$ 
  has $\Zmod{2}$ grading specified by $\ell-i+1$,
  which coincides with the length of 
  $\phi(\rho_i^*\otimes\rho_{i+1}^*\otimes\dots\otimes \rho_{\ell-1}^*\otimes \rho_\ell^*)$.
  
  Equation~\eqref{eq:HomotopyFormula} (now in characteristic different than $2$) follows, using the following sign
  refinement of the homotopy operator from Equation~\eqref{eq:DefH}:
  \[
    H(\overbrace{\rho_i^*\otimes \rho_{i+1}^*\otimes\dots\otimes \rho_\ell^*}^k
    \otimes a_1^*\otimes\dots\otimes a_m^*)=
    \begin{cases}
      (-1)^{\ell-i}\rho_i^*\otimes\rho_{i+1}^*\otimes\dots\otimes \rho_{\ell-1}^*
      \otimes (a_1\cdot \rho_\ell)^*\otimes a_2^*\otimes\dots\otimes a_m^* &
      k>0 \\
      0 &k=0.
    \end{cases}
  \]
  This completes the proof.
\end{proof}

We have the following sign refinement of
Proposition~\ref{prop:AsUnDefHoch}:

\begin{proposition}\label{prop:AsUnDefHochZ}
  The graded Hochschild cohomology
  $\HH^{*,*}_\Gamma(\AsUnDefAlgZ\otimes\ZZ[U])$ of
  $\AsUnDefAlgZ\otimes_\ZZ\ZZ[U]$ over $\Ground\otimes\ZZ[U]$
  satisfies
  \begin{align*}
  \HH^{n,-1}_\Gamma(\AsUnDefAlgZ\otimes\ZZ[U])&=
  \begin{cases}
   \ZZ & n=4\\
    0 &\text{otherwise}
  \end{cases}\\
  \HH^{n,-2}_\Gamma(\AsUnDefAlgZ\otimes\ZZ[U])&=
  \begin{cases}
   \ZZ & n=5\\
    0 &\text{otherwise}
  \end{cases}
  \end{align*}
  Moreover, suppose $\xi \in
  \HC^{4,-1}_\Gamma(\AsUnDefAlgZ\otimes\ZZ[U])$ is a cycle and
  $\xi(\rho_4\otimes\rho_3\otimes\rho_2\otimes\rho_1)=U$. Then $\xi$
  represents a generator of
  $\HH^{4,-1}_\Gamma(\AsUnDefAlgZ\otimes\ZZ[U])$.
\end{proposition}

\begin{proof}
  As in the proof of Proposition~\ref{prop:AsUnDefHoch}, we will
  perform the Hochschild cohomology computations with the help of a
  small model of the cobar algebra.

  Specifically, define $C^*_{\ZZ}$ to be the sign-refined
  analogue of Definition~\ref{def:SmallModel},
  defined using $\AsUnDefAlgZ$ in place of $\AsUnDefAlg$. That is,
  $C^*_{\ZZ}$ is the free $\ZZ$-module generated by elements of the form
  $a\otimes [b]$ where $a\in \AsUnDefAlgZ$ and $[b]\in\AsUnDefAlgZ$ are
  basic elements with the property that $i\cdot a \cdot j=a$ and
  $[j\cdot b \cdot i]=[b]$, for some idempotents
  $i,j\in\{\iota_0,\iota_1\}$. The differential is defined by the following
  analogue of Equation~\eqref{eq:DefCdiff}:
  \[
  \partial (a\otimes [b])=\sum_{i=1}^4 \left((-1)^{|b|}\rho_i\cdot a \otimes [b \cdot \rho_i]+ a\cdot \rho_i \otimes [\rho_i\cdot b]]\right). 
    \]
  
  Define
  $C^*_{\Gamma;\ZZ}\subset C^*$ to be the portion in grading
  $0\times \ZZ\subset G\times \ZZ$. The analogue  of Proposition~\ref{prop:SmallerModel} gives a quasi-isomorphism
  \[
    H^{n,k}(C_{\Gamma;\ZZ})\cong \HH^{n,k}(\AsUnDefAlgZ\otimes\ZZ[U]).
  \]
  (This proof uses Lemma~\ref{lem:Koszul-Z} in place of Lemma~\ref{lem:Koszul}, which was used to prove Proposition~\ref{prop:SmallerModel}.)
  
  With these remarks in place, the proof of
  Proposition~\ref{prop:AsUnDefHoch} applies, with minor modifications. The
  vanishing of homology for grading reasons follows exactly as before.
  In the present case, the differentials appearing in the computation
  of that proposition involving elements with $k=-1$ and $-2$ have
  signs in them.  For example:
  \begin{align*}
    \partial(\rho_{123}[\rho_{123}])&=
    \rho_{1234}[\rho_{4123}]-
    \rho_{4123}[\rho_{1234}] \\
    \partial(U[\rho_{1234}])&=
    U\rho_1[\rho_{12341}]+U\rho_4[\rho_{41234}].
  \end{align*}
  The homology class $\HH^{4,-1}_\Gamma$ is represented by the cycle
  \[ U[\rho_{1234}]-U[\rho_{2341}]+U[\rho_{3412}]-U[\rho_{4123}] \]
  The homology class $\HH^{5,-1}$ is represented by 
  \[
    U\rho_1 [\rho_{12341}]\sim
    U\rho_2 [\rho_{23412}]\sim
    U\rho_3 [\rho_{34123}]\sim
    U\rho_4 [\rho_{41234}].\qedhere
  \]
\end{proof}
    
\begin{proof}[Proof of Theorem~\ref{thm:UnDefAlg-unique-Z}]
  This proof is the same as the proof of
  Theorem~\ref{thm:UnDefAlg-unique}, replacing the use of
  Proposition~\ref{prop:deform-assoc-alg} (which required
  characteristic $2$) with its analogue,
  Proposition~\ref{prop:deform-assoc-algZ}, and replacing the
  characteristic $2$ Hochschild cohomology computation of
  Proposition~\ref{prop:AsUnDefHoch} with the more general
  Proposition~\ref{prop:AsUnDefHochZ}.
\end{proof}

\subsection{Weighted algebras and Hochschild cohomology revisited}

We explain how to put signs into the discussion of
Section~\ref{sec:weighted-deformation}.

Let $\Ground$ be an algebra that is  free as a $\ZZ$-module.
Fix an augmented $\Ainf$-algebra $\Alg^0=(A,\{\smu_m\})$
over $\Ground$ with underlying
$\ZZ$-module $A$ and augmentation ideal $A_+\subset A$. By a
\emph{weighted deformation} of $\Alg^0$ we mean a weighted
$\Ainf$-algebra $(A,\{\smu_m^k\})$ with the same underlying vector
space as $\Alg^0$ and whose weight-zero operations are the same as for
$\Alg^0$: $\smu_m^0=\smu_m$ for all $m\geq 0$. If $\Alg$ and $\Blg$ 
are both weighted deformations of the same undeformed $\Ainf$-algebra, a {\em homomorphism of deformations}
from $\Alg$ to $\Blg$ is a sequence of $\Zmod{2}$-grading-preserving maps 
$\smap^{\bullet}=\{\smap^W\co \Tensor^*(A_+[1])\to A_+[1]\}_{W=0}^{\infty}$
satisfying the weighted $\Ainf$ relation
\[
  \sum_{a+b=W} \smap^a \circ (\Id\otimes \smu^b\otimes \Id)\circ \Delta^3 
  - \sum_{w_1+\dots+w_m=W} (\smap^{w_1}\otimes\dots\otimes \smap^{w_m})\circ
  \Delta^m=0
\]
for each $W\geq 0$.

In this case, the Hochschild complex of $\Alg^0$ is given, as a $\ZZ$-module, by
\[
\HC^*(\Alg^0)=\prod_{n=0}^\infty \Hom_{\Ground\otimes\Ground}(\Ground\otimes (A_+)^{\otimes
  n}, A\otimes\Ground)=\prod_{n=0}^\infty \Hom_{\Ground\otimes\Ground}((A_+)^{\otimes
  n}, A),\]
with differential specified by
\[ \delta(\smap)=\smu^0\star \smap - (-1)^{\|\smap\|} \smap\star  \smu^0,\]
with $\star$ as in Proposition~\ref{prop:deform-assoc-algZ}.

\begin{proposition}\label{prop:deform-Ainf-algZ}
  The statement of Proposition~\ref{prop:deform-Ainf-alg}
  holds in arbitrary characteristic.
\end{proposition}

\begin{proof}
  In place of Equation~\eqref{eq:WeightedAinftyRelation}, the signed weight $W$ $\Ainf$ relation takes the form 
  \[
    \delta\smu^{W}+(\smu^{\bullet\geq 1}\star\smu^{\bullet\geq 1})^W=0.
  \]
  Letting 
  \[
    \obstr^W=-(\smu^{\bullet\geq 1}\star \smu^{\bullet\geq 1})^W,
  \]
  the $\Ainf$ relation is equivalent to 
  \[
    \delta \smu^W = \obstr^W.
  \]
  Hence, Property~\ref{wo:Extension} follows in arbitrary characteristic.

  We verify Property~\ref{wo:Cocycle} as follows.
  We modify the definition of $\eta$: if $\|\smap^\bullet\|=\|\smapg^\bullet\|=1$, then
  let 
  \[ \eta^W(\smap^\bullet,\smapg^\bullet)=\smu^0\circ (\Id \otimes \smap^{\bullet}\otimes \Id\otimes \smapg^\bullet\otimes \Id)\circ \Delta^5-\smu^0\circ (\Id \otimes \smapg^\bullet\otimes \Id\otimes \smap^\bullet \otimes \Id)\circ\Delta^5.\]
  Equation~\eqref{eq:FirstLeibnizW2} generalizes to 
  \begin{equation}
    \label{eq:FirstLeibnizW2Z}
    \delta(f^\bullet\star g^\bullet)=(\delta f^\bullet)\star g^\bullet+(-1)^{\|f^\bullet\|} f^{\bullet}\star (\delta g^\bullet) -\eta^\bullet(f^\bullet,g^\bullet)
  \end{equation}
  (cf.\ Equation~\eqref{eq:LeibnizStar-signed}).
  (We are using here that $\|\smap^\bullet\|=\|\smapg^\bullet\|=1$.)

  To verify Property~\ref{wo:Cocycle}, observe that
  \begin{align*}
    \delta \obstr^W= -\delta(\smu^{\bullet\geq 1} \star \smu^{\bullet\geq 1})^W
    &= -\bigl(\delta(\smu^{\bullet\geq 1})\star \smu^{\bullet\geq 1}\bigr)^W 
     +\bigl(\smu^{\bullet\geq 1}\star \delta(\smu^{\bullet\geq 1})\bigr)^W
    + \eta^{\bullet}(\smu^{\bullet\geq 1},\smu^{\bullet^\geq 1}) \\
    &= -\bigl((\smu^{\bullet \geq 1}\star\smu^{\bullet\geq 1})\star \smu^{\bullet \geq 1}\bigr)^W +
    \bigl(\smu^{\bullet \geq 1}\star (\smu^{\bullet \geq 1}\star\smu^{\bullet\geq 1})\bigr)^W\\ &= 0,
  \end{align*}
  in view of Equation~\eqref{eq:AssociatorSign}.
  
  Defining $\dtildew{\smap}^\bullet$ as in Equation~\eqref{eq:DefdTildew}, we find that the
  weight $W$  $\Ainf$-homomorphism relation (Equation~\eqref{eq:WeightedAinfHom}) takes the form
  \[
    (\smap^{\bullet}\star \smu^{\bullet})^W - (\smu^{\bullet}\circ \dtildew{\smap}^{\bullet})^W=0
  \]
  (cf.\ Equation~\eqref{eq:AinfMor}).
  Analogous to Equation~\eqref{eq:AinfMap1w2}, we can formulate this as
  \[
    \delta \smap^W=(\smap^\bullet\star \smu^{\bullet\geq 1})^W - (\smu^{\bullet\geq 1}\circ \dtildew{\smap}^\bullet)^W - \smu^0\circ (\dtildew{\smap^{\bullet<W}})^W.
  \]
  Equivalently, if we let
  \begin{align*}
    \fobstr^W &= (\smap^\bullet\star \smu^{\bullet\geq 1})^W - (\smu^{\bullet\geq 1}\circ \dtildew{\smap}^\bullet)^W - \smu^0\circ (\dtildew{\smap^{\bullet<W}})^W,
 \\
    &= (\smap^\bullet\star \smu^{\bullet\geq 1})^W - (\smu^{\bullet\geq 1}\circ \dtildew{\smap}^\bullet)^W - \smu^0\circ (\dtildew{\smap^{\bullet}})^W +
 \smu^0\star \smap^{\bullet}
    \end{align*}
  then the $\Ainf$-homomorphism relation takes the form
  \[
    \delta \smap^W=\fobstr^W.
  \]
  Property~\ref{w:coboundary} follows.

  To verify Property~\ref{w:mapcocycle}, we use the following signed version
  of Equation~\eqref{eq:AinfMap2w}:
  \[
    D(\dtildew{\smap}^\bullet)^w+ (\overline{\smu^{\bullet\geq 1}}\circ \dtildew{\smap}^\bullet)^w - (\dtildew{\smap}^\bullet\circ \overline{\smu^{\bullet\geq 1}})^w=0
  \]
  (cf.\ Equation~\eqref{eq:AinfMap2-signed})
  where
  \[
    D(\Phi)=\overline{\smu^0}\circ \Phi-(-1)^{\|\Phi\|}\Phi\circ \overline{\smu^0}.
  \]

  In place of Equation~\eqref{eq:Map1wStarMu}, we have
  \[
    (\delta \smap^\bullet\star \smu^{\bullet\geq 1})^W = 
    \Big(\Big((\smap^\bullet\star \smu^{\bullet\geq 1})
    -  (\smu^{\bullet \geq 1} \circ \dtildew{\smap}^\bullet) - 
    (\smu^0\circ \dtildew{{\smap}^{\bullet<W}})\Big)\star\smu^{\bullet\geq1}\big)^W.
  \]
  In place of Equation~\eqref{eq:SecondLeibnizW}, we have:
  \[
    \delta(\smap^\bullet\circ \dtildew{\phi}^\bullet)= (\delta(\smap^\bullet)\circ \dtildew{\phi}^\bullet)
    + (-1)^{\|\smap^\bullet\|}(\smap^\bullet \circ D \dtildew{\phi}^\bullet) 
    - (\smu^0\star \smap^\bullet)\circ \dtildew{\phi}^\bullet
    + \smu^0\star (\smap^\bullet\circ \dtildew{\phi}^\bullet)
  \]
  (cf.\ Equation~\eqref{eq:AinfMap1-signed}).  In place of
  Equation~\eqref{eq:AinfMap2w2}, if we know that the weight $<W$
  $\Ainf$\hyp homomorphism relations hold, we have
  \begin{align*}
    D (\dtildew{\smap^{\bullet<W}})^W&+ 
    ({\overline{\smu^{\bullet\geq 1}}}\circ {\dtildew \smap}^\bullet)^W 
    - (\dtildew{\smap}^\bullet\circ \overline{\smu^{\bullet\geq 1}})^W \nonumber \\
    &
    = {\overline {(\smu^0\circ \dtildew{\smap^{\bullet<W}})^W}} +
    {\overline{(\smu^{\bullet\geq 1}\circ \dtildew{\smap^{\bullet<W}})^W}}
    - {\overline{(\smap^{\bullet<W}\star \smu^{\bullet\geq 1})^W}}
  \end{align*}

  In place of Equation~\eqref{eq:Part1}:
  \begin{align*}
    \delta(\smap^\bullet\star \smu^{\bullet\geq 1})^W &=
            ((\delta \smap^\bullet)\star \smu^{\bullet\geq 1})^W
    + (\smap^\bullet\star \delta(\smu^{\bullet \geq 1}))^W- 
    \eta^W(\smap^\bullet,\smu^{\bullet \geq 1}) \nonumber \\
    &= 
    \left((\smap^\bullet\star\smu^{\bullet\geq 1})\star\smu^{\bullet\geq 1}
-(\smu^{\bullet \geq 1}\circ \dtildew{\smap}^\bullet)\star \smu^{\bullet \geq 1}
- (\smu^0\circ  \dtildew{\smap}^\bullet) \star \smu^{\bullet\geq 1}
+ (\smu^0\star  \smap^\bullet) \star \smu^{\bullet\geq 1}\right)^W \\
    &\qquad - \left(\smap^{\bullet}\star (\smu^{\bullet\geq 1}\star \smu^{\bullet\geq 1})\right)^W
    -  \eta^W(\smap^\bullet,\smu^{\bullet \geq 1}) \\
    &= -((\smu^{\bullet \geq 1}\circ \dtildew{\smap}^\bullet)\star \smu^{\bullet \geq 1})^W
    - ((\smu^0\circ  \dtildew{\smap}^\bullet) \star \smu^{\bullet\geq 1})^W
    + \bigl((\smu^0\star  \smap^\bullet) \star \smu^{\bullet\geq 1}\bigr)^W\\
    &\qquad-  \eta^W(\smap^\bullet,\smu^{\bullet \geq 1}). 
    \end{align*}
    
    In place of Equation~\eqref{eq:Part2}:
      \begin{align*}  \delta(\smu^{\bullet \geq 1}\circ \dtildew{\smap}^\bullet)^W&=
        \Big(\delta(\smu^{\bullet \geq 1})\circ \dtildew{\smap}^\bullet
        - \smu^{\bullet \geq 1}\circ (D \dtildew{\smap}^\bullet)-(\smu^0\star\smu^{\bullet \geq 1})\circ \dtildew{\smap}^\bullet+\smu^0\star(\smu^{\bullet \geq 1}\circ\dtildew{\smap}^\bullet)\Big)^W
        \nonumber \\
        &=\Big(-(\smu^{\bullet \geq 1}\star \smu^{\bullet \geq 1})\circ {\dtildew{\smap}^\bullet} 
        + \smu^{\bullet \geq 1}\circ (\overline{\smu^{\bullet \geq 1}}\circ \dtildew{\smap}^\bullet)
        - \smu^{\bullet \geq 1}\circ (\dtildew{\smap}^\bullet \circ\overline{\smu^{\bullet \geq 1}})\nonumber \\
        &\qquad - (\smu^0\circ(\overline{\smu^{\bullet \geq 1}}\circ \dtildew{\smap}^\bullet)+\smu^0\star(\smu^{\bullet \geq 1}\circ \dtildew{\smap}^\bullet)\Big)^W \\
        &=\Big(- \smu^{\bullet \geq 1}\circ (\dtildew{\smap}^\bullet \circ\overline{\smu^{\bullet \geq 1}})
        - \smu^0\circ(\overline{\smu^{\bullet \geq 1}}\circ \dtildew{\smap}^\bullet)+\smu^0\star(\smu^{\bullet \geq 1}\circ \dtildew{\smap}^\bullet)\Big)^W 
      \end{align*}
      In place of Equation~\eqref{eq:Part3}:
    \begin{align*}
    \delta(\smu^0\circ (\dtildew{\smap^{\bullet<W}})^W)&=(\delta\smu^0)\circ (\dtildew{\smap^{\bullet<W}})^W
    - \smu^0\circ (D\dtildew{\smap^{\bullet<W}})^W - (\smu^0\star\smu^0)\circ (\dtildew{\smap^{\bullet<W}})^W\\
    &\qquad+ \smu^0\star(\smu^0\circ \dtildew{\smap^{\bullet<W}})^W \\
    &= -\smu^0\circ (D\dtildew{\smap^{\bullet<W}})^W 
    + \smu^0\star(\smu^0\circ \dtildew{\smap^{\bullet<W}})^W.
    \end{align*}
    In place of Equation~\eqref{eq:Part4}:
    \begin{align*}
      \smu^0\circ D (\dtildew{\smap^{\bullet<W}})^W&= 
      -\smu^0\circ({\overline{\smu^{\bullet\geq 1}}}\circ {\dtildew \smap}^\bullet)^W 
      + \smu^0\circ(\dtildew{\smap}^\bullet\circ \overline{\smu^{\bullet\geq 1}})^W \nonumber \\
      &
      +\smu^0\star (\smu^0\circ \dtildew{\smap^{\bullet<W}})^W +
      \smu^0\star (\smu^{\bullet\geq 1}\circ \dtildew{\smap^{\bullet<W}})^W
      - \smu^0\star(\smap^{\bullet<W}\star \smu ^{\bullet\geq 1})^W.
    \end{align*}

    Finally, using 
    \[ (\smu^0\star \smap^\bullet)\star \smu^{\bullet\geq 1}
    -\smu^0\star(\smap^{\bullet}\star \smu^{\bullet\geq 1})-\eta^W(\smap^\bullet,\smu^{\bullet\geq 1})=0\]
    (a special case of  Equation~\eqref{eq:AssociatorSign}),
    it now follows that  $\delta(\fobstr^W)=0$, 
    that is, that Property~\ref{w:mapcocycle} holds in arbitrary characteristic.
\end{proof}

\subsection{Uniqueness of \texorpdfstring{$\MAlgZ$}{the torus algebra over 
the integers}}\label{sec:MAlg-uniqueZ}

In this section, we view the ground ring for $\UnDefAlgZ$ as
$\Ground=\ZZ\oplus \ZZ$, rather than $\ZZ[U]$; so our
augmentation is a map $\UnDefAlgZ\to \Ground$, and there is a
corresponding augmentation ideal.

\begin{theorem}\label{thm:MAlg-uniqueZ}
  Up to isomorphism, there is a unique weighted deformation $\MAlgZ$ of $\UnDefAlgZ$ such
  that:
  \begin{enumerate}
  \item $\MAlgZ$ is $\Gamma=G\times \ZZ$-graded and
  \item $\mu^1_0=\rho_{1234}+\rho_{2341}+\rho_{3412}+\rho_{4123}$.
  \end{enumerate}
\end{theorem}

We sketch the modifications needed to make to the discussion in
Section~\ref{sec:MAlg-unique} to hold over $\ZZ$.  Endow
$\SmallCobarZ=\AsUnDefAlgZ[h]/(h^2)$, with the $\Zmod{2}$-grading
which coincides with the mod $2$ reduction of the length grading on
$\AsUnDefAlgZ\subset \SmallCobarZ$, and so that $\|h\|$ is odd.  With this
understood, we have the following analogue of Lemma~\ref{lem:Koszul2}:

\begin{lemma}
  \label{lem:Koszul2Z}
  There is a quasi-isomorphism of $\Gamma\times \Zmod{2}$-graded algebras
  $\phi'\co \Cobarop(\UnDefAlgZ)\to \SmallCobarZ$, determined by
  $\phi'(\iota_0)=\iota_1$, $\phi'(\iota_1)=\iota_0$,
  $\phi'(\rho_i^*)=[\rho_i]$ for $i=1,\dots,4$, $\phi'(U^*)=[h]$, and
  $\phi'(a^*)=0$ if $a\neq U$ and length $|a|>1$.
  This map is $\Zmod{2}$-grading preserving, using the induced $\Zmod{2}$ grading on
  $\Cobarop(\UnDefAlgZ)$, corresponding to a length grading on $\SmallCobarZ$
  (with the understanding
  that $|h|$ is odd).
\end{lemma}

\begin{proof}
  The proof of Lemma~\ref{lem:Koszul2} applies, with signs added
  as in the proof of Lemma~\ref{lem:Koszul-Z}.
\end{proof}

The small model of the Hochschild complex is defined as before, as
$\mathfrak{C}^*_{\ZZ}=\UnDefAlgZ\hotimes_{\Ground\otimes\Ground}\SmallCobarZ$ as before.
Signs are now inserted into the differential, as follows:
\begin{align*}
  \partial (a [b])&=a [\partial b]+\sum_{i=1}^4
  \left((-1)^{|b|}\rho_i\cdot a  [b \cdot \rho_i]+ a\cdot \rho_i 
    [\rho_i\cdot b]\right) \\
      &\quad+\sum_{i=1}^4 
      \bigl((-1)^{|b|}\mu_4(\rho_{i+3},\rho_{i+2},\rho_{i+1}, a) [b \cdot\rho_{i+1,i+2,i+3}]
        +\mu_4(\rho_{i+2},\rho_{i+1},a,\rho_{i-1}) [\rho_{i-1}\cdot b \cdot\rho_{i+1,i+2}]\\       
        &\quad\qquad+(-1)^{|b|}\mu_4(\rho_{i+1},a,\rho_{i-1},\rho_{i-2}) [\rho_{i-2,i-1}\cdot b \cdot\rho_{i+1}]
        +\mu_4(a,\rho_{i-1},\rho_{i-2},\rho_{i-3},)  [\rho_{i-3,i-2,i-1}\cdot b ].
\end{align*}
For example, Equation~\eqref{eq:diff3} is replaced by
\[
  \partial(\rho_1[\rho_1])= -U[\rho_{1234}]+U[\rho_{2341}]-U[\rho_{3412}]+ U[\rho_{4123}].
\]

Lemma~\ref{lem:C-is-cx} has the following analogue:

\begin{lemma}
This differential makes $\mathfrak{C}^*_{\ZZ}$ into a chain complex.
\end{lemma}

\begin{proof}
  The proof follows along the lines of Lemma~\ref{lem:C-is-cx}.
  Again, we break up the differential into its components $\partial_i$
  for $i=1,2,4$, according to which $\mu_i$ action contributes.  It is
  immediate that $\partial_1^2=0$. The identity $\partial_1\circ
  \partial_2+\partial_2\circ \partial_1$ follows quickly from the fact
  that $|h|\equiv1\pmod{2}$, 
  as does $\partial_1 \partial_4+\partial_4\partial_1=0$.

  Again, verifying $\partial_2\partial_4+\partial_4\partial_2=0$ is a case check.
  For example, as in the proof of Lemma~\ref{lem:C-is-cx}, only now
  keeping track of signs, we find that
  \[ \partial_2\partial_4(\rho_1[b])+\partial_4\partial_2(\rho_1[b])=
  U\rho_1[(\rho_{1234}+\rho_{3412})\cdot b]-U\rho_1[b\cdot(\rho_{4123}+\rho_{2341})]=0.\]

  The cancellation in $\partial_4^2=0$ is also straightforward,
  and is left to the reader.
\end{proof}

Specializing gradings, we have
$\mathfrak{C}^{W,\ell}_{\Gamma;\ZZ}\subset \mathfrak{C}^*_{\ZZ}$
(analogous to $\mathfrak{C}^{W,\ell}_\Gamma\subset \mathfrak{C}^*$;
see Equation~\eqref{eq:SpecialGradings}).  We now have the following
analogue of Proposition~\ref{prop:SmallerModel2}:

\begin{proposition}
  \label{prop:SmallerModel2Z}
  The chain complex $\mathfrak{C}^*_{\Gamma;\ZZ}$ is quasi-isomorphic to 
  the complex $\HC^*_\Gamma(\UnDefAlgZ)$; in particular
  $H^{w,k}(\mathfrak{C}_{\Gamma;\ZZ})\cong \HH^{w,k}_\Gamma(\UnDefAlgZ)$.
\end{proposition}

\begin{proof}
  The proof of Proposition~\ref{prop:SmallerModel2} applies with minor changes.
\end{proof}
  
Proposition~\ref{prop:UnDefHoch} now has the following analogue:
\begin{proposition}\label{prop:UnDefHochZ} The Hochschild cohomology
  groups $\HH_{\Gamma}^{w,k}(\UnDefAlgZ)$, $w>0$, have
  \begin{align*}
  \HH^{w,-1}_\Gamma(\UnDefAlgZ)&=
  \begin{cases}
   \ZZ^{2} & w=1\\
    0 &\text{otherwise}
  \end{cases}
  \end{align*}
  and $\HH^{w,-2}_\Gamma(\UnDefAlgZ)$ is entirely supported in weight
  ($w$) grading $1$.
  Moreover, one can choose a basis for $\HH^{1,-1}_\Gamma(\UnDefAlgZ)$ so
  that one basis element sends $1\in\Ground$ to
  $\rho_{1234}+\rho_{2341}+\rho_{3412}+\rho_{4123}$ and the other
  sends $1\in\Ground$ to $U=U(\iota_0+\iota_1)$.
\end{proposition}

\begin{proof}
  This follows as in the proof of Proposition~\ref{prop:UnDefHoch}.
  Again, the generators of the homology are
  $\rho_{1234}[\iota_1]+\rho_{2341}[\iota_0]+\rho_{3412}[\iota_1]+\rho_{4123}[\iota_0]$ and
  $U\iota_0[\iota_1]+U\iota_1[\iota_0]$.
\end{proof}

\begin{proof}[Proof of Theorem~\ref{thm:MAlg-uniqueZ}]
  As in the case of Theorem~\ref{thm:MAlg-unique}, this follows from
  the above Hochschild computation (Proposition~\ref{prop:UnDefHochZ})
  and deformation theory (Proposition~\ref{prop:deform-Ainf-algZ}).
\end{proof}

\subsection{Application to the Fukaya category}
Sheridan showed that the anchored Fukaya category can be defined with
$\ZZ$, instead of $\FF_2$, coefficients~\cite{Sheridan:CY-hyp}. In
particular, his construction makes $\End_{\WFuk_z}(\alpha_1\oplus\alpha_2;\ZZ)$
into an $\Ainf$-algebra over $\ZZ[U]$. For his, or any other, way of
lifting the wrapped Fukaya category of the torus to
$\ZZ$-coefficients, we have:

\begin{theorem}\label{thm:UnDef-is-Fuk-Z}
  There is an $\Ainf$-quasi-isomorphism
  $\UnDefAlgZ\simeq \End_{\WFuk_z}(\alpha_1\oplus\alpha_2;\ZZ).$
\end{theorem}
\begin{proof}
  This follows from Theorem~\ref{thm:UnDefAlg-unique-Z} and the
  observation that for any way of assigning signs in the
  multiplication on
  $\AsUnDefAlgZ=(\End_{\WFuk_z}(\alpha_1\oplus\alpha_2),\mu_2)$ there
  is an isomorphism to $\AsUnDefAlgZ$, and similarly if one negates
  the equations in Theorem~\ref{thm:UnDefAlg-unique-Z} (\ref{item:Uuniqe-mu-4-Z}) above one obtains an
  isomorphic $\Ainf$-algebra (by sending $U$ to $-U$). (Note that the
  $\Ainf$-relation with inputs $(\rho_4,\rho_3,\rho_2,\rho_1,\rho_4)$
  ensures that both equations in
  Theorem~\ref{thm:UnDefAlg-unique-Z} (\ref{item:Uuniqe-mu-4-Z}) have
  the same sign.)
\end{proof}
